\documentclass[12pt]{article}
\usepackage{indentfirst}
\usepackage[top=2.4cm,bottom=2cm,left=2cm,right=2cm,%
headheight=0.25cm,headsep=0.55cm,footskip=0.7cm]{geometry}

\usepackage{fancyhdr}
\fancypagestyle{plain}{
  \fancyhf{}
  
  \cfoot{\small \thepage}
}
\usepackage{titlesec}
\titleformat{\section}{\normalfont\bfseries\large}{Section~\thesection}{1em}{}
\titleformat{\subsection}{\normalfont\bfseries}{\thesubsection}{0.8em}{}

\renewcommand\appendix{%
\par
\setcounter{section}{0}
\setcounter{subsection}{0}
\gdef\thesection{\Alph{section}}
\titleformat{\section}{\normalfont\bfseries\large}{Appendix~\thesection}{1em}{}}

\usepackage[hyperfootnotes=false]{hyperref}
\hypersetup{
  pdftitle={$\mathrm{C}^2$ estimates for general $p$-Hessian equations on closed Riemannian manifolds},
  pdfauthor={Yuxiang Qiao},
  bookmarks=true,
  bookmarksnumbered=true,
  bookmarksopen=true,
  colorlinks=true,
  allcolors=blue
}
\usepackage{bookmark}

\newcommand\secref[1]{Section~\ref{#1}}
\newcommand\subsecref[1]{Subsection~\ref{#1}}

\newcommand\thmref[1]{Theorem~\ref{#1}}
\newcommand\lemref[1]{Lemma~\ref{#1}}
\newcommand\propref[1]{Proposition~\ref{#1}}
\newcommand\corref[1]{Corollary~\ref{#1}}
\newcommand\defnref[1]{Definition~\ref{#1}}
\newcommand\rmkref[1]{Remark~\ref{#1}}

\usepackage{setspace}

\usepackage[perpage,bottom,hang]{footmisc}
\usepackage[square,numbers]{natbib}
\providecommand\bibcommenthead{}
\usepackage{amsmath}
\usepackage{amssymb}
\usepackage{bm}
\usepackage{mathrsfs}
\usepackage{upgreek}

\newcommand\comma{\text{, }}
\numberwithin{equation}{section}
\usepackage{mathtools}

\usepackage[amsmath,hyperref,thmmarks]{ntheorem}
\newlength\mylen
\makeatletter
\newtheoremstyle{indent}%
  {\item[\hskip\mylen \theorem@headerfont ##1 ##2\theorem@separator]}%
  {\item[\hskip\mylen \theorem@headerfont ##1 ##2 (##3)\theorem@separator]} %
\newtheoremstyle{nonumberindent}%
  {\item[\hskip\mylen \theorem@headerfont ##1\theorem@separator]}%
  {\item[\hskip\mylen \theorem@headerfont ##1（##3）\theorem@separator]} %
\makeatother
\theoremstyle{indent}
\theorembodyfont{\normalfont}
\theorempreskip{4pt}
\theorempostskip{4pt}
\newtheorem{thm}{Theorem}[section]
\newtheorem{lem}[thm]{Lemma}
\newtheorem{prop}[thm]{Proposition}
\newtheorem{cor}[thm]{Corollary}
\newtheorem{defn}[thm]{Definition}
\newtheorem{rmk}[thm]{Remark}
{
  \theorembodyfont{\normalfont}
  \theoremsymbol{\mbox{$\blacksquare$}}
  \newtheorem*{proof}{Proof}
  
}
\qedsymbol{\mbox{$\blacksquare$}}

\usepackage{environ}
\NewEnviron{eq}{\setstretch{1} \begin{equation} \BODY \end{equation}}
\NewEnviron{eq*}{\setstretch{1} \begin{equation*} \BODY \end{equation*}}\NewEnviron{ga}{\setstretch{1} \begin{gather} \BODY \end{gather}}
\NewEnviron{ga*}{\setstretch{1} \begin{gather*} \BODY \end{gather*}}
\NewEnviron{al}{\setstretch{1} \begin{align} \BODY \end{align}}
\NewEnviron{al*}{\setstretch{1} \begin{align*} \BODY \end{align*}}

\usepackage{enumitem}
\newlength{\mylenspace}
\setlist{nosep} 
\setlist[enumerate,1]{left=0pt,label=(\arabic*),labelsep=\mylenspace,ref=\arabic*}

\DeclareMathOperator\dif{d\!}
\DeclareMathOperator*\osc{osc}

\DeclareMathOperator*\essinf{ess\,inf}

\DeclareMathOperator\diag{diag}
\DeclareMathOperator\tr{tr}

\begin{document}

\allowdisplaybreaks
\setlength{\baselineskip}{16pt}
\setlength{\lineskiplimit}{4pt}
\setlength{\lineskip}{4pt}
\setlength{\emergencystretch}{3em}
\setlength{\parskip}{0pt}
\setlength{\abovedisplayskip}{0pt}
\setlength{\belowdisplayskip}{0pt}
\setlength{\abovedisplayshortskip}{0pt}
\setlength{\belowdisplayshortskip}{0pt}

\thispagestyle{plain}
{
  \centering
  \setstretch{1}
  {\bfseries\Large $\mathrm{C}^2$ estimates for general $p$-Hessian equations on closed Riemannian manifolds}

  \vspace{12pt}
  {
  \large Yuxiang Qiao\footnote{School of Mathematical Sciences, Peking University, Beijing, China \\
  Corresponding author: Yuxiang Qiao, E-mail: qiaoyx@stu.pku.edu.cn}
  }
  
}

\vspace{12pt}
\begin{abstract}
\small
\setlength{\baselineskip}{13pt}
We study the $\mathrm{C}^2$ estimates for $p$-Hessian equations with general left-hand and right-hand terms on closed Riemannian manifolds of dimension $n$. To overcome the constraints of closed manifolds, we advance a new kind of ``subsolution'', called pseudo-solution, which generalizes ``$\mathcal{C}$-subsolution'' to some extent and is well-defined for fully general $p$-Hessian equations. Based on pseudo-solutions, we prove the $\mathrm{C}^1$ estimates for general $p$-Hessian equations, and the corresponding second-order estimates when $p\in\{2\comma n-1\comma n\}$, under sharp conditions --- we don't impose curvature restrictions, convexity conditions or ``MTW condition'' on our main results. Some other conclusions related to a~priori estimates and different kinds of ``subsolutions'' are also given, including estimates for ``semi-convex'' solutions and when there exists a pseudo-solution.
\end{abstract}


{}


\section{Introduction} \label{sec: Introduction}

Hessian equations appear naturally in classical differential geometry, Riemannian geometry, conformal geometry, convex geometry, calibrated geometry and complex geometry. The classical $p$-Hessian equation in $\mathbb{R}^n$ has the following form:
\begin{eq} \label{eq: classcial p-Hessian}
\sigma_p^{\frac1p}\Bigl(\bm\lambda\bigl(\mathrm{D}^2u(\bm{x})\bigr)\Bigr)=\varphi(\bm{x})\comma\forall\bm{x}\in\Omega\comma
\end{eq}
where $\sigma_p$ denotes the $p$-th elementary symmetric $n$-polynomial (cf. \eqref{eq: sigma_p}), $\bm\lambda$ denotes the eigenvalue map (cf. \defnref{defn: lambda}), $\mathrm{D}^2u$ denotes the Hessian matrix of $u$, $\Omega$ is a bounded domain in $\mathbb{R}^n$ with $\partial\Omega$ smooth, and $\varphi$ is a positive smooth function in $\overline{\Omega}$. When $p=1$, \eqref{eq: classcial p-Hessian} is the Possion equation; when $p=n$, \eqref{eq: classcial p-Hessian} is the Monge-Amp\`ere equation. For $p>1$, \eqref{eq: classcial p-Hessian} as a second-order fully non-linear PDE is not always elliptic, but it is elliptic with respect to every $p$-admissible solution --- a function $u\in\mathrm{C}^2(\overline{\Omega})$ is called $p$-admissible if and only if
\begin{eq}
\bm\lambda\bigl(\mathrm{D}^2u(\bm{x})\bigr)\in\Gamma_p\comma\forall\bm{x}\in\overline{\Omega}\comma
\end{eq}
where $\Gamma_p$ denotes the $p$-th G\r{a}rding cone with respect to $\mathbb{R}^n$ (cf. \defnref{defn: Garding cone}). We are mainly interested in the $p$-admissible solutions of the $p$-Hessian equation which in turn can be viewed as a fully non-linear elliptic equation when $p>1$. \citeauthor{Caffarelli1985a} \cite{Caffarelli1985a} solved the following Dirichlet problem of \eqref{eq: classcial p-Hessian}:
\begin{eq} \label{eq: classcial p-Hessian, Dirichlet}
\left\{ \begin{gathered}
\sigma_p^{\frac1p}\Bigl(\bm\lambda\bigl(\mathrm{D}^2u(\bm{x})\bigr)\Bigr)=\varphi(\bm{x})\comma\forall\bm{x}\in\Omega \\
\bm\lambda\bigl(\mathrm{D}^2u(\bm{x})\bigr)\in\Gamma_p\comma\forall\bm{x}\in\overline{\Omega} \\
u(\bm{x})=\psi(\bm{x})\comma\forall\bm{x}\in\partial\Omega
\end{gathered} \right. \comma
\end{eq}
where $\psi$ was a smooth function on $\partial\Omega$, and $\Omega$ was required to be strictly $(p-1)$-convex, namely the vector composed of the $n-1$ principal curvatures of $\partial\Omega$ at $\bm{x}$ with respect to the inward normal was required to be in $\Gamma_{p-1}^{(n-1)}$ for any $\bm{x}\in\partial\Omega$. As is illustrated below, it's natural to consider the following Dirichlet problem of the $p$-Hessian equation
\begin{eq} \label{eq: p-Hessian, Dirichlet}
\left\{ \begin{gathered}
\sigma_p^{\frac1p}\Bigl(\bm\lambda\bigl(\mathrm{D}^2u(\bm{x})\bigr)\Bigr)=\varphi\bigl(\bm{x}\comma\mathrm{D}u(\bm{x})\comma u(\bm{x})\bigr)\comma\forall\bm{x}\in\Omega \\
\bm\lambda\bigl(\mathrm{D}^2u(\bm{x})\bigr)\in\Gamma_p\comma\forall\bm{x}\in\overline{\Omega} \\
u(\bm{x})=\psi(\bm{x})\comma\forall\bm{x}\in\partial\Omega
\end{gathered} \right.
\end{eq}
assuming that there exists a $p$-admissible subsolution $\underline{u}\in\mathrm{C}^2(\overline{\Omega})$ to \eqref{eq: p-Hessian, Dirichlet} which by definition satisfies
\begin{eq} \label{eq: p-Hessian, Dirichlet, subsolution}
\left\{ \begin{gathered}
\sigma_p^{\frac1p}\Bigl(\bm\lambda\bigl(\mathrm{D}^2\underline{u}(\bm{x})\bigr)\Bigr)\geqslant\varphi\bigl(\bm{x}\comma\mathrm{D}\underline{u}(\bm{x})\comma\underline{u}(\bm{x})\bigr)\comma\forall\bm{x}\in\Omega \\
\bm\lambda\bigl(\mathrm{D}^2\underline{u}(\bm{x})\bigr)\in\Gamma_p\comma\forall\bm{x}\in\overline{\Omega} \\
\underline{u}(\bm{x})=\psi(\bm{x})\comma\forall\bm{x}\in\partial\Omega
\end{gathered} \right. \comma
\end{eq}
where $\mathrm{D}u$ denotes the gradient of $u$, $\Omega$ is not required to be strictly $(p-1)$-convex or satisfy any other non-trivial conditions, and $\varphi(\bm{x}\comma\bm{v}\comma t)$ is a positive smooth function in $\overline{\Omega}\times\mathbb{R}^n\times\mathbb{R}$. We find that there does exist a $p$-admissible subsolution $\underline{u}\in\mathrm{C}^\infty(\overline{\Omega})$ to \eqref{eq: p-Hessian, Dirichlet} when we additionally assume that $\Omega$ is strictly $(p-1)$-convex and
\begin{eq} \label{eq: varphi(x, v, t) leqslant tilde{varphi}(x, t)(1+|v|+|t|^alpha)}
\varphi(\bm{x}\comma\bm{v}\comma t)\leqslant\tilde{\varphi}(\bm{x}\comma t)\left(1+\left|\bm{v}\right|+\left|t\right|^\alpha\right)\comma\forall(\bm{x}\comma\bm{v}\comma t)\in\overline{\Omega}\times\mathbb{R}^n\times\mathbb{R}\comma
\end{eq}
where $\alpha\in(0\comma1)$ and $\tilde{\varphi}(\bm{x}\comma t)$ is a positive continuous function in $\overline{\Omega}\times\mathbb{R}$ non-decreasing with respect to $t$ (cf. \propref{prop: existence of a classical p-admissible subsolution}). \citeauthor{Guan1994} \cite{Guan1994,Guan1999} generalized the result of \citeauthor{Caffarelli1985a} and solved \eqref{eq: p-Hessian, Dirichlet} for the case when $\varphi(\bm{x}\comma\bm{v}\comma t)$ was convex with respect to $\bm{v}$ and satisfied several technical conditions. A typical example of the $p$-Hessian equation in \eqref{eq: p-Hessian, Dirichlet} is the prescribed Gauss curvature equation for graphic hypersurfaces in $\mathbb{R}^{n+1}$:
\begin{eq} \label{eq: prescribed Gauss curvature}
\Bigl(\det\bigl(\mathrm{D}^2u(\bm{x})\bigr)\Bigr)^{\frac1n}=\Bigl(K\bigl(\bm{x}\comma u(\bm{x})\bigr)\Bigr)^{\frac1n}\bigl(1+\left|\mathrm{D}u(\bm{x})\right|^2\bigr)^{\frac{n+2}{2n}}\comma\forall\bm{x}\in\Omega\comma
\end{eq}
where $\Omega$ is a bounded (not necessarily convex) domain in $\mathbb{R}^n$ with $\partial\Omega$ smooth, $K(\bm{x}\comma t)$ is a positive smooth function in $\overline{\Omega}\times\mathbb{R}$, and $K\bigl(\bm{x}\comma u(\bm{x})\bigr)$ corresponds to the Gauss curvature of the graph of some $u\in\mathrm{C}^\infty(\overline{\Omega})$. One can refer to e.g. \cite{Trudinger1983a,Guan1993,Guan1998} for results related to the prescribed Gauss curvature equation in the Euclidean space. Note that $\left(1+\left|\bm{v}\right|^2\right)^{\frac{n+2}{2n}}$ is convex. On the other hand, $\left(1-\left|\bm{v}\right|^2\right)^{\frac{n+2}{2n}}$ ($\left|\bm{v}\right|<1$) is concave, which is related to the prescribed Gauss-Kronecker curvature equation for spacelike graphic hypersurfaces in the Minkowski space $\mathbb{R}^{n\comma1}$:
\begin{eq} \label{eq: prescribed Gauss-Kronecker curvature, Minkowski space}
\left\{ \begin{gathered}
\Bigl(\det\bigl(\mathrm{D}^2u(\bm{x})\bigr)\Bigr)^{\frac1n}=\Bigl(K\bigl(\bm{x}\comma u(\bm{x})\bigr)\Bigr)^{\frac1n}\bigl(1-\left|\mathrm{D}u(\bm{x})\right|^2\bigr)^{\frac{n+2}{2n}}\comma\forall\bm{x}\in\Omega \\
\left|\mathrm{D}u(\bm{x})\right|<1\comma\forall\bm{x}\in\overline{\Omega}
\end{gathered} \right. \comma
\end{eq}
refer to e.g. \cite{Delanoe1990,Guan1998}. \citeauthor{Guan1998} \cite{Guan1998} solved \eqref{eq: p-Hessian, Dirichlet} for $p=n$ and general positive smooth function $\varphi(\bm{x}\comma\bm{v}\comma t)$, not necessarily convex with respect to $\bm{v}$.

It's a long-standing question how to (or whether one can) solve \eqref{eq: p-Hessian, Dirichlet} when $1<p<n$ and $\varphi(\bm{x}\comma\bm{v}\comma t)$ is not convex with respect to $\bm{v}$. By the classical Schauder theory and continuity method (cf. e.g. \cite{Gilbarg2001}), the degree theory in non-linear functional analysis (cf. e.g. \cite[p.~67]{Guan1999} for its application to \eqref{eq: p-Hessian, Dirichlet} and \cite{Schwartz1969} for the general theory) and the Evans-Krylov estimate \cite{Evans1982,Krylov1984}, it's well-known that the existence of a $p$-admissible solution $u\in\mathrm{C}^\infty(\overline{\Omega})$ to \eqref{eq: p-Hessian, Dirichlet} can be reduced to the global $\mathrm{C}^2$ estimate for \eqref{eq: p-Hessian, Dirichlet}. In the proof of the global $\mathrm{C}^2$ estimate for \eqref{eq: p-Hessian, Dirichlet} in \cite{Guan1999}, the only step that required $\varphi(\bm{x}\comma\bm{v}\comma t)$ to be convex with respect to $\bm{v}$ was reducing the global second-order estimate to the boundary second-order estimate. Thus, based on \cite{Guan1999}, the only difficulty of solving the long-standing question mentioned above is how to estimate $\sup\limits_{\overline{\Omega}} \left|\mathrm{D}^2u\right|$ by $\sup\limits_{\partial\Omega} \left|\mathrm{D}^2u\right|$, $\sup\limits_{\overline{\Omega}} \left|\mathrm{D}u\right|$, $\sup\limits_{\overline{\Omega}} \left|u\right|$, the subsolution $\underline{u}$ and other trivial quantities for any $p$-admissible solution $u\in\mathrm{C}^\infty\left(\overline{\Omega}\right)$ to \eqref{eq: p-Hessian, Dirichlet}. So far, this difficulty has only been resolved in some special cases: \citeauthor{Guan2015a} \cite{Guan2015a} solved \eqref{eq: p-Hessian, Dirichlet} for $p=2$, and reduced the global second-order estimate to the boundary second-order estimate for $(p+1)$-admissible solutions to \eqref{eq: p-Hessian, Dirichlet}, see also \cite{Li2016,Chu2021a,Zhang2025}; \citeauthor{Ren2019} \cite{Ren2019,Ren2023} managed to solve \eqref{eq: p-Hessian, Dirichlet} for $p=n-1$ and $p=n-2$, see also \cite{Chen2018,Lu2024,Tu2024} for $p=n-1$; how to prove the global second-order estimate for $p$-admissible solutions to \eqref{eq: p-Hessian, Dirichlet} is still unknown when $2<p<n-2$. The $p$-Hessian equation in \eqref{eq: p-Hessian, Dirichlet} is closely related to the general prescribed $p$-th Weingarten curvature equation for closed\footnote{A closed manifold is by definition a compact manifold without boundary.} hypersurfaces:
\begin{eq} \label{eq: prescribed p-th Weingarten curvature}
\sigma_p^{\frac1p}\bigl(\bm\kappa(\bm{x})\bigr)=\varphi\bigl(\bm{x}\comma\bm\nu(\bm{x})\bigr)\comma\forall\bm{x}\in\mathcal{M}\comma
\end{eq}
where $\mathcal{M}$ is a closed hypersurface in $\mathbb{R}^{n+1}$, $\bm\kappa(\bm{x})$ denotes the vector composed of the $n$ principal curvatures of $\mathcal{M}$ at $\bm{x}$ with respect to the inward normal, $\bm\nu$ denotes the outward (or inward) unit normal vector field on $\mathcal{M}$, and $\varphi$ is a positive smooth function on the unit normal bundle of $\mathcal{M}$. When $p=1\comma2$ and $n$, the $p$-th Weingarten curvature $\sigma_p\bigl(\bm\kappa(\bm{x})\bigr)$ corresponds to the mean curvature, scalar curvature and Gauss curvature respectively. ``$\bm\nu(\bm{x})$'' in \eqref{eq: prescribed p-th Weingarten curvature} can be viewed as some kind of gradient term. Similar to the global second-order estimate for \eqref{eq: p-Hessian, Dirichlet}, how to prove the curvature estimate for \eqref{eq: prescribed p-th Weingarten curvature} is still unknown when $2<p<n-2$. One can refer to e.g. \cite{Caffarelli1988,Ivochkina1990,Guan2002a,Guan2012a,%
Guan2015a,Ren2019,Ren2023} for results about the prescribed $p$-th Weingarten curvature equation.

Let $(\mathcal{M}\comma\bm g)$ be a closed connected Riemannian manifold of dimension $n$, $\mathrm{T}^*\mathcal{M}$ (resp. $\mathrm{T}^*\mathcal{M}\mathrm{T}^*\mathcal{M}$) denote the cotangent bundle (resp. symmetric $(0\comma2)$-tensor bundle) on $\mathcal{M}$ and $\pi$ (resp. $\tilde{\pi}$) denote the canonical projection of $\mathrm{T}^*\mathcal{M}$ (resp. $\mathrm{T}^*\mathcal{M}\mathrm{T}^*\mathcal{M}$) onto $\mathcal{M}$. In this paper, we study the $\mathrm{C}^2$ estimate for the following general $p$-Hessian equation\footnote{In literature, general $p$-Hessian equations are also called ``augmented Hessian equations''.} on $(\mathcal{M}\comma\bm g)$:
\begin{eq} \label{eq: general p-Hessian}
\left\{ \begin{gathered}
\sigma_p^{\frac{1}{p}}\biggl(\bm\lambda\Bigl({\bm g}^{-1}\bigl(\bm{A}(\dif u\comma u)+\nabla^2u\bigr)\Bigr)(\bm{x})\biggr)=\varphi\bigl(\dif u(\bm{x})\comma u(\bm{x})\bigr)\comma\forall\bm{x}\in\mathcal{M} \\
\bm\lambda\Bigl({\bm g}^{-1}\bigl(\bm{A}(\dif u\comma u)+\nabla^2u\bigr)\Bigr)(\bm{x})\in\Gamma_p\comma\forall\bm{x}\in\mathcal{M}
\end{gathered} \right. \comma
\end{eq}
where $\dif u$ and $\nabla^2u$ denote the differential and the Hessian of $u$ respectively (cf. \eqref{eq: du and nabla^2u}),
\begin{eq} \label{eq: A(alpha, t)}
\bm{A}(\bm{\alpha}\comma t)\in\mathrm{C}^\infty\left(\mathrm{T}^*\mathcal{M}\times\mathbb{R}\comma\mathrm{T}^*\mathcal{M}\mathrm{T}^*\mathcal{M}\right)\colon\tilde{\pi}\bigl(\bm{A}(\bm{\alpha}\comma t)\bigr)=\pi(\bm{\alpha})\comma
\end{eq}
$\bm{A}(\dif u\comma u)$ is defined in \eqref{eq: A(du, u)} (see also \eqref{eq: A(v, t)} and \eqref{eq: A_{jk}(v, t)}),
\begin{eq}
\bm\lambda\Bigl({\bm g}^{-1}\bigl(\bm{A}(\dif u\comma u)+\nabla^2u\bigr)\Bigr)
\end{eq}
denotes the $\mathbb{R}^n$-valued function \eqref{eq: lambda(g^{-1}(A(du, u)+nabla^2u))} well-defined on $\mathcal M$, and $\varphi(\bm{\alpha}\comma t)$ is a positive smooth function on $\mathrm{T}^*\mathcal{M}\times\mathbb{R}$. Many equations from various geometric problems are special cases of \eqref{eq: general p-Hessian}, some of them listed below:

\begin{enumerate}[label=(\Alph*),ref=\Alph*]
\item \label{item: example, p-Christoffel-Minkowski} Let $(\mathcal{M}\comma\bm{g})$ be the unit $n$-sphere $\mathbb{S}^n$ equipped with the canonical metric $\bm{g}_{\mathbb{S}^n}$ and
\begin{eq}
\text{$\bm{A}(\bm{\alpha}\comma t)=t\bm{g}_{\mathbb{S}^n}\bigl(\pi(\bm{\alpha})\bigr)$, i.e. $\bm{A}(\dif u\comma u)=u\bm{g}_{\mathbb{S}^n}$.}
\end{eq}
When $\varphi(\bm{\alpha}\comma t)$ depends only on $\pi(\bm{\alpha})$, \eqref{eq: general p-Hessian} corresponds to the $p$-Christoffel-Minkowski problem (cf. e.g. \cite{Guan2003,Guan2006}) and in particular the classical Minkowski problem (cf. e.g. \cite{Cheng1976}) if $p=n$. On the other hand, when $p=n$ and
\begin{eq}
\varphi(\bm{\alpha}\comma t)=\phi\bigl(\pi(\bm{\alpha})\bigr)t^{-\frac1n}\left(t^2+\left|\bm{\alpha}\right|_{\bm g}^2\right)^{\frac{n+1}{2n}}\ (t\geqslant\varepsilon)
\end{eq}
for some $\varepsilon>0$ and positive smooth function $\phi$ on $\mathcal{M}$ (cf. \eqref{eq: moduli of alpha and B} for the definition of $\left|\bm{\alpha}\right|_{\bm g}$), \eqref{eq: general p-Hessian} corresponds to the Alexandrov problem of prescribing Gauss curvature measure (cf. e.g. \cite[p.~793]{Guan1997}).
\item \label{item: example, geometric optics} Let $(\mathcal{M}\comma\bm{g})$ be the unit $n$-sphere $\mathbb{S}^n$ equipped with the canonical metric $\bm{g}_{\mathbb{S}^n}$. When $p=n$,
\begin{al}
\bm{A}(\bm{\alpha}\comma t)&=\frac{t^2-\left|\bm{\alpha}\right|_{\bm g}^2}{2t}\bm{g}_{\mathbb{S}^n}\bigl(\pi(\bm{\alpha})\bigr)\ (t\geqslant\varepsilon)\comma \\
\varphi(\bm{\alpha}\comma t)&=\phi\bigl(\pi(\bm{\alpha})\bigr)\frac{t^2+\left|\bm{\alpha}\right|_{\bm g}^2}{2t}\ (t\geqslant\varepsilon)
\end{al}
for some $\varepsilon>0$ and positive smooth function $\phi$ on $\mathcal{M}$, \eqref{eq: general p-Hessian} corresponds to a problem in geometric optics (cf. e.g. \cite[p.~209]{Guan1998b}).
\item \label{item: example, p-Yamabe} Let $(\mathcal{M}\comma\bm g)$ be any closed Riemannian manifold of dimension $n\geqslant3$ and $\bm{Ric}_{\bm g}$ denote the Ricci tensor field of $\bm{g}$. The Schouten tensor field of $\bm{g}$ is by definition (cf. e.g. \cite[p.~1413]{Guan2003b})
\begin{eq}
\bm{S}_{\bm g}\triangleq\frac{1}{n-2}\left(\bm{Ric}_{\bm g}-\frac{\tr_{\bm g}(\bm{Ric}_{\bm g})}{2(n-1)}\bm{g}\right)\comma
\end{eq}
where $\tr_{\bm g}$ denotes the trace operator with respect to $\bm{g}$. Let $\bm{vw}$ denote the symmetric tensor product of $\bm{\alpha}$, $\bm{\beta}\in\mathrm{T}^*\mathcal{M}\colon\pi(\bm{\alpha})=\pi(\bm{\beta})$, namely
\begin{eq}
\bm{\alpha\beta}\triangleq\tfrac12(\bm{\alpha}\otimes\bm{\beta}+\bm{\beta}\otimes\bm{\alpha}).
\end{eq}
When
\begin{ga}
\bm{A}(\bm{\alpha}\comma t)=\bm{S}_{\bm g}\bigl(\pi(\bm{\alpha})\bigr)-\frac{\left|\bm{\alpha}\right|_{\bm g}^2}{2}\bm{g}\bigl(\pi(\bm{\alpha})\bigr)+\bm{\alpha}\bm{\alpha}\comma \\
\varphi(\bm{\alpha}\comma t)=\phi\bigl(\pi(\bm{\alpha})\bigr)\mathrm{e}^{-2t}
\end{ga}
for some positive smooth function $\phi$ on $\mathcal{M}$, \eqref{eq: general p-Hessian} corresponds to the $p$-Yamabe problem in conformal geometry (cf. e.g. \cite{Viaclovsky2000}, \cite[p.~1414]{Guan2003b}). Next, fix $\psi\in\mathrm{C}^\infty(\mathcal{M})$. The general Schouten tensor field of $\bm{g}$ with respect to $\psi$ is defined as
\begin{eq}
\bm{S}_{\bm g\comma\psi}\triangleq\frac{1}{n-2}\left(\bm{Ric}_{\bm{g}\comma\psi}-\frac{\tr_{\bm g}(\bm{Ric}_{\bm{g}\comma\psi})}{2(n-1)}\bm{g}\right)\comma
\end{eq}
where $\bm{Ric}_{\bm{g}\comma\psi}$ denotes the Bakry-\'Emery Ricci tensor field (cf. e.g. \cite[p.~377]{Wei2009}) of the smooth metric measure space $\left(\mathcal{M}\comma\bm g\comma\mathrm{e}^{-\psi}\dif V_{\bm g}\right)$ (cf. \eqref{eq: dV_g} for the definition of $\dif V_{\bm g}$), namely
\begin{eq}
\bm{Ric}_{\bm{g}\comma\psi}\triangleq\bm{Ric}_{\bm g}+\nabla^2\psi.
\end{eq}
When
\begin{eq} \label{eq: general p-Yamabe, A(alpha, t)} \begin{aligned}
\bm{A}(\bm{\alpha}\comma t)&=\bm{S}_{\bm g\comma\psi}\bigl(\pi(\bm{\alpha})\bigr)-\frac{\left|\bm{\alpha}\right|_{\bm g}^2}{2}\bm{g}\bigl(\pi(\bm{\alpha})\bigr)+\bm{\alpha}\bm{\alpha} \\
&\quad+\frac{1}{n-2}\left(2\dif\psi\bigl(\pi(\bm{\alpha})\bigr)\bm{\alpha}-\frac{n}{2(n-1)}\tr_{\bm g}\Bigl(\dif\psi\bigl(\pi(\bm{\alpha})\bigr)\bm{\alpha}\Bigr)\bm{g}\bigl(\pi(\bm{\alpha})\bigr)\right)
\end{aligned} \end{eq}
and $\varphi(\bm{\alpha}\comma t)=\phi\bigl(\pi(\bm{\alpha})\bigr)\mathrm{e}^{-2t}$ for some positive smooth function $\phi$ on $\mathcal{M}$, \eqref{eq: general p-Hessian} corresponds to the general $p$-Yamabe problem with respect to $\bm{S}_{\bm g\comma\psi}$.
\end{enumerate}
These specific examples motivate us to study the general $p$-Hessian equation \eqref{eq: general p-Hessian} on a closed Riemannian manifold. It's noteworthy that the complex version of \eqref{eq: general p-Hessian} on a closed Hermitian manifold is also important, of which the Fu-Yau equation (cf. e.g. \cite{Fu2008}, \cite[p.~74]{Chu2019}) from the superstring theory is a special case.

In \cite{Li1990}, \citeauthor{Li1990} solved \eqref{eq: general p-Hessian} on any closed connected Riemannian manifold of dimension $n$ with non-negative sectional curvature, for the case when $\bm{A}(\bm{\alpha}\comma t)=\bm g\bigl(\pi(\bm{\alpha})\bigr)$, $\varphi(\bm{\alpha}\comma t)=\phi\bigl(\pi(\bm{\alpha})\comma t\bigr)$ for some positive smooth function $\phi(\bm{x}\comma t)$ on $\mathcal{M}\times\mathbb{R}$ and there existed two constants $a$, $A\colon a<A$ so that
\begin{eq} \label{eq: a is subsolution, A is supersolution}
\max_{\bm{x}\in\mathcal{M}} \phi(\bm{x}\comma a)\leqslant(\mathrm{C}_n^p)^{\frac1p}\leqslant\min_{\bm{x}\in\mathcal{M}} \phi(\bm{x}\comma A).
\end{eq}
\citeauthor{Urbas2002} \cite{Urbas2002} generalized the result of \citeauthor{Li1990} by removing the curvature restriction. Note that \eqref{eq: a is subsolution, A is supersolution} implies either $\phi(\bm{x}\comma t)\equiv(\mathrm{C}_n^p)^{\frac1p}$ or $\phi(\bm{x}\comma t)$ does depend on $t$.

To solve the Dirichlet problem (e.g. \eqref{eq: p-Hessian, Dirichlet}) of the $p$-Hessian equation in a general bounded domain with smooth boundary, or on a general compact Riemannian manifold with non-empty boundary (cf. e.g. \cite{Guan1999}), one has to impose some kind of extra assumption on the domain or manifold since, for example, the Monge-Amp\`ere equation in a non-convex domain has no strictly convex solutions with zero boundary values. The optimal assumption seems to be the existence of a $p$-admissible subsolution (cf. \eqref{eq: p-Hessian, Dirichlet, subsolution}) --- this assumption is necessary and almost sufficient (cf. e.g. \cite{Guan1999,Guan2014,Guan2016}) for the existence of a $p$-admissible solution, especially useful for the global $\mathrm{C}^2$ estimate, and not too difficult to verify in many specific situations, refer to e.g. \propref{prop: existence of a classical p-admissible subsolution}, \cite[pp.~50--52]{Guan1999}, \cite[pp.~207--209]{Guan2002}. See also \cite{Jiang2015,Jiang2020}.

Analogous to the case of Dirichlet problem, for the general $p$-Hessian equation \eqref{eq: general p-Hessian} on a closed connected Riemannian manifold, it's natural to assume the existence of some kind of ``subsolution'' --- not necessarily in the classical sense of \eqref{eq: p-Hessian, Dirichlet, subsolution}. In the following, by a ``classical subsolution'' $\underline{u}$ to \eqref{eq: general p-Hessian} we mean $\underline{u}\in\mathrm{C}^2(\mathcal{M})$ satisfies
\begin{eq} \label{eq: general p-Hessian, classical subsolution}
\left\{ \begin{gathered}
\sigma_p^{\frac{1}{p}}\biggl(\bm\lambda\Bigl({\bm g}^{-1}\bigl(\bm{A}(\dif\underline{u}\comma\underline{u})+\nabla^2\underline{u}\bigr)\Bigr)(\bm{x})\biggr)\geqslant\varphi\bigl(\dif\underline{u}(\bm{x})\comma\underline{u}(\bm{x})\bigr)\comma\forall\bm{x}\in\mathcal{M} \\
\bm\lambda\Bigl({\bm g}^{-1}\bigl(\bm{A}(\dif\underline{u}\comma\underline{u})+\nabla^2\underline{u}\bigr)\Bigr)(\bm{x})\in\Gamma_p\comma\forall\bm{x}\in\mathcal{M}
\end{gathered} \right. .
\end{eq}
Obviously \eqref{eq: a is subsolution, A is supersolution} implies that $\underline{u}(\bm{x})\equiv a$ is a ``classical subsolution'' in the setting of \cite{Li1990,Urbas2002}. \citeauthor{Guan2016} \cite{Guan2015,Guan2016} proved the second-order estimate for \eqref{eq: general p-Hessian} based on ``classical subsolution'' when there held for any $(\bm{x}\comma t)\in\mathcal{M}\times\mathbb{R}$ and $\bm\alpha\in\mathrm{T}_{\bm x}^*\mathcal{M}$:
\begin{al}
\Bigl(\mathrm{D}_{\bm\alpha}^2\bigl[\bigl(\bm{A}(\cdot\comma t)\bigr)(\bm{v}\comma\bm{v})\bigr]\Bigr)(\bm{\beta}\comma\bm{\beta})&\leqslant0\comma \forall(\bm\beta\comma\bm{v})\in\mathrm{T}_{\bm x}^*\mathcal{M}\times\mathrm{T}_{\bm x}\mathcal{M}; \label{eq: A_{vv}(alpha, t) is concave with respect to alpha} \\
\left.\frac{\partial\bigl(\bm{A}(\bm\alpha\comma s)\bigr)(\bm{v}\comma\bm{v})}{\partial s}\right|_{s=t}&\geqslant0\comma \forall\bm{v}\in\mathrm{T}_{\bm x}\mathcal{M}; \label{eq: A_{vv}(alpha, t) is non-decreasing with respect to t} \\
\bigl(\mathrm{D}_{\bm\alpha}^2[\varphi(\cdot\comma t)]\bigr)(\bm{\beta}\comma\bm{\beta})&\geqslant0\comma \forall\bm\beta\in\mathrm{T}_{\bm x}^*\mathcal{M}; \label{eq: varphi(alpha, t) is convex with respect to alpha} \\
\left.\frac{\partial\varphi(\bm\alpha\comma s)}{\partial s}\right|_{s=t}&\leqslant0. \label{eq: varphi(alpha, t) is non-increasing with repect to t}
\end{al}
Here $\mathrm{T}_{\bm x}^*\mathcal{M}$ denotes the cotangent space at $\bm{x}$ and $\mathrm{D}_{\bm\alpha}^2[\cdot]$ denotes the second-order Fr\'echet differentiation at $\bm{\alpha}\in(\mathrm{T}_{\bm x}^*\mathcal{M}\comma\left|\cdot\right|_{\bm g})$ --- thus, for any $\psi\in\mathrm{C}^2(\mathrm{T}_{\bm x}^*\mathcal{M})$, $\mathrm{D}_{\bm\alpha}^2[\psi]$ is a symmetric bilinear function on $\mathrm{T}_{\bm x}^*\mathcal{M}\times\mathrm{T}_{\bm x}^*\mathcal{M}$ (cf. \eqref{eq: second-order Frechet}). \citeauthor{Jiang2017} \cite{Jiang2017,Jiang2020} improved the above result of \citeauthor{Guan2016} by replacing the concavity condition \eqref{eq: A_{vv}(alpha, t) is concave with respect to alpha} with a weaker one:
\begin{eq} \label{eq: A(alpha, t), weak MTW condition}
\Bigl(\mathrm{D}_{\bm\alpha}^2\bigl[\bigl(\bm{A}(\cdot\comma t)\bigr)(\bm{v}\comma\bm{v})\bigr]\Bigr)(\bm{\beta}\comma\bm{\beta})\leqslant0\comma \forall(\bm\beta\comma\bm{v})\in\mathrm{T}_{\bm x}^*\mathcal{M}\times\mathrm{T}_{\bm x}\mathcal{M}\colon\bm\beta(\bm{v})=0.
\end{eq}
\eqref{eq: A(alpha, t), weak MTW condition} is called ``weak MTW condition''. It's well-known that the second-order estimate for \eqref{eq: general p-Hessian} or even the general Monge-Amp\`ere equation will not hold without any structural conditions imposed on $\mathbf{A}(\bm{\alpha}\comma t)$, refer to e.g. the Heinz-Lewy counterexample in \cite[pp.~104--105]{Schulz1990}. In fact, the regularity result in \cite[p.~256]{Loeper2009} corresponding to the optimal transport problem shows that \eqref{eq: A(alpha, t), weak MTW condition} is almost necessary for the existence of a solution to \eqref{eq: general p-Hessian}. Note that the equations in examples \eqref{item: example, p-Christoffel-Minkowski}, \eqref{item: example, geometric optics}, \eqref{item: example, p-Yamabe} listed on pages~\pageref{item: example, p-Christoffel-Minkowski}--\pageref{item: example, p-Yamabe} all satisfy \eqref{eq: A(alpha, t), weak MTW condition}.

Similar to the case of Dirichlet problem, it's an interesting question how to (or whether one can) prove the second-order estimate for \eqref{eq: general p-Hessian} without any structural conditions imposed on $\varphi(\bm{\alpha}\comma t)$ when $p>1$. For the general Monge-Amp\`ere equation, it's well-known that the second-order estimate holds for any positive smooth function $\varphi(\bm{\alpha}\comma t)$, refer to e.g. \cite{Jiang2014}. When $p=2$, $n-1$ or $n-2$, although we could not find such a reference in literature, it's natural to believe that the second-order estimate for \eqref{eq: general p-Hessian} also holds for general $\varphi(\bm{\alpha}\comma t)$ in view of the corresponding result for the case of Dirichlet problem with $\bm{A}(\bm\alpha\comma t)\equiv 0$ (cf. e.g. \cite{Guan2015a,Ren2019,Ren2023}). It's noteworthy that there is more difficulty in proving the second-order estimate for \eqref{eq: general p-Hessian} than the case of Dirichlet problem with $\bm{A}(\bm\alpha\comma t)\equiv 0$. For example, it's unsatisfactory to prove an estimate of $\max\limits_{\mathcal M} \left|\nabla^2u\right|_{\bm g}$ (cf. \eqref{eq: moduli of alpha and B} for the definition of $\left|\nabla^2u\right|_{\bm g}$) dependent on a ``classical subsolution'' to \eqref{eq: general p-Hessian} --- when $\bm{A}(\bm{\alpha}\comma t)$, $\varphi(\bm{\alpha}\comma t)$ are both independent of $t$, $\bm{A}(\bm{\alpha}\comma t)$ satisfies \eqref{eq: A_{vv}(alpha, t) is concave with respect to alpha} and the equation \eqref{eq: general p-Hessian} does have a solution, any ``classical subsolution'' to \eqref{eq: general p-Hessian} is nothing but a solution (cf. \propref{prop: ``classical subsolution'' is solution}). To overcome this problem, \citeauthor{Guan2014} \cite[p.~1493]{Guan2014} advanced a new kind of ``subsolution'', which we call ``weak subsolution'', to orthogonally-invariant\footnote{In this paper, by ``orthogonally-invariant'' we mean the second-order term of an elliptic equation can be expressed as ``$f\circ\bm\lambda\bigl({\bm g}^{-1}(\bm{A}(\dif u\comma u)+\nabla^2u)\bigr)$'', where $f$ is a symmetric function.} elliptic equations on closed Riemannian manifolds. Later, \citeauthor{Szekelyhidi2018} \cite[pp.~344,~374]{Szekelyhidi2018} advanced another kind of ``subsolution'' called $\mathcal{C}$-subsolution for the case of closed Hermitian or Riemannian manifolds, which generalized the ``cone condition'' advanced in \cite{Song2008,Fang2011} for a class of geometric flows on closed Kähler manifolds including the $J$-flow, see also \cite{Sun2016}. $\mathcal{C}$-subsolution coincides with ``weak subsolution'' for the case of $p$-Hessian equations when $p>1$, refer to e.g. \cite[p.~18]{Guan2024}. No matter what kind, the ``subsolutions'' are mainly used to construct good globally-defined functions in the case of second-order estimates for elliptic equations on closed manifolds, refer to e.g. \cite{Guan2014,Guan2015,Szekelyhidi2018,Feng2020,Guan2023a}.

When $\varphi(\bm{\alpha}\comma t)=\phi\bigl(\pi(\bm{\alpha})\bigr)$, a $\mathcal{C}$-subsolution $\underline{u}\in\mathrm{C}^2(\mathcal{M})$ to the equation \eqref{eq: general p-Hessian} satisfies by definition that there exist positive constants $\delta$ and $R$, depending only on $\mathcal{M}$, $\bm{g}$, $n$, $p$, $\bm{A}(\bm{\alpha}\comma t)$, $\phi$ and $\underline{u}$, so that for any $\bm{x}\in\mathcal{M}$, there holds
\begin{eq} \label{eq: p-Hessian, C-subsolution}
\left\{\bm\mu\in\Gamma_p\middle|\sigma_p^{\frac1p}(\bm\mu)=\phi(\bm{x})\right\}\cap\biggl(\Bigl\{\bm\lambda\Bigl({\bm g}^{-1}\bigl(\bm{A}(\dif\underline{u}\comma\underline{u})+\nabla^2\underline{u}\bigr)\Bigr)(\bm{x})-\delta\bm{1}_n\Bigr\}+\Gamma_n\biggr)\subset\mathrm{B}_R(\bm{0}).
\end{eq}
Here $\phi$ is a positive smooth function on $\mathcal{M}$, $\bm{1}_n$ is defined in \eqref{eq: 1_n}, and $\mathrm{B}_R(\bm{0})$ denotes the open ball in $\mathbb{R}^n$ whose centre and radius are the origin $\bm{0}$ and $R$ respectively. A general $p$-admissible function (cf. \eqref{eq: general p-admissible}) must be a $\mathcal{C}$-subsolution to \eqref{eq: general p-Hessian} --- in fact, every $v\in\mathrm{C}^2(\mathcal{M})$ satisfying for some $R\in[0\comma{+}\infty)$ that (here $\mathbf{e}_j$ denotes the $j$-th standard basis vector in $\mathbb{R}^n$)
\begin{eq} \label{eq: general p-admissible, weaker}
\bm\lambda\Bigl({\bm g}^{-1}\bigl(\bm{A}(\dif v\comma v)+\nabla^2v\bigr)\Bigr)(\bm{x})+R\mathbf{e}_j\in\Gamma_p\comma\forall\bm{x}\in\mathcal{M}\comma\forall j\in\{1\comma2\comma\cdots\comma n\}
\end{eq}
is a $\mathcal{C}$-subsolution to \eqref{eq: general p-Hessian}, refer to e.g. \cite[p.~345]{Szekelyhidi2018}.

For general $\varphi(\bm{\alpha}\comma t)$, however, one can hardly define the corresponding $\mathcal{C}$-subsolution. And so far, to the best of our knowledge, no literature has advanced any kind of ``subsolution'' (except ``classical subsolution'') well-defined for the fully general $p$-Hessian equation \eqref{eq: general p-Hessian} on a closed Riemannian manifold. In this paper, we advance such one.

\begin{defn} \label{defn: pseudo-solution}
We call $\underline{u}\in\mathrm{C}^2(\mathcal{M})$ a \emph{pseudo-solution} to \eqref{eq: general p-Hessian} if it satisfies:
\begin{enumerate}
\item \label{item: pseudo-subsolution condition} \emph{Pseudo-subsolution condition}: for any $C\in[0\comma{+}\infty)$, there exist a positive constant $\delta_1$ and a non-negative constant $M_1$, depending only on $\mathcal{M}$, $\bm{g}$, $n$, $p$, $\bm{A}(\bm{\alpha}\comma t)$, $\varphi(\bm{\alpha}\comma t)$, $\underline{u}$ and $C$, so that for any solution
\begin{eq}
u\in\mathrm{C}^2(\mathcal{M})\colon\max\limits_{\mathcal{M}}\left|u\right|+\max\limits_{\mathcal{M}}\left|\dif u\right|_{\bm g}\leqslant C
\end{eq}
and $\bm{x}\in\mathcal{M}$, there holds
\begin{eq} \label{eq: general p-Hessian, pseudo-subsolution condition} \begin{aligned}
&\mathrel{\hphantom{=}}F_{u\comma\bm{x}}^{jk}\nabla_{jk}(\underline{u}-u)(\bm{x})+F_{u\comma\bm{x}}^{jk}\Bigl(\mathrm{D}_{\dif u(\bm{x})}\bigl[A_{jk}\bigl(\cdot\comma u(\bm{x})\bigr)\bigr]\Bigr)\bigl(\dif\left(\underline{u}-u\right)(\bm{x})\bigr) \\
&\geqslant\delta_1F_{u\comma\bm{x}}^{jk}g_{jk}(\bm{x})-M_1\lambda_1\Bigl((F_{u\comma\bm x}^{jk})\bigl(g_{jk}(\bm{x})\bigr)\Bigr)-M_1\comma
\end{aligned} \end{eq}
where $(\mathcal{U}\comma\bm{\psi}_{\mathcal{U}}\mbox{; }x^j)$ is a local coordinate system containing $\bm{x}$, $F_{u\comma\bm x}^{jk}$ is defined as \eqref{eq: F_{u, x}^{jk}}, $\mathrm{D}_{\bm\alpha}[\cdot]$ denotes the Fr\'echet differentiation at $\bm{\alpha}\in(\mathrm{T}_{\bm x}^*\mathcal{M}\comma\left|\cdot\right|_{\bm g})$ (cf. \eqref{eq: Frechet}) and $\lambda_1$ denotes the minimum eigenvalue function (cf. \defnref{defn: lambda_q} and the last paragraph before \subsecref{subsec: Hessian equations on compact Riemannian manifolds});
\item \label{item: pseudo-supersolution condition} \emph{Pseudo-supersolution condition}: there exist a positive constant $\delta_2$ and a non-negative constant $M_2$, depending only on $\mathcal{M}$, $\bm{g}$, $n$, $p$, $\bm{A}(\bm{\alpha}\comma t)$, $\varphi(\bm{\alpha}\comma t)$ and $\underline{u}$, so that for any solution $u\in\mathrm{C}^2(\mathcal{M})$ and
\begin{eq}
\bm{x}\in\mathcal{M}\colon\bm\lambda\Bigl({\bm g}^{-1}\bigl(\nabla^2(u-\underline{u})\bigr)\Bigr)(\bm{x})+\delta_2\bm{1}_n\in\overline{\Gamma}_n\comma\left|\dif\left(u-\underline{u}\right)(\bm{x})\right|_{\bm g}\leqslant\delta_2\comma
\end{eq}
there holds
\begin{eq} \label{eq: general p-Hessian, pseudo-supersolution condition}
\sigma_n\biggl(\bm\lambda\Bigl({\bm g}^{-1}\bigl(\delta_2\bm{g}+\nabla^2(u-\underline{u})\bigr)\Bigr)(\bm{x})\biggr)\leqslant M_2.
\end{eq}
\end{enumerate}
\end{defn}

\begin{rmk}
The meaning of the name ``pseudo-solution'' is that every solution to \eqref{eq: general p-Hessian} must be a pseudo-solution while a pseudo-solution is not necessarily a solution (cf. \propref{prop: pseudo-solution, basic}). On the other hand, the names ``pseudo-subsolution condition'' and ``pseudo-supersolution condition'' are only for ease of reference.
\end{rmk}

Pseudo-solution plays an important role in our study of \eqref{eq: general p-Hessian}. When there exists a pseudo-solution is a non-trivial question. When $\varphi(\bm{\alpha}\comma t)=\phi\bigl(\pi(\bm{\alpha})\bigr)$, pseudo-solution is weaker to some extent than $\mathcal{C}$-subsolution. In this case, when $\bm{A}(\bm\alpha\comma t)$ is independent of $t$ and satisfies the concavity condition \eqref{eq: A_{vv}(alpha, t) is concave with respect to alpha}, every $\mathcal{C}$-subsolution to \eqref{eq: general p-Hessian} defined by \eqref{eq: p-Hessian, C-subsolution} is a pseudo-solution to \eqref{eq: general p-Hessian} (cf. \propref{prop: C-subsolution is pseudo-solution}). We are mainly interested in the case of general $\varphi(\bm{\alpha}\comma t)$ and have the following conclusion.

\begin{prop} \label{prop: pseudo-solution, basic}
Assume that $\bm{A}(\bm\alpha\comma t)$, $\varphi(\bm\alpha\comma t)$ are both independent of $t$ and $\bm{A}(\bm\alpha\comma t)$ satisfies \eqref{eq: A_{vv}(alpha, t) is concave with respect to alpha}. Then every $v\in\mathrm{C}^2(\mathcal{M})$ satisfying \eqref{eq: general p-admissible, weaker} for some $R\in[0\comma{+}\infty)$ is a pseudo-solution to \eqref{eq: general p-Hessian}. In particular, every general $p$-admissible function (cf. \eqref{eq: general p-admissible}) is a pseudo-solution.
\end{prop}

The first main result of this paper is as follows.

\begin{thm} \label{thm: general p-Hessian, second-order estimates}
Let $(\mathcal{M}\comma\bm g)$ be a closed connected Riemannian manifold of dimension $n$. Assume that $n\geqslant 3$, $p\in\{2\comma n-1\comma n\}$ and $u\in\mathrm{C}^4(\mathcal M)$ satisfies the general $p$-Hessian equation \eqref{eq: general p-Hessian} on $\mathcal{M}$. Assume additionally that $\underline{u}\in\mathrm{C}^2(\mathcal{M})$ satisfies pseudo-subsolution condition, and $\bm{A}(\bm{\alpha}\comma t)$ satisfies ``weak MTW condition'' \eqref{eq: A(alpha, t), weak MTW condition}. Then there exists a positive constant $M$ depending only on $\mathcal{M}$, $\bm{g}$, $n$, $p$, $\bm{A}(\bm{\alpha}\comma t)$, $\varphi(\bm{\alpha}\comma t)$, $\underline{u}$, $\max\limits_{\mathcal M} \left|u\right|$ and $\max\limits_{\mathcal M} \left|\dif u\right|_{\bm g}$ so that
\begin{eq}
\max\limits_{\mathcal M} \left|\nabla^2u\right|_{\bm g}\leqslant M.
\end{eq}
\end{thm}

For this second-order estimate, no special restrictions are imposed on the curvature, $\varphi(\bm\alpha\comma t)$ or $\bm{A}(\bm{\alpha}\comma t)$ (except ``weak MTW condition''). One can refer to \propref{prop: general p-Hessian, pseudo-subsolution condition} for when $\underline{u}$ satisfies pseudo-subsolution condition, and refer to \thmref{thm: semi-convex, second-order estimates} for a second-order estimate for ``semi-convex solutions'' without special restrictions on $p$ or others (see also \cite{Lu2023,Zhang2025,Chen2025}). \thmref{thm: general p-Hessian, second-order estimates} is, to some extent, stronger than similar results in literature, cf. e.g. \cite{Ren2019}. A new concavity inequality (cf. \thmref{thm: concavity inequality}) is used to prove the case $p=n-1$ of \thmref{thm: general p-Hessian, second-order estimates}. This concavity inequality follows essentially from the ideas in \cite{Lu2024,Chen2025,Zhang2025,Tu2024}, and may be also useful for the study of general complex $(n-1)$-Hessian equations.

Next, we discuss the $\mathrm{C}^1$ estimate which, like the second-order estimate, is essential to the theory of orthogonally-invariant elliptic equations on closed Riemannian manifolds. \citeauthor{Sui2025} \cite[pp.~5--8]{Sui2025} proved the sharp $\mathrm{C}^0$ estimate for \eqref{eq: general p-Hessian} with $\bm{A}(\bm\alpha\comma t)$, $\varphi(\bm\alpha\comma t)$ depending only on $\pi(\bm\alpha)$ and $u\equiv 0$ being a general $p$-admissible function. See also \cite[pp.~346--349]{Szekelyhidi2018} for the complex case, where \citeauthor{Szekelyhidi2018} pointed out that his method could also be used to obtain $\mathrm{C}^0$ estimates for more general equations. The second main result of this paper is as follows.

\begin{thm} \label{thm: general p-Hessian, C^0 estimates}
Let $(\mathcal{M}\comma\bm g)$ be a closed connected Riemannian manifold of dimension $n$. Assume that $n\geqslant 3$, $p\in\{2\comma3\comma\cdots\comma n\}$ and $u\in\mathrm{C}^2(\mathcal M)$ satisfies the general $p$-Hessian equation \eqref{eq: general p-Hessian} on $\mathcal{M}$. Assume additionally that $\underline{u}\in\mathrm{C}^2(\mathcal{M})$ satisfies pseudo-supersolution condition, and there exist non-negative continuous functions $\phi_1$, $\phi_2$ on $\mathcal{M}$ so that for any $(\bm{x}\comma t)\in\mathcal{M}\times\mathbb{R}$,
\begin{eq} \label{eq: C^0 estimates, structural condition}
g^{jk}(\bm{x})A_{jk}(\bm\alpha\comma t)\leqslant \phi_1(\bm{x})\left|\bm\alpha\right|_{\bm g}+\phi_2(\bm{x})\comma\forall\bm\alpha\in\mathrm{T}_{\bm x}^*\mathcal{M}\comma
\end{eq}
where $(\mathcal{U}\comma\bm{\psi}_{\mathcal{U}}\mbox{; }x^j)$ is a local coordinate system containing $\bm{x}$. Then there exists a positive constant $M$ depending only on $\mathcal{M}$, $\bm{g}$, $n$, $p$, $\bm{A}(\bm{\alpha}\comma t)$, $\varphi(\bm{\alpha}\comma t)$ and $\underline{u}$ so that
\begin{eq}
\osc\limits_{\mathcal M} u\leqslant M.
\end{eq}
Moreover, if $\max\limits_{\mathcal{M}}u=0$ or $\max\limits_{\mathcal{M}}(u-\underline{u})=0$, the assumption \eqref{eq: C^0 estimates, structural condition} can be replaced with a weaker one:
\begin{eq} \label{eq: C^0 estimates, structural condition, weaker}
g^{jk}(\bm{x})A_{jk}(\bm\alpha\comma t)\leqslant \phi_1(\bm{x})\left|\bm\alpha\right|_{\bm g}+\phi_2(\bm{x})(\left|t\right|+1)\comma\forall\bm\alpha\in\mathrm{T}_{\bm x}^*\mathcal{M}.
\end{eq}
\end{thm}

For this $\mathrm{C}^0$ estimate, no special restrictions are imposed on $\varphi(\bm\alpha\comma t)$. One can refer to \propref{prop: general p-Hessian, pseudo-supersolution condition} for when $\underline{u}$ satisfies pseudo-supersolution condition. The structural condition \eqref{eq: C^0 estimates, structural condition} or \eqref{eq: C^0 estimates, structural condition, weaker} is used for the $\mathrm{L}^1$ estimate. The proof of \thmref{thm: general p-Hessian, C^0 estimates} is based on the ideas in \cite{Szekelyhidi2018,Sui2025}, including a generalized Alexandrov lemma (cf. \propref{prop: generalized Alexandrov}). Note that the equations in example \eqref{item: example, p-Yamabe} on page~\pageref{item: example, p-Yamabe} all satisfy \eqref{eq: C^0 estimates, structural condition}, while those in examples \eqref{item: example, p-Christoffel-Minkowski}, \eqref{item: example, geometric optics} all satisfy \eqref{eq: C^0 estimates, structural condition, weaker}.

The first-order estimate\footnote{In literature, first-order estimate is also called ``gradient estimate''.} for \eqref{eq: general p-Hessian} is similar to that for the case of Dirichlet problem, refer to e.g. \cite{Jiang2018,Jiang2020,Guan2016}, see also \cite{Sui2025}. One purpose of this paper is to prove a priori estimates for \eqref{eq: general p-Hessian} under sharp conditions, so we also give such a first-order estimate which is, to some extent, stronger than similar results in literature for the case of general $p$-Hessian equations.

\begin{thm} \label{thm: general p-Hessian, first-order estimates}
Let $(\mathcal{M}\comma\bm g)$ be a closed connected Riemannian manifold of dimension $n$. Assume that $n\geqslant 3$, $p\in\{2\comma3\comma\cdots\comma n\}$ and $u\in\mathrm{C}^3(\mathcal M)$ satisfies the general $p$-Hessian equation \eqref{eq: general p-Hessian} on $\mathcal{M}$. Assume additionally that there exist non-negative continuous functions $\phi_3$, $\phi_4$, $\cdots$, $\phi_8$ on $\mathcal{M}\times\mathbb{R}$ and
\begin{eq}
\omega_1\comma\omega_2\comma\cdots\comma\omega_5\in\left\{\omega\in\mathrm{C}\bigl([0\comma{+}\infty)\comma[0\comma{+}\infty)\bigr)\middle|\lim_{t\to{+}\infty}\frac{\omega(t)}{t}=0\comma\sup_{t\in[0\comma{+}\infty)}\bigl(\omega(t)-t\bigr)\leqslant1\right\}
\end{eq}
so that for any $(\bm{x}\comma t)\in\mathcal{M}\times\mathbb{R}$ and $(\bm\alpha\comma\bm{v})\in\mathrm{T}_{\bm x}^*\mathcal{M}\times\mathrm{T}_{\bm x}\mathcal{M}$, there hold
\begin{ga}
\bigl(\bm{A}(\bm\alpha\comma t)\bigr)(\bm{v}\comma\bm{v})\leqslant\phi_3(\bm{x}\comma t)\bm{g}(\bm{v}\comma\bm{v})\omega_1(\left|\bm\alpha\right|_{\bm g}^2)\comma \label{eq: first-order estimates, structural conditions, A_{vv}(alpha, t) leqslant phi_3(x, t)g(v, v)omega_1(|alpha|_g^2)} \\%
\bigl(\bm{A}(\bm\alpha\comma t)\bigr)(\bm{v}\comma\bm{w})\leqslant\phi_4(\bm{x}\comma t)\bigl(\bm{g}(\bm{v}\comma\bm{v})\bm{g}(\bm{w}\comma\bm{w})\bigr)^{\frac12}\omega_2(\left|\bm\alpha\right|_{\bm g}^2)\comma\forall\bm{w}\in\mathrm{T}_{\bm x}\mathcal{M}\colon\bm{g}(\bm{v}\comma\bm{w})=0\comma \label{eq: first-order estimates, structural conditions, A_{vw}(alpha, t) leqslant phi_4(x, t)(g(v, v)g(w, w))^{frac 12}omega_2(|alpha|_g^2)} \\%
g^{jk}(\bm{x})\tilde{\nabla}_{\tilde{x}^j}A_{rs}(\bm\alpha\comma t)v^rv^s\alpha_k+\frac{\partial\bigl(\bm{A}(\bm\alpha\comma t)\bigr)(\bm{v}\comma\bm{v})}{\partial t}\left|\bm\alpha\right|_{\bm g}^2\leqslant\phi_5(\bm{x}\comma t)\bm{g}(\bm{v}\comma\bm{v})\omega_3(\left|\bm\alpha\right|_{\bm g}^4)\comma \label{eq: first-order estimates, structural conditions, g^{jk}(x)tilde{nabla}_{tilde{x}^j}A_{rs}(alpha, t)v^rv^s alpha_k+D_tA_{vv}(alpha, t)|alpha|_g^2 leqslant phi_5(x, t)g(v, v)omega_3(|alpha|_g^4)} \\%
\int_0^1 \theta\Bigl(\mathrm{D}_{\theta\bm\alpha}^2\bigl[\bigl(\bm{A}(\cdot\comma t)\bigr)(\bm{v}\comma\bm{v})\bigr]\Bigr)(\bm{\alpha}\comma\bm{\alpha})\dif\theta\geqslant-\phi_6(\bm{x}\comma t)\bm{g}(\bm{v}\comma\bm{v})(\left|\bm\alpha\right|_{\bm g}^2+1)\comma \label{eq: first-order estimates, structural conditions, D^2A_{vv}(alpha, t) geqslant -phi_6(x, t)g(v, v)(|alpha|_g^2+1)} \\%
g^{jk}(\bm{x})\tilde{\nabla}_{\tilde{x}^j}\varphi(\bm\alpha\comma t)\alpha_k+\frac{\partial\varphi(\bm\alpha\comma t)}{\partial t}\left|\bm\alpha\right|_{\bm g}^2\geqslant-\phi_7(\bm{x}\comma t)\omega_4(\left|\bm\alpha\right|_{\bm g}^4)\max\left\{1\comma\frac{\left|\bm\alpha\right|_{\bm g}^{2(p-1)}}{\varphi^{p-1}(\bm\alpha\comma t)}\right\}\comma \label{eq: first-order estimates, structural conditions, g^{jk}(x)tilde{nabla}_{tilde{x}^j}varphi(alpha, t)alpha_k+D_t varphi(alpha, t)|alpha|_g^2 geqslant -phi_7(x, t)omega_4(|alpha|_g^4)max{1, frac{|alpha|_g^{2(p-1)}}{varphi^{p-1}(alpha, t)}}} \\%
\int_0^1 \theta\bigl(\mathrm{D}_{\theta\bm\alpha}^2[\varphi(\cdot\comma t)]\bigr)(\bm{\alpha}\comma\bm{\alpha})\dif\theta\leqslant\phi_8(\bm{x}\comma t)(\left|\bm\alpha\right|_{\bm g}^2+1)\max\left\{1\comma\frac{\omega_5(\left|\bm\alpha\right|_{\bm g}^{2(p-1)})}{\varphi^{p-1}(\bm\alpha\comma t)}\right\}\comma \label{eq: first-order estimates, structural conditions, D^2 varphi(alpha, t) leqslant phi_8(x, t)(|alpha|_g^2+1)max{1, frac{omega_5(|alpha|_g^{2(p-1)})}{varphi^{p-1}(alpha, t)}}}
\end{ga}
where $(\mathcal{U}\comma\bm{\psi}_{\mathcal{U}}\mbox{; }x^j)$ is a local coordinate system containing $\bm{x}$ and $\tilde{\nabla}_{\tilde{x}^j}A_{rs}(\bm\alpha\comma t)$, $\tilde{\nabla}_{\tilde{x}^j}\varphi(\bm\alpha\comma t)$ are defined as \eqref{eq: tilde{nabla}_{tilde{x}^j}A_{x^rx^s}}, \eqref{eq: tilde{nabla}_{tilde{x}^j}varphi} respectively. Then there exists a positive constant $M$ depending only on $\mathcal{M}$, $\bm{g}$, $n$, $p$, $\bm{A}(\bm{\alpha}\comma t)$, $\varphi(\bm{\alpha}\comma t)$ and $\max\limits_{\mathcal{M}}\left|u\right|$ so that
\begin{eq}
\max\limits_{\mathcal M} \left|\dif u\right|_{\bm g}\leqslant M.
\end{eq}
Moreover, if $p=n$, the assumptions \eqref{eq: first-order estimates, structural conditions, A_{vw}(alpha, t) leqslant phi_4(x, t)(g(v, v)g(w, w))^{frac 12}omega_2(|alpha|_g^2)}--\eqref{eq: first-order estimates, structural conditions, D^2 varphi(alpha, t) leqslant phi_8(x, t)(|alpha|_g^2+1)max{1, frac{omega_5(|alpha|_g^{2(p-1)})}{varphi^{p-1}(alpha, t)}}} can be omitted and \eqref{eq: first-order estimates, structural conditions, A_{vv}(alpha, t) leqslant phi_3(x, t)g(v, v)omega_1(|alpha|_g^2)} can be replaced with a weaker one:
\begin{eq} \label{eq: first-order estimates, structural conditions, A_{vv}(alpha, t) leqslant phi_3(x, t)g(v, v)(|alpha|_g^2+1)}
\bigl(\bm{A}(\bm\alpha\comma t)\bigr)(\bm{v}\comma\bm{v})\leqslant\phi_3(\bm{x}\comma t)\bm{g}(\bm{v}\comma\bm{v})(\left|\bm\alpha\right|_{\bm g}^2+1). 
\end{eq}
\end{thm}

This first-order estimate can be applied to all the equations in examples \eqref{item: example, p-Christoffel-Minkowski}, \eqref{item: example, geometric optics} on page \pageref{item: example, p-Christoffel-Minkowski}, and the equation \eqref{eq: general p-Hessian} with
\begin{al}
\bm{A}(\bm\alpha\comma t)&=\Bigl(1-\mathrm{e}^t-\tilde{\phi}\bigl(\pi(\bm\alpha)\bigr)\left|\bm\alpha\right|_{\bm g}^2\Bigr)\bm{g}\bigl(\pi(\bm\alpha)\bigr)\comma \label{eq: A(alpha, t)=(1-e^t-tilde{phi}(pi(alpha))|alpha|_g^2)g(pi(alpha))} \\
\varphi(\bm\alpha\comma t)&=\phi\bigl(\pi(\bm{\alpha})\comma t\bigr)\bigl(\sin(\left|\bm\alpha\right|_{\bm g}^{2p-1})+2\bigr)
\end{al}
for any positive smooth function $\phi$ on $\mathcal{M}\times\mathbb{R}$ and non-negative smooth function $\tilde\phi$ on $\mathcal{M}$. The first-order estimate for the latter equation, which reads
\begin{eq} \label{eq: example, quite general}
\sigma_p^{\frac{1}{p}}\biggl(\bm\lambda\Bigl({\bm g}^{-1}\bigl((1-\mathrm{e}^u-\tilde{\phi}\left|\dif u\right|_{\bm g}^2)\bm{g}+\nabla^2u\bigr)\Bigr)(\bm{x})\biggr)=\phi\bigl(\bm{x}\comma u(\bm{x})\bigr)\Bigl(\sin\bigl(\left|\dif u(\bm{x})\right|_{\bm g}^{2p-1}\bigr)+2\Bigr)\comma
\end{eq}
is not covered by the results in \cite{Jiang2018,Jiang2020,Guan2016,Sui2025}. Note that \eqref{eq: A(alpha, t)=(1-e^t-tilde{phi}(pi(alpha))|alpha|_g^2)g(pi(alpha))} also satisfies \eqref{eq: A(alpha, t), weak MTW condition} and \eqref{eq: C^0 estimates, structural condition}. The proof of \thmref{thm: general p-Hessian, first-order estimates} uses in particular a trick in \cite[p.~1039]{Chou2001}, cf. \eqref{eq: first-order estimates, F^{n_0n_0}>frac{(zeta'(u))^{p-1}|du|_g^{2(p-1)}}{C_3varphi^{p-1}(du, u)}}. The case $p=n$ was proved in \cite[pp.~1232--1233]{Jiang2014a}, see also \thmref{thm: semi-convex, first-order estimates}. One can refer to \cite[pp.~384,~389]{Jiang2018} for another first-order estimate for the case $p\in\left(\frac n2\comma n\right)\cap\mathbb{N}$.

In \cite[p.~164]{Ma2005}, \citeauthor{Ma2005} advanced an essential condition, now called ``MTW condition'', to prove the interior second-order estimate for the optimal transport problem. By definition, \eqref{eq: general p-Hessian} satisfies ``MTW condition'' if and only if there exists such a positive constant $c_0$ that for any $(\bm{x}\comma t)\in\mathcal{M}\times\mathbb{R}$ and $\bm\alpha\in\mathrm{T}_{\bm x}^*\mathcal{M}$, there holds
\begin{eq} \label{eq: A(alpha, t), MTW condition}
\Bigl(\mathrm{D}_{\bm\alpha}^2\bigl[\bigl(\bm{A}(\cdot\comma t)\bigr)(\bm{v}\comma\bm{v})\bigr]\Bigr)(\bm{\beta}\comma\bm{\beta})\leqslant-c_0\bm{g}(\bm{v}\comma\bm{v})\left|\bm\beta\right|_{\bm g}^2\comma \forall(\bm\beta\comma\bm{v})\in\mathrm{T}_{\bm x}^*\mathcal{M}\times\mathrm{T}_{\bm x}\mathcal{M}\colon\bm\beta(\bm{v})=0.
\end{eq}
This explains why \eqref{eq: A(alpha, t), weak MTW condition} is called ``weak MTW condition''. The equations in example \eqref{item: example, p-Yamabe} on page \pageref{item: example, p-Yamabe}, including the $p$-Yamabe equation, are typical ones satisfying ``MTW condition'' (cf. \eqref{eq: general p-Yamabe, A(alpha, t), MTW condition}). The first-order estimates for such equations, although not necessarily covered by \thmref{thm: general p-Hessian, first-order estimates}, can be proved in a similar manner (cf. e.g. \cite[pp.~709--710]{Guan2016}). The second-order estimates for such equations don't need any kind of ``subsolution'' (cf. \rmkref{rmk: second-order estimates, p=2 or A(alpha, t) satisfies MTW condition} and e.g. \cite[p.~2697]{Guan2015}). One can refer to e.g. \cite{Gursky2007,Li2003,Li2005,Chang2002,Sheng2007,Ge2006,Yuan2025} for more about the $p$-Yamabe equation and other equations arising from conformal geometry.

Combining \defnref{defn: pseudo-solution}, \thmref{thm: general p-Hessian, second-order estimates}, \thmref{thm: general p-Hessian, C^0 estimates}, \thmref{thm: general p-Hessian, first-order estimates} and the Evans-Krylov estimate, we obtain the following $\mathrm{C}^{2\comma\gamma}$ estimate. Compared to results in literature, the conditions of this theorem for $\bm{A}(\bm{\alpha}\comma t)$ and $\varphi(\bm{\alpha}\comma t)$ are quite sharp, satisfied by \eqref{eq: example, quite general} for example.

\begin{thm} \label{thm: general p-Hessian, C^{2, gamma} estimates}
Let $(\mathcal{M}\comma\bm g)$ be a closed connected Riemannian manifold of dimension $n$. Assume that $n\geqslant 3$, $u\in\mathrm{C}^4(\mathcal M)$ satisfies the general $p$-Hessian equation \eqref{eq: general p-Hessian} on $\mathcal{M}$, $\underline{u}\in\mathrm{C}^2(\mathcal{M})$ is a pseudo-solution to \eqref{eq: general p-Hessian}, $\max\limits_{\mathcal{M}}(u-\underline{u})=0$ and $\bm{A}(\bm{\alpha}\comma t)$ satisfies \eqref{eq: A(alpha, t), weak MTW condition}, \eqref{eq: C^0 estimates, structural condition, weaker}. Assume additionally that one of the following conditions holds:
\begin{enumerate}
\item $p=n$, and $\bm{A}(\bm{\alpha}\comma t)$ satisfies \eqref{eq: first-order estimates, structural conditions, A_{vv}(alpha, t) leqslant phi_3(x, t)g(v, v)(|alpha|_g^2+1)};
\item \label{item: general p-Hessian, C^{2, gamma} estimates, p<n} $p\in\{2\comma n-1\}$, $\bm{A}(\bm{\alpha}\comma t)$ satisfies \eqref{eq: first-order estimates, structural conditions, A_{vv}(alpha, t) leqslant phi_3(x, t)g(v, v)omega_1(|alpha|_g^2)}--\eqref{eq: first-order estimates, structural conditions, D^2A_{vv}(alpha, t) geqslant -phi_6(x, t)g(v, v)(|alpha|_g^2+1)}, and $\varphi(\bm{\alpha}\comma t)$ satisfies \eqref{eq: first-order estimates, structural conditions, g^{jk}(x)tilde{nabla}_{tilde{x}^j}varphi(alpha, t)alpha_k+D_t varphi(alpha, t)|alpha|_g^2 geqslant -phi_7(x, t)omega_4(|alpha|_g^4)max{1, frac{|alpha|_g^{2(p-1)}}{varphi^{p-1}(alpha, t)}}}, \eqref{eq: first-order estimates, structural conditions, D^2 varphi(alpha, t) leqslant phi_8(x, t)(|alpha|_g^2+1)max{1, frac{omega_5(|alpha|_g^{2(p-1)})}{varphi^{p-1}(alpha, t)}}}.
\end{enumerate}
Then there exist positive constants $M$ and $\gamma\in(0\comma1)$ depending only on $\mathcal{M}$, $\bm{g}$, $n$, $p$, $\bm{A}(\bm{\alpha}\comma t)$, $\varphi(\bm{\alpha}\comma t)$ and $\underline{u}$ so that
\begin{eq}
\left\|u\right\|_{\mathrm{C}^{2\comma\gamma}(\mathcal M)}\leqslant M.
\end{eq}
\end{thm}

\begin{rmk}
In the case \eqref{item: general p-Hessian, C^{2, gamma} estimates, p<n} of \thmref{thm: general p-Hessian, C^{2, gamma} estimates}, the restriction on $p$ can be removed and even the equations can be more general orthogonally-invariant elliptic equations when $\bm{A}(\bm\alpha\comma t)$ additionally satisfies ``MTW condition'' \eqref{eq: A(alpha, t), MTW condition} or $\varphi(\bm{\alpha}\comma t)$ additionally satisfies the convexity condition \eqref{eq: varphi(alpha, t) is convex with respect to alpha} (cf. \thmref{thm: p=2 or A(alpha, t) satisfies MTW condition or varphi(alpha, t) satisfies the convexity condition, second-order estimates} and \cite{Jiang2018,Jiang2020,Guan2015}). Our focus is not in this direction, however.
\end{rmk}

As \citeauthor{Szekelyhidi2018} \cite[p.~368]{Szekelyhidi2018} pointed out, it's hard to formulate satisfactory existence results for \eqref{eq: general p-Hessian} in general because non-zero constant functions on $\mathcal{M}$ are not in the image of the linearized operator of \eqref{eq: general p-Hessian}, at least when $\bm{A}(\bm{\alpha}\comma t)$ and $\varphi(\bm{\alpha}\comma t)$ are both independent of $t$.

This paper is organized as follows. Preliminaries to elementary symmetric polynomials, elementary spectral polynomials and Hessian equations on Riemannian manifolds are given successively in the three subsections of \secref{sec: Preliminaries}. Some of our notations clarified in \secref{sec: Preliminaries}, e.g. the tensor fields defined by \eqref{eq: tilde{nabla}_{tilde{x}^j}varphi}--\eqref{eq: tilde{nabla}_{tilde{x}^j tilde{x}^k}A_{x^rx^s}}, are unusual but seem to be suitable for the study of general $p$-Hessian equations on Riemannian manifolds. In \secref{sec: ``Subsolutions''}, we prove \propref{prop: pseudo-solution, basic} and some other conclusions (e.g. \lemref{lem: C-subsolution, key lemma}) related to different kinds of ``subsolutions''. In \secref{sec: Concavity inequalities related to (n-1)-Hessian equations}, we discuss concavity inequalities related to real and complex $(n-1)$-Hessian equations. Second-order estimates (\thmref{thm: general p-Hessian, second-order estimates}), $\mathrm{C}^0$ estimates (\thmref{thm: general p-Hessian, C^0 estimates}) and first-order estimates (\thmref{thm: general p-Hessian, first-order estimates}) are proved in \secref{sec: Second-order estimates for general p-Hessian equations}, \secref{sec: C^0 estimates for general p-Hessian equations}, \secref{sec: First-order estimates for general p-Hessian equations} respectively.

Last but not least, the notations $C_1$, $C_2$, $C_3$ and so forth in our proofs denote positive constants depending only on some trivial quantities unless otherwise noted. What quantities such a constant exactly depends on will be clear by reading the context. Throughout this paper we follow the Einstein summation convention, and ``$\mathrm{D}$'' never denotes the covariant derivative --- we calculate mainly by ordinary derivatives like e.g. \cite{Li1990}.


\section{Preliminaries} \label{sec: Preliminaries}

We fix a positive integer $n$ and let $\bm\mu$ denote a generic real $n$-vector $(\mu_1\comma\mu_2\comma\cdots\comma\mu_n)^{\mathrm T}$ throughout this section.

\subsection{Properties of elementary symmetric polynomials} \label{subsec: Properties of elementary symmetric polynomials}

To begin with, we define an equivalence relation in $\mathbb{R}^n$: $\bm{x}\sim\bm{y}$ if and only if there exists such an $n$-permutation $\rho$ that $x_j=y_{\rho(j)}$ for any $j\in\{1\comma2\comma\cdots\comma n\}$. 

\begin{defn}
A non-empty subset $\Gamma$ of $\mathbb{R}^n$ is called \emph{symmetric} if for any $\bm{x}\in\Gamma$ and $\bm{y}\in\mathbb{R}^n\colon \bm{y}\sim\bm{x}$, there holds $\bm{y}\in\Gamma$. A real-valued function $f$ defined on a symmetric set $\Gamma\subset\mathbb{R}^n$ is called \emph{symmetric} if for any $\bm{x}\comma\bm{y}\in\Gamma\colon\bm{x}\sim\bm{y}$, there holds $f(\bm{x})=f(\bm{y})$.
\end{defn}

For any $p\in\mathbb{N}$, we define a function $\sigma_p$ as follows:
\begin{eq} \label{eq: sigma_p}
\sigma_p\colon\bigsqcup_{m=1}^\infty \mathbb{R}^m\longrightarrow\mathbb{R}\comma(\nu_1\comma\nu_2\comma\cdots\comma\nu_m)^{\mathrm T}\longmapsto
\left\{ \begin{aligned}
\sum\limits_{1\leqslant j_1<j_2<\dots<j_p\leqslant m}\nu_{j_1}\nu_{j_2}\dots\nu_{j_p}\comma & \ m\geqslant p \\
0\comma & \ m<p
\end{aligned} \right. .
\end{eq}
We also define a function $\sigma_p$ for any non-positive integer $p$ as follows:
\begin{eq}
\sigma_p\colon\bigsqcup_{m=1}^\infty \mathbb{R}^m\longrightarrow\mathbb{R}\comma(\nu_1\comma\nu_2\comma\cdots\comma\nu_m)^{\mathrm T}\longmapsto
\left\{ \begin{aligned}
1\comma & \ p=0 \\
0\comma & \ p<0
\end{aligned} \right. .
\end{eq}
The function $\sigma_p\bigr|_{\mathbb{R}^n}$ is symmetric --- $\sigma_0\bigr|_{\mathbb{R}^n}=1$; when $p<0$ or $p>n$, $\sigma_p\bigr|_{\mathbb{R}^n}=0$; when $p\in\{1\comma2\comma\cdots\comma n\}$, $\sigma_p\bigr|_{\mathbb{R}^n}$ is just the $p$-th elementary symmetric $n$-polynomial. For any $p\in\mathbb{Z}$, $m\in\mathbb{N}$ and $j_1\comma j_2\comma\cdots\comma j_m\in\{1\comma2\comma\cdots\comma n\}$, let the notation $\sigma_p(\mu_1\comma\mu_2\comma\cdots\comma\mu_n)$ denote $\sigma_p(\bm\mu)$ and the notation $\sigma_p(\bm\mu|j_1j_2\cdots j_m)$ denote
\begin{eq}
\frac{\partial^m\sigma_{p+m}(\bm\mu)}{\partial\mu_{j_m}\cdots\partial\mu_{j_2}\partial\mu_{j_1}}.
\end{eq}
It's obvious that
\begin{eq} \label{eq: sigma_p, derivatives}
\frac{\partial^m\sigma_{p}(\bm{\mu})}{\partial\mu_{j_m}\cdots\partial\mu_{j_2}\partial\mu_{j_1}}=\sigma_{p-m}(\bm{\mu}|j_1j_2\cdots j_m)
\end{eq}
and when $p>n-m$, $\sigma_p(\bm\mu|j_1j_2\cdots j_m)=0$. The following lemma contains some basic properties of the functions $\sigma_p$'s.

\begin{lem} \label{lem: basic properties of sigma_p}
Assume that $p\in\mathbb{Z}$ and $j\in\{1\comma2\comma\cdots\comma n\}$. Then there hold
\begin{al}
\sigma_p(\bm\mu|j)&=\sigma_p(\mu_1\comma\cdots\comma\mu_{j-1}\comma0\comma\mu_{j+1}\comma\cdots\comma\mu_n)\comma \label{eq: sigma_p(mu|j)} \\
\sigma_p(\bm\mu)&=\mu_j\sigma_{p-1}(\bm\mu|j)+\sigma_p(\bm\mu|j)\comma \label{eq: sigma_p(mu)=mu_j sigma_{p-1}(mu|j)+sigma_p(mu|j)} \\
\sum_{k=1}^n\sigma_p(\bm\mu|k)&=(n-p)\sigma_p(\bm\mu)\comma \label{eq: sum_{k=1}^n sigma_p(mu|k)} \\
\sum_{k=1}^n\mu_k\sigma_p(\bm\mu|k)&=(p+1)\sigma_{p+1}(\bm\mu)\comma \label{eq: sum_{k=1}^n mu_k sigma_p(mu|k)} \\
\sum_{k=1}^n\mu_k^2\sigma_p(\bm\mu|k)&=\sigma_1(\bm\mu)\sigma_{p+1}(\bm\mu)-(p+2)\sigma_{p+2}(\bm\mu). \label{eq: sum_{k=1}^n mu_k^2 sigma_p(mu|k)}
\end{al}
\end{lem}

\begin{proof}
Refer to e.g. \cite[p.~47]{Huisken1999}.
\end{proof}

The following proposition is essentially a trivial corollary of \eqref{eq: sigma_p(mu|j)} and \eqref{eq: sigma_p(mu)=mu_j sigma_{p-1}(mu|j)+sigma_p(mu|j)}.

\begin{prop} \label{prop: basic properties of sigma_p(mu|j_1j_2...j_m)}
Assume that $p\in\mathbb{Z}$, $m\in\mathbb{N}\colon m\geqslant 2$ and $j_1\comma j_2\comma\cdots\comma j_m\in\{1\comma2\comma\cdots\comma n\}$.
\begin{enumerate}
\item If there exist two indices $a$, $b\in\{1\comma2\comma\cdots\comma m\}$ so that $a\ne b$ and $j_a=j_b$, then there holds $\sigma_p(\bm\mu|j_1j_2\cdots j_m)=0$. In particular, $\sigma_p(\bm\mu|j_1j_2\cdots j_m)=0$ if $m>n$.
\item If $j_a\ne j_b$ for any two indices $a$, $b\in\{1\comma2\comma\cdots\comma m\}\colon a\ne b$, then there hold
\begin{al}
\sigma_p(\bm\mu|j_1j_2\cdots j_m)&=\sigma_p(\bm\mu)\Bigr|_{\mu_{j_1}=\mu_{j_2}=\cdots=\mu_{j_m}=0}\comma \\
\sigma_p(\bm\mu|j_1j_2\cdots j_{m-1})&=\mu_{j_m}\sigma_{p-1}(\bm\mu|j_1j_2\cdots j_m)+\sigma_p(\bm\mu|j_1j_2\cdots j_m) \label{eq: sigma_p(mu|j_1j_2...j_{m-1})=mu_{j_m} sigma_{p-1}(mu|j_1j_2...j_m)+sigma_p(mu|j_1j_2...j_m)} \comma
\end{al}
and
\begin{eq} \label{eq: sigma_p(mu), expansion} \begin{aligned}
\sigma_p(\bm{\mu})&=\mu_{j_1}\mu_{j_2}\cdots\mu_{j_m}\sigma_{p-m}(\bm\mu|j_1j_2\cdots j_m)+\sigma_p(\bm\mu|j_1) \\
&+\sum_{q=1}^{m-1} \mu_{j_1}\mu_{j_2}\cdots\mu_{j_q}\sigma_{p-q}(\bm\mu|j_1j_2\cdots j_{q+1}).
\end{aligned} \end{eq}
Furthermore, if we additionally assume that $m<n$ and $\{1\comma2\comma\cdots\comma n\}\setminus\{j_1\comma j_2\comma\cdots\comma j_m\}=\{k_1\comma k_2\comma\cdots\comma k_{n-m}\}$, then there holds
\begin{eq} \label{eq: sigma_p(mu|j_1j_2...j_m)=sigma_p(mu_{k_1},mu_{k_2},...,mu_{k_{n-m}})}
\sigma_p(\bm\mu|j_1j_2\cdots j_m)=\sigma_p(\mu_{k_1}\comma\mu_{k_2}\comma\cdots\comma\mu_{k_{n-m}}).
\end{eq}
\item There holds
\begin{eq} \label{eq: sigma_{p-1}(mu|j_1)-sigma_{p-1}(mu|j_2)}
\sigma_{p-1}(\bm\mu|j_1)-\sigma_{p-1}(\bm\mu|j_2)=(\mu_{j_2}-\mu_{j_1})\sigma_{p-2}(\bm\mu|j_1j_2).
\end{eq}
\end{enumerate}
\end{prop}

\begin{proof}
\eqref{eq: sigma_p(mu), expansion} can be proved by \eqref{eq: sigma_p(mu|j_1j_2...j_{m-1})=mu_{j_m} sigma_{p-1}(mu|j_1j_2...j_m)+sigma_p(mu|j_1j_2...j_m)} and induction with respect to $m$. The others are obvious.
\end{proof}

For the rest of this subsection, we assume that $p\in\{1\comma2\comma\cdots\comma n\}$. Recall that we always let $\bm\mu$ denote a generic real $n$-vector in this section.

\begin{defn} \label{defn: Garding cone}
The $p$-th \emph{G\r{a}rding cone} $\Gamma_p^{(n)}$ with respect to $\mathbb{R}^n$ is defined as
\begin{eq}
\bigl\{\bm\mu\in\mathbb{R}^n\bigm|\sigma_q(\bm\mu)>0\comma\forall q\in\{1\comma2\comma\cdots\comma p\}\bigr\}.
\end{eq}
$\Gamma_0^{(n)}$ is defined as $\mathbb{R}^n$. When $n$ is fixed, the notations $\Gamma_p^{(n)}$, $\Gamma_0^{(n)}$ will be abbreviated to $\Gamma_p$, $\Gamma_0$.
\end{defn}

Obviously
\begin{eq} \label{eq: 1_n}
\bm{1}_n\triangleq(\underbrace{1\comma1\comma\cdots\comma1}_{\text{$n$ times}})^{\mathrm T}\in\Gamma_n.
\end{eq}
The following proposition contains some basic properties of the G\r{a}rding cones $\Gamma_p$'s.

\begin{prop} \label{prop: basic properties of Gamma_p}
$\Gamma_p$ is a symmetric, open, proper subset of $\mathbb{R}^n$ and indeed a cone whose vertex is $\bm 0$. Moreover, we have
\begin{eq}
\Gamma_1\supset\Gamma_2\supset\dots\supset\Gamma_n=\bigl\{\bm\mu\in\mathbb{R}^n\bigm|\mu_j>0\comma\forall j\in\{1\comma2\comma\cdots\comma n\}\bigr\}.
\end{eq}
\end{prop}

\begin{proof}
Obvious.
\end{proof}

Let $\overline{\Gamma}_p$ denote the closure of $\Gamma_p$. It's obvious that
\begin{eq} \label{eq: overline{Gamma}_p subset}
\overline{\Gamma}_p\subset\bigl\{\bm\mu\in\mathbb{R}^n\bigm|\sigma_q(\bm\mu)\geqslant 0\comma\forall q\in\{1\comma2\comma\cdots\comma p\}\bigr\}.
\end{eq}
In fact, the ``$\subset$'' in \eqref{eq: overline{Gamma}_p subset} can be replaced by ``=''.

\begin{prop} \label{prop: overline{Gamma}_p}
There holds
\begin{eq}
\overline{\Gamma}_p=\bigl\{\bm\mu\in\mathbb{R}^n\bigm|\sigma_q(\bm\mu)\geqslant 0\comma\forall q\in\{1\comma2\comma\cdots\comma p\}\bigr\}.
\end{eq}
\end{prop}

\begin{proof}
Assume that $\bm\mu\in\mathbb{R}^n\colon\sigma_q(\bm\mu)\geqslant 0\comma\forall q\in\{1\comma2\comma\cdots\comma p\}$. We only need to prove that for any $\varepsilon>0$, there holds $\bm\mu+\varepsilon\bm{1}_n\in\Gamma_p$, i.e.
\begin{eq} \label{eq: (mu_1+varepsilon, mu_2+varepsilon, ..., mu_n+varepsilon)^T}
(\mu_1+\varepsilon\comma\mu_2+\varepsilon\comma\cdots\comma\mu_n+\varepsilon)^{\mathrm T}\in\Gamma_p.
\end{eq}
For any $q\in\{1\comma2\comma\cdots\comma p\}$, we calculate as follows:
\begin{eq} \begin{aligned}
&\mathrel{\hphantom{=}}\sigma_q(\mu_1+\varepsilon\comma\mu_2+\varepsilon\comma\cdots\comma\mu_n+\varepsilon) \\
&=\sum_{1\leqslant j_1<j_2<\cdots<j_q\leqslant n} (\mu_{j_1}+\varepsilon)(\mu_{j_2}+\varepsilon)\cdots(\mu_{j_q}+\varepsilon) \\
&=\mathrm{C}_n^q\varepsilon^q+\sum_{1\leqslant j_1<j_2<\cdots<j_q\leqslant n} \sum_{r=1}^q \varepsilon^{q-r}\sum_{1\leqslant k_1<k_2<\cdots<k_r\leqslant q} \mu_{j_{k_1}}\mu_{j_{k_2}}\cdots\mu_{j_{k_r}} \\
&=\mathrm{C}_n^q\varepsilon^q+\sum_{r=1}^q \varepsilon^{q-r}\sum_{1\leqslant k_1<k_2<\cdots<k_r\leqslant q} \sum_{1\leqslant j_{k_1}<j_{k_2}<\cdots<j_{k_r}\leqslant n} \mathrm{C}_{n-r}^{q-r}\mu_{j_{k_1}}\mu_{j_{k_2}}\cdots\mu_{j_{k_r}} \\
&=\mathrm{C}_n^q\varepsilon^q+\sum_{r=1}^q \varepsilon^{q-r}\mathrm{C}_q^r\mathrm{C}_{n-r}^{q-r}\sigma_r(\bm\mu)>0.
\end{aligned} \end{eq}
This justifies \eqref{eq: (mu_1+varepsilon, mu_2+varepsilon, ..., mu_n+varepsilon)^T}.
\end{proof}

The $p$-th G\r{a}rding cone $\Gamma_p$ is closely related to the $p$-th elementary symmetric polynomial $\sigma_p\bigr|_{\mathbb{R}^n}$. In fact, the function $\sigma_p^{\frac 1p}$ is strictly increasing, with respect to each variable, and concave in $\Gamma_p$.

\begin{prop} \label{prop: increasing with respect to each variable}
Assume that $q\in\{0\comma1\comma\cdots\comma p-1\}$ and $j\in\{1\comma2\comma\cdots\comma n\}$. Then there holds $\sigma_q(\bm\mu|j)\geqslant0$ for any $\bm\mu\in\overline{\Gamma}_p$. Moreover, if $\bm\mu\in\Gamma_p$, then there holds $\sigma_q(\bm\mu|j)>0$.
\end{prop}

\begin{proof}
Refer to e.g. \cite[p.~403]{Lieberman2005} for the proof that $\sigma_q(\bm\mu|j)>0$ if $\bm\mu\in\Gamma_p$. It can be proved by continuity of $\sigma_q\bigr|_{\mathbb{R}^n}$ that $\sigma_q(\bm\mu|j)\geqslant0$ for any $\bm\mu\in\overline{\Gamma}_p$.
\end{proof}

A corollary of \propref{prop: increasing with respect to each variable} and \propref{prop: basic properties of sigma_p(mu|j_1j_2...j_m)} is as follows.

\begin{cor} \label{cor: corollary of increasing with respect to each variable}
Assume that $m\in\{1\comma2\comma\cdots\comma p\}$, $q\in\{0\comma1\comma\cdots\comma p-m\}$ and $j_1\comma j_2\comma\cdots\comma j_m\in\{1\comma2\comma\cdots\comma n\}$. Then there holds $\sigma_q(\bm\mu|j_1j_2\cdots j_m)\geqslant 0$ for any $\bm\mu\in\overline{\Gamma}_p$. Moreover, if $\bm\mu\in\Gamma_p$ and $j_a\ne j_b$ for any $a$, $b\in\{1\comma2\comma\cdots\comma m\}\colon a\ne b$, then there holds $\sigma_q(\bm\mu|j_1j_2\cdots j_m)>0$.
\end{cor}

\begin{proof}
This corollary can be proved by induction with respect to $m$ and by continuity of $\sigma_q\bigr|_{\mathbb{R}^n}$.
\end{proof}

The following proposition is concerned with the concavity of the function $\left(\frac{\sigma_p}{\sigma_q}\right)^{\frac{1}{p-q}}$ in $\Gamma_p$.

\begin{prop} \label{prop: concavity-1}
Assume that $\bm\mu$, $\bm\nu\in\Gamma_p$ and $q\in\{0\comma1\comma\cdots\comma p-1\}$. Then there holds
\begin{eq}
\left(\frac{\sigma_p(\bm\mu+\bm\nu)}{\sigma_q(\bm\mu+\bm\nu)}\right)^{\frac{1}{p-q}}\geqslant \left(\frac{\sigma_p(\bm\mu)}{\sigma_q(\bm\mu)}\right)^{\frac{1}{p-q}}+\left(\frac{\sigma_p(\bm\nu)}{\sigma_q(\bm\nu)}\right)^{\frac{1}{p-q}}.
\end{eq}
\end{prop}

\begin{proof}
Refer to e.g. \cite[pp.~404--407]{Lieberman2005}.
\end{proof}

A trivial corollary of \propref{prop: concavity-1} is as follows.

\begin{cor} \label{cor: concavity-1}
The following statements are true:
\begin{enumerate}
\item For any $\bm\mu$, $\bm\nu\in\overline{\Gamma}_p$, there holds
\begin{eq}
\sigma_p^{\frac 1p}(\bm\mu+\bm\nu)\geqslant \sigma_p^{\frac 1p}(\bm\mu)+\sigma_p^{\frac 1p}(\bm\nu).
\end{eq}
\item Assume that $\bm\mu\in\overline{\Gamma}_p$. Then there holds $\sigma_p(\bm\mu+\bm\nu)>\sigma_p(\bm\mu)$ for any $\bm\nu\in\Gamma_p$.
\item $\Gamma_p$ is convex and therefore connected, so is $\overline{\Gamma}_p$.
\item The function $\sigma_p^{\frac{1}{p}}$ is concave in $\overline{\Gamma}_p$. Moreover, for any $q\in\{1\comma2\comma\cdots\comma p-1\}$, the function $\left(\frac{\sigma_p}{\sigma_q}\right)^{\frac{1}{p-q}}$ is concave in $\Gamma_p$.
\end{enumerate}
\end{cor}

\begin{proof}
(1), (2) and (4) are obvious. (3) follows from (1).
\end{proof}

A useful property of the functions $\sigma_p$'s is the following generalized Newton-Maclaurin inequality.

\begin{prop}[Generalized Newton-Maclaurin inequality] \label{prop: Newton-Maclaurin}
Assume that $\bm\mu\in\Gamma_p$, $j$, $k$, $l$, $m\in\mathbb{Z}$, $0\leqslant k<j\leqslant p+1$, $0\leqslant m<l\leqslant p$, $l\leqslant j$ and $m\leqslant k$. Then there holds
\begin{eq} \label{eq: Newton-Maclaurin}
\frac{\displaystyle \frac{\sigma_j(\bm\mu)}{\mathrm{C}_n^j}}{\displaystyle \ \frac{\sigma_k(\bm\mu)}{\mathrm{C}_n^k}\ }\leqslant\left(\frac{\displaystyle \frac{\sigma_l(\bm\mu)}{\mathrm{C}_n^l}}{\displaystyle \ \frac{\sigma_m(\bm\mu)}{\mathrm{C}_n^m}\ }\right)^{\textstyle \frac{j-k}{l-m}}.
\end{eq}
\end{prop}

\begin{proof}
For the case when $j\leqslant p$, see \cite[pp.~290--291]{Spruck2005}. For the case when $j=p+1$, if $\sigma_j(\bm\mu)\leqslant0$, then \eqref{eq: Newton-Maclaurin} is trivial; if $\sigma_j(\bm\mu)>0$, then $\bm\mu\in\Gamma_{p+1}$ and \eqref{eq: Newton-Maclaurin} follows from the case when $j\leqslant p$.
\end{proof}

A special case of the generalized Newton-Maclaurin inequality is the well-known Maclaurin inequality.

\begin{prop}[Maclaurin inequality] \label{prop: Maclaurin}
Assume that $\bm\mu\in\overline{\Gamma}_p$, $j$, $k\in\mathbb{N}$ and $k<j\leqslant p+1$. Then there holds
\begin{eq} \label{eq: Maclaurin}
\frac{\sigma_j(\bm\mu)}{\mathrm{C}_n^j}\leqslant\left(\frac{\sigma_k(\bm\mu)}{\mathrm{C}_n^k}\right)^{\frac jk}.
\end{eq}
\end{prop}

\begin{proof}
When $\bm\mu\in\Gamma_p$, \eqref{eq: Maclaurin} is just a special case of \propref{prop: Newton-Maclaurin}. It can be proved by continuity of $\sigma_q\bigr|_{\mathbb{R}^n}$ that \eqref{eq: Maclaurin} holds for any $\bm\mu\in\overline{\Gamma}_p$.
\end{proof}

Let $\partial\Gamma_p$ denote the boundary of $\Gamma_p$. A corollary of the Maclaurin inequality is as follows.

\begin{cor} \label{cor: partial Gamma_p}
There holds
\begin{eq} \label{eq: partial Gamma_p}
\partial\Gamma_p=\bigl\{\bm\mu\in\mathbb{R}^n\bigm|\text{$\sigma_p(\bm\mu)=0$ and $\sigma_q(\bm\mu)\geqslant 0\comma\forall q\in\{1\comma2\comma\cdots\comma p-1\}$}\bigr\}.
\end{eq}
Moreover, for any $q\in\{0\comma1\comma\cdots\comma p-1\}$ and $\bm\nu\in\partial\Gamma_p$, there holds
\begin{eq} \label{eq: lim_{mu to partial Gamma_p}}
\lim_{\Gamma_p\ni\bm\mu\to\bm\nu}\frac{\sigma_p(\bm\mu)}{\sigma_q(\bm\mu)}=0.
\end{eq}
\end{cor}

\begin{proof}
By \propref{prop: Maclaurin}, it's easy to find that $\sigma_p\bigr|_{\partial\Gamma_p}=0$. Thus \eqref{eq: partial Gamma_p} can be deduced from \propref{prop: overline{Gamma}_p}, and \eqref{eq: lim_{mu to partial Gamma_p}} can be deduced from \propref{prop: Maclaurin}.
\end{proof}

At the end of this subsection, we state a proposition with $n\geqslant p\geqslant 2$ which contains many technical inequalities used in the proof of our main results.

\begin{prop} \label{prop: sigma_{p-1}(mu|n)}
For any $\bm\mu\in\Gamma_p\colon \mu_1\leqslant\mu_2\leqslant\cdots\leqslant\mu_n$, there hold
\begin{ga}
\sum_{j=1}^{n-p+1} \mu_j>0\comma\quad (n-p)\mu_{n-p+1}>-\mu_1\comma \label{eq: sum_{j=1}^{n-p+1} mu_j>0} \\%
\mu_1>-\frac{n-p}{p(n-1)}\sum_{j=2}^n \mu_j\geqslant -\frac{n-p}{p}\mu_n \comma \label{eq: mu_1>-frac{n-p}{p(n-1)} sum_{j=2}^n mu_j} \\%
\sigma_{p-1}(\bm\mu)>\mu_{n-p+2}\mu_{n-p+3}\cdots\mu_n\comma \label{eq: sigma_{p-1}(mu), lower bound} \\%
\sigma_{p-1}(\bm\mu|1)\geqslant\sigma_{p-1}(\bm\mu|2)\geqslant\cdots\geqslant\sigma_{p-1}(\bm\mu|n)>0\comma \label{eq: sigma_{p-1}(mu|j) decreases} \\%
\mu_n\sigma_{p-1}(\bm\mu|n)\geqslant\frac pn\sigma_p(\bm\mu)\comma \label{eq: mu_n sigma_{p-1}(mu|n)} \\%
\sigma_p^{\frac1p-1}(\bm\mu)\sigma_{p-1}(\bm\mu|n)\geqslant\frac{1}{C(n\comma p)\left(\sigma_p^{\frac1p-1}(\bm\mu)\sigma_{p-1}(\bm\mu)\right)^{p-1}}\comma \label{eq: sigma_p^{1/p-1}(mu)sigma_{p-1}(mu|n)} \\%
\sum_{j=1}^n \left(\frac1p \sigma_p^{\frac1p-1}(\bm\mu)\sigma_{p-1}(\bm\mu|j)\right)\geqslant n\prod_{j=1}^n \left(\frac1p \sigma_p^{\frac1p-1}(\bm\mu)\sigma_{p-1}(\bm\mu|j)\right)^{\frac1n}\geqslant(\mathrm{C}_n^p)^{\frac1p}\comma \label{eq: product of sigma_{p-1}(mu|j), lower bound} \\%
\mu_1\geqslant-C(n\comma p)\max\left\{\sigma_p^{\frac1p}(\bm\mu)\comma\bigl(\max\left\{-\sigma_{p+1}(\bm\mu)\comma 0\right\}\bigr)^{\frac{1}{p+1}}\right\}\comma \label{eq: mu_1, lower bound}
\end{ga}
where every ``$C(n\comma p)$'' denotes some positive constant depending only on $n$ and $p$.
\end{prop}

\begin{proof}
\eqref{eq: sum_{j=1}^{n-p+1} mu_j>0} follows from \corref{cor: corollary of increasing with respect to each variable}. \eqref{eq: mu_1>-frac{n-p}{p(n-1)} sum_{j=2}^n mu_j} can be proved by combining \eqref{eq: sigma_p(mu)=mu_j sigma_{p-1}(mu|j)+sigma_p(mu|j)}, \propref{prop: increasing with respect to each variable} and the fact that
\begin{eq}
\frac{\sigma_p(\bm\mu|1)}{\sigma_{p-1}(\bm\mu|1)}\leqslant\frac{n-p}{p(n-1)}\sigma_1(\bm\mu|1)\comma
\end{eq}
which in turn follows from \propref{prop: Newton-Maclaurin} since
\begin{eq}
(\mu_2\comma\mu_3\comma\cdots\comma\mu_n)^{\mathrm T}\in\Gamma_{p-1}^{(n-1)}.
\end{eq}
\eqref{eq: sigma_{p-1}(mu), lower bound} follows from \eqref{eq: sigma_p(mu), expansion} and \corref{cor: corollary of increasing with respect to each variable}. \eqref{eq: sigma_{p-1}(mu|j) decreases} follows from \eqref{eq: sigma_{p-1}(mu|j_1)-sigma_{p-1}(mu|j_2)} and \corref{cor: corollary of increasing with respect to each variable}. Refer to \cite[pp.~557--558]{Hou2010} for the proofs of \eqref{eq: mu_n sigma_{p-1}(mu|n)} and \eqref{eq: sigma_p^{1/p-1}(mu)sigma_{p-1}(mu|n)}. Refer to e.g. \cite[pp.~27--28]{Wang1994} for the proof of \eqref{eq: product of sigma_{p-1}(mu|j), lower bound}. 

Next, we give the proof of \eqref{eq: mu_1, lower bound}, similar to that in \cite[p.~6]{Zhang2025}. We may as well assume that $\mu_1<0$. By \eqref{eq: sum_{k=1}^n mu_k^2 sigma_p(mu|k)}, \eqref{eq: sigma_{p-1}(mu|j) decreases} and \eqref{eq: sum_{k=1}^n sigma_p(mu|k)} we have
\begin{eq}
\sigma_1(\bm\mu)\sigma_p(\bm\mu)-(p+1)\sigma_{p+1}(\bm\mu)\geqslant\mu_1^2\sigma_{p-1}(\bm\mu|1)\geqslant\frac{n-p+1}{n}\mu_1^2\sigma_{p-1}(\bm\mu).
\end{eq}
In view of \eqref{eq: sigma_{p-1}(mu), lower bound} and \eqref{eq: sum_{j=1}^{n-p+1} mu_j>0}, there holds
\begin{eq}
n\mu_n\sigma_p(\bm\mu)-(p+1)\sigma_{p+1}(\bm\mu)\geqslant\frac{n-p+1}{n}\mu_1^2\mu_{n-p+2}\mu_{n-p+3}\cdots\mu_n\geqslant\frac{1}{C_1}(-\mu_1)^p\mu_n.
\end{eq}
It follows that
\begin{eq}
-(p+1)\sigma_{p+1}(\bm\mu)\geqslant\frac{1}{2C_1}(-\mu_1)^p\mu_n\geqslant\frac{1}{C_2}(-\mu_1)^{p+1}
\end{eq}
as long as
\begin{eq}
(-\mu_1)^p\geqslant 2nC_1\sigma_p(\bm\mu).
\end{eq}
We have justified \eqref{eq: mu_1, lower bound}.
\end{proof}


\subsection{Elementary spectral polynomials} \label{subsec: Elementary spectral polynomials}

In this subsection, we let $a_{jk}$ (resp. $b_{jk}$, $c_{jk}$) denote the $(j\comma k)$-th element of the matrix $\mathbf A$ (resp. $\mathbf B$, $\mathbf{C}$). Recall that $n$ is a fixed positive integer and $\bm\mu$ denotes a generic real $n$-vector throughout this section.

\begin{defn} \label{defn: lambda_q}
For any $q\in\{1\comma2\comma\cdots\comma n\}$, we define the $q$-th \emph{minimum eigenvalue function} $\lambda_q$ as follows:
\begin{eq} \label{eq: lambda_q} \begin{lgathered}
\lambda_q\colon \left\{\mathbf{A}\in\mathbb{R}^{n\times n}\middle|\text{$\mathbf{A}+\mathbf{A}^{\mathrm T}$ is positively definite}\right\}\times\mathbb{R}^{n\times n}\longrightarrow\mathbb{R}\comma \\
\mathrel{\phantom{\lambda_j\colon}}(\mathbf{A}\comma\mathbf{B})\longmapsto\text{the $q$-th minimum eigenvalue of }\frac 14\left(\mathbf{A}+\mathbf{A}^{\mathrm T}\right)\left(\mathbf{B}+\mathbf{B}^{\mathrm T}\right).
\end{lgathered} \end{eq}
\end{defn}

The theory of linear algebra shows that the real-valued function $\lambda_q$ is well-defined. When $\mathbf A$ is a positive real symmetric $n\times n$ matrix and $\mathbf B$ a real symmetric $n\times n$ matrix, $\lambda_q(\mathbf{A}\comma\mathbf{B})$ is just the $q$-th minimum eigenvalue of $\mathbf{AB}$.

\begin{lem} \label{lem: Weyl}
Assume that $\mathbf{A}$, $\mathbf{B}$, $\mathbf{C}\in\mathbb{R}^{n\times n}$ and $\mathbf{A}+\mathbf{A}^{\mathrm T}$ is positively definite. Then for any $q\in\{1\comma2\comma\cdots\comma n\}$, there holds
\begin{eq} \label{eq: Weyl}
\lambda_q(\mathbf{A}\comma\mathbf{B})+\lambda_1(\mathbf{A}\comma\mathbf{C})\leqslant\lambda_q(\mathbf{A}\comma\mathbf{B}+\mathbf{C})\leqslant\lambda_q(\mathbf{A}\comma\mathbf{B})+\lambda_n(\mathbf{A}\comma\mathbf{C}).
\end{eq}
\end{lem}

\begin{proof}
We may as well assume that $\mathbf A$, $\mathbf{B}$, $\mathbf{C}$ are all symmetric. Then there exists such $\mathbf{P}\in\mathbb{R}^{n\times n}$ that $\det(\mathbf{P})\ne 0$ and $\mathbf{A}=\mathbf{P}^{\mathrm T}\mathbf{P}$. In view of the theory of linear algebra, we have
\begin{ga}
\begin{aligned}
\lambda_q(\mathbf{A}\comma\mathbf{B}+\mathbf{C})&=\lambda_q\left(\mathbf{I}_n\comma\mathbf{P}(\mathbf{B}+\mathbf{C})\mathbf{P}^{\mathrm T}\right) \\
&\leqslant\lambda_q\left(\mathbf{I}_n\comma\mathbf{P}\mathbf{B}\mathbf{P}^{\mathrm T}\right)+\lambda_n\left(\mathbf{I}_n\comma\mathbf{P}\mathbf{C}\mathbf{P}^{\mathrm T}\right)=\lambda_q(\mathbf{A}\comma\mathbf{B})+\lambda_n(\mathbf{A}\comma\mathbf{C})\comma
\end{aligned} \\
\lambda_q(\mathbf{A}\comma\mathbf{B}+\mathbf{C})\geqslant\lambda_q\left(\mathbf{I}_n\comma\mathbf{P}\mathbf{B}\mathbf{P}^{\mathrm T}\right)+\lambda_1\left(\mathbf{I}_n\comma\mathbf{P}\mathbf{C}\mathbf{P}^{\mathrm T}\right)=\lambda_q(\mathbf{A}\comma\mathbf{B})+\lambda_1(\mathbf{A}\comma\mathbf{C})\comma
\end{ga}
where the ``$\leqslant$'' and ``$\geqslant$'' hold due to the Weyl inequalities, e.g. \cite[p.~112]{Serre2010}.
\end{proof}

Note that
\begin{eq}
\left\{\mathbf{A}\in\mathbb{R}^{n\times n}\middle|\text{$\mathbf{A}+\mathbf{A}^{\mathrm T}$ is positively definite}\right\}
\end{eq}
is an open convex set in $\mathbb{R}^{n\times n}$, and therefore $\lambda_q$ can be viewed as a continuous real-valued function defined on an open convex set in $\mathbb{R}^{2n^2}$. The following important lemma is concerned with the first-order and second-order partial derivatives of $\lambda_q$ at some special ``point'' $(\mathbf{I}_n\comma\mathbf{D})$.

\begin{lem} \label{lem: lambda_q, derivatives}
Fix $q\in\{1\comma 2\comma\cdots\comma n\}$. Let $\mathbf D$ be a real diagonal matrix $\diag\{\mu_1\comma\mu_2\comma\cdots\comma\mu_n\}$ where $\mu_1\leqslant\cdots\leqslant\mu_{q-1}<\mu_q<\mu_{q+1}\leqslant\cdots\leqslant\mu_n$ $(\mu_0\triangleq{-}\infty$, $\mu_{n+1}\triangleq{+}\infty)$.
\begin{enumerate}
\item The function $\lambda_q$ is smooth in some neighbourhood of $(\mathbf{I}_n\comma\mathbf{D})$.
\item For any $j$, $k\in\{1\comma 2\comma\cdots\comma n\}$, we have
\begin{eq} \label{eq: lambda_q, first-order derivatives}
\left.\frac{\partial\lambda_q(\mathbf{I}_n\comma\mathbf{B})}{\partial b_{jk}}\right|_{\mathbf{B}=\mathbf{D}}=\updelta_{qj}\updelta_{jk}.
\end{eq}
\item For any $j$, $k\in\{1\comma 2\comma\cdots\comma n\}$, we have
\begin{eq}
\left.\frac{\partial\lambda_q(\mathbf{A}\comma\mathbf{D})}{\partial a_{jk}}\right|_{\mathbf{A}=\mathbf{I}_n}=\mu_q\updelta_{qj}\updelta_{jk}.
\end{eq} 
\item For any $j$, $k$, $l$, $m\in\{1\comma 2\comma\cdots\comma n\}$, we have
\begin{eq} \label{eq: lambda_q, second-order derivatives}
\left.\frac{\partial^2\lambda_q(\mathbf{I}_n\comma\mathbf{B})}{\partial b_{lm}\partial b_{jk}}\right|_{\mathbf{B}=\mathbf{D}}=
\left\{ \begin{aligned}
\tfrac{1}{2(\mu_q-\mu_k)}\comma & \ \text{if $j=l=q$ and $k=m\ne j$} \\
\tfrac{1}{2(\mu_q-\mu_k)}\comma & \ \text{if $j=m=q$ and $k=l\ne j$} \\
\tfrac{1}{2(\mu_q-\mu_j)}\comma & \ \text{if $k=m=q$ and $j=l\ne k$} \\
\tfrac{1}{2(\mu_q-\mu_j)}\comma & \ \text{if $k=l=q$ and $j=m\ne k$} \\
0\comma & \ \text{otherwise}
\end{aligned} \right. .
\end{eq}
\end{enumerate}
\end{lem}

\begin{proof}
Note that
\begin{eq} \label{eq: eigenvalue equality}
\det\left(\lambda_q(\mathbf{A}\comma\mathbf{B})\mathbf{I}_n-\frac 14\left(\mathbf{A}+\mathbf{A}^{\mathrm T}\right)\left(\mathbf{B}+\mathbf{B}^{\mathrm T}\right)\right)=0.
\end{eq}
By the implicit function theorem, it's easy to find that $\lambda_q(\mathbf{A}\comma\mathbf{B})$ is a smooth function in some neighbourhood of $(\mathbf{I}_n\comma\mathbf{D})$. Then the other results can be obtained by differentiating \eqref{eq: eigenvalue equality}. By way of illustration, we provide detailed proofs of \eqref{eq: lambda_q, first-order derivatives} and \eqref{eq: lambda_q, second-order derivatives}. Straightforward calculations show that
\begin{eq} \begin{aligned}
0&=\left.\frac{\partial\det(\mathbf{C})}{\partial c_{rs}}\right|_{\mathbf{C}=\lambda_q(\mathbf{A}\comma\mathbf{B})\mathbf{I}_n-\frac 14\left(\mathbf{A}+\mathbf{A}^{\mathrm T}\right)\left(\mathbf{B}+\mathbf{B}^{\mathrm T}\right)}\frac{\partial\left(\lambda_q(\mathbf{A}\comma\mathbf{B})\updelta_{rs}-\frac 14 \sum\limits_{t=1}^n \left(a_{rt}+a_{tr}\right)\left(b_{ts}+b_{st}\right)\right)}{\partial b_{jk}} \\
&=\mathbf{C}^{r\comma s}\bigr|_{\mathbf{C}=\lambda_q(\mathbf{A}\comma\mathbf{B})\mathbf{I}_n-\frac 14\left(\mathbf{A}+\mathbf{A}^{\mathrm T}\right)\left(\mathbf{B}+\mathbf{B}^{\mathrm T}\right)}\left(\frac{\partial\lambda_q(\mathbf{A}\comma\mathbf{B})}{\partial b_{jk}}\updelta_{rs}-\frac 14\bigl(\left(a_{rj}+a_{jr}\right)\updelta_{sk}+\left(a_{rk}+a_{kr}\right)\updelta_{sj}\bigr)\right)
\end{aligned} \end{eq}
in some neighbourhood of $(\mathbf{I}_n\comma\mathbf{D})$, where the notation $\mathbf{C}^{r\comma s}$ denotes the cofactor of the matrix $\mathbf{C}$ with respect to its $(r\comma s)$-th element. Thus we have
\begin{eq} \begin{aligned}
0&=\sum_{r=1}^n \mathbf{C}^{r\comma r}\bigr|_{\mathbf{C}=\mu_q\mathbf{I}_n-\mathbf{D}}\left.\frac{\partial\lambda_q(\mathbf{I}_n\comma\mathbf{B})}{\partial b_{jk}}\right|_{\mathbf{B}=\mathbf{D}}-\frac 12\left(\mathbf{C}^{j\comma k}\bigr|_{\mathbf{C}=\mu_q\mathbf{I}_n-\mathbf{D}}+\mathbf{C}^{k\comma j}\bigr|_{\mathbf{C}=\mu_q\mathbf{I}_n-\mathbf{D}}\right) \\
&=\mathbf{C}^{q\comma q}\bigr|_{\mathbf{C}=\mu_q\mathbf{I}_n-\mathbf{D}}\left.\frac{\partial\lambda_q(\mathbf{I}_n\comma\mathbf{B})}{\partial b_{jk}}\right|_{\mathbf{B}=\mathbf{D}}-\mathbf{C}^{q\comma q}\bigr|_{\mathbf{C}=\mu_q\mathbf{I}_n-\mathbf{D}}\updelta_{jq}\updelta_{kq}\comma
\end{aligned} \end{eq}
and
\begin{eq}
\left.\frac{\partial\lambda_q(\mathbf{I}_n\comma\mathbf{B})}{\partial b_{jk}}\right|_{\mathbf{B}=\mathbf{D}}=\updelta_{qj}\updelta_{jk}
\end{eq}
since
\begin{eq}
\mathbf{C}^{q\comma q}\bigr|_{\mathbf{C}=\mu_q\mathbf{I}_n-\mathbf{D}}=\prod_{t\ne q} (\mu_q-\mu_t)\ne 0.
\end{eq}
We have justified \eqref{eq: lambda_q, first-order derivatives}, and \eqref{eq: lambda_q, second-order derivatives} remains to be proved. Differentiating \eqref{eq: eigenvalue equality} twice, we obtain
\begin{eq} \label{eq: det(lambda_q(A,B)I_n-AB), second-order derivatives} \begin{aligned}
0&=\left.\frac{\partial^2\det(\mathbf{C})}{\partial c_{\alpha\beta}\partial c_{rs}}\right|_{\mathbf{C}=\mu_q\mathbf{I}_n-\mathbf{D}}
\left(\updelta_{ql}\updelta_{lm}\updelta_{\alpha\beta}-\frac 12(\updelta_{\alpha l}\updelta_{\beta m}+\updelta_{\alpha m}\updelta_{\beta l})\right)
\left(\updelta_{qj}\updelta_{jk}\updelta_{rs}-\frac 12(\updelta_{rj}\updelta_{sk}+\updelta_{rk}\updelta_{sj})\right) \\
&\quad+\left.\frac{\partial\det(\mathbf{C})}{\partial c_{rs}}\right|_{\mathbf{C}=\mu_q\mathbf{I}_n-\mathbf{D}}
\left.\frac{\partial^2\lambda_q(\mathbf{I}_n\comma\mathbf{B})}{\partial b_{lm}\partial b_{jk}}\right|_{\mathbf{B}=\mathbf{D}}\updelta_{rs} \\
&=\sum_{r\comma\alpha=1}^n \left.\frac{\partial^2\det(\mathbf{C})}{\partial c_{\alpha\alpha}\partial c_{rr}}\right|_{\mathbf{C}=\mu_q\mathbf{I}_n-\mathbf{D}}\updelta_{qj}\updelta_{ql}\updelta_{jk}\updelta_{lm}+\mathbf{C}^{q\comma q}\bigr|_{\mathbf{C}=\mu_q\mathbf{I}_n-\mathbf{D}} \left.\frac{\partial^2\lambda_q(\mathbf{I}_n\comma\mathbf{B})}{\partial b_{lm}\partial b_{jk}}\right|_{\mathbf{B}=\mathbf{D}} \\
&\quad-\frac12 \sum_{\alpha=1}^n \left(\left.\frac{\partial^2\det(\mathbf{C})}{\partial c_{\alpha\alpha}\partial c_{jk}}\right|_{\mathbf{C}=\mu_q\mathbf{I}_n-\mathbf{D}}+\left.\frac{\partial^2\det(\mathbf{C})}{\partial c_{\alpha\alpha}\partial c_{kj}}\right|_{\mathbf{C}=\mu_q\mathbf{I}_n-\mathbf{D}}\right)\updelta_{ql}\updelta_{lm} \\
&\quad-\frac12 \sum_{r=1}^n \left(\left.\frac{\partial^2\det(\mathbf{C})}{\partial c_{lm}\partial c_{rr}}\right|_{\mathbf{C}=\mu_q\mathbf{I}_n-\mathbf{D}}+\left.\frac{\partial^2\det(\mathbf{C})}{\partial c_{ml}\partial c_{rr}}\right|_{\mathbf{C}=\mu_q\mathbf{I}_n-\mathbf{D}}\right)\updelta_{qj}\updelta_{jk} \\
&\quad+\frac14\left.\frac{\partial^2\det(\mathbf{C})}{\partial c_{\alpha\beta}\partial c_{rs}}\right|_{\mathbf{C}=\mu_q\mathbf{I}_n-\mathbf{D}}(\updelta_{\alpha l}\updelta_{\beta m}\updelta_{rj}\updelta_{sk}+\updelta_{\alpha l}\updelta_{\beta m}\updelta_{rk}\updelta_{sj}+\updelta_{\alpha m}\updelta_{\beta l}\updelta_{rj}\updelta_{sk}+\updelta_{\alpha m}\updelta_{\beta l}\updelta_{rk}\updelta_{sj}).
\end{aligned} \end{eq}
Note that
\begin{eq} \label{eq: det(C), second-order derivatives}
\left.\frac{\partial^2\det(\mathbf{C})}{\partial c_{\alpha\beta}\partial c_{rs}}\right|_{\mathbf{C}=\mu_q\mathbf{I}_n-\mathbf{D}}=
\left\{ \begin{aligned}
\textstyle \prod\limits_{t\ne q\comma\alpha} (\mu_q-\mu_t)\comma & \ \text{if $r=s=q$ and $\alpha=\beta\ne q$} \\
\textstyle \prod\limits_{t\ne q\comma r} (\mu_q-\mu_t)\comma & \ \text{if $\alpha=\beta=q$ and $r=s\ne q$} \\
\textstyle -\prod\limits_{t\ne q\comma\alpha} (\mu_q-\mu_t)\comma & \ \text{if $r=\beta=q$ and $\alpha=s\ne q$} \\
\textstyle -\prod\limits_{t\ne q\comma r} (\mu_q-\mu_t)\comma & \ \text{if $\alpha=s=q$ and $r=\beta\ne q$} \\
0\comma & \ \text{otherwise} 
\end{aligned} \right. .
\end{eq}
Combining \eqref{eq: det(lambda_q(A,B)I_n-AB), second-order derivatives} and \eqref{eq: det(C), second-order derivatives}, we can deduce that
\begin{eq} \begin{aligned}
&\mathrel{\phantom{=}}\mathbf{C}^{q\comma q}\bigr|_{\mathbf{C}=\mu_q\mathbf{I}_n-\mathbf{D}} \left.\frac{\partial^2\lambda_q(\mathbf{I}_n\comma\mathbf{B})}{\partial b_{lm}\partial b_{jk}}\right|_{\mathbf{B}=\mathbf{D}} \\
&=-\left(\sum_{r\ne q} \left.\frac{\partial^2\det(\mathbf{C})}{\partial c_{qq}\partial c_{rr}}\right|_{\mathbf{C}=\mu_q\mathbf{I}_n-\mathbf{D}}\updelta_{qj}\updelta_{ql}\updelta_{jk}\updelta_{lm}
+\sum_{\alpha\ne q} \left.\frac{\partial^2\det(\mathbf{C})}{\partial c_{\alpha\alpha}\partial c_{qq}}\right|_{\mathbf{C}=\mu_q\mathbf{I}_n-\mathbf{D}}\updelta_{qj}\updelta_{ql}\updelta_{jk}\updelta_{lm}\right) \\
&\quad+\left(\sum_{\alpha\ne q} \left.\frac{\partial^2\det(\mathbf{C})}{\partial c_{\alpha\alpha}\partial c_{qq}}\right|_{\mathbf{C}=\mu_q\mathbf{I}_n-\mathbf{D}}\updelta_{ql}\updelta_{lm}\updelta_{jk}\updelta_{qj}
+\left.\frac{\partial^2\det(\mathbf{C})}{\partial c_{qq}\partial c_{jj}}\right|_{\mathbf{C}=\mu_q\mathbf{I}_n-\mathbf{D}}\updelta_{ql}\updelta_{lm}\updelta_{jk}(1-\updelta_{qj})\right) \\
&\quad+\left(\sum_{r\ne q} \left.\frac{\partial^2\det(\mathbf{C})}{\partial c_{qq}\partial c_{rr}}\right|_{\mathbf{C}=\mu_q\mathbf{I}_n-\mathbf{D}}\updelta_{qj}\updelta_{jk}\updelta_{lm}\updelta_{ql}
+\left.\frac{\partial^2\det(\mathbf{C})}{\partial c_{ll}\partial c_{qq}}\right|_{\mathbf{C}=\mu_q\mathbf{I}_n-\mathbf{D}}\updelta_{qj}\updelta_{jk}\updelta_{lm}(1-\updelta_{ql})\right) \\
&\quad-\frac14\left.\frac{\partial^2\det(\mathbf{C})}{\partial c_{\alpha\beta}\partial c_{rs}}\right|_{\mathbf{C}=\mu_q\mathbf{I}_n-\mathbf{D}}(\updelta_{\alpha l}\updelta_{\beta m}\updelta_{rj}\updelta_{sk}+\updelta_{\alpha l}\updelta_{\beta m}\updelta_{rk}\updelta_{sj}+\updelta_{\alpha m}\updelta_{\beta l}\updelta_{rj}\updelta_{sk}+\updelta_{\alpha m}\updelta_{\beta l}\updelta_{rk}\updelta_{sj}) \\
&=\left(\left.\frac{\partial^2\det(\mathbf{C})}{\partial c_{qq}\partial c_{jj}}\right|_{\mathbf{C}=\mu_q\mathbf{I}_n-\mathbf{D}}\updelta_{ql}(1-\updelta_{qj})
+\left.\frac{\partial^2\det(\mathbf{C})}{\partial c_{ll}\partial c_{qq}}\right|_{\mathbf{C}=\mu_q\mathbf{I}_n-\mathbf{D}}\updelta_{qj}(1-\updelta_{ql})\right)\updelta_{jk}\updelta_{lm} \\
&\quad-\frac14\left.\frac{\partial^2\det(\mathbf{C})}{\partial c_{\alpha\beta}\partial c_{rs}}\right|_{\mathbf{C}=\mu_q\mathbf{I}_n-\mathbf{D}}(\updelta_{\alpha l}\updelta_{\beta m}\updelta_{rj}\updelta_{sk}+\updelta_{\alpha l}\updelta_{\beta m}\updelta_{rk}\updelta_{sj}+\updelta_{\alpha m}\updelta_{\beta l}\updelta_{rj}\updelta_{sk}+\updelta_{\alpha m}\updelta_{\beta l}\updelta_{rk}\updelta_{sj}).
\end{aligned} \end{eq}
When $j=k$, we have
\begin{eq} \begin{aligned}
&\mathrel{\phantom{=}}\mathbf{C}^{q\comma q}\bigr|_{\mathbf{C}=\mu_q\mathbf{I}_n-\mathbf{D}} \left.\frac{\partial^2\lambda_q(\mathbf{I}_n\comma\mathbf{B})}{\partial b_{lm}\partial b_{jj}}\right|_{\mathbf{B}=\mathbf{D}}+\left.\frac{\partial^2\det(\mathbf{C})}{\partial c_{ll}\partial c_{jj}}\right|_{\mathbf{C}=\mu_q\mathbf{I}_n-\mathbf{D}}\updelta_{lm}(1-\updelta_{jl}) \\
&=\left(\left.\frac{\partial^2\det(\mathbf{C})}{\partial c_{qq}\partial c_{jj}}\right|_{\mathbf{C}=\mu_q\mathbf{I}_n-\mathbf{D}}\updelta_{ql}
+\left.\frac{\partial^2\det(\mathbf{C})}{\partial c_{ll}\partial c_{qq}}\right|_{\mathbf{C}=\mu_q\mathbf{I}_n-\mathbf{D}}\updelta_{qj}\right)\updelta_{lm}(1-\updelta_{jl})
\end{aligned} \end{eq}
and therefore
\begin{eq}
\left.\frac{\partial^2\lambda_q(\mathbf{I}_n\comma\mathbf{B})}{\partial b_{lm}\partial b_{jj}}\right|_{\mathbf{B}=\mathbf{D}}=0;
\end{eq}
when $j\ne k$, we have
\begin{eq}
\mathbf{C}^{q\comma q}\bigr|_{\mathbf{C}=\mu_q\mathbf{I}_n-\mathbf{D}} \left.\frac{\partial^2\lambda_q(\mathbf{I}_n\comma\mathbf{B})}{\partial b_{lm}\partial b_{jk}}\right|_{\mathbf{B}=\mathbf{D}}=-\frac12\left.\frac{\partial^2\det(\mathbf{C})}{\partial c_{kj}\partial c_{jk}}\right|_{\mathbf{C}=\mu_q\mathbf{I}_n-\mathbf{D}}(\updelta_{jm}\updelta_{kl}+\updelta_{jl}\updelta_{km})
\end{eq}
and therefore
\begin{eq}
\left.\frac{\partial^2\lambda_q(\mathbf{I}_n\comma\mathbf{B})}{\partial b_{lm}\partial b_{jk}}\right|_{\mathbf{B}=\mathbf{D}}=
\left\{ \begin{aligned}
\tfrac{1}{2(\mu_q-\mu_k)}\comma & \ \text{if $j=m=q$ and $k=l$} \\
\tfrac{1}{2(\mu_q-\mu_k)}\comma & \ \text{if $j=l=q$ and $k=m$} \\
\tfrac{1}{2(\mu_q-\mu_j)}\comma & \ \text{if $k=l=q$ and $j=m$} \\
\tfrac{1}{2(\mu_q-\mu_j)}\comma & \ \text{if $k=m=q$ and $j=l$} \\
0\comma & \ \text{otherwise}
\end{aligned} \right. .
\end{eq}
As a result, we obtain \eqref{eq: lambda_q, second-order derivatives}. The others can be likewise proved.
\end{proof}

\begin{defn} \label{defn: lambda}
We define the \emph{eigenvalue map} $\bm\lambda$ as follows:
\begin{eq} \label{eq: lambda} \begin{gathered}
\bm\lambda\colon \left\{\mathbf{A}\in\mathbb{R}^{n\times n}\middle|\text{$\mathbf{A}+\mathbf{A}^{\mathrm T}$ is positively definite}\right\}\times\mathbb{R}^{n\times n}\longrightarrow\mathbb{R}^n\comma \\
\mathrel{\phantom{\bm\lambda\colon}}(\mathbf{A}\comma\mathbf{B})\longmapsto\bigl(\lambda_1(\mathbf{A}\comma\mathbf{B})\comma\lambda_2(\mathbf{A}\comma\mathbf{B})\comma\cdots\comma\lambda_n(\mathbf{A}\comma\mathbf{B})\bigr)^{\mathrm T}.
\end{gathered} \end{eq}
\end{defn}

When $\mathbf A$ is a positive real symmetric $n\times n$ matrix and $\mathbf B$ a real symmetric $n\times n$ matrix, $\bm\lambda(\mathbf{A}\comma\mathbf{B})$ is a vector composed of all eigenvalues of $\mathbf{AB}$ in non-decreasing order. $\bm\lambda$ can be viewed as a continuous map from an open convex set in $\mathbb{R}^{2n^2}$ to $\mathbb{R}^n$.

\begin{lem} \label{lem: If lambda(I_n, B) in overline{Gamma}_p}
Assume that $p\in\{1\comma2\comma\cdots\comma n\}$ and $\mathbf{B}\in\mathbb{R}^{n\times n}$. If $\bm\lambda(\mathbf{I}_n\comma\mathbf{B})\in\overline{\Gamma}_p$, then there hold
\begin{ga}
(b_{11}\comma b_{22}\comma\cdots\comma b_{nn})^{\mathrm T}\in\overline{\Gamma}_p\comma \label{eq: (b_{11}, b_{22}, cdots, b_{nn}) in overline{Gamma}_p} \\
\sigma_p(b_{11}\comma b_{22}\comma\cdots\comma b_{nn})\geqslant\sigma_p\bigl(\bm\lambda(\mathbf{I}_n\comma\mathbf{B})\bigr). \label{eq: sigma_p(b_{11}, b_{22}, cdots, b_{nn}) leqslant sigma_p(lambda(I_n, B))}
\end{ga}
In particular, if $\bm\lambda(\mathbf{I}_n\comma\mathbf{B})\in\Gamma_p$, then $(b_{11}\comma b_{22}\comma\cdots\comma b_{nn})^{\mathrm T}\in\Gamma_p$.
\end{lem}

\begin{proof}
We may as well assume that $\mathbf{B}$ is symmetric. Since $\overline{\Gamma}_p$ is symmetric and convex (see \corref{cor: concavity-1}), \eqref{eq: (b_{11}, b_{22}, cdots, b_{nn}) in overline{Gamma}_p} is a direct consequence of the Schur–Horn theorem, e.g. \cite[p.~624]{Horn1954}. Moreover, since the function $\sigma_p^{\frac1p}$ is symmetric and concave in $\overline{\Gamma}_p$ (see \corref{cor: concavity-1}), by the Schur–Horn theorem we find that
\begin{eq}
\sigma_p^{\frac1p}(b_{11}\comma b_{22}\comma\cdots\comma b_{nn})\geqslant\sigma_p^{\frac1p}\bigl(\bm\lambda(\mathbf{I}_n\comma\mathbf{B})\bigr)\comma
\end{eq}
which indicates \eqref{eq: sigma_p(b_{11}, b_{22}, cdots, b_{nn}) leqslant sigma_p(lambda(I_n, B))}.
\end{proof}

When $p\in\{1\comma 2\comma\cdots\comma n\}$, the theory of linear algebra shows that
\begin{eq} \label{eq: spectral function} \begin{aligned}
\sigma_p\bigl(\bm\lambda(\mathbf{A}\comma\mathbf{B})\bigr)&=\sum_{1\leqslant j_1<j_2<\dots<j_p\leqslant n}\left(\frac 14\left(\mathbf{A}+\mathbf{A}^{\mathrm T}\right)\left(\mathbf{B}+\mathbf{B}^{\mathrm T}\right)\right)\binom{j_1\comma j_2\comma\cdots\comma j_p}{j_1\comma j_2\comma\cdots\comma j_p} \\
&=\sum_{\substack{1\leqslant j_1<j_2<\dots<j_p\leqslant n \\ 1\leqslant k_1<k_2<\dots<k_p\leqslant n}}\left(\frac 12\left(\mathbf{A}+\mathbf{A}^{\mathrm T}\right)\right)\binom{j_1\comma j_2\comma\cdots\comma j_p}{k_1\comma k_2\comma\cdots\comma k_p} \\
&\mathrel{\phantom{=}}\hphantom{\sum_{\substack{1\leqslant j_1<j_2<\dots<j_p\leqslant n \\ 1\leqslant k_1<k_2<\dots<k_p\leqslant n}}}\cdot\left(\frac 12\left(\mathbf{B}+\mathbf{B}^{\mathrm T}\right)\right)\binom{k_1\comma k_2\comma\cdots\comma k_p}{j_1\comma j_2\comma\cdots\comma j_p}\comma
\end{aligned} \end{eq}
where the notation $\mathbf{C}\binom{j_1\comma j_2\comma\cdots\comma j_p}{k_1\comma k_2\comma\cdots\comma k_p}$ denotes the determinant of the submatrix composed of the rows $j_1\comma j_2\comma\cdots\comma j_p$ and columns $k_1\comma k_2\comma\cdots\comma k_p$ of the matrix $\mathbf{C}$.

\begin{defn} \label{defn: elementary spectral polynomial}
For any $p\in\{1\comma2\comma\cdots\comma n\}$, the smooth function
\begin{eq}
\sigma_p\circ\bm\lambda\colon\left\{\mathbf{A}\in\mathbb{R}^{n\times n}\middle|\text{$\mathbf{A}+\mathbf{A}^{\mathrm T}$ is positively definite}\right\}\times\mathbb{R}^{n\times n}\longrightarrow\mathbb{R}
\end{eq}
is called the $p$-th \emph{elementary spectral polynomial}.
\end{defn}

Since the $p$-th elementary spectral polynomial $\sigma_p\circ\bm\lambda$ is smooth, we are interested in its partial derivatives.

\begin{prop} \label{prop: sigma_p circ lambda, derivatives, basic}
Assume that $\mathbf{A}$, $\mathbf{C}$, $\mathbf{P}\in\mathbb{R}^{n\times n}$, $\mathbf{A}+\mathbf{A}^{\mathrm T}$ is positively definite and $\det(\mathbf{P})\ne 0$.
\begin{enumerate}
\item For any $j$, $k\in\{1\comma 2\comma\cdots\comma n\}$, there holds
\begin{eq}
\left.\frac{\partial\sigma_p\bigl(\bm\lambda(\mathbf{A}\comma\mathbf{B})\bigr)}{\partial b_{kj}}\right|_{\mathbf{B}=\mathbf{C}}=\left.\frac{\partial\sigma_p\bigl(\bm\lambda(\mathbf{A}\comma\mathbf{B})\bigr)}{\partial b_{jk}}\right|_{\mathbf{B}=\mathbf{C}^{\mathrm T}}.
\end{eq}
\item There holds
\begin{eq}
\left(\left.\frac{\partial\sigma_p\bigl(\bm\lambda(\mathbf{A}\comma\mathbf{B})\bigr)}{\partial b_{jk}}\right|_{\mathbf{B}=\mathbf{P}^{\mathrm T}\mathbf{C}\mathbf{P}}\right)=\mathbf{P}^{-1}\left(\left.\frac{\partial\sigma_p\bigl(\bm\lambda(\mathbf{P}\mathbf{A}\mathbf{P}^{\mathrm T}\comma\mathbf{B})\bigr)}{\partial b_{jk}}\right|_{\mathbf{B}=\mathbf{C}}\right)\left(\mathbf{P}^{-1}\right)^{\mathrm T}.
\end{eq}
\end{enumerate}
\end{prop}

\begin{proof}
Applying the chain rule, we have
\begin{eq} \begin{aligned}
\left.\frac{\partial\sigma_p\bigl(\bm\lambda(\mathbf{A}\comma\mathbf{B})\bigr)}{\partial b_{kj}}\right|_{\mathbf{B}=\mathbf{C}}&=\left.\frac{\partial\sigma_p\bigl(\bm\lambda(\mathbf{A}\comma\mathbf{B}^{\mathrm T})\bigr)}{\partial b_{kj}}\right|_{\mathbf{B}=\mathbf{C}} \\
&=\left.\frac{\partial\sigma_p\bigl(\bm\lambda(\mathbf{A}\comma\mathbf{B})\bigr)}{\partial b_{lm}}\right|_{\mathbf{B}=\mathbf{C}^{\mathrm T}}\left.\frac{\partial b_{ml}}{\partial b_{kj}}\right|_{\mathbf{B}=\mathbf{C}}=\left.\frac{\partial\sigma_p\bigl(\bm\lambda(\mathbf{A}\comma\mathbf{B})\bigr)}{\partial b_{jk}}\right|_{\mathbf{B}=\mathbf{C}^{\mathrm T}}\comma
\end{aligned} \end{eq}
and
\begin{eq} \begin{aligned}
\left.\frac{\partial\sigma_p\bigl(\bm\lambda(\mathbf{A}\comma\mathbf{B})\bigr)}{\partial b_{jk}}\right|_{\mathbf{B}=\mathbf{P}^{\mathrm T}\mathbf{C}\mathbf{P}}&=\left.\frac{\partial\sigma_p\Bigl(\bm\lambda\bigl(\mathbf{P}\mathbf{A}\mathbf{P}^{\mathrm T}\comma(\mathbf{P}^{\mathrm T})^{-1}\mathbf{B}\mathbf{P}^{-1}\bigr)\Bigr)}{\partial b_{jk}}\right|_{\mathbf{B}=\mathbf{P}^{\mathrm T}\mathbf{C}\mathbf{P}} \\
&=\left.\frac{\partial\sigma_p\bigl(\bm\lambda(\mathbf{P}\mathbf{A}\mathbf{P}^{\mathrm T}\comma\mathbf{B})\bigr)}{\partial b_{lm}}\right|_{\mathbf{B}=\mathbf{C}}\left.\frac{\partial p^{rl}b_{rs}p^{sm}}{\partial b_{jk}}\right|_{\mathbf{B}=\mathbf{P}^{\mathrm T}\mathbf{C}\mathbf{P}} \\
&=p^{jl}\left.\frac{\partial\sigma_p\bigl(\bm\lambda(\mathbf{P}\mathbf{A}\mathbf{P}^{\mathrm T}\comma\mathbf{B})\bigr)}{\partial b_{lm}}\right|_{\mathbf{B}=\mathbf{C}}p^{km}\comma
\end{aligned} \end{eq}
where $\left(p^{jk}\right)$ denotes $\mathbf{P}^{-1}$.
\end{proof}

The following proposition is concerned with the first-order and second-order partial derivatives of $\sigma_p\circ\bm\lambda$ at some special ``point'' $(\mathbf{I}_n\comma\mathbf{D})$.

\begin{prop} \label{prop: sigma_p circ lambda, derivatives}
Let $\mathbf D$ be a real diagonal matrix $\diag\{\mu_1\comma\mu_2\comma\cdots\comma\mu_n\}$ where $\mu_1\leqslant\mu_2\leqslant\cdots\leqslant\mu_n$.
\begin{enumerate}
\item For any $j$, $k\in\{1\comma 2\comma\cdots\comma n\}$, we have
\begin{eq}
\left.\frac{\partial\sigma_p\bigl(\bm\lambda(\mathbf{I}_n\comma\mathbf{B})\bigr)}{\partial b_{jk}}\right|_{\mathbf{B}=\mathbf{D}}=\sigma_{p-1}(\bm\mu|j)\updelta_{jk}.
\end{eq}
\item For any $j$, $k\in\{1\comma 2\comma\cdots\comma n\}$, we have
\begin{eq}
\left.\frac{\partial\sigma_p\bigl(\bm\lambda(\mathbf{A}\comma\mathbf{D})\bigr)}{\partial a_{jk}}\right|_{\mathbf{A}=\mathbf{I}_n}=\mu_j\sigma_{p-1}(\bm\mu|j)\updelta_{jk}.
\end{eq} 
\item For any $j$, $k$, $l$, $m\in\{1\comma 2\comma\cdots\comma n\}$, we have
\begin{eq} \label{eq: sigma_p circ lambda, second-order derivatives}
\left.\frac{\partial^2\sigma_p\bigl(\bm\lambda(\mathbf{I}_n\comma\mathbf{B})\bigr)}{\partial b_{lm}\partial b_{jk}}\right|_{\mathbf{B}=\mathbf{D}}=
\left\{ \begin{aligned}
\sigma_{p-2}(\bm\mu|jl)\comma & \ \text{if $j=k$ and $l=m\ne j$} \\
-\tfrac{1}{2}\sigma_{p-2}(\bm\mu|jk)\comma & \ \text{if $j=l$ and $k=m\ne j$} \\
-\tfrac{1}{2}\sigma_{p-2}(\bm\mu|jk)\comma & \ \text{if $j=m$ and $k=l\ne j$} \\
0\comma & \ \text{otherwise}
\end{aligned} \right. .
\end{eq}
\end{enumerate}
\end{prop}

\begin{proof}
We may as well assume that $\mu_1<\mu_2<\cdots<\mu_n$ --- otherwise, we apply the perturbation method, i.e. replace each $\mu_j$ with $\mu_j+j\varepsilon$ ($\varepsilon>0$) at first and let $\varepsilon$ tend to $0^+$ in the end. By \lemref{lem: lambda_q, derivatives} we find that the functions $\lambda_1$, $\lambda_2$, $\cdots$, $\lambda_n$ are all smooth in some neighbourhood of $(\mathbf{I}_n\comma\mathbf{D})$. Then the results can be obtained by applying the chain rule, \lemref{lem: lambda_q, derivatives}, \eqref{eq: sigma_p, derivatives} and \eqref{eq: sigma_{p-1}(mu|j_1)-sigma_{p-1}(mu|j_2)}. By way of illustration, we provide a detailed proof of \eqref{eq: sigma_p circ lambda, second-order derivatives}. Straightforward calculations show that
\begin{eq} \begin{aligned}
&\mathrel{\phantom{=}}\frac{\partial^2\sigma_p\bigl(\bm\lambda(\mathbf{A}\comma\mathbf{B})\bigr)}{\partial b_{lm}\partial b_{jk}} \\
&=\left.\frac{\partial^2\sigma_p(\bm\nu)}{\partial\nu_r\partial\nu_q}\right|_{\bm\nu=\bm\lambda(\mathbf{A}\comma\mathbf{B})}\frac{\partial\lambda_r(\mathbf{A}\comma\mathbf{B})}{\partial b_{lm}}\frac{\partial\lambda_q(\mathbf{A}\comma\mathbf{B})}{\partial b_{jk}}+\left.\frac{\partial\sigma_p(\bm\nu)}{\partial\nu_q}\right|_{\bm\nu=\bm\lambda(\mathbf{A}\comma\mathbf{B})}\frac{\partial^2\lambda_q(\mathbf{A}\comma\mathbf{B})}{\partial b_{lm}\partial b_{jk}} \\
&=\sigma_{p-2}\bigl(\bm\lambda(\mathbf{A}\comma\mathbf{B})\bigm|qr\bigr)\frac{\partial\lambda_r(\mathbf{A}\comma\mathbf{B})}{\partial b_{lm}}\frac{\partial\lambda_q(\mathbf{A}\comma\mathbf{B})}{\partial b_{jk}}+\sigma_{p-1}\bigl(\bm\lambda(\mathbf{A}\comma\mathbf{B})\bigm|q\bigr)\frac{\partial^2\lambda_q(\mathbf{A}\comma\mathbf{B})}{\partial b_{lm}\partial b_{jk}}
\end{aligned} \end{eq}
in some neighbourhood of $(\mathbf{I}_n\comma\mathbf{D})$. Thus we have
\begin{eq} \label{eq: sigma_p circ lambda, second-order derivatives, intermediate step} \begin{aligned}
&\mathrel{\phantom{=}}\left.\frac{\partial^2\sigma_p\bigl(\bm\lambda(\mathbf{I}_n\comma\mathbf{B})\bigr)}{\partial b_{lm}\partial b_{jk}}\right|_{\mathbf{B}=\mathbf{D}} \\
&=\sigma_{p-2}(\bm\mu|jl)\updelta_{lm}\updelta_{jk}+\sigma_{p-1}(\bm\mu|j)\left.\frac{\partial^2\lambda_j(\mathbf{I}_n\comma\mathbf{B})}{\partial b_{lm}\partial b_{jk}}\right|_{\mathbf{B}=\mathbf{D}}+\sigma_{p-1}(\bm\mu|k)\left.\frac{\partial^2\lambda_k(\mathbf{I}_n\comma\mathbf{B})}{\partial b_{lm}\partial b_{jk}}\right|_{\mathbf{B}=\mathbf{D}}.
\end{aligned} \end{eq}
When $j=k$, the right-hand side of \eqref{eq: sigma_p circ lambda, second-order derivatives, intermediate step} equals to
\begin{eq}
\sigma_{p-2}(\bm\mu|jl)\updelta_{lm}=
\left\{ \begin{aligned}
\sigma_{p-2}(\bm\mu|jl)\comma & \ \text{if $l=m\ne j$} \\
0\comma & \ \text{otherwise}
\end{aligned} \right. ;
\end{eq}
when $j\ne k$, the right-hand side of \eqref{eq: sigma_p circ lambda, second-order derivatives, intermediate step} equals to
\begin{eq} \begin{aligned}
&\mathrel{\phantom{=}}\sigma_{p-1}(\bm\mu|j)\frac{1}{2(\mu_j-\mu_k)}(\updelta_{jl}\updelta_{km}+\updelta_{jm}\updelta_{kl})+\sigma_{p-1}(\bm\mu|k)\frac{1}{2(\mu_k-\mu_j)}(\updelta_{jl}\updelta_{km}+\updelta_{jm}\updelta_{kl}) \\
&=\left\{ \begin{aligned}
-\tfrac12\sigma_{p-2}(\bm\mu|jk)\comma & \ \text{if $j=l$ and $k=m$} \\
-\tfrac12\sigma_{p-2}(\bm\mu|jk)\comma & \ \text{if $j=m$ and $k=l$} \\
0\comma & \ \text{otherwise}
\end{aligned} \right. .
\end{aligned} \end{eq}
As a result, we obtain \eqref{eq: sigma_p circ lambda, second-order derivatives}. The others can be likewise proved.
\end{proof}

The following proposition is similar to a special case of the main result in \cite{Davis1957}.

\begin{prop} \label{prop: concavity-2}
Assume that $p\in\{1\comma2\comma\cdots\comma n\}$, $\mathbf{A}\in\mathbb{R}^{n\times n}$ and $\mathbf{A}+\mathbf{A}^{\mathrm T}$ is positively definite. Then
\begin{eq} \label{eq: lambda^{-1}(Gamma_p)}
\left\{\mathbf{B}\in\mathbb{R}^{n\times n}\middle|\bm\lambda(\mathbf{A}\comma\mathbf{B})\in\Gamma_p\right\}
\end{eq}
is an open convex set in $\mathbb{R}^{n\times n}$, and the function $\sigma_p^{\frac1p}\circ\bm\lambda\left(\mathbf{A}\comma\cdot\right)$ is concave in the closure of the above set, i.e. in the closed convex set
\begin{eq} \label{eq: lambda^{-1}(overline{Gamma}_p)}
\left\{\mathbf{B}\in\mathbb{R}^{n\times n}\middle|\bm\lambda(\mathbf{A}\comma\mathbf{B})\in\overline{\Gamma}_p\right\}.
\end{eq}
\end{prop}

\begin{proof}
We may as well assume that $\mathbf A$ is symmetric, and therefore there exists such $\mathbf{P}\in\mathbb{R}^{n\times n}$ that $\det(\mathbf{P})\ne 0$ and $\mathbf{A}=\mathbf{P}^{\mathrm T}\mathbf{P}$. Since $\bm\lambda$ is continuous and $\Gamma_p$ is open, obviously the set \eqref{eq: lambda^{-1}(Gamma_p)} is open and the set \eqref{eq: lambda^{-1}(overline{Gamma}_p)} is closed. Note that the set \eqref{eq: lambda^{-1}(overline{Gamma}_p)} is indeed the closure of the set \eqref{eq: lambda^{-1}(Gamma_p)} since for any $\mathbf{B}$ in the set \eqref{eq: lambda^{-1}(overline{Gamma}_p)} and $\varepsilon>0$, there holds
\begin{eq}
\bm\lambda\left(\mathbf{A}\comma\mathbf{B}+\varepsilon\mathbf{A}^{-1}\right)=\bm\lambda(\mathbf{A}\comma\mathbf{B})+\varepsilon\bm{1}_n\in\Gamma_p
\end{eq}
by the proof of \propref{prop: overline{Gamma}_p}. Let $\mathbf{B}$ and $\mathbf{C}$ be real symmetric $n\times n$ matrices in the set \eqref{eq: lambda^{-1}(overline{Gamma}_p)}. For any $t\in(0\comma1)$, there exists such an orthogonal $n\times n$ matrix $\mathbf{Q}$ that
\begin{eq}
(1-t)\mathbf{Q}^{\mathrm T}\mathbf{P}\mathbf{B}\mathbf{P}^{\mathrm T}\mathbf{Q}+t\mathbf{Q}^{\mathrm T}\mathbf{P}\mathbf{C}\mathbf{P}^{\mathrm T}\mathbf{Q}
\end{eq}
is a diagonal matrix. Note that
\begin{al}
\bm\lambda\left(\mathbf{I}_n\comma\mathbf{Q}^{\mathrm T}\mathbf{P}\mathbf{B}\mathbf{P}^{\mathrm T}\mathbf{Q}\right)&=\bm\lambda\left(\mathbf{I}_n\comma\mathbf{P}\mathbf{B}\mathbf{P}^{\mathrm T}\right)=\bm\lambda(\mathbf{A}\comma\mathbf{B})\in\overline{\Gamma}_p\comma \\
\bm\lambda\left(\mathbf{I}_n\comma \mathbf{Q}^{\mathrm T}\mathbf{P}\mathbf{C}\mathbf{P}^{\mathrm T}\mathbf{Q}\right)&=\bm\lambda\left(\mathbf{I}_n\comma\mathbf{P}\mathbf{C}\mathbf{P}^{\mathrm T}\right)=\bm\lambda(\mathbf{A}\comma\mathbf{C})\in\overline{\Gamma}_p.
\end{al}
Since $\overline{\Gamma}_p$ is symmetric and convex (see \corref{cor: concavity-1}), by \lemref{lem: If lambda(I_n, B) in overline{Gamma}_p} we have
\begin{eq}
\bm\lambda\bigl(\mathbf{A}\comma(1-t)\mathbf{B}+t\mathbf{C}\bigr)=\bm\lambda\left(\mathbf{I}_n\comma(1-t)\mathbf{Q}^{\mathrm T}\mathbf{P}\mathbf{B}\mathbf{P}^{\mathrm T}\mathbf{Q}+t\mathbf{Q}^{\mathrm T}\mathbf{P}\mathbf{C}\mathbf{P}^{\mathrm T}\mathbf{Q}\right)\in\overline{\Gamma}_p.
\end{eq}
This indicates that the set \eqref{eq: lambda^{-1}(overline{Gamma}_p)} is convex, so is the set \eqref{eq: lambda^{-1}(Gamma_p)}. Moreover, since the function $\sigma_p^{\frac1p}$ is concave in $\overline{\Gamma}_p$ (see \corref{cor: concavity-1}), by \lemref{lem: If lambda(I_n, B) in overline{Gamma}_p} we have
\begin{eq} \begin{aligned}
&\mathrel{\hphantom{=}}\sigma_p^{\frac1p}\circ\bm\lambda\bigl(\mathbf{A}\comma(1-t)\mathbf{B}+t\mathbf{C}\bigr) \\
&=\sigma_p^{\frac1p}\circ\bm\lambda\left(\mathbf{I}_n\comma(1-t)\mathbf{Q}^{\mathrm T}\mathbf{P}\mathbf{B}\mathbf{P}^{\mathrm T}\mathbf{Q}+t\mathbf{Q}^{\mathrm T}\mathbf{P}\mathbf{C}\mathbf{P}^{\mathrm T}\mathbf{Q}\right) \\
&\geqslant(1-t)\sigma_p^{\frac1p}\circ\bm\lambda\left(\mathbf{I}_n\comma\mathbf{Q}^{\mathrm T}\mathbf{P}\mathbf{B}\mathbf{P}^{\mathrm T}\mathbf{Q}\right)+t\sigma_p^{\frac1p}\circ\bm\lambda\left(\mathbf{I}_n\comma\mathbf{Q}^{\mathrm T}\mathbf{P}\mathbf{C}\mathbf{P}^{\mathrm T}\mathbf{Q}\right) \\
&=(1-t)\sigma_p^{\frac1p}\circ\bm\lambda(\mathbf{A}\comma\mathbf{B})+t\sigma_p^{\frac1p}\circ\bm\lambda(\mathbf{A}\comma\mathbf{C}).
\end{aligned} \end{eq}
This indicates that the function $\sigma_p^{\frac1p}\circ\bm\lambda\left(\mathbf{A}\comma\cdot\right)$ is concave in the set \eqref{eq: lambda^{-1}(overline{Gamma}_p)}.
\end{proof}

When $\mathbf A$ is a positive real symmetric $n\times n$ matrix and $\mathbf B$ is a real symmetric $n\times n$ matrix, the notations $\lambda_j(\mathbf{A}\comma\mathbf{B})$, $\bm\lambda(\mathbf{A}\comma\mathbf{B})$, $\sigma_p\bigl(\bm\lambda(\mathbf{A}\comma\mathbf{B})\bigr)$, $\lambda_j(\mathbf{I}_n\comma\mathbf{B})$, $\bm\lambda(\mathbf{I}_n\comma\mathbf{B})$ and $\sigma_p\bigl(\bm\lambda(\mathbf{I}_n\comma\mathbf{B})\bigr)$ may be abbreviated to $\lambda_j(\mathbf{AB})$, $\bm\lambda(\mathbf{AB})$, $\sigma_p(\mathbf{AB})$, $\lambda_j(\mathbf{B})$, $\bm\lambda(\mathbf{B})$ and $\sigma_p(\mathbf{B})$ respectively, as long as this causes no ambiguity.


\subsection{Hessian equations on Riemannian manifolds} \label{subsec: Hessian equations on compact Riemannian manifolds}

In this subsection, we give some preliminaries to Hessian equations on Riemannian manifolds. First, we sketch out some basic concepts and facts of Riemannian geometry. Let $(\mathcal{M}\comma\bm{g})$ be a Riemannian manifold of dimension $n$ with the smooth structure $\left\{(\mathcal{U}\comma\bm{\psi}_{\mathcal{U}}\mbox{; }x^j)\right\}$ and the Riemannian metric
\begin{eq}
\bm{g}\triangleq g_{jk}\dif x^j\otimes\dif x^k
\end{eq}
in any local coordinate system $(\mathcal{U}\comma\bm{\psi}_{\mathcal{U}}\mbox{; }x^j)$, where $(g_{jk})$ is symmetric and positively definite. The volumn form of $\bm{g}$ is by definition
\begin{eq} \label{eq: dV_g}
\dif V_{\bm g}\triangleq\bigl(\det(g_{jk})\bigr)^{\frac12}\dif x^1\wedge\dif x^2\wedge\cdots\wedge\dif x^n.
\end{eq}
Let $\left(g^{jk}\right)$ denote $(g_{jk})^{-1}$. Then there hold
\begin{ga}
g^{jl}g_{kl}=\updelta_{jk}\comma \quad \mathrm{D}_lg^{jk}=-g^{js}g^{rk}\mathrm{D}_lg_{rs}\comma  \label{eq: g^{jk}, first-order} \\
\mathrm{D}_{lm}g^{jk}=-g^{js}g^{rk}\mathrm{D}_{lm}g_{rs}+g^{jq}g^{ps}g^{rk}\mathrm{D}_lg_{rs}\mathrm{D}_mg_{pq}+g^{js}g^{rq}g^{pk}\mathrm{D}_lg_{rs}\mathrm{D}_mg_{pq}. \label{eq: g^{jk}, second-order}
\end{ga}
Note that ``$\mathrm{D}$'' denotes the ordinary derivative instead of the covariant derivative, e.g.
\begin{eq}
\mathrm{D}_{lm}g^{jk}\triangleq\frac{\partial g^{jk}}{\partial x^m\partial x^l}.
\end{eq}
Let $\nabla$ denote the Levi-Civita connection of $(\mathcal{M}\comma\bm g)$. By definition, $\nabla$ is a torsion-free connection on the tangent bundle $\mathrm{T}\mathcal{M}$ compatible with $\bm g$. Let $\mathscr{R}$ denote the curvature operator of $\nabla$, namely for any smooth sections $\bm{X}$, $\bm{Y}$, $\bm{Z}$ of $\mathrm{T}\mathcal{M}$,
\begin{eq}
\mathscr{R}(\bm{X}\comma\bm{Y})\bm{Z}\triangleq\nabla_{\bm{X}}\nabla_{\bm{Y}}\bm{Z}-\nabla_{\bm{Y}}\nabla_{\bm{X}}\bm{Z}-\nabla_{[\bm{X}\comma\bm{Y}]}\bm{Z}.
\end{eq}
The curvature tensor field $\bm{R}$ of $\nabla$ is defined as follows: for any smooth sections $\bm{X}$, $\bm{Y}$, $\bm{Z}$, $\bm{W}$ of $\mathrm{T}\mathcal{M}$,
\begin{eq} \label{eq: curvature tensor field}
\bm{R}(\bm{X}\comma\bm{Y}\comma\bm{Z}\comma\bm{W})\triangleq\bm{g}\bigl(\mathscr{R}(\bm{Z}\comma\bm{W})\bm{X}\comma\bm{Y}\bigr).
\end{eq}
It's well-known that
\begin{ga}
\bm{R}(\bm{Y}\comma\bm{X}\comma\bm{Z}\comma\bm{W})=\bm{R}(\bm{X}\comma\bm{Y}\comma\bm{W}\comma\bm{Z})=-\bm{R}(\bm{X}\comma\bm{Y}\comma\bm{Z}\comma\bm{W})\comma \\
\bm{R}(\bm{Z}\comma\bm{W}\comma\bm{X}\comma\bm{Y})=\bm{R}(\bm{X}\comma\bm{Y}\comma\bm{Z}\comma\bm{W}) \label{eq: R(Z, W, X, Y)=R(X, Y, Z, W)}.
\end{ga}
In the local coordinate system $(\mathcal{U}\comma\bm{\psi}_{\mathcal{U}}\mbox{; }x^j)$, there hold
\begin{ga}
\nabla\frac{\partial}{\partial x^j}\triangleq\Gamma_{jk}^l\frac{\partial}{\partial x^l}\otimes\dif x^k\comma\quad \Gamma_{jk}^l=\frac12 g^{lr}(\mathrm{D}_kg_{jr}+\mathrm{D}_jg_{rk}-\mathrm{D}_rg_{jk})\comma \label{eq: connection} \\
\Gamma_{kj}^l=\Gamma_{jk}^l\comma\quad \mathrm{D}_lg_{jk}=\Gamma_{jl}^rg_{rk}+\Gamma_{kl}^rg_{jr}\comma\quad \mathrm{D}_lg^{jk}=-\Gamma_{rl}^jg^{rk}-\Gamma_{rl}^kg^{jr}. \label{eq: properties of connection}
\end{ga}
and
\begin{eq} \label{eq: R_{jklm}} \begin{aligned}
R_{jklm}&=(\mathrm{D}_l\Gamma_{jm}^r-\mathrm{D}_m\Gamma_{jl}^r+\Gamma_{jm}^s\Gamma_{sl}^r-\Gamma_{jl}^s\Gamma_{sm}^r)g_{rk} \\
&=\frac12(\mathrm{D}_{km}g_{jl}+\mathrm{D}_{jl}g_{km}-\mathrm{D}_{kl}g_{jm}-\mathrm{D}_{jm}g_{kl})+\Gamma_{jl}^r\Gamma_{km}^sg_{rs}-\Gamma_{jm}^r\Gamma_{kl}^sg_{rs}.
\end{aligned} \end{eq}
Now we turn our attention to Hessian equations. For any $u\in\mathrm{C}^2(\mathcal M)$, let $\dif u$ and $\nabla^2u$ denote the differential and the Hessian of $u$ respectively, i.e.
\begin{eq} \label{eq: du and nabla^2u}
\dif u=\mathrm{D}_ju\dif x^j\comma\quad \nabla^2u=\nabla_{jk}u\dif x^j\otimes\dif x^k\comma\quad \nabla_{jk}u\triangleq\mathrm{D}_{jk}u-\Gamma_{jk}^l\mathrm{D}_lu
\end{eq}
in any local coordinate system $(\mathcal{U}\comma\bm{\psi}_{\mathcal{U}}\mbox{; }x^j)$. Indeed, $\nabla^2u$ is the second-order covariant differentiation of $u$. Moreover, $\nabla^2u$ is a symmetric $(0\comma2)$-tensor field on $\mathcal M$, so is $\bm{g}$, and we let $\bm{A}(\dif u\comma u)$ be another one --- to be precise,
\begin{ga}
\bm{A}(\dif u\comma u)\colon\mathcal{M}\longrightarrow\mathrm{T}^*\mathcal{M}\mathrm{T}^*\mathcal{M}\comma\bm{x}\longmapsto \bm{A}\bigl(\dif u(\bm{x})\comma u(\bm{x})\bigr)\comma \label{eq: A(du, u)} \\
\bm{A}(\bm{\alpha}\comma t)\in\mathrm{C}^\infty\left(\mathrm{T}^*\mathcal{M}\times\mathbb{R}\comma\mathrm{T}^*\mathcal{M}\mathrm{T}^*\mathcal{M}\right)\colon\tilde{\pi}\bigl(\bm{A}(\bm{\alpha}\comma t)\bigr)=\pi(\bm{\alpha})\comma \label{eq: A(v, t), definition}
\end{ga}
where $\mathrm{T}^*\mathcal{M}$ denotes the cotangent bundle on $\mathcal{M}$, $\mathrm{T}^*\mathcal{M}\mathrm{T}^*\mathcal{M}$ denotes the symmetric tensor product of $\mathrm{T}^*\mathcal{M}$ and itself, i.e. the symmetric $(0\comma2)$-tensor bundle on $\mathcal{M}$, and $\pi$ (resp. $\tilde{\pi}$) denotes the canonical projection of $\mathrm{T}^*\mathcal{M}$ (resp. $\mathrm{T}^*\mathcal{M}\mathrm{T}^*\mathcal{M}$) onto $\mathcal{M}$. \eqref{eq: A(v, t), definition} implies for any local coordinate system $(\mathcal{U}\comma\bm{\psi}_{\mathcal{U}}\mbox{; }x^j)$ that
\begin{ga}
\bm{A}(\bm{\alpha}\comma t)\bigr|_{\mathrm{T}^*\mathcal{U}\times\mathbb{R}}=A_{jk}(\bm{\alpha}\comma t)\dif x^j\bigr|_{\pi(\bm\alpha)}\otimes\dif x^k\bigr|_{\pi(\bm\alpha)}\comma \label{eq: A(v, t)} \\
A_{jk}(\bm{\alpha}\comma t)\in\mathrm{C}^\infty\left(\mathrm{T}^*\mathcal{U}\times\mathbb{R}\right)\comma\quad A_{jk}(\bm{\alpha}\comma t)=A_{kj}(\bm{\alpha}\comma t). \label{eq: A_{jk}(v, t)}
\end{ga}
The moduli of $\bm{\alpha}\in\mathrm{T}^*\mathcal{M}$ and $\bm{B}\in\mathrm{T}^*\mathcal{M}\mathrm{T}^*\mathcal{M}$ are defined as follows:
\begin{eq} \label{eq: moduli of alpha and B}
\left|\bm{\alpha}\right|_{\bm g}\triangleq\Bigl(g^{jk}\bigl(\pi(\bm{\alpha})\bigr)\alpha_j\alpha_k\Bigr)^{\frac12}\comma\quad \left|\bm{B}\right|_{\bm g}\triangleq\Bigl(g^{jm}\bigl(\tilde{\pi}(\bm{B})\bigr)g^{lk}\bigl(\tilde{\pi}(\bm{B})\bigr)B_{jk}B_{lm}\Bigr)^{\frac12}.
\end{eq}
In particular, we have
\begin{eq} \label{eq: moduli of du and nabla^2u}
\left|\dif u\right|_{\bm g}=\left(g^{jk}\mathrm{D}_ju\mathrm{D}_ku\right)^{\frac12}\comma\quad \left|\nabla^2u\right|_{\bm g}=\left(g^{jm}g^{lk}\nabla_{jk}u\nabla_{lm}u\right)^{\frac12}.
\end{eq}
Let $\bm\lambda\Bigl({\bm g}^{-1}\bigl(\bm{A}(\dif u\comma u)+\nabla^2u\bigr)\Bigr)$ denote the $\mathbb{R}^n$-valued function
\begin{eq} \label{eq: lambda(g^{-1}(A(du, u)+nabla^2u))}
\bm\lambda\Bigl(\left(g^{jk}\right)\bigl(A_{jk}(\dif u\comma u)+\nabla_{jk}u\bigr)\Bigr)
\end{eq}
well-defined on $\mathcal M$ (recall \defnref{defn: lambda} and the last paragraph before \subsecref{subsec: Hessian equations on compact Riemannian manifolds}).

For the rest of this subsection, we focus on the case when $\mathcal M$ is a closed manifold (i.e. compact manifold without boundary) for simplicity. Fix $p\in\{1\comma2\comma\cdots\comma n\}$. The $p$-Hessian equation with general left-hand and right-hand items on $(\mathcal{M}\comma\bm{g})$ is as follows (recall \defnref{defn: elementary spectral polynomial}):
\begin{eq} \label{eq: p-Hessian equation}
\sigma_p^{\frac 1p}\biggl(\bm\lambda\Bigl({\bm g}^{-1}\bigl(\bm{A}(\dif u\comma u)+\nabla^2u\bigr)\Bigr)(\bm{x})\biggr)=\varphi\bigl(\dif u(\bm{x})\comma u(\bm{x})\bigr)\comma\forall\bm{x}\in\mathcal{M}\comma
\end{eq}
where $\varphi(\bm{\alpha}\comma t)$ is a positive smooth function on $\mathrm{T}^*\mathcal{M}\times\mathbb{R}$. When $p=1$, \eqref{eq: p-Hessian equation} is of Laplace type; when $p=n$, \eqref{eq: p-Hessian equation} is of Monge-Amp\`ere type. We call $u\in\mathrm{C}^2(\mathcal{M})$ a general $p$-admissible function if
\begin{eq} \label{eq: general p-admissible}
\bm\lambda\Bigl({\bm g}^{-1}\bigl(\bm{A}(\dif u\comma u)+\nabla^2u\bigr)\Bigr)(\bm{x})\in\Gamma_p\comma\forall\bm{x}\in\mathcal{M}\comma
\end{eq}
recall \defnref{defn: Garding cone}. A $\mathrm{C}^2$ solution $u$ to the general $p$-Hessian equation \eqref{eq: p-Hessian equation} is called a $p$-admissible solution if it satisfies \eqref{eq: general p-admissible}. We are mainly interested in the $p$-admissible solutions of \eqref{eq: p-Hessian equation} because by the following proposition, \eqref{eq: p-Hessian equation} is an elliptic equation with respect to any general $p$-admissible function $u$ in the sense that for any $\bm{x}\in\mathcal{M}$, the matrix $(F_{u\comma\bm x}^{jk})$ with
\begin{eq} \label{eq: F_{u, x}^{jk}}
F_{u\comma\bm x}^{jk}\triangleq\left.\frac{\partial}{\partial b_{jk}}\sigma_p^{\frac 1p}\biggl(\bm\lambda\Bigl(\bigl(g^{jk}(\bm{x})\bigr)\comma\mathbf{B}\Bigr)\biggr)\right|_{\mathbf{B}=\left(A_{jk}(\dif u\comma u)+\nabla_{jk}u\right)(\bm{x})}
\end{eq}
is positively definite\footnote{It's independent of the choice of local coordinate systems whether $\left(F_{u\comma\bm x}^{jk}\right)$ is a real symmetric positively definite matrix.}.

\begin{prop} \label{prop: properties of F_{u, x}^jk}
Let $u\in\mathrm{C}^2(\mathcal M)$ be a general $p$-admissible function.
\begin{enumerate}
\item Assume that $g_{jk}(\bm{x}_0)=\updelta_{jk}$ and $A_{jk}\bigl(\dif u(\bm{x}_0)\comma u(\bm{x}_0)\bigr)+\nabla_{jk}u(\bm{x}_0)=\mu_j\updelta_{jk}$ for a point $\bm{x}_0\in\mathcal M$, a local coordinate system $(\mathcal{U}\comma\bm{\psi}_{\mathcal{U}}\mbox{; }x^j)$ containing $\bm{x}_0$ and a vector $\bm\mu\in\mathbb{R}^n$. Then we have
\begin{eq}
F_{u\comma\bm{x}_0}^{jk}=\frac1p\sigma_p^{\frac1p-1}(\bm\mu)\sigma_{p-1}(\bm\mu|j)\updelta_{jk}.
\end{eq}
\item For any $\bm{x}\in\mathcal{M}$, $F_{u\comma\bm x}^{jk}$ is the component of a symmetric $(2\comma0)$-tensor at $\bm x$. 
\end{enumerate}
\end{prop}

\begin{proof}
(1) follows from \propref{prop: sigma_p circ lambda, derivatives}. (2) follows from \propref{prop: sigma_p circ lambda, derivatives, basic} and the fact that $g^{jk}$ is the component of a $(2\comma0)$-tensor at $\bm x$.
\end{proof}

Next, we outline the geometry of $\mathrm{T}^*\mathcal{M}$. Let $(\mathrm{T}^*\mathcal{U}\comma\bm{\Psi}_{\mathrm{T}^*\mathcal{U}}\mbox{; }\tilde{x}^j\comma\tilde\alpha_j)$ be the canonical coordinate system on $\mathrm{T}^*\mathcal{U}=\pi^{-1}(\mathcal{U})$ with respect to a local coordinate system $(\mathcal{U}\comma\bm{\psi}_{\mathcal{U}}\mbox{; }x^j)$ of $\mathcal{M}$, namely
\begin{ga}
\bm{\Psi}_{\mathrm{T}^*\mathcal{U}}\colon\mathrm{T}^*\mathcal{U}\longrightarrow\bm\psi_{\mathcal{U}}(\mathcal{U})\times\mathbb{R}^n\comma\bm\alpha\longmapsto\bigl(\tilde{x}^1(\bm\alpha)\comma\cdots\comma\tilde{x}^n(\bm\alpha)\comma\tilde\alpha_1(\bm\alpha)\comma\cdots\comma\tilde\alpha_n(\bm\alpha)\bigr)^{\mathrm T}\comma \\
\tilde{x}^j(\bm\alpha)\triangleq x^j\bigl(\pi(\bm\alpha)\bigr)\comma\quad\tilde\alpha_j(\bm\alpha)\triangleq\bm\alpha\left(\left.\frac{\partial}{\partial x^j}\right|_{\pi(\bm\alpha)}\right)\comma \label{eq: tilde x^j, tilde alpha_j} \\
\bm{\Psi}_{\mathrm{T}^*\mathcal{U}}^{-1}\colon\bm\psi_{\mathcal{U}}(\mathcal{U})\times\mathbb{R}^n\longrightarrow\mathrm{T}^*\mathcal{U}\comma(\tilde{x}^1\comma\cdots\comma\tilde{x}^n\comma\tilde\alpha_1\comma\cdots\comma\tilde\alpha_n)^{\mathrm T}\longmapsto\tilde\alpha_j\dif x^j\bigr|_{\bm{\psi}_{\mathcal{U}}^{-1}(\tilde{x}^1\comma\cdots\comma\tilde{x}^n)}.
\end{ga}
For any two joint local coordinate systems $(\mathcal{U}\comma\bm{\psi}_{\mathcal{U}}\mbox{; }x^j)$, $(\mathcal{V}\comma\bm{\psi}_{\mathcal{V}}\mbox{; }y^j)$ of $\mathcal{M}$, there holds
\begin{eq} \begin{gathered}
\bm{\Psi}_{\mathrm{T}^*\mathcal{U}}\circ\bm{\Psi}_{\mathrm{T}^*\mathcal{V}}^{-1}\colon\bm\Psi_{\mathrm{T}^*\mathcal{V}}(\mathrm{T}^*\mathcal{U}\cap\mathrm{T}^*\mathcal{V})\longrightarrow\bm\Psi_{\mathrm{T}^*\mathcal{U}}(\mathrm{T}^*\mathcal{U}\cap\mathrm{T}^*\mathcal{V})\comma \\
(\tilde{y}^1\comma\cdots\comma\tilde{y}^n\comma\tilde\beta_1\comma\cdots\comma\tilde\beta_n)^{\mathrm T}\longmapsto%
\begin{pmatrix}
\bm{\psi}_{\mathcal{U}}\circ\bm{\psi}_{\mathcal{V}}^{-1}(\tilde{y}^1\comma\cdots\comma\tilde{y}^n) \\
\Bigl(\bigl(\mathrm{J}(\bm{\psi}_{\mathcal{V}}\circ\bm{\psi}_{\mathcal{U}}^{-1})\bigr)(\tilde{y}^1\comma\cdots\comma\tilde{y}^n)\Bigr)^{\mathrm T}\tilde{\bm\beta}
\end{pmatrix}\comma
\end{gathered} \end{eq}
where $(\mathrm{T}^*\mathcal{U}\comma\bm{\Psi}_{\mathrm{T}^*\mathcal{U}}\mbox{; }\tilde{x}^j\comma\tilde\alpha_j)$ and $(\mathrm{T}^*\mathcal{V}\comma\bm{\Psi}_{\mathrm{T}^*\mathcal{V}}\mbox{; }\tilde{y}^j\comma\tilde\beta_j)$ are the canonical coordinate systems on $\mathrm{T}^*\mathcal{U}$ and $\mathrm{T}^*\mathcal{V}$ respectively, $\tilde{\bm\beta}\triangleq(\tilde\beta_1\comma\cdots\comma\tilde\beta_n)^{\mathrm T}$, and $\mathrm{J}(\bm{\psi}_{\mathcal{V}}\circ\bm{\psi}_{\mathcal{U}}^{-1})$ denotes the matrix-valued function
\begin{eq}
\left(\frac{\partial y^j\circ\bm{\psi}_{\mathcal{U}}^{-1}}{\partial x^k}\right).
\end{eq}
The canonical topological structure and smooth structure of $\mathrm{T}^*\mathcal{M}$ are both determined by
\begin{eq}
\left\{(\mathrm{T}^*\mathcal{U}\comma\bm{\Psi}_{\mathrm{T}^*\mathcal{U}}\mbox{; }\tilde{x}^j\comma\tilde\alpha_j)\colon\text{$(\mathcal{U}\comma\bm{\psi}_{\mathcal{U}}\mbox{; }x^j)$ is a local coordinate system of $\mathcal{M}$}\right\}
\end{eq}
so that $\bm{\Psi}_{\mathrm{T}^*\mathcal{U}}$ is a diffeomorphism. Since $\mathcal{M}$ is compact, it's easy to find that for any $C\in[0\comma{+}\infty)$,
\begin{eq}
\bigl\{\bm\alpha\in\mathrm{T}^*\mathcal{M}\bigm|\left|\bm\alpha\right|_{\bm g}\leqslant C\bigr\}
\end{eq}
is a compact subset of $\mathrm{T}^*\mathcal{M}$.

Obviously, $\pi$ is a smooth map from $\mathrm{T}^*\mathcal{M}$ onto $\mathcal{M}$ and for any $\phi\in\mathrm{C}^\infty(\mathcal{U})$, there hold
\begin{eq} \label{eq: phi circ pi, derivatives}
\frac{\partial\phi\circ\pi}{\partial\tilde{x}^j}=\frac{\partial\phi}{\partial x^j}\circ\pi\comma\quad \frac{\partial\phi\circ\pi}{\partial\tilde{\alpha}_j}=0.
\end{eq}
Another common smooth function on $\mathrm{T}^*\mathcal{M}$ is $\left|\cdot\right|_{\bm g}^2$, which satisfies
\begin{eq} \label{eq: |cdot|_g^2, derivatives}
\frac{\partial\left|\cdot\right|_{\bm g}^2}{\partial\tilde{x}^j}=\left(\frac{\partial g^{x^lx^m}}{\partial x^j}\circ\pi\right)\tilde{\alpha}_l\tilde{\alpha}_m\comma\quad \frac{\partial\left|\cdot\right|_{\bm g}^2}{\partial\tilde{\alpha}_j}=2\bigl(g^{x^jx^m}\circ\pi\bigr)\tilde{\alpha}_m\comma\quad \frac{\partial^2\left|\cdot\right|_{\bm g}^2}{\partial\tilde{\alpha}_k\partial\tilde{\alpha}_j}=2g^{x^jx^k}\circ\pi.
\end{eq}
Moreover, there hold
\begin{ga}
\frac{\partial\tilde{x}^j}{\partial\tilde{y}^k}=\frac{\partial x^j}{\partial y^k}\circ\pi\comma\quad \frac{\partial\tilde{x}^j}{\partial\tilde{\beta}_k}=0\comma\quad \tilde\alpha_j=\left(\frac{\partial y^k}{\partial x^j}\circ\pi\right)\tilde\beta_k\comma \\
\frac{\partial\tilde{\alpha}_j}{\partial\tilde{y}^k}=\left(\left(\frac{\partial^2 y^l}{\partial x^m\partial x^j}\frac{\partial x^m}{\partial y^k}\right)\circ\pi\right)\tilde\beta_l\comma\quad \frac{\partial\tilde{\alpha}_j}{\partial\tilde{\beta}_k}=\frac{\partial y^k}{\partial x^j}\circ\pi.
\end{ga}
Note that
\begin{eq}
\Gamma_{x^jx^k}^{x^l}=\Gamma_{y^qy^r}^{y^s}\frac{\partial y^q}{\partial x^j}\frac{\partial y^r}{\partial x^k}\frac{\partial x^l}{\partial y^s}+\frac{\partial^2 y^m}{\partial x^k\partial x^j}\frac{\partial x^l}{\partial y^m}.
\end{eq}
It's easy to verify that for any $(\bm\alpha\comma t)\in\mathrm{T}^*\mathcal{M}\times\mathbb{R}$,
\begin{eq} \label{eq: tilde{nabla}_{tilde{x}^j}varphi}
\tilde{\nabla}_{\tilde{x}^j}\varphi(\bm\alpha\comma t)\triangleq\left.\frac{\partial\varphi(\cdot\comma t)}{\partial \tilde{x}^j}\right|_{\bm\alpha}+\Gamma_{x^jx^m}^{x^l}\bigl(\pi(\bm\alpha)\bigr)\left.\frac{\partial\varphi(\cdot\comma t)}{\partial \tilde{\alpha}_m}\right|_{\bm\alpha}\alpha_{x^l}
\end{eq}
is the component of a $(0\comma1)$-tensor at $\pi(\bm\alpha)$ and
\begin{eq} \label{eq: tilde{nabla}_{tilde{x}^j}A_{x^rx^s}} \begin{aligned}
\tilde{\nabla}_{\tilde{x}^j}A_{x^rx^s}(\bm\alpha\comma t)&\triangleq\left.\frac{\partial A_{x^rx^s}(\cdot\comma t)}{\partial \tilde{x}^j}\right|_{\bm\alpha}+\Gamma_{x^jx^m}^{x^l}\bigl(\pi(\bm\alpha)\bigr)\left.\frac{\partial A_{x^rx^s}(\cdot\comma t)}{\partial \tilde{\alpha}_m}\right|_{\bm\alpha}\alpha_{x^l} \\
&\quad-\Gamma_{x^rx^j}^{x^q}\bigl(\pi(\bm\alpha)\bigr)A_{x^qx^s}(\bm\alpha\comma t)-\Gamma_{x^sx^j}^{x^q}\bigl(\pi(\bm\alpha)\bigr)A_{x^rx^q}(\bm\alpha\comma t)
\end{aligned} \end{eq}
is the component of a $(0\comma3)$-tensor at $\pi(\bm\alpha)$. In particular, \eqref{eq: tilde{nabla}_{tilde{x}^j}varphi} equals to 0 when $\varphi(\cdot\comma t)=\left|\cdot\right|_{\bm g}^2$, so does \eqref{eq: tilde{nabla}_{tilde{x}^j}A_{x^rx^s}} when $\bm{A}(\cdot\comma t)=\bm{g}\circ\pi$ (recall \eqref{eq: A(v, t)}). Other common tensor components in the style of \eqref{eq: tilde{nabla}_{tilde{x}^j}varphi} and \eqref{eq: tilde{nabla}_{tilde{x}^j}A_{x^rx^s}} include
\begin{ga}
\tilde{\nabla}_{\tilde{\alpha}_j}\varphi(\bm\alpha\comma t)\triangleq\left.\frac{\partial\varphi(\cdot\comma t)}{\partial\tilde{\alpha}_j}\right|_{\bm\alpha}\comma\quad \tilde{\nabla}_{\tilde{\alpha}_j\tilde{\alpha}_k}\varphi(\bm\alpha\comma t)\triangleq\left.\frac{\partial^2\varphi(\cdot\comma t)}{\partial\tilde{\alpha}_k\partial\tilde{\alpha}_j}\right|_{\bm\alpha}\comma \label{eq: tilde{nabla}_{tilde{alpha}_j}varphi, tilde{nabla}_{tilde{alpha}_j tilde{alpha}_k}varphi} \\%
\tilde{\nabla}_{\tilde{\alpha}_j}A_{x^rx^s}(\bm\alpha\comma t)\triangleq\left.\frac{\partial A_{x^rx^s}(\cdot\comma t)}{\partial\tilde{\alpha}_j}\right|_{\bm\alpha}\comma\quad \tilde{\nabla}_{\tilde{\alpha}_j\tilde{\alpha}_k}A_{x^rx^s}(\bm\alpha\comma t)\triangleq\left.\frac{\partial^2 A_{x^rx^s}(\cdot\comma t)}{\partial\tilde{\alpha}_k\partial\tilde{\alpha}_j}\right|_{\bm\alpha}\comma \label{eq: tilde{nabla}_{tilde{alpha}_j}A_{x^rx^s}, tilde{nabla}_{tilde{alpha}_j tilde{alpha}_k}A_{x^rx^s}} \\%
\tilde{\nabla}_{\tilde{x}^j\tilde{\alpha}_k}\varphi(\bm\alpha\comma t)\triangleq\left.\frac{\partial^2\varphi(\cdot\comma t)}{\partial\tilde{\alpha}_k\partial\tilde{x}^j}\right|_{\bm\alpha}+\Gamma_{x^jx^m}^{x^l}\bigl(\pi(\bm\alpha)\bigr)\left(\left.\frac{\partial^2\varphi(\cdot\comma t)}{\partial\tilde{\alpha}_k\partial\tilde{\alpha}_m}\right|_{\bm\alpha}\alpha_{x^l}+\left.\frac{\partial\varphi(\cdot\comma t)}{\partial\tilde{\alpha}_m}\right|_{\bm\alpha}\updelta_{lk}\right)\comma \label{eq: tilde{nabla}_{tilde{x}^j tilde{alpha}_k}varphi} \\%
\begin{aligned}
\tilde{\nabla}_{\tilde{x}^j\tilde{\alpha}_k}A_{x^rx^s}(\bm\alpha\comma t)&\triangleq\left.\frac{\partial^2A_{x^rx^s}(\cdot\comma t)}{\partial\tilde{\alpha}_k\partial\tilde{x}^j}\right|_{\bm\alpha}+\Gamma_{x^jx^m}^{x^l}\bigl(\pi(\bm\alpha)\bigr)\left(\left.\frac{\partial^2A_{x^rx^s}(\cdot\comma t)}{\partial\tilde{\alpha}_k\partial\tilde{\alpha}_m}\right|_{\bm\alpha}\alpha_{x^l}+\left.\frac{\partial A_{x^rx^s}(\cdot\comma t)}{\partial\tilde{\alpha}_m}\right|_{\bm\alpha}\updelta_{lk}\right) \\
&\quad-\Gamma_{x^rx^j}^{x^q}\bigl(\pi(\bm\alpha)\bigr)\left.\frac{\partial A_{x^qx^s}(\cdot\comma t)}{\partial\tilde{\alpha}_k}\right|_{\bm\alpha}-\Gamma_{x^sx^j}^{x^q}\bigl(\pi(\bm\alpha)\bigr)\left.\frac{\partial A_{x^rx^q}(\cdot\comma t)}{\partial\tilde{\alpha}_k}\right|_{\bm\alpha}\comma
\end{aligned} \label{eq: tilde{nabla}_{tilde{x}^j tilde{alpha}_k}A_{x^rx^s}} \\%
\begin{aligned}
\tilde{\nabla}_{\tilde{x}^j\tilde{x}^k}\varphi(\bm\alpha\comma t)&\triangleq\left.\frac{\partial^2\varphi(\cdot\comma t)}{\partial\tilde{x}^k\partial\tilde{x}^j}\right|_{\bm\alpha}+\mathrm{D}_{x^k}\Gamma_{x^jx^m}^{x^l}\bigl(\pi(\bm\alpha)\bigr)\left.\frac{\partial\varphi(\cdot\comma t)}{\partial\tilde{\alpha}_m}\right|_{\bm\alpha}\alpha_{x^l}+\Gamma_{x^jx^m}^{x^l}\bigl(\pi(\bm\alpha)\bigr)\left.\frac{\partial^2\varphi(\cdot\comma t)}{\partial\tilde{x}^k\partial\tilde{\alpha}_m}\right|_{\bm\alpha}\alpha_{x^l} \\
&\quad+\Gamma_{x^kx^m}^{x^l}\bigl(\pi(\bm\alpha)\bigr)\tilde{\nabla}_{\tilde{x}^j\tilde{\alpha}_m}\varphi(\bm\alpha\comma t)\alpha_{x^l}-\Gamma_{x^jx^k}^{x^q}\bigl(\pi(\bm\alpha)\bigr)\tilde{\nabla}_{\tilde{x}^q}\varphi(\bm\alpha\comma t)\comma
\end{aligned} \label{eq: tilde{nabla}_{tilde{x}^j tilde{x}^k}varphi} \\%
\begin{aligned}
\tilde{\nabla}_{\tilde{x}^j\tilde{x}^k}A_{x^rx^s}(\bm\alpha\comma t)&\triangleq\left.\frac{\partial^2A_{x^rx^s}(\cdot\comma t)}{\partial\tilde{x}^k\partial\tilde{x}^j}\right|_{\bm\alpha}+\mathrm{D}_{x^k}\Gamma_{x^jx^m}^{x^l}\bigl(\pi(\bm\alpha)\bigr)\left.\frac{\partial A_{x^rx^s}(\cdot\comma t)}{\partial\tilde{\alpha}_m}\right|_{\bm\alpha}\alpha_{x^l} \\
&\quad-\mathrm{D}_{x^k}\Gamma_{x^rx^j}^{x^q}\bigl(\pi(\bm\alpha)\bigr)A_{x^qx^s}(\bm\alpha\comma t)-\mathrm{D}_{x^k}\Gamma_{x^sx^j}^{x^q}\bigl(\pi(\bm\alpha)\bigr)A_{x^rx^q}(\bm\alpha\comma t) \\
&\quad+\Gamma_{x^jx^m}^{x^l}\bigl(\pi(\bm\alpha)\bigr)\left.\frac{\partial^2A_{x^rx^s}(\cdot\comma t)}{\partial\tilde{x}^k\partial\tilde{\alpha}_m}\right|_{\bm\alpha}\alpha_{x^l}-\Gamma_{x^rx^j}^{x^q}\bigl(\pi(\bm\alpha)\bigr)\left.\frac{\partial A_{x^qx^s}(\cdot\comma t)}{\partial\tilde{x}^k}\right|_{\bm\alpha} \\
&\quad-\Gamma_{x^sx^j}^{x^q}\bigl(\pi(\bm\alpha)\bigr)\left.\frac{\partial A_{x^rx^q}(\cdot\comma t)}{\partial\tilde{x}^k}\right|_{\bm\alpha}+\Gamma_{x^kx^m}^{x^l}\bigl(\pi(\bm\alpha)\bigr)\tilde{\nabla}_{\tilde{x}^j\tilde{\alpha}_m}A_{x^rx^s}(\bm\alpha\comma t)\alpha_{x^l} \\
&\quad-\Gamma_{x^rx^k}^{x^q}\bigl(\pi(\bm\alpha)\bigr)\tilde{\nabla}_{\tilde{x}^j}A_{x^qx^s}(\bm\alpha\comma t)-\Gamma_{x^sx^k}^{x^q}\bigl(\pi(\bm\alpha)\bigr)\tilde{\nabla}_{\tilde{x}^j}A_{x^rx^q}(\bm\alpha\comma t) \\
&\quad-\Gamma_{x^jx^k}^{x^q}\bigl(\pi(\bm\alpha)\bigr)\tilde{\nabla}_{\tilde{x}^q}A_{x^rx^s}(\bm\alpha\comma t).
\end{aligned} \label{eq: tilde{nabla}_{tilde{x}^j tilde{x}^k}A_{x^rx^s}}
\end{ga}
These tensor components are used for the a priori estimates. For any $\bm{x}\in\mathcal{M}$ and $\psi\in\mathrm{C}^2(\mathrm{T}_{\bm x}^*\mathcal{M})$, the Fr\'echet differentiation and second-order Fr\'echet differentiation (cf. e.g. \cite[pp.~11,~28]{Schwartz1969}) of $\psi$ at $\bm{\alpha}\in(\mathrm{T}_{\bm x}^*\mathcal{M}\comma\left|\cdot\right|_{\bm g})$ are by definition
\begin{al}
&\mathrm{D}_{\bm\alpha}[\psi]\colon\mathrm{T}_{\bm x}^*\mathcal{M}\longrightarrow\mathbb{R}\comma\bm\beta\longmapsto\left.\frac{\partial\psi}{\partial \tilde{\alpha}_j}\right|_{\bm\alpha}\beta_{x^j}\comma \label{eq: Frechet} \\
&\mathrm{D}_{\bm\alpha}^2[\psi]\colon\mathrm{T}_{\bm x}^*\mathcal{M}\times\mathrm{T}_{\bm x}^*\mathcal{M}\longrightarrow\mathbb{R}\comma(\bm\beta\comma\bm\gamma)\longmapsto\left.\frac{\partial^2\psi}{\partial \tilde{\alpha}_k\partial \tilde{\alpha}_j}\right|_{\bm\alpha}\beta_{x^j}\gamma_{x^k} \label{eq: second-order Frechet}
\end{al}
respectively. When $\bm{A}(\bm{\alpha}\comma t)$ is defined e.g. as \eqref{eq: general p-Yamabe, A(alpha, t)}, for any $(\bm{\alpha}\comma t)\in\mathrm{T}_{\bm x}^*\mathcal{M}\times\mathbb{R}$ there holds
\begin{eq} \label{eq: general p-Yamabe, A(alpha, t), MTW condition}
\Bigl(\mathrm{D}_{\bm\alpha}^2\bigl[\bigl(\bm{A}(\cdot\comma t)\bigr)(\bm{v}\comma\bm{v})\bigr]\Bigr)(\bm{\beta}\comma\bm{\beta})=-\bm{g}(\bm{v}\comma\bm{v})\left|\bm{\beta}\right|_{\bm g}^2+2\left|\bm\beta(\bm{v})\right|^2\comma \forall(\bm\beta\comma\bm{v})\in\mathrm{T}_{\bm x}^*\mathcal{M}\times\mathrm{T}_{\bm x}\mathcal{M}.
\end{eq}


\section{\texorpdfstring{``Subsolutions''}{"Subsolutions"}} \label{sec: ``Subsolutions''}

In this section, we prove some conclusions which are connected with different kinds of ``subsolutions'', including ``classical subsolution'', $\mathcal{C}$-subsolution and pseudo-solution. The first proposition is related to when there exists a $p$-admissible subsolution to the Dirichlet problem of the $p$-Hessian equation (cf. \eqref{eq: p-Hessian, Dirichlet}).

\begin{prop} \label{prop: existence of a classical p-admissible subsolution}
Fix $n\in\mathbb{N}$, $p\in\{1\comma2\comma\cdots\comma n\}$ and $\alpha\in(0\comma1)$. Let $\Omega$ be a strictly $(p-1)$-convex bounded domain in $\mathbb{R}^n$ with $\partial\Omega$ smooth, $\psi$ be a smooth function on $\partial\Omega$, and $\tilde{\varphi}(\bm{x}\comma t)$ be a positive continuous function in $\overline{\Omega}\times\mathbb{R}$ non-decreasing with respect to $t$. Then there exists $v\in\mathrm{C}^\infty(\overline{\Omega})$ satisfying
\begin{eq} \label{eq: v}
\left\{ \begin{gathered}
\sigma_p^{\frac1p}\Bigl(\bm\lambda\bigl(\mathrm{D}^2v(\bm{x})\bigr)\Bigr)\geqslant\tilde{\varphi}\bigl(\bm{x}\comma v(\bm{x})\bigr)\bigl(1+\left|\mathrm{D}v(\bm{x})\right|+\left|v(\bm{x})\right|^\alpha\bigr)\comma\forall\bm{x}\in\overline{\Omega} \\
\bm\lambda\bigl(\mathrm{D}^2v(\bm{x})\bigr)\in\Gamma_p\comma\forall\bm{x}\in\overline{\Omega} \\
v(\bm{x})=\psi(\bm{x})\comma\forall\bm{x}\in\partial\Omega
\end{gathered} \right. .
\end{eq}
\end{prop}

\begin{proof}
By \cite[pp.~274--276]{Caffarelli1985a}, there exists $u\in\mathrm{C}^\infty(\overline{\Omega})$ satisfying
\begin{eq} \label{eq: u}
\left\{ \begin{gathered}
\bm\lambda\bigl(\mathrm{D}^2u(\bm{x})\bigr)\in\Gamma_p\comma\forall\bm{x}\in\overline{\Omega} \\
u(\bm{x})=0\comma\left|\mathrm{D}u(\bm{x})\right|>0\comma\forall\bm{x}\in\partial\Omega \\
u(\bm{x})<0\comma\forall\bm{x}\in\Omega
\end{gathered} \right. \comma
\end{eq}
and $\psi$ can be extended smoothly inside $\Omega$ so that
\begin{eq} \label{eq: lambda(D^2psi(x))}
\bm\lambda\bigl(\mathrm{D}^2\psi(\bm{x})\bigr)\in\Gamma_p\comma\forall\bm{x}\in\overline{\Omega}.
\end{eq}
Following P.~L.~Lions' idea for the case $p=n$ (cf. \cite[pp.~391--392]{Caffarelli1984}), we define
\begin{eq}
v\triangleq\psi+A\left(\mathrm{e}^{Bu}-1\right)\comma
\end{eq}
where $A$ and $B$ are positive constants to be determined later. Obviously we have $v\in\mathrm{C}^\infty(\overline{\Omega})$,
\begin{ga}
v(\bm{x})=\psi(\bm{x})\comma\forall\bm{x}\in\partial\Omega\comma\label{eq: v, partial Omega} \\
\psi(\bm{x})-A<v(\bm{x})<\psi(\bm{x})\comma\forall\bm{x}\in\Omega\comma\label{eq: v, Omega} \\
\mathrm{D}v=\mathrm{D}\psi+AB\mathrm{e}^{Bu}\mathrm{D}u\comma \label{eq: Dv} \\
\mathrm{D}^2v=\mathrm{D}^2\psi+AB\mathrm{e}^{Bu}\left(\mathrm{D}^2u+B\mathrm{D}u(\mathrm{D}u)^{\mathrm T}\right). \label{eq: D^2v}
\end{ga}
Combining \eqref{eq: D^2v}, \eqref{eq: lambda(D^2psi(x))}, \eqref{eq: u} and \propref{prop: concavity-2}, we find that
\begin{eq}
\bm\lambda\bigl(\mathrm{D}^2v(\bm{x})\bigr)\in\Gamma_p\comma\forall\bm{x}\in\overline{\Omega}.
\end{eq}

Now fix $\bm{x}_0\in\overline{\Omega}$. We may as well assume that
\begin{eq}
\mathrm{D}^2u(\bm{x}_0)=\diag\{\mu_1\comma\mu_2\comma\cdots\comma\mu_n\}\comma
\end{eq}
and therefore
\begin{eq}
\bm\mu\triangleq(\mu_1\comma\mu_2\comma\cdots\comma\mu_n)^{\mathrm T}\in\Gamma_p.
\end{eq}
The theory of linear algebra shows that for any $m\in\mathbb{N}$, $\mathbf{A}\in\mathbb{R}^{m\times m}$ and $\bm\nu\in\mathbb{R}^m$, there holds
\begin{eq}
\det\left(\mathbf{A}+B\bm{\nu}\bm{\nu}^{\mathrm T}\right)=\det(\mathbf{A})+B\bm{\nu}^{\mathrm T}\mathbf{A}^*\bm{\nu}.
\end{eq}
Thus, for any $\bm\nu\in\mathbb{R}^n$, we have
\begin{eq} \begin{aligned}
&\mathrel{\hphantom{=}}\sigma_p\left(\mathrm{D}^2u(\bm{x}_0)+B\bm{\nu}\bm{\nu}^{\mathrm T}\right) \\
&=\sum_{1\leqslant j_1<j_2<\cdots<j_p\leqslant n} \left(\mathrm{D}^2u(\bm{x}_0)+B\bm{\nu}\bm{\nu}^{\mathrm T}\right)\binom{j_1\comma j_2\comma\cdots\comma j_p}{j_1\comma j_2\comma\cdots\comma j_p} \\
&=\sum_{1\leqslant j_1<j_2<\cdots<j_p\leqslant n} \left(\mu_{j_1}\mu_{j_2}\cdots\mu_{j_p}+B\sum_{k=1}^p \nu_{j_k}^2\mu_{j_1}\cdots\mu_{j_{k-1}}\mu_{j_{k+1}}\cdots\mu_{j_p}\right) \\
&=\sigma_p(\bm{\mu})+\frac{B}{p!}\sum_{\substack{1\leqslant j_1\comma j_2\comma\cdots\comma j_p\leqslant n \\ \text{$j_1$, $j_2$, $\cdots$, $j_p$ differ from each other}}} \sum_{k=1}^p \nu_{j_k}^2\mu_{j_1}\cdots\mu_{j_{k-1}}\mu_{j_{k+1}}\cdots\mu_{j_p} \\
&=\sigma_p(\bm{\mu})+\frac{B}{p!}\sum_{k=1}^p \sum_{j_k=1}^n \nu_{j_k}^2\sum_{\substack{j_1\comma\cdots\comma j_{k-1}\comma j_{k+1}\comma\cdots\comma j_p\in\{1\comma2\comma\cdots\comma n\}\setminus\{j_k\} \\ \text{$j_1$, $\cdots$, $j_{k-1}$, $j_{k+1}$, $\cdots$, $j_p$ differ from each other}}} \mu_{j_1}\cdots\mu_{j_{k-1}}\mu_{j_{k+1}}\cdots\mu_{j_p} \\
&=\sigma_p(\bm{\mu})+\frac{B}{p}\sum_{k=1}^p \sum_{j_k=1}^n \nu_{j_k}^2\sigma_{p-1}(\bm\mu|j_k)=\sigma_p(\bm{\mu})+B\sum_{j=1}^n \nu_j^2\sigma_{p-1}(\bm\mu|j)\geqslant\varepsilon_1+B\varepsilon_2\left|\bm\nu\right|^2\comma
\end{aligned} \end{eq}
where
\begin{ga}
\varepsilon_1\triangleq\min_{\bm{x}\in\overline{\Omega}} \sigma_p\left(\mathrm{D}^2u(\bm{x})\right)>0\comma \\
\varepsilon_2\triangleq\min_{\bm{x}\in\overline{\Omega}} \sigma_{p-1}\Bigl(\bm\lambda\bigl(\mathrm{D}^2u(\bm{x})\bigr)\Bigm|n\Bigr)>0.
\end{ga}
Letting $\bm\nu$ be $\mathrm{D}u(\bm{x}_0)$, we obtain
\begin{eq}
\sigma_p\left(\mathrm{D}^2u(\bm{x}_0)+B\mathrm{D}u(\bm{x}_0)\bigl(\mathrm{D}u(\bm{x}_0)\bigr)^{\mathrm T}\right)\geqslant\varepsilon_1+B\varepsilon_2\left|\mathrm{D}u(\bm{x}_0)\right|^2.
\end{eq}

As a result, there holds
\begin{eq} \label{eq: sigma_p(D^2u+BDu(Du)^T)}
\sigma_p\left(\mathrm{D}^2u+B\mathrm{D}u(\mathrm{D}u)^{\mathrm T}\right)\geqslant\varepsilon_1+B\varepsilon_2\left|\mathrm{D}u\right|^2
\end{eq}
in $\overline{\Omega}$. Combining \eqref{eq: D^2v}, \eqref{eq: lambda(D^2psi(x))}, \eqref{eq: sigma_p(D^2u+BDu(Du)^T)} and \propref{prop: concavity-2}, we find that
\begin{eq} \label{eq: (sigma_p(D^2v))^{frac 1p}} \begin{aligned}
\sigma_p^{\frac1p}\left(\mathrm{D}^2v\right)&\geqslant AB\mathrm{e}^{Bu}\left(\varepsilon_1+B\varepsilon_2\left|\mathrm{D}u\right|^2\right)^{\frac1p} \\
&\geqslant2^{\frac{1-p}{p}}\varepsilon_1^{\frac1p}AB\mathrm{e}^{Bu}+2^{\frac{1-p}{p}}\varepsilon_2^{\frac1p}AB^{1+\frac1p}\mathrm{e}^{Bu}\left|\mathrm{D}u\right|^{\frac2p}
\end{aligned} \end{eq}
in $\overline{\Omega}$. On the other hand, combining \eqref{eq: v, partial Omega}, \eqref{eq: v, Omega} and \eqref{eq: Dv}, for any $\bm{x}\in\overline{\Omega}$ we obtain
\begin{eq} \label{eq: tilde{varphi}(x, v(x))(1+|Dv(x)|+|v(x)|^alpha)} \begin{aligned}
&\mathrel{\hphantom{=}}\tilde{\varphi}\bigl(\bm{x}\comma v(\bm{x})\bigr)\bigl(1+\left|\mathrm{D}v(\bm{x})\right|+\left|v(\bm{x})\right|^\alpha\bigr) \\
&\leqslant\tilde{\varphi}\bigl(\bm{x}\comma \psi(\bm{x})\bigr)\left(1+\left|\mathrm{D}\psi(\bm{x})\right|+AB\mathrm{e}^{Bu(\bm{x})}\left|\mathrm{D}u(\bm{x})\right|+\left|\psi(\bm{x})\right|^\alpha+A^\alpha\right) \\
&\leqslant C_1\left(C_2+A^\alpha+AB\mathrm{e}^{Bu(\bm{x})}\left|\mathrm{D}u(\bm{x})\right|\right)
\end{aligned} \end{eq}

For the case when $p\geqslant2$, we choose the constants $B$ and $A$ as follows:
\begin{al}
B&=C_1^p2^{p-1}\varepsilon_2^{-1}\left(\max_{\overline{\Omega}}\left|\mathrm{D}u\right|\right)^{p-2}\comma \\
A&=C_2^{\frac{1}{\alpha}}+\left(C_12^{\frac{2p-1}{p}}\varepsilon_1^{-\frac1p}B^{-1}\mathrm{e}^{-B\min_{\overline{\Omega}} u}\right)^{\frac{1}{1-\alpha}}.
\end{al}
Then for any $\bm{x}\in\overline{\Omega}$, there hold
\begin{ga}
2^{\frac{1-p}{p}}\varepsilon_2^{\frac1p}B^{\frac1p}\left|\mathrm{D}u(\bm{x})\right|^{\frac2p}\geqslant C_1\left|\mathrm{D}u(\bm{x})\right|\comma \\
2^{\frac{1-p}{p}}\varepsilon_1^{\frac1p}AB\mathrm{e}^{Bu(\bm{x})}\geqslant2C_1A^\alpha\geqslant C_1\left(C_2+A^\alpha\right).
\end{ga}
Recalling \eqref{eq: (sigma_p(D^2v))^{frac 1p}} and \eqref{eq: tilde{varphi}(x, v(x))(1+|Dv(x)|+|v(x)|^alpha)}, we obtain
\begin{eq}
\sigma_p^{\frac1p}\left(\mathrm{D}^2v(\bm{x})\right)\geqslant\tilde{\varphi}\bigl(\bm{x}\comma v(\bm{x})\bigr)\bigl(1+\left|\mathrm{D}v(\bm{x})\right|+\left|v(\bm{x})\right|^\alpha\bigr).
\end{eq}

For the case when $p=1$, we choose the constants $B$ and $A$ as follows:
\begin{eq}
B=\frac{C_1^2}{2\varepsilon_1\varepsilon_2}\comma\quad A=C_2^{\frac{1}{\alpha}}+\left(\frac{4}{\varepsilon_1}C_1B^{-1}\mathrm{e}^{-B\min_{\overline{\Omega}} u}\right)^{\frac{1}{1-\alpha}}.
\end{eq}
Then for any $\bm{x}\in\overline{\Omega}$, there hold
\begin{ga}
\varepsilon_2B\left|\mathrm{D}u(\bm{x})\right|^2+\frac{\varepsilon_1}{2}\geqslant C_1\left|\mathrm{D}u(\bm{x})\right|\comma \\
\frac{\varepsilon_1}{2}AB\mathrm{e}^{Bu(\bm{x})}\geqslant2C_1A^\alpha\geqslant C_1\left(C_2+A^\alpha\right)\comma
\end{ga}
Recalling \eqref{eq: (sigma_p(D^2v))^{frac 1p}} and \eqref{eq: tilde{varphi}(x, v(x))(1+|Dv(x)|+|v(x)|^alpha)}, we obtain
\begin{eq}
\sigma_1\left(\mathrm{D}^2v(\bm{x})\right)\geqslant\tilde{\varphi}\bigl(\bm{x}\comma v(\bm{x})\bigr)\bigl(1+\left|\mathrm{D}v(\bm{x})\right|+\left|v(\bm{x})\right|^\alpha\bigr).
\end{eq}

In conclusion, we have justified \eqref{eq: v}.
\end{proof}

The theme of this paper is the general $p$-Hessian equation, but we will temporarily consider more general orthogonally-invariant elliptic equations in the two lemmas below. Let $\Gamma\subsetneqq\mathbb{R}^n$ be a symmetric, open, convex cone containing $\Gamma_n$ whose vertex is $\bm{0}$ and $f$ be a symmetric, concave $\mathrm{C}^1$ function defined in $\Gamma$ which satisfies:
\begin{ga}
\sup_{\partial\Gamma} f\triangleq\sup_{\bm\nu\in\partial\Gamma}\varlimsup_{\Gamma\ni\bm\mu\to\bm\nu}f(\bm\mu)<\sup_{\Gamma} f\comma \\
\varlimsup_{s\to{+}\infty}f(s\bm\mu)=\sup_{\Gamma} f\comma\forall\bm\mu\in\Gamma. \label{eq: varlimsup_{s to{+}infty}f(s mu)=sup_{Gamma} f}
\end{ga}
Note that we haven't yet assumed ``$\left.\frac{\partial f(\bm\mu)}{\partial \mu_j}\right|_{\bm\mu=\bm\nu}>0\ (\forall\bm\nu\in\Gamma)$''. The following lemma is basic.

\begin{lem} \label{lem: basic properties of Gamma and f}
There hold $\overline\Gamma+\Gamma\subset\Gamma$ and
\begin{eq} \label{eq: varlimsup_{s to{+}infty}f(nu+s mu)=sup_{Gamma} f}
\varlimsup_{s\to{+}\infty}f(\bm\nu+s\bm\mu)=\sup_{\Gamma} f\comma\forall\bm\mu\comma\bm\nu\in\Gamma.
\end{eq}
\end{lem}

\begin{proof}
Assume that $\bm\xi\in\overline\Gamma$ and $\bm\eta\in\Gamma$. Since $\Gamma$ is open, there exists $r\in(0\comma{+}\infty)$ so that $\mathrm{B}_r(\bm\eta)\subset\Gamma$. On the other hand, there exists $\bm\zeta\in\mathrm{B}_r(\bm\xi)\cap\Gamma$. Thus, we have $\bm\eta-\bm\zeta+\bm\xi\in\Gamma$ and therefore
\begin{eq}
\bm\xi+\bm\eta=2\left(\frac12\bm\zeta+\frac12(\bm\eta-\bm\zeta+\bm\xi)\right)\in\Gamma
\end{eq} 
since $\Gamma$ is a convex cone. This implies $\overline\Gamma+\Gamma\subset\Gamma$ and it remains to prove \eqref{eq: varlimsup_{s to{+}infty}f(nu+s mu)=sup_{Gamma} f}. For any $C<\sup\limits_{\Gamma}f$, by \eqref{eq: varlimsup_{s to{+}infty}f(s mu)=sup_{Gamma} f} there exists such $r\in(0\comma1)$ that $f\left(\frac1r\bm\nu\right)>C$. Since $f$ is concave, \eqref{eq: varlimsup_{s to{+}infty}f(s mu)=sup_{Gamma} f} also implies
\begin{eq}
\varlimsup_{s\to{+}\infty}f(\bm\nu+s\bm\mu)\geqslant\varlimsup_{s\to{+}\infty}\left(rf\left(\frac1r\bm\nu\right)+(1-r)f\left(\frac{s}{1-r}\bm\mu\right)\right)\geqslant rC+(1-r)\sup_{\Gamma} f\geqslant C.
\end{eq}
As a result, we have $\varlimsup\limits_{s\to{+}\infty}f(\bm\nu+s\bm\mu)=\sup\limits_{\Gamma} f$.
\end{proof}

Next, we prove a key lemma which improves that in \cite[p.~342]{Szekelyhidi2018}, see also \cite[p.~1505]{Guan2014}, \cite[p.~2696]{Guan2015}, \cite[p.~803]{Phong2021} and \cite[pp.~14016--14017]{Guan2023a}.

\begin{lem} \label{lem: C-subsolution, key lemma}
Assume that $\delta$, $R\in(0\comma{+}\infty)$, $\bm\mu\in\mathbb{R}^n$ and $a\in\left(\sup\limits_{\partial\Gamma} f\comma\sup\limits_{\Gamma} f\right)$ satisfy
\begin{eq} \label{eq: C-subsolution, key lemma, assumption}
\left\{\bm\xi\in\Gamma\middle|f(\bm\xi)=a\right\}\cap\left(\{\bm\mu-\delta\bm{1}_n\}+\Gamma_n\right)\subset\mathrm{B}_R(\bm{0}).
\end{eq}
Then for any $\bm\nu\in\Gamma$, there holds
\begin{eq} \label{eq: C-subsolution, key lemma, conclusion}
\left.\frac{\partial f}{\partial\mu_j}\right|_{\bm\nu}(\mu_j-\nu_j)\geqslant\delta\sum_{j=1}^n\left.\frac{\partial f}{\partial\mu_j}\right|_{\bm\nu}-\left(R+\left|\bm\mu-\delta\bm{1}_n\right|\right)\min\left\{\left.\frac{\partial f}{\partial\mu_1}\right|_{\bm\nu}\comma\left.\frac{\partial f}{\partial\mu_2}\right|_{\bm\nu}\comma\cdots\comma\left.\frac{\partial f}{\partial\mu_n}\right|_{\bm\nu}\right\}+a-f(\bm\nu).
\end{eq}
\end{lem}

\begin{proof}
First, we prove that
\begin{eq} \label{eq: C-subsolution, key lemma, claim}
\bigl(\mathbb{R}^n\setminus\left\{\bm\xi\in\Gamma\middle|f(\bm\xi)>a\right\}\bigr)\cap\left(\{\bm\mu-\delta\bm{1}_n\}+\Gamma_n\right)\subset\mathrm{B}_{R+\left|\bm\mu-\delta\bm{1}_n\right|}(\bm\mu-\delta\bm{1}_n).
\end{eq}
For any $\bm\eta\in\mathbb{R}^n\setminus\Gamma\colon\bm\eta-\bm\mu+\delta\bm{1}_n\in\Gamma_n$, there exists $t_0\in[0\comma{+}\infty)$ so that
\begin{eq}
\bm\eta+t_0(\bm\eta-\bm\mu+\delta\bm{1}_n)\in\partial\Gamma.
\end{eq}
Note that $\partial\Gamma+\Gamma_n\subset\Gamma$ and
\begin{eq}
\varlimsup_{\Gamma\ni\bm\mu\to\bm\eta+t_0(\bm\eta-\bm\mu+\delta\bm{1}_n)}f(\bm\mu)\leqslant\sup_{\partial\Gamma} f<a.
\end{eq}
Thus, there exists $t_1\in(0\comma{+}\infty)$ so that
\begin{eq}
\bm\eta+t_0(\bm\eta-\bm\mu+\delta\bm{1}_n)+t_1(\bm\eta-\bm\mu+\delta\bm{1}_n)\in\left\{\bm\xi\in\Gamma\middle|f(\bm\xi)<a\right\}.
\end{eq}
Now let $\bm\eta$ be any element in the left-hand set of \eqref{eq: C-subsolution, key lemma, claim}. There exists $t_2\in[0\comma{+}\infty)$ so that
\begin{eq}
\bm\eta+t_2(\bm\eta-\bm\mu+\delta\bm{1}_n)\in\left\{\bm\xi\in\Gamma\middle|f(\bm\xi)\leqslant a\right\}.
\end{eq}
By \eqref{eq: varlimsup_{s to{+}infty}f(nu+s mu)=sup_{Gamma} f} we have
\begin{eq}
\varlimsup_{s\to{+}\infty} f\bigl(\bm\eta+t_2(\bm\eta-\bm\mu+\delta\bm{1}_n)+s(\bm\eta-\bm\mu+\delta\bm{1}_n)\bigr)=\sup_{\Gamma} f>a.
\end{eq}
As a result, there exists $s_0\in[0\comma{+}\infty)$ so that
\begin{eq}
\bm\eta+s_0(\bm\eta-\bm\mu+\delta\bm{1}_n)\in\left\{\bm\xi\in\Gamma\middle|f(\bm\xi)=a\right\}.
\end{eq}
Since
\begin{eq}
\bm\eta+s_0(\bm\eta-\bm\mu+\delta\bm{1}_n)=\bm\mu-\delta\bm{1}_n+(s_0+1)(\bm\eta-\bm\mu+\delta\bm{1}_n)\in\{\bm\mu-\delta\bm{1}_n\}+\Gamma_n\comma
\end{eq}
by \eqref{eq: C-subsolution, key lemma, assumption} there holds
\begin{eq}
\left|\bm\mu-\delta\bm{1}_n+(s_0+1)(\bm\eta-\bm\mu+\delta\bm{1}_n)\right|<R.
\end{eq}
This implies
\begin{eq}
(s_0+1)\left|\bm\eta-\bm\mu+\delta\bm{1}_n\right|<R+\left|\bm\mu-\delta\bm{1}_n\right|
\end{eq}
and therefore justifies \eqref{eq: C-subsolution, key lemma, claim}.

Now fix $\varepsilon\in(0\comma\delta)$. For any $k\in\{1\comma2\comma\cdots\comma n\}$, let
\begin{eq} \label{eq: C-subsolution, key lemma, zeta^{(k)}}
\bm\zeta^{(k)}\triangleq\left(R+\left|\bm\mu-\delta\bm{1}_n\right|-\varepsilon\right)\mathbf{e}_k+\varepsilon\bm{1}_n\in\Gamma_n\comma
\end{eq}
where $\mathbf{e}_k$ denotes the $k$-th standard basis vector in $\mathbb{R}^n$. Note that $|\bm\zeta^{(k)}|\geqslant R+\left|\bm\mu-\delta\bm{1}_n\right|$. Thus, by \eqref{eq: C-subsolution, key lemma, claim} there holds
\begin{eq}
\bm\mu-\delta\bm{1}_n+\bm\zeta^{(k)}\in\left\{\bm\xi\in\Gamma\middle|f(\bm\xi)>a\right\}.
\end{eq}
Since $f$ is concave, we have
\begin{eq} \begin{aligned}
a&<f(\bm\mu-\delta\bm{1}_n+\bm\zeta^{(k)}) \\
&\leqslant f(\bm\nu)+\sum_{j=1}^n\left.\frac{\partial f}{\partial\mu_j}\right|_{\bm\nu}(\mu_j-\delta+\zeta_j^{(k)}-\nu_j) \\
&=f(\bm\nu)+\left.\frac{\partial f}{\partial\mu_j}\right|_{\bm\nu}(\mu_j-\nu_j)-(\delta-\varepsilon)\sum_{j=1}^n\left.\frac{\partial f}{\partial\mu_j}\right|_{\bm\nu}+(R+\left|\bm\mu-\delta\bm{1}_n\right|-\varepsilon)\left.\frac{\partial f}{\partial\mu_k}\right|_{\bm\nu}.
\end{aligned} \end{eq}
As a result, there holds
\begin{eq}
\left.\frac{\partial f}{\partial\mu_j}\right|_{\bm\nu}(\mu_j-\nu_j)>(\delta-\varepsilon)\sum_{j=1}^n\left.\frac{\partial f}{\partial\mu_j}\right|_{\bm\nu}-\left(R+\left|\bm\mu-\delta\bm{1}_n\right|-\varepsilon\right)\min\left\{\left.\frac{\partial f}{\partial\mu_1}\right|_{\bm\nu}\comma\cdots\comma\left.\frac{\partial f}{\partial\mu_n}\right|_{\bm\nu}\right\}+a-f(\bm\nu).
\end{eq}
Letting $\varepsilon\to0+0$, we obtain \eqref{eq: C-subsolution, key lemma, conclusion}.
\end{proof}
 
In the following, we always fix $p\in\{1\comma2\comma\cdots\comma n\}$ and let $\Gamma$, $f$ be $\Gamma_p$, $\sigma_p^{\frac1p}$ respectively (cf. \corref{cor: concavity-1} and \corref{cor: partial Gamma_p}). Recall \defnref{defn: lambda_q}, \defnref{defn: lambda} and the last paragraph before \subsecref{subsec: Hessian equations on compact Riemannian manifolds}. The proposition below is essentially a corollary of \lemref{lem: C-subsolution, key lemma}, see also \cite[p.~344]{Szekelyhidi2018}, \cite[p.~1507]{Guan2014} and \cite[p.~2699]{Guan2015}.

\begin{prop} \label{prop: C-subsolution, key lemma, matrix form}
Assume that $\delta$, $R$, $a\in(0\comma{+}\infty)$ and $\mathbf C$ is a real symmetric $n\times n$ matrix. If
\begin{eq} \label{eq: C-subsolution, key lemma, matrix form, assumption}
\left\{\bm\xi\in\Gamma\middle|f(\bm\xi)=a\right\}\cap\bigl(\left\{\bm\lambda(\mathbf{C})-\delta\bm{1}_n\right\}+\Gamma_n\bigr)\subset\mathrm{B}_R(\bm{0})\comma
\end{eq}
then for any real symmetric $n\times n$ matrix $\mathbf{D}$ satisfying $\bm\lambda(\mathbf{D})\in\Gamma$, there holds
\begin{eq} \label{eq: C-subsolution, key lemma, matrix form, conclusion} \begin{aligned}
\left.\frac{\partial f\bigl(\bm\lambda(\mathbf{I}_n\comma\mathbf{B})\bigr)}{\partial b_{jk}}\right|_{\mathbf{B}=\mathbf{D}}(c_{jk}-d_{jk})&\geqslant\delta\sum_{j=1}^n\left.\frac{\partial f\bigl(\bm\lambda(\mathbf{I}_n\comma\mathbf{B})\bigr)}{\partial b_{jj}}\right|_{\mathbf{B}=\mathbf{D}}+a-f\bigl(\bm\lambda(\mathbf{D})\bigr) \\
&-\bigl(R+\left|\bm\lambda(\mathbf{C})-\delta\bm{1}_n\right|\bigr)\lambda_1\left(\left(\left.\frac{\partial f\bigl(\bm\lambda(\mathbf{I}_n\comma\mathbf{B})\bigr)}{\partial b_{jk}}\right|_{\mathbf{B}=\mathbf{D}}\right)\right).
\end{aligned} \end{eq}
\end{prop}

\begin{proof}
Since $\mathbf D$ is a real symmetric matrix, in view of \propref{prop: sigma_p circ lambda, derivatives, basic} we may as well assume that $\mathbf{D}$ is a diagonal matrix $\diag\{\nu_1\comma\nu_2\comma\cdots\comma\nu_n\}$ where $\nu_1\leqslant\nu_2\leqslant\cdots\leqslant\nu_n$ --- in particular, $\bm\lambda(\mathbf{D})=\bm\nu\in\Gamma$. By \propref{prop: sigma_p circ lambda, derivatives} (see also \cite[p.~577]{Lewis1996}) there holds
\begin{eq}
\left.\frac{\partial f\bigl(\bm\lambda(\mathbf{I}_n\comma\mathbf{B})\bigr)}{\partial b_{jk}}\right|_{\mathbf{B}=\mathbf{D}}=\left.\frac{\partial f}{\partial \mu_j}\right|_{\bm\nu}\updelta_{jk}.
\end{eq}
So we only need to prove
\begin{eq} \label{eq: C-subsolution, key lemma, matrix form, conclusion, equivalent}
\left.\frac{\partial f}{\partial \mu_j}\right|_{\bm\nu}(c_{jj}-\nu_j)\geqslant\delta\sum_{j=1}^n\left.\frac{\partial f}{\partial \mu_j}\right|_{\bm\nu}+a-f(\bm\nu)-\bigl(R+\left|\bm\lambda(\mathbf{C})-\delta\bm{1}_n\right|\bigr)\left.\frac{\partial f}{\partial\mu_n}\right|_{\bm\nu}.
\end{eq}
Here we have used
\begin{eq}
\left.\frac{\partial f}{\partial\mu_1}\right|_{\bm\nu}\geqslant\left.\frac{\partial f}{\partial\mu_2}\right|_{\bm\nu}\geqslant\cdots\geqslant\left.\frac{\partial f}{\partial\mu_n}\right|_{\bm\nu}\comma
\end{eq}
which in turn follows from the concavity and symmetry of $f$, see also \eqref{eq: sigma_{p-1}(mu|j_1)-sigma_{p-1}(mu|j_2)} and \corref{cor: corollary of increasing with respect to each variable}. On the other hand, \lemref{lem: C-subsolution, key lemma} and \eqref{eq: C-subsolution, key lemma, matrix form, assumption} imply
\begin{eq}
\left.\frac{\partial f}{\partial \mu_j}\right|_{\bm\nu}\bigl(\lambda_j(\mathbf{C})-\nu_j\bigr)\geqslant\delta\sum_{j=1}^n\left.\frac{\partial f}{\partial \mu_j}\right|_{\bm\nu}+a-f(\bm\nu)-\bigl(R+\left|\bm\lambda(\mathbf{C})-\delta\bm{1}_n\right|\bigr)\left.\frac{\partial f}{\partial\mu_n}\right|_{\bm\nu}.
\end{eq}
Since $\left.\frac{\partial f}{\partial\mu_n}\right|_{\bm\nu}>0$ (cf. \propref{prop: increasing with respect to each variable}), combining the above inequality with a linear-algebraic lemma in \cite[p.~287]{Caffarelli1985a} (or the Schur–Horn theorem, e.g. \cite[p.~624]{Horn1954}), we obtain \eqref{eq: C-subsolution, key lemma, matrix form, conclusion, equivalent}.
\end{proof}

Now let $(\mathcal{M}\comma\bm{g})$ be a closed connected Riemannian manifold of dimension $n$ and turn our attention to the general $p$-Hessian equation \eqref{eq: general p-Hessian} on $\mathcal{M}$. The following proposition explains why ``classical subsolution'' is not powerful in the study of \eqref{eq: general p-Hessian}. This conclusion may be well-known, but we could not find its proof in literature.

\begin{prop} \label{prop: ``classical subsolution'' is solution}
Assume that $\bm{A}(\bm\alpha\comma t)$, $\varphi(\bm\alpha\comma t)$ are both independent of $t$, $\bm{A}(\bm\alpha\comma t)$ satisfies the concavity condition \eqref{eq: A_{vv}(alpha, t) is concave with respect to alpha}, $u\in\mathrm{C}^2(\mathcal{M})$ satisfies \eqref{eq: general p-Hessian}, and $\underline{u}\in\mathrm{C}^2(\mathcal{M})$ is a ``classical subsolution'' to \eqref{eq: general p-Hessian} in the sense of \eqref{eq: general p-Hessian, classical subsolution}. Then $u-\underline{u}$ is a constant function on $\mathcal{M}$. In particular, $\underline{u}$ also satisfies \eqref{eq: general p-Hessian}.
\end{prop}

\begin{proof}
First, we prove that for any $\theta\in[0\comma1]$, $\theta\underline{u}+(1-\theta)u$ is a general $p$-admissible function (cf. \eqref{eq: general p-admissible}). Fix $\bm{y}\in\mathcal{M}$ and let $(\mathcal{V}\comma\bm{\psi}_{\mathcal{V}}\mbox{; }y^j)$ be such a local coordinate system containing $\bm{y}$ that $g_{jk}(\bm{y})=\updelta_{jk}$. By \eqref{eq: lambda(g^{-1}(A(du, u)+nabla^2u))} and \propref{prop: concavity-2}, we have
\begin{eq}
\bm\lambda\biggl(\theta\Bigl(A_{jk}\bigl(\dif\underline{u}(\bm{y})\comma\underline{u}(\bm{y})\bigr)+\nabla_{jk}\underline{u}(\bm{y})\Bigr)+(1-\theta)\Bigl(A_{jk}\bigl(\dif u(\bm{y})\comma u(\bm{y})\bigr)+\nabla_{jk}u(\bm{y})\Bigr)\biggr)\in\Gamma.
\end{eq}
Since $\bm{A}(\bm\alpha\comma t)$ satisfies \eqref{eq: A_{vv}(alpha, t) is concave with respect to alpha} and is independent of $t$, for any $\bm{v}\in\mathbb{R}^n$ there holds
\begin{eq} \begin{aligned}
&\mathrel{\hphantom{=}}\sum_{j\comma k=1}^n A_{jk}\bigl(\theta\dif\underline{u}(\bm{y})+(1-\theta)\dif u(\bm{y})\comma\theta\underline{u}(\bm{y})+(1-\theta)u(\bm{y})\bigr)v_jv_k \\
&\geqslant\sum_{j\comma k=1}^n \Bigl(\theta A_{jk}\bigl(\dif\underline{u}(\bm{y})\comma\theta\underline{u}(\bm{y})+(1-\theta)u(\bm{y})\bigr)v_jv_k+(1-\theta)A_{jk}\bigl(\dif u(\bm{y})\comma\theta\underline{u}(\bm{y})+(1-\theta)u(\bm{y})\bigr)v_jv_k\Bigr) \\
&=\sum_{j\comma k=1}^n \Bigl(\theta A_{jk}\bigl(\dif\underline{u}(\bm{y})\comma\underline{u}(\bm{y})\bigr)+(1-\theta)A_{jk}\bigl(\dif u(\bm{y})\comma u(\bm{y})\bigr)\Bigr)v_jv_k.
\end{aligned} \end{eq}
Thus, we have
\begin{eq} \begin{aligned}
&\bm\lambda\biggl(\Bigl(A_{jk}\bigl(\theta\dif\underline{u}(\bm{y})+(1-\theta)\dif u(\bm{y})\comma\theta\underline{u}(\bm{y})+(1-\theta)u(\bm{y})\bigr) \\
&\hphantom{\bm\lambda\biggl(\Bigl(}-\theta A_{jk}\bigl(\dif\underline{u}(\bm{y})\comma\underline{u}(\bm{y})\bigr)-(1-\theta)A_{jk}\bigl(\dif u(\bm{y})\comma u(\bm{y})\bigr)\Bigr)\biggr)\in\overline{\Gamma}_n.
\end{aligned} \end{eq}
Using \propref{prop: concavity-2} and \eqref{eq: lambda(g^{-1}(A(du, u)+nabla^2u))} again, it's easy to find that
\begin{eq}
\bm\lambda\biggl(\Bigl(A_{jk}\bigl(\theta\dif\underline{u}(\bm{y})+(1-\theta)\dif u(\bm{y})\comma\theta\underline{u}(\bm{y})+(1-\theta)u(\bm{y})\bigr)+\theta\nabla_{jk}\underline{u}(\bm{y})+(1-\theta)\nabla_{jk}u(\bm{y})\Bigr)\biggr)\in\Gamma
\end{eq}
and therefore $\theta\underline{u}+(1-\theta)u$ is a general $p$-admissible function.

Now fix $\bm{x}_0\in\mathcal{M}\colon (\underline{u}-u)(\bm{x}_0)=\max\limits_{\mathcal{M}}(\underline{u}-u)$, and let $(\mathcal{U}\comma\bm{\psi}_{\mathcal{U}}\mbox{; }x^j)$ be a local coordinate system containing $\bm{x}_0$. For any $\bm{x}\in\mathcal{U}$, there holds
\begin{eq} \begin{aligned}
0&\leqslant f\biggl(\bm\lambda\Bigl({\bm g}^{-1}\bigl(\bm{A}(\dif\underline{u}\comma\underline{u})+\nabla^2\underline{u}\bigr)\Bigr)(\bm{x})\biggr)-f\biggl(\bm\lambda\Bigl({\bm g}^{-1}\bigl(\bm{A}(\dif u\comma u)+\nabla^2u\bigr)\Bigr)(\bm{x})\biggr) \\
&\quad-\varphi\bigl(\dif\underline{u}(\bm{x})\comma\underline{u}(\bm{x})\bigr)+\varphi\bigl(\dif u(\bm{x})\comma u(\bm{x})\bigr) \\
&=\int_0^1\frac{\dif}{\dif\theta}f\Biggl(\bm\lambda\biggl((g^{jk})\Bigl(A_{jk}\bigl(\theta\dif\underline{u}+(1-\theta)\dif u\comma\theta\underline{u}+(1-\theta)u\bigr)+\theta\nabla_{jk}\underline{u}+(1-\theta)\nabla_{jk}u\Bigr)\biggr)(\bm{x})\Biggr)\dif\theta \\
&\quad-\int_0^1\frac{\dif}{\dif\theta}\varphi\bigl(\theta\dif\underline{u}(\bm{x})+(1-\theta)\dif u(\bm{x})\comma\theta\underline{u}(\bm{x})+(1-\theta)u(\bm{x})\bigr)\dif\theta.
\end{aligned} \end{eq}
Let $(\tilde{x}^j\comma\tilde\alpha_j)$ denote the canonical coordinate on $\mathrm{T}^*\mathcal{U}$ with respect to $(\mathcal{U}\comma\bm{\psi}_{\mathcal{U}}\mbox{; }x^j)$ (cf. \eqref{eq: tilde x^j, tilde alpha_j}) and recall \eqref{eq: F_{u, x}^{jk}}. It follows that
\begin{eq} \begin{aligned}
0&\leqslant\int_0^1 F_{\theta\underline{u}+(1-\theta)u\comma\bm{x}}^{jk}\left(\left.\frac{\partial A_{jk}\bigl(\cdot\comma t_\theta(\bm{x})\bigr)}{\partial\tilde{\alpha}_l}\right|_{\bm{\alpha}_\theta}\mathrm{D}_l(\underline{u}-u)+\left.\frac{\partial A_{jk}\bigl(\bm{\alpha}_\theta(\bm{x})\comma\cdot\bigr)}{\partial t}\right|_{t_\theta}(\underline{u}-u)\right. \\
&\phantom{=}\left.\vphantom{\left.\frac{\partial A_{jk}\bigl(\cdot\comma t_\theta(\bm{x})\bigr)}{\partial\tilde{\alpha}_l}\right|_{\bm{\alpha}_\theta}}%
+\nabla_{jk}(\underline{u}-u)\right)(\bm{x})\dif\theta-\int_0^1\left(\left.\frac{\partial\varphi\bigl(\cdot\comma t_\theta(\bm{x})\bigr)}{\partial\tilde{\alpha}_l}\right|_{\bm{\alpha}_\theta}\mathrm{D}_l(\underline{u}-u)+\left.\frac{\partial\varphi\bigl(\bm{\alpha}_\theta(\bm{x})\comma\cdot\bigr)}{\partial t}\right|_{t_\theta}(\underline{u}-u)\right)(\bm{x})\dif\theta \\
&=\int_0^1 F_{\theta\underline{u}+(1-\theta)u\comma\bm{x}}^{jk}\dif\theta\cdot\nabla_{jk}(\underline{u}-u)(\bm{x}) \\
&\quad+\int_0^1 \left(F_{\theta\underline{u}+(1-\theta)u\comma\bm{x}}^{jk}\left.\frac{\partial A_{jk}\bigl(\cdot\comma t_\theta(\bm{x})\bigr)}{\partial\tilde{\alpha}_l}\right|_{\bm{\alpha}_\theta(\bm{x})}-\left.\frac{\partial\varphi\bigl(\cdot\comma t_\theta(\bm{x})\bigr)}{\partial\tilde{\alpha}_l}\right|_{\bm{\alpha}_\theta(\bm{x})}\right)\dif\theta\cdot\mathrm{D}_l(\underline{u}-u)(\bm{x})\comma
\end{aligned} \end{eq}
where $\bm{\alpha}_\theta(\bm{x})\triangleq\theta\dif\underline{u}(\bm{x})+(1-\theta)\dif u(\bm{x})$, $t_\theta(\bm{x})\triangleq\theta\underline{u}(\bm{x})+(1-\theta)u(\bm{x})$ and we have used in the ``$=$'' the assumption that $\bm{A}(\bm\alpha\comma t)$, $\varphi(\bm\alpha\comma t)$ are both independent of $t$. In view of \propref{prop: properties of F_{u, x}^jk} and \propref{prop: increasing with respect to each variable}, the matrix $\bigl(F_{\theta\underline{u}+(1-\theta)u\comma\bm{x}}^{jk}\bigr)$ is positively definite for any $\theta\in[0\comma1]$. Recall \eqref{eq: du and nabla^2u} and define
\begin{ga}
a^{jk}(\bm{x})\triangleq\int_0^1 F_{\theta\underline{u}+(1-\theta)u\comma\bm{x}}^{jk}\dif\theta\comma \\
b^l(\bm{x})\triangleq a^{jk}(\bm{x})\Gamma_{jk}^l(\bm{x})-\int_0^1 \left(F_{\theta\underline{u}+(1-\theta)u\comma\bm{x}}^{jk}\left.\frac{\partial A_{jk}\bigl(\cdot\comma t_\theta(\bm{x})\bigr)}{\partial\tilde{\alpha}_l}\right|_{\bm{\alpha}_\theta(\bm{x})}-\left.\frac{\partial\varphi\bigl(\cdot\comma t_\theta(\bm{x})\bigr)}{\partial\tilde{\alpha}_l}\right|_{\bm{\alpha}_\theta(\bm{x})}\right)\dif\theta.
\end{ga}
Then there holds
\begin{eq} \label{eq: -a^{jk}(x)D_{jk}(underline{u}-u)(x)+b^l(x)D_l(underline{u}-u)(x) leqslant 0}
-a^{jk}(\bm{x})\mathrm{D}_{jk}(\underline{u}-u)(\bm{x})+b^l(\bm{x})\mathrm{D}_l(\underline{u}-u)(\bm{x})\leqslant 0\comma\forall\bm{x}\in\mathcal{U}.
\end{eq}
Let $\mathcal{U}_0$ be such a connected open neighbourhood of $\bm{x}_0$ that $\overline{\mathcal{U}}_0$ is compact and $\overline{\mathcal{U}}_0\subset\mathcal{U}$. Since $(\underline{u}-u)(\bm{x}_0)=\max\limits_{\overline{\mathcal{U}}_0}(\underline{u}-u)$, by the classical strong maximum principle (cf. e.g. \cite[p.~35]{Gilbarg2001}), \eqref{eq: -a^{jk}(x)D_{jk}(underline{u}-u)(x)+b^l(x)D_l(underline{u}-u)(x) leqslant 0} implies
\begin{eq}
(\underline{u}-u)(\bm{x})=(\underline{u}-u)(\bm{x}_0)=\max\limits_{\mathcal{M}}(\underline{u}-u)\comma\forall\bm{x}\in\mathcal{U}_0.
\end{eq}
As a result, we find that
\begin{eq}
\left\{\bm{x}\in\mathcal{M}\middle|(\underline{u}-u)(\bm{x})=\max\limits_{\mathcal{M}}(\underline{u}-u)\right\}
\end{eq}
is a non-empty open (and closed) subset of $\mathcal{M}$, so this subset is just $\mathcal{M}$ itself.
\end{proof}

Next, we prove several propositions related to when there exists a pseudo-solution to \eqref{eq: general p-Hessian}. One can refer to \defnref{defn: pseudo-solution} for the definitions of pseudo-subsolution condition, pseudo-supersolution condition and pseudo-solution.

\begin{prop} \label{prop: general p-Hessian, pseudo-subsolution condition}
Assume that $\bm{A}(\bm\alpha\comma t)$ is independent of $t$ and satisfies the concavity condition \eqref{eq: A_{vv}(alpha, t) is concave with respect to alpha}, $\underline{u}\in\mathrm{C}^2(\mathcal{M})$ and $\tilde\phi$ is a positive function on $\mathcal{M}$. If there exist positive constants $\delta$ and $R$, depending only on $\mathcal{M}$, $\bm{g}$, $n$, $p$, $\bm{A}(\bm{\alpha}\comma t)$, $\varphi(\bm{\alpha}\comma t)$ and $\underline{u}$, so that for any $\bm{x}\in\mathcal{M}$ there holds
\begin{eq} \label{eq: general p-Hessian, ``C-subsolution''}
\left\{\bm\xi\in\Gamma\middle|f(\bm\xi)=\tilde\phi(\bm{x})\right\}\cap\biggl(\Bigl\{\bm\lambda\Bigl({\bm g}^{-1}\bigl(\bm{A}(\dif\underline{u}\comma \underline{u})+\nabla^2\underline{u}\bigr)\Bigr)(\bm{x})-\delta\bm{1}_n\Bigr\}+\Gamma_n\biggr)\subset\mathrm{B}_R(\bm{0})\comma
\end{eq}
then $\underline{u}$ satisfies pseudo-subsolution condition (cf. \eqref{eq: general p-Hessian, pseudo-subsolution condition}).
\end{prop}

\begin{proof}
Assume that $C\in[0\comma{+}\infty)$ and
\begin{eq}
u\in\mathrm{C}^2(\mathcal{M})\colon\max\limits_{\mathcal{M}}\left|u\right|+\max\limits_{\mathcal{M}}\left|\dif u\right|_{\bm g}\leqslant C
\end{eq}
is a solution to \eqref{eq: general p-Hessian}. Let $(\mathcal{U}\comma\bm{\psi}_{\mathcal{U}}\mbox{; }x^j)$ be such a local coordinate system containing $\bm{x}$ that $g_{jk}(\bm{x})=\updelta_{jk}$. Then there holds
\begin{eq}
\bm\lambda\Bigl(\bigl(A_{jk}(\dif u\comma u)+\nabla_{jk}u\bigr)(\bm{x})\Bigr)\in\Gamma.
\end{eq}
By \eqref{eq: general p-Hessian, ``C-subsolution''} we also have
\begin{eq}
\left\{\bm\xi\in\Gamma\middle|f(\bm\xi)=\tilde\phi(\bm{x})\right\}\cap\biggl(\Bigl\{\bm\lambda\Bigl(\bigl(A_{jk}(\dif\underline{u}\comma\underline{u})+\nabla_{jk}\underline{u}\bigr)(\bm{x})\Bigr)-\delta\bm{1}_n\Bigr\}+\Gamma_n\biggr)\subset\mathrm{B}_R(\bm{0}).
\end{eq}
Thus, it follows from \propref{prop: C-subsolution, key lemma, matrix form} and \eqref{eq: F_{u, x}^{jk}} that
\begin{eq} \label{eq: F_{u, x}^{jk}(A_{jk}(d underline{u}(x), underline{u}(x))+nabla_{jk}underline{u}(x)-A_{jk}(du(x), u(x))-nabla_{jk}u(x)) geqslant delta sum_{j=1}^nF_{u, x}^{jj}-M lambda_1((F_{u, x}^{jk}))-M} \begin{aligned}
&\mathrel{\hphantom{=}}F_{u\comma\bm{x}}^{jk}\Bigl(A_{jk}\bigl(\dif\underline{u}(\bm{x})\comma\underline{u}(\bm{x})\bigr)+\nabla_{jk}\underline{u}(\bm{x})-A_{jk}\bigl(\dif u(\bm{x})\comma u(\bm{x})\bigr)-\nabla_{jk}u(\bm{x})\Bigr) \\
&\geqslant\delta\sum_{j=1}^nF_{u\comma\bm{x}}^{jj}+\tilde\phi(\bm{x})-\varphi\bigl(\dif u(\bm{x})\comma u(\bm{x})\bigr) \\
&\mathrel{\hphantom{=}}-\biggl(R+\left|\bm\lambda\Bigl(\bigl(A_{jk}(\dif\underline{u}\comma\underline{u})+\nabla_{jk}\underline{u}\bigr)(\bm{x})\Bigr)-\delta\bm{1}_n\right|\biggr)\lambda_1\Bigl(\left(F_{u\comma\bm{x}}^{jk}\right)\Bigr) \\
&\geqslant\delta\sum_{j=1}^nF_{u\comma\bm{x}}^{jj}-M\lambda_1\Bigl(\left(F_{u\comma\bm{x}}^{jk}\right)\Bigr)-M\comma
\end{aligned} \end{eq}
where
\begin{eq}
M\triangleq R+\max_{\mathcal{M}}\left|\bm\lambda\Bigl(\bm{g}^{-1}\bigl(\bm{A}(\dif\underline{u}\comma\underline{u})+\nabla^2\underline{u}\bigr)\Bigr)-\delta\bm{1}_n\right|+\max_{(\bm\alpha\comma t)\in\mathrm{T}^*\mathcal{M}\times\mathbb{R}\colon\left|\bm\alpha\right|_{\bm g}+\left|t\right|\leqslant C}\varphi(\bm\alpha\comma t).
\end{eq}
Since $\bm{A}(\bm\alpha\comma t)$ is independent of $t$ and satisfies \eqref{eq: A_{vv}(alpha, t) is concave with respect to alpha}, for any $\bm{v}\in\mathbb{R}^n$ there holds
\begin{eq} \begin{aligned}
&\mathrel{\hphantom{=}}\sum_{j\comma k=1}^n A_{jk}\bigl(\dif\underline{u}(\bm{x})\comma\underline{u}(\bm{x})\bigr)v_jv_k=\sum_{j\comma k=1}^n A_{jk}\bigl(\dif\underline{u}(\bm{x})\comma u(\bm{x})\bigr)v_jv_k \\
&\leqslant\sum_{j\comma k=1}^n A_{jk}\bigl(\dif u(\bm{x})\comma u(\bm{x})\bigr)v_jv_k+\left(\mathrm{D}_{\dif u(\bm{x})}\left[\sum_{j\comma k=1}^n A_{jk}\bigl(\cdot\comma u(\bm{x})\bigr)v_jv_k\right]\right)\bigl(\dif\underline{u}(\bm{x})-\dif u(\bm{x})\bigr) \\
&=\sum_{j\comma k=1}^n \biggl(A_{jk}\bigl(\dif u(\bm{x})\comma u(\bm{x})\bigr)+\Bigl(\mathrm{D}_{\dif u(\bm{x})}\bigl[A_{jk}\bigl(\cdot\comma u(\bm{x})\bigr)\bigr]\Bigr)\bigl(\dif\underline{u}(\bm{x})-\dif u(\bm{x})\bigr)\biggr)v_jv_k\comma
\end{aligned} \end{eq}
where $\mathrm{D}_{\bm\alpha}[\cdot]$ denotes the Fr\'echet differentiation at $\bm{\alpha}\in(\mathrm{T}_{\bm x}^*\mathcal{M}\comma\left|\cdot\right|_{\bm g})$ (cf. \eqref{eq: Frechet}). Thus, we have
\begin{eq}
F_{u\comma\bm{x}}^{jk}A_{jk}\bigl(\dif\underline{u}(\bm{x})\comma\underline{u}(\bm{x})\bigr)\leqslant F_{u\comma\bm{x}}^{jk}\biggl(A_{jk}\bigl(\dif u(\bm{x})\comma u(\bm{x})\bigr)+\Bigl(\mathrm{D}_{\dif u(\bm{x})}\bigl[A_{jk}\bigl(\cdot\comma u(\bm{x})\bigr)\bigr]\Bigr)\bigl(\dif\underline{u}(\bm{x})-\dif u(\bm{x})\bigr)\biggr)
\end{eq}
since $(F_{u\comma\bm{x}}^{jk})$ is positively definite. Combining \eqref{eq: F_{u, x}^{jk}(A_{jk}(d underline{u}(x), underline{u}(x))+nabla_{jk}underline{u}(x)-A_{jk}(du(x), u(x))-nabla_{jk}u(x)) geqslant delta sum_{j=1}^nF_{u, x}^{jj}-M lambda_1((F_{u, x}^{jk}))-M} and the above inequality, we obtain
\begin{eq} \begin{aligned}
&\mathrel{\hphantom{=}}F_{u\comma\bm{x}}^{jk}\nabla_{jk}(\underline{u}-u)(\bm{x})+F_{u\comma\bm{x}}^{jk}\Bigl(\mathrm{D}_{\dif u(\bm{x})}\bigl[A_{jk}\bigl(\cdot\comma u(\bm{x})\bigr)\bigr]\Bigr)\bigl(\dif\left(\underline{u}-u\right)(\bm{x})\bigr) \\
&\geqslant\delta F_{u\comma\bm{x}}^{jk}g_{jk}(\bm{x})-M\lambda_1\Bigl((F_{u\comma\bm x}^{jk})\bigl(g_{jk}(\bm{x})\bigr)\Bigr)-M.
\end{aligned} \end{eq}
\end{proof}

For the case when $\varphi(\bm{\alpha}\comma t)$ only depends on $\pi(\bm\alpha)$, recall the definition of $\mathcal{C}$-subsolution (cf. \eqref{eq: p-Hessian, C-subsolution}) and we have the following proposition.

\begin{prop} \label{prop: C-subsolution is pseudo-solution}
Assume that $\varphi(\bm{\alpha}\comma t)=\phi\bigl(\pi(\bm{\alpha})\bigr)$ for some positive smooth function $\phi$ on $\mathcal{M}$, $\bm{A}(\bm\alpha\comma t)$ is independent of $t$ and satisfies the concavity condition \eqref{eq: A_{vv}(alpha, t) is concave with respect to alpha}. Then every $\mathcal{C}$-subsolution $\underline{u}\in\mathrm{C}^2(\mathcal{M})$ to \eqref{eq: general p-Hessian} defined by \eqref{eq: p-Hessian, C-subsolution} is a pseudo-solution to \eqref{eq: general p-Hessian}.
\end{prop}

\begin{proof}
By definition, there exist positive constants $\delta$ and $R$, depending only on $\mathcal{M}$, $\bm{g}$, $n$, $p$, $\bm{A}(\bm{\alpha}\comma t)$, $\phi$ and $\underline{u}$, so that for any $\bm{x}\in\mathcal{M}$, there holds
\begin{eq} \label{eq: p-Hessian, C-subsolution, section ``Subsolutions''}
\left\{\bm\xi\in\Gamma\middle|f(\bm\xi)=\phi(\bm{x})\right\}\cap\biggl(\Bigl\{\bm\lambda\Bigl({\bm g}^{-1}\bigl(\bm{A}(\dif\underline{u}\comma\underline{u})+\nabla^2\underline{u}\bigr)\Bigr)(\bm{x})-\delta\bm{1}_n\Bigr\}+\Gamma_n\biggr)\subset\mathrm{B}_R(\bm{0}).
\end{eq}
By \propref{prop: general p-Hessian, pseudo-subsolution condition}, $\underline{u}$ satisfies pseudo-subsolution condition and we only need to prove that $\underline{u}$ satisfies pseudo-supersolution condition (cf. \eqref{eq: general p-Hessian, pseudo-supersolution condition}). Assume that $u\in\mathrm{C}^2(\mathcal{M})$ is a solution to \eqref{eq: general p-Hessian}, $\varepsilon\in(0\comma\delta)$ and
\begin{eq}
\bm{x}_0\in\mathcal{M}\colon\bm\lambda\Bigl({\bm g}^{-1}\bigl(\nabla^2(u-\underline{u})\bigr)\Bigr)(\bm{x}_0)+\varepsilon\bm{1}_n\in\overline{\Gamma}_n\comma\left|\dif\left(u-\underline{u}\right)(\bm{x}_0)\right|_{\bm g}\leqslant\varepsilon.
\end{eq}
Let $(\mathcal{U}\comma\bm{\psi}_{\mathcal{U}}\mbox{; }x^j)$ be such a local coordinate system containing $\bm{x}_0$ that $g_{jk}(\bm{x}_0)=\updelta_{jk}$. In view of \lemref{lem: Weyl}, for any $q\in\{1\comma2\comma\cdots\comma n\}$ there holds
\begin{eq} \begin{aligned}
&\mathrel{\hphantom{=}}\lambda_q\Bigl(\bigl(A_{jk}(\dif u\comma u)+\nabla_{jk}u\bigr)(\bm{x}_0)\Bigr)-\lambda_q\Bigl(\bigl(A_{jk}(\dif\underline{u}\comma\underline{u})+\nabla_{jk}\underline{u}\bigr)(\bm{x}_0)\Bigr)+\delta \\
&\geqslant\lambda_1\Bigl(\bigl(A_{jk}(\dif u\comma u)+\nabla_{jk}u-A_{jk}(\dif\underline{u}\comma\underline{u})-\nabla_{jk}\underline{u}\bigr)(\bm{x}_0)\Bigr)+\delta \\
&\geqslant\lambda_1\Bigl(\bigl(A_{jk}(\dif u\comma u)-A_{jk}(\dif\underline{u}\comma\underline{u})\bigr)(\bm{x}_0)\Bigr)+\lambda_1\Bigl(\bigl(\nabla_{jk}(u-\underline{u})\bigr)(\bm{x}_0)\Bigr)+\delta.
\end{aligned} \end{eq}
Since $\bm{A}(\bm\alpha\comma t)$ is independent of $t$ and $\left|\dif u(\bm{x}_0)\right|_{\bm g}\leqslant\left|\dif\underline{u}(\bm{x}_0)\right|_{\bm g}+\varepsilon$, for any $\bm{v}\in\mathbb{R}^n$ there holds
\begin{eq} \begin{aligned}
&\mathrel{\hphantom{=}}\sum_{j\comma k=1}^n A_{jk}\bigl(\dif u(\bm{x}_0)\comma u(\bm{x}_0)\bigr)v_jv_k=\sum_{j\comma k=1}^n A_{jk}\bigl(\dif u(\bm{x}_0)\comma\underline{u}(\bm{x}_0)\bigr)v_jv_k \\
&\geqslant\sum_{j\comma k=1}^n A_{jk}\bigl(\dif\underline{u}(\bm{x}_0)\comma\underline{u}(\bm{x}_0)\bigr)v_jv_k-C_1\left|\dif u(\bm{x}_0)-\dif\underline{u}(\bm{x}_0)\right|_{\bm g}\left|\bm{v}\right|^2\comma
\end{aligned} \end{eq}
where we have used the mean value theorem for Fr\'echet differentiation, the Cauchy inequality and
\begin{eq} \label{eq: pseudo-supersolution condition, C_1}
C_1\triangleq\max_{\substack{(\bm\alpha\comma\bm\beta\comma t)\in\mathrm{T}^*\mathcal{M}\times\mathrm{T}^*\mathcal{M}\times\mathbb{R} \\ \left|\bm\alpha\right|_{\bm g}\leqslant\max\limits_{\mathcal{M}}\left|\dif\underline{u}\right|_{\bm g}+\delta\comma|t|\leqslant\max\limits_{\mathcal{M}}\left|\underline{u}\right|\comma\left|\bm\beta\right|_{\bm g}=1\comma\pi(\bm\alpha)=\pi(\bm\beta)}} \left|\Bigl(\mathrm{D}_{\bm\alpha}\bigl[\bigl(\bm{A}(\cdot\comma t)\bigr)(\cdot\comma\cdot)\bigr]\Bigr)(\bm\beta)\right|_{\bm g}\comma
\end{eq}
recall \eqref{eq: moduli of alpha and B} for the definition of $\left|\bm{B}\right|_{\bm g}$. As a result, letting $\varepsilon=\frac{\delta}{C_1+2}$, we have
\begin{eq} \begin{aligned}
&\mathrel{\hphantom{=}}\lambda_q\Bigl(\bigl(A_{jk}(\dif u\comma u)+\nabla_{jk}u\bigr)(\bm{x}_0)\Bigr)-\lambda_q\Bigl(\bigl(A_{jk}(\dif\underline{u}\comma\underline{u})+\nabla_{jk}\underline{u}\bigr)(\bm{x}_0)\Bigr)+\delta \\
&\geqslant-C_1\left|\dif u(\bm{x}_0)-\dif\underline{u}(\bm{x}_0)\right|_{\bm g}+\lambda_1\Bigl(\bigl(\nabla_{jk}(u-\underline{u})\bigr)(\bm{x}_0)\Bigr)+\delta\geqslant\delta-(C_1+1)\varepsilon>0.
\end{aligned} \end{eq}
This implies
\begin{eq}
\bm\lambda\Bigl(\bigl(A_{jk}(\dif u\comma u)+\nabla_{jk}u\bigr)(\bm{x}_0)\Bigr)-\bm\lambda\Bigl(\bigl(A_{jk}(\dif\underline{u}\comma\underline{u})+\nabla_{jk}\underline{u}\bigr)(\bm{x}_0)\Bigr)+\delta\bm{1}_n\in\Gamma_n
\end{eq}
and therefore
\begin{eq}
\left|\bm\lambda\Bigl(\bigl(A_{jk}(\dif u\comma u)+\nabla_{jk}u\bigr)(\bm{x}_0)\Bigr)\right|<R\comma
\end{eq}
in view of \eqref{eq: p-Hessian, C-subsolution, section ``Subsolutions''}. It follows that
\begin{eq} \begin{aligned}
R&>\lambda_n\Bigl(\bigl(A_{jk}(\dif u\comma u)+\nabla_{jk}u\bigr)(\bm{x}_0)\Bigr) \\
&\geqslant\lambda_n\Bigl(\bigl(\nabla_{jk}(u-\underline{u})\bigr)(\bm{x}_0)\Bigr)+\lambda_1\Bigl(\bigl(A_{jk}(\dif u\comma\underline{u})+\nabla_{jk}\underline{u}\bigr)(\bm{x}_0)\Bigr)\geqslant\lambda_n\Bigl(\bigl(\nabla_{jk}(u-\underline{u})\bigr)(\bm{x}_0)\Bigr)-C_2\comma
\end{aligned} \end{eq}
where
\begin{eq}
C_2\triangleq\max_{\substack{(\bm\alpha\comma t)\in\mathrm{T}^*\mathcal{M}\times\mathbb{R} \\ \left|\bm\alpha\right|_{\bm g}\leqslant\max\limits_{\mathcal{M}}\left|\dif\underline{u}\right|_{\bm g}+\delta\comma|t|\leqslant\max\limits_{\mathcal{M}}\left|\underline{u}\right|}}\left|\bm{A}(\bm\alpha\comma t)\right|_{\bm g}+\max_{\mathcal{M}}|\nabla^2\underline{u}|_{\bm g}.
\end{eq}
Finally, since
\begin{eq}
\lambda_1\Bigl(\bigl(\nabla_{jk}(u-\underline{u})\bigr)(\bm{x}_0)\Bigr)+\varepsilon\geqslant 0\comma
\end{eq}
there holds
\begin{eq} \begin{aligned}
\sigma_n\biggl(\bm\lambda\Bigl({\bm g}^{-1}\bigl(\varepsilon\bm{g}+\nabla^2(u-\underline{u})\bigr)\Bigr)(\bm{x}_0)\biggr)&=\sigma_n\biggl(\bm\lambda\Bigl(\bigl(\nabla_{jk}(u-\underline{u})\bigr)(\bm{x}_0)\Bigr)+\varepsilon\bm{1}_n\biggr) \\
&\leqslant\biggl(\lambda_n\Bigl(\bigl(\nabla_{jk}(u-\underline{u})\bigr)(\bm{x}_0)\Bigr)+\varepsilon\biggr)^n\leqslant(C_2+R+\delta)^n
\end{aligned} \end{eq}
and therefore $\underline{u}$ also satisfies pseudo-supersolution condition.
\end{proof}

We point out again that $\Gamma$, $f$ denote $\Gamma_p$, $\sigma_p^{\frac1p}$ respectively in \propref{prop: C-subsolution, key lemma, matrix form}--\propref{prop: C-subsolution is pseudo-solution}. For the rest of this section, however, we directly use the notations $\Gamma_p$, $\sigma_p^{\frac1p}$ instead of $\Gamma$, $f$, because the corresponding conclusions are significantly different from those for the case of general orthogonally-invariant elliptic equations.

\begin{prop} \label{prop: general p-Hessian, pseudo-supersolution condition}
Assume that $R\in[0\comma{+}\infty)$, $\bm{A}(\bm\alpha\comma t)$, $\varphi(\bm\alpha\comma t)$ are both independent of $t$, and $\underline{u}\in\mathrm{C}^2(\mathcal{M})$ satisfies
\begin{eq}
\bm\lambda\Bigl({\bm g}^{-1}\bigl(\bm{A}(\dif\underline{u}\comma\underline{u})+\nabla^2\underline{u}\bigr)\Bigr)(\bm{x})+R\mathbf{e}_n\in\Gamma_p\comma\forall\bm{x}\in\mathcal{M}.
\end{eq}
Then $\underline{u}$ satisfies pseudo-supersolution condition (cf. \eqref{eq: general p-Hessian, pseudo-supersolution condition}).
\end{prop}

\begin{proof}
Note that there exists such $\delta\in(0\comma{+}\infty)$ that
\begin{eq}
\bm\lambda\Bigl({\bm g}^{-1}\bigl(\bm{A}(\dif\underline{u}\comma\underline{u})+\nabla^2\underline{u}\bigr)\Bigr)(\bm{x})+R\mathbf{e}_n-\delta\bm{1}_n\in\Gamma_p\comma\forall\bm{x}\in\mathcal{M}.
\end{eq}
Assume that $u\in\mathrm{C}^2(\mathcal{M})$ is a solution to \eqref{eq: general p-Hessian}, $\varepsilon\in(0\comma\delta)$ and
\begin{eq}
\bm{x}_0\in\mathcal{M}\colon\bm\lambda\Bigl({\bm g}^{-1}\bigl(\nabla^2(u-\underline{u})\bigr)\Bigr)(\bm{x}_0)+\varepsilon\bm{1}_n\in\overline{\Gamma}_n\comma\left|\dif\left(u-\underline{u}\right)(\bm{x}_0)\right|_{\bm g}\leqslant\varepsilon.
\end{eq}
Let $(\mathcal{U}\comma\bm{\psi}_{\mathcal{U}}\mbox{; }x^j)$ be such a local coordinate system containing $\bm{x}_0$ that $g_{jk}(\bm{x}_0)=\updelta_{jk}$. By the proof of \propref{prop: C-subsolution is pseudo-solution}, for any $q\in\{1\comma2\comma\cdots\comma n-1\}$ we have
\begin{eq}
\lambda_q\Bigl(\bigl(A_{jk}(\dif u\comma u)+\nabla_{jk}u\bigr)(\bm{x}_0)\Bigr)-\lambda_q\Bigl(\bigl(A_{jk}(\dif\underline{u}\comma\underline{u})+\nabla_{jk}\underline{u}\bigr)(\bm{x}_0)\Bigr)\geqslant-(C_1+1)\varepsilon\comma
\end{eq}
where $C_1$ is defined as \eqref{eq: pseudo-supersolution condition, C_1}. We may as well assume that
\begin{eq}
\lambda_n\Bigl(\bigl(A_{jk}(\dif u\comma u)+\nabla_{jk}u\bigr)(\bm{x}_0)\Bigr)-\lambda_n\Bigl(\bigl(A_{jk}(\dif\underline{u}\comma\underline{u})+\nabla_{jk}\underline{u}\bigr)(\bm{x}_0)\Bigr)\geqslant R-(C_1+1)\varepsilon\comma
\end{eq}
As a result, letting $\varepsilon=\frac{\delta}{C_1+2}$, we have
\begin{eq}
\bm\mu\triangleq\bm\lambda\Bigl(\bigl(A_{jk}(\dif u\comma u)+\nabla_{jk}u\bigr)(\bm{x}_0)\Bigr)-\bm\lambda\Bigl(\bigl(A_{jk}(\dif\underline{u}\comma\underline{u})+\nabla_{jk}\underline{u}\bigr)(\bm{x}_0)\Bigr)+\delta\bm{1}_n-R\mathbf{e}_n\in\Gamma_n.
\end{eq}
Let $\bm\nu$ denote
\begin{eq}
\bm\lambda\Bigl(\bigl(A_{jk}(\dif\underline{u}\comma\underline{u})+\nabla_{jk}\underline{u}\bigr)(\bm{x}_0)\Bigr)+R\mathbf{e}_n-\delta\bm{1}_n\in\Gamma_p.
\end{eq}
Then there holds
\begin{eq} \begin{aligned}
\max_{\substack{(\bm\alpha\comma t)\in\mathrm{T}^*\mathcal{M}\times\mathbb{R} \\ \left|\bm\alpha\right|_{\bm g}\leqslant\max\limits_{\mathcal{M}}\left|\dif\underline{u}\right|_{\bm g}+\delta\comma|t|\leqslant\max\limits_{\mathcal{M}}\left|\underline{u}\right|}}\varphi^p(\bm\alpha\comma t)&\geqslant\varphi^p\bigl(\dif u(\bm{x}_0)\comma\underline{u}(\bm{x}_0)\bigr)=\varphi^p\bigl(\dif u(\bm{x}_0)\comma u(\bm{x}_0)\bigr) \\
&=\sigma_p(\bm\nu+\bm\mu)\geqslant\sigma_p(\bm\nu)+\sum_{j=1}^n\sigma_{p-1}(\bm\nu|j)\mu_j \\
&>\min_{\mathcal{M}}\sigma_{p-1}\biggl(\bm\lambda\Bigl({\bm g}^{-1}\bigl(\bm{A}(\dif\underline{u}\comma\underline{u})+\nabla^2\underline{u}\bigr)\Bigr)+R\mathbf{e}_n-\delta\bm{1}_n\biggm|n\biggr)\cdot\mu_n\comma
\end{aligned} \end{eq}
where the second ``$\geqslant$'' uses the mean value theorem for Fr\'echet differentiation, \propref{prop: increasing with respect to each variable} and \corref{cor: concavity-1}. It follows that $\mu_n\leqslant C_3$ and
\begin{eq}
\lambda_n\Bigl(\bigl(A_{jk}(\dif u\comma u)+\nabla_{jk}u\bigr)(\bm{x}_0)\Bigr)\leqslant C_3+R+\lambda_n\Bigl(\bigl(A_{jk}(\dif\underline{u}\comma\underline{u})+\nabla_{jk}\underline{u}\bigr)(\bm{x}_0)\Bigr).
\end{eq}
The rest of the proof is similar to that of \propref{prop: C-subsolution is pseudo-solution} and therefore omitted.
\end{proof}

At last, we give the proof of \propref{prop: pseudo-solution, basic}.

\begin{proof}
In view of \propref{prop: general p-Hessian, pseudo-supersolution condition}, $v$ obviously satisfies pseudo-supersolution condition. Note that there exists such $\delta\in(0\comma{+}\infty)$ that
\begin{eq}
\bm\lambda\Bigl({\bm g}^{-1}\bigl(\bm{A}(\dif v\comma v)+\nabla^2v\bigr)\Bigr)(\bm{x})+R\mathbf{e}_j-\delta\bm{1}_n\in\Gamma_p\comma\forall\bm{x}\in\mathcal{M}\comma\forall j\in\{1\comma2\comma\cdots\comma n\}.
\end{eq}
Fix $\bm{x}_0\in\mathcal{M}$ and let $\bm\mu$ denote
\begin{eq}
\bm\lambda\Bigl({\bm g}^{-1}\bigl(\bm{A}(\dif v\comma v)+\nabla^2v\bigr)\Bigr)(\bm{x}_0).
\end{eq}
For any
\begin{eq}
\bm\nu\in\Gamma_n\colon\bm\mu-\delta\bm{1}_n+\bm\nu\in\Gamma_p\setminus\mathrm{B}_{\sqrt{n}R+\left|\bm\mu-\delta\bm{1}_n\right|+1}(\bm{0})\comma
\end{eq}
there exists such $k\in\{1\comma2\comma\cdots\comma n\}$ that $\bm\nu-R\mathbf{e}_k\in\Gamma_n$ and therefore
\begin{eq}
\sigma_p^{\frac1p}(\bm\mu-\delta\bm{1}_n+\bm\nu)=\sigma_p^{\frac1p}(\bm\mu+R\mathbf{e}_k-\delta\bm{1}_n+\bm\nu-R\mathbf{e}_k)>\sigma_p^{\frac1p}(\bm\mu+R\mathbf{e}_k-\delta\bm{1}_n)\geqslant\varepsilon_0\comma
\end{eq}
where we have used \corref{cor: concavity-1} and
\begin{eq}
\varepsilon_0\triangleq\min_{\bm{x}\in\mathcal{M}}\min_{j\in\{1\comma2\comma\cdots\comma n\}} \sigma_p^{\frac1p}\biggl(\bm\lambda\Bigl({\bm g}^{-1}\bigl(\bm{A}(\dif v\comma v)+\nabla^2v\bigr)\Bigr)(\bm{x})+R\mathbf{e}_j-\delta\bm{1}_n\biggr)>0.
\end{eq}
Thus, there holds
\begin{eq}
\left\{\bm\xi\in\Gamma_p\middle|\sigma_p^{\frac1p}(\bm\xi)=\frac{\varepsilon_0}{2}\right\}\cap\biggl(\Bigl\{\bm\lambda\Bigl({\bm g}^{-1}\bigl(\bm{A}(\dif v\comma v)+\nabla^2v\bigr)\Bigr)(\bm{x}_0)-\delta\bm{1}_n\Bigr\}+\Gamma_n\biggr)\subset\mathrm{B}_M(\bm{0})\comma
\end{eq}
where
\begin{eq}
M\triangleq\sqrt{n}R+\max_{\mathcal{M}}\left|\bm\lambda\Bigl({\bm g}^{-1}\bigl(\bm{A}(\dif v\comma v)+\nabla^2v\bigr)\Bigr)-\delta\bm{1}_n\right|+1.
\end{eq}
By \propref{prop: general p-Hessian, pseudo-subsolution condition} we find that $v$ also satisfies pseudo-subsolution condition, so $v$ is a pseudo-solution to \eqref{eq: general p-Hessian}.
\end{proof}


\section{Concavity inequalities related to $(n-1)$-Hessian equations} \label{sec: Concavity inequalities related to (n-1)-Hessian equations}

In this section, we discuss concavity inequalities related to real and complex $(n-1)$-Hessian equations. We fix a positive integer $n$ and let $\bm\mu$ denote a generic real $n$-vector $(\mu_1\comma\mu_2\comma\cdots\comma\mu_n)^{\mathrm T}$ throughout this section. The following proposition generalizes the concavity inequality in \cite[p.~4]{Lu2024}, see also \cite[p.~7]{Tu2024}.

\begin{prop} \label{prop: concavity inequality, -mu_1 is large}
Assume that $n\geqslant 3$, $\tau\in[0\comma1]$, $\varepsilon\in(0\comma{+}\infty)$, $a\in\left(\frac{1-\tau}{1+\tau}\comma n-1\right]$ and
\begin{eq} \label{eq: concavity inequality, -mu_1 is large, assumption}
\bm\mu\in\Gamma_{n-1}\colon\mu_1\leqslant\mu_2\leqslant\cdots\leqslant\mu_n\comma%
\mu_n\geqslant\frac{\varepsilon\left(a+\frac{1-\tau}{1+\tau}\right)}{a-\frac{1-\tau}{1+\tau}}\comma\mu_1\leqslant-\left(\frac{2\sigma_{n-1}(\bm\mu)}{a-\frac{1-\tau}{1+\tau}}\right)^{\frac{1}{n-1}}.
\end{eq}
Then there holds
\begin{eq} \label{eq: concavity inequality, -mu_1 is large, conclusion} \begin{aligned}
&\mathrel{\hphantom{=}}-\sum_{\substack{1\leqslant j\comma k\leqslant n \\ j\ne k}} \sigma_{n-3}(\bm\mu|jk)w_j\overline{w}_k-\frac{1-\tau}{\mu_n}\sigma_{n-2}(\bm\mu|n)\left|w_n\right|^2 \\
&\geqslant -\frac{1}{\sigma_{n-1}(\bm\mu)}\left|\sum_{j=1}^n \sigma_{n-2}(\bm\mu|j)w_j\right|^2-\frac{2a}{n-1}\sum_{j=1}^{n-1} \frac{\sigma_{n-2}(\bm\mu|j)\left|w_j\right|^2}{\mu_n+\varepsilon-\mu_j}
\end{aligned} \end{eq}
for any $\bm w=(w_1\comma w_2\comma\cdots\comma w_n)^{\mathrm T}\in\mathbb{C}^n$.
\end{prop}

\begin{proof}
The proof is somewhat similar to that in \cite{Lu2024}. Since $\bm\mu\in\Gamma_{n-1}$ and $\mu_1<0$ (cf. \eqref{eq: concavity inequality, -mu_1 is large, assumption}), by \eqref{eq: sum_{j=1}^{n-p+1} mu_j>0} we have $\mu_2>-\mu_1>0$ and therefore
\begin{eq}
-\sigma_n(\bm\mu)>(-\mu_1)^{n-1}\mu_n.
\end{eq}
It follows from \eqref{eq: concavity inequality, -mu_1 is large, assumption} and the above inequality that
\begin{eq} \label{eq: concavity inequality, -mu_1 is large, -sigma_n(mu)>}
-\sigma_n(\bm\mu)>\frac{2\sigma_{n-1}(\bm\mu)\mu_n}{a-\frac{1-\tau}{1+\tau}}\geqslant\frac{\sigma_{n-1}(\bm\mu)\mu_n}{\frac{a\mu_n}{\mu_n+\varepsilon}-\frac{1-\tau}{1+\tau}}>0.
\end{eq}
The equalities (3.2), (3.3) of \cite[pp.~4--5]{Lu2024} read
\begin{ga}
\sigma_{n-2}(\bm\mu|j)=\frac{\sigma_{n-1}(\bm\mu)}{\mu_j}-\frac{\sigma_n(\bm\mu)}{\mu_j^2}\comma \\
\sigma_{n-3}(\bm\mu|jk)=\frac{\sigma_{n-1}(\bm\mu)}{\mu_j\mu_k}-\frac{\sigma_n(\bm\mu)(\mu_j+\mu_k)}{\mu_j^2\mu_k^2}\ (j\ne k).
\end{ga}
Analogous to the equalities (3.4), (3.5) of \cite[p.~5]{Lu2024}, we have
\begin{eq} \begin{aligned}
&\mathrel{\hphantom{=}}-\sum_{\substack{1\leqslant j\comma k\leqslant n \\ j\ne k}} \sigma_{n-3}(\bm\mu|jk)w_j\overline{w}_k+\frac{1}{\sigma_{n-1}(\bm\mu)}\left|\sum_{j=1}^n \sigma_{n-2}(\bm\mu|j)w_j\right|^2 \\
&=\frac{\sigma_n^2(\bm\mu)}{\sigma_{n-1}(\bm\mu)\mu_j^2\mu_k^2}w_j\overline{w}_k+\frac{\sigma_{n-1}(\bm\mu)}{\mu_j^2}\left|w_j\right|^2-\frac{2\sigma_n(\bm\mu)}{\mu_j^3}\left|w_j\right|^2
\end{aligned} \end{eq}
and
\begin{eq} \begin{aligned}
&\mathrel{\hphantom{=}}\frac{2a}{n-1}\sum_{j=1}^{n-1} \frac{\sigma_{n-2}(\bm\mu|j)\left|w_j\right|^2}{\mu_n+\varepsilon-\mu_j}-\frac{1-\tau}{\mu_n}\sigma_{n-2}(\bm\mu|n)\left|w_n\right|^2 \\
&=\sigma_{n-1}(\bm\mu)\left(\frac{2a}{n-1}\sum_{j=1}^{n-1}\frac{\left|w_j\right|^2}{\mu_j(\mu_n+\varepsilon-\mu_j)}-\frac{(1-\tau)\left|w_n\right|^2}{\mu_n^2}\right) \\
&\quad-\sigma_n(\bm\mu)\left(\frac{2a}{n-1}\sum_{j=1}^{n-1}\frac{\left|w_j\right|^2}{\mu_j^2(\mu_n+\varepsilon-\mu_j)}-\frac{(1-\tau)\left|w_n\right|^2}{\mu_n^3}\right).
\end{aligned} \end{eq}
Define
\begin{ga}
T_1\triangleq\frac{1}{\mu_j^2}\left|w_j\right|^2+\frac{2a}{n-1}\sum_{j=1}^{n-1}\frac{\left|w_j\right|^2}{\mu_j(\mu_n+\varepsilon-\mu_j)}-\frac{(1-\tau)\left|w_n\right|^2}{\mu_n^2}\comma \\
T_2\triangleq\frac{-\sigma_n(\bm\mu)}{\sigma_{n-1}(\bm\mu)\mu_j^2\mu_k^2}w_j\overline{w}_k+\frac{2}{\mu_j^3}\left|w_j\right|^2+\frac{2a}{n-1}\sum_{j=1}^{n-1}\frac{\left|w_j\right|^2}{\mu_j^2(\mu_n+\varepsilon-\mu_j)}-\frac{(1-\tau)\left|w_n\right|^2}{\mu_n^3}.
\end{ga}
Then there holds
\begin{eq} \label{eq: concavity inequality, -mu_1 is large, 1} \begin{aligned}
&\mathrel{\hphantom{=}}-\sum_{\substack{1\leqslant j\comma k\leqslant n \\ j\ne k}} \sigma_{n-3}(\bm\mu|jk)w_j\overline{w}_k-\frac{1-\tau}{\mu_n}\sigma_{n-2}(\bm\mu|n)\left|w_n\right|^2 \\
&=-\frac{1}{\sigma_{n-1}(\bm\mu)}\left|\sum_{j=1}^n \sigma_{n-2}(\bm\mu|j)w_j\right|^2-\frac{2a}{n-1}\sum_{j=1}^{n-1} \frac{\sigma_{n-2}(\bm\mu|j)\left|w_j\right|^2}{\mu_n+\varepsilon-\mu_j}+\sigma_{n-1}(\bm\mu)T_1-\sigma_n(\bm\mu)T_2.
\end{aligned} \end{eq}
Recalling $\mu_2>-\mu_1>0$, note that
\begin{eq} \label{eq: concavity inequality, -mu_1 is large, T_1 geqslant 0} \begin{aligned}
T_1&=\sum_{j=1}^{n-1}\frac{\left|w_j\right|^2}{\mu_j^2}+\frac{\tau\left|w_n\right|^2}{\mu_n^2}+\frac{2a}{n-1}\sum_{j=1}^{n-1}\frac{\left|w_j\right|^2}{\mu_j(\mu_n+\varepsilon-\mu_j)} \\
&\geqslant\frac{\left|w_1\right|^2}{\mu_1^2}+\frac{2a}{n-1}\frac{\left|w_1\right|^2}{\mu_1(\mu_n+\varepsilon-\mu_1)}=\frac{\mu_n+\varepsilon+\left(\frac{2a}{n-1}-1\right)\mu_1}{\mu_1^2(\mu_n+\varepsilon-\mu_1)}\left|w_1\right|^2\geqslant 0\comma
\end{aligned} \end{eq}
where we have used $a\leqslant n-1$ in the last ``$\geqslant$''.

It remains to prove $T_2\geqslant 0$. Letting $v_j$ denote $\frac{w_j}{\mu_j^2}$, we find that
\begin{eq} \label{eq: concavity inequality, -mu_1 is large, T_2=v^TA overline{v}}
T_2=\frac{-\sigma_n(\bm\mu)}{\sigma_{n-1}(\bm\mu)}\sum_{1\leqslant j\comma k\leqslant n}v_j\overline{v}_k+(1+\tau)\mu_n\left|v_n\right|^2+\sum_{j=1}^{n-1}\frac{2\mu_j\left(\mu_n+\varepsilon+\frac{a-n+1}{n-1}\mu_j\right)}{\mu_n+\varepsilon-\mu_j}\left|v_j\right|^2=\bm{v}^{\mathrm T}\mathbf{A}\overline{\bm{v}}\comma
\end{eq}
where $\bm{v}\triangleq(v_1\comma v_2\comma\cdots\comma v_n)^{\mathrm T}$ and $\mathbf{A}$ denotes the following matrix ($\bm{1}_n$ is defined in \eqref{eq: 1_n}):
\begin{eq}
\frac{-\sigma_n(\bm\mu)}{\sigma_{n-1}(\bm\mu)}\bm{1}_n\bm{1}_n^{\mathrm T}+\diag\left\{\frac{2\mu_1\left(\mu_n+\varepsilon+\frac{a-n+1}{n-1}\mu_1\right)}{\mu_n+\varepsilon-\mu_1}\comma\cdots\comma\frac{2\mu_{n-1}\left(\mu_n+\varepsilon+\frac{a-n+1}{n-1}\mu_{n-1}\right)}{\mu_n+\varepsilon-\mu_{n-1}}\comma(1+\tau)\mu_n\right\}.
\end{eq}
Recall that $a\in\left(\frac{1-\tau}{1+\tau}\comma n-1\right]\subset(0\comma n-1]$ and $\varepsilon>0$. For any $j\in\{1\comma2\comma\cdots\comma n-1\}$, there holds
\begin{eq}
\mu_n+\varepsilon+\frac{a-n+1}{n-1}\mu_j>0.
\end{eq}
Since $\bm{1}_n\bm{1}_n^{\mathrm T}$ is positively semi-definite, by the Weyl inequalities (cf. e.g. \cite[p.~112]{Serre2010}) we have
\begin{eq} \label{eq: concavity inequality, -mu_1 is large, lambda_2(A)>0}
\lambda_2(\mathbf{A})>0.
\end{eq}
Here $\lambda_2$ denotes the second minimum eigenvalue function (cf. \defnref{defn: lambda_q} and the last paragraph before \subsecref{subsec: Hessian equations on compact Riemannian manifolds}). The theory of linear algebra shows that
\begin{eq} \label{eq: concavity inequality, -mu_1 is large, det(A)}
\det(\mathbf{A})=(1+\tau)\mu_n\prod_{j=1}^{n-1}\frac{2\mu_j\left(\mu_n+\varepsilon+\frac{a-n+1}{n-1}\mu_j\right)}{\mu_n+\varepsilon-\mu_j}\left(1+\frac{-\sigma_n(\bm\mu)}{\sigma_{n-1}(\bm\mu)}T_3\right)\comma
\end{eq}
where
\begin{eq}
T_3\triangleq\sum_{j=1}^{n-1}\frac{\mu_n+\varepsilon-\mu_j}{2\mu_j\left(\mu_n+\varepsilon+\frac{a-n+1}{n-1}\mu_j\right)}+\frac{1}{(1+\tau)\mu_n}.
\end{eq}
In view of \eqref{eq: concavity inequality, -mu_1 is large, T_2=v^TA overline{v}}, \eqref{eq: concavity inequality, -mu_1 is large, lambda_2(A)>0} and \eqref{eq: concavity inequality, -mu_1 is large, det(A)}, (in order to prove $T_2\geqslant 0$) we only need to prove that
\begin{eq} \label{eq: concavity inequality, -mu_1 is large, aim}
1+\frac{-\sigma_n(\bm\mu)}{\sigma_{n-1}(\bm\mu)}T_3\leqslant 0.
\end{eq}
Note that
\begin{eq} \begin{aligned}
T_3&=\sum_{j=1}^{n-1}\left(\frac{1}{2\mu_j}-\frac{\frac{a}{2(n-1)}}{\mu_n+\varepsilon+\frac{a-n+1}{n-1}\mu_j}\right)+\frac{1}{(1+\tau)\mu_n} \\
&=\frac{\sigma_{n-1}(\bm\mu)}{2\sigma_n(\bm\mu)}-\sum_{j=1}^{n-1}\frac{\frac{a}{2(n-1)}}{\mu_n+\varepsilon+\frac{a-n+1}{n-1}\mu_j}+\frac{1-\tau}{2(1+\tau)\mu_n}.
\end{aligned} \end{eq}
Thus, \eqref{eq: concavity inequality, -mu_1 is large, aim} is equivalent to
\begin{eq} \label{eq: concavity inequality, -mu_1 is large, aim, new}
\frac{\sigma_{n-1}(\bm\mu)}{-\sigma_n(\bm\mu)}\leqslant\sum_{j=1}^{n-1}\frac{\frac{a}{n-1}}{\mu_n+\varepsilon+\frac{a-n+1}{n-1}\mu_j}-\frac{1-\tau}{(1+\tau)\mu_n}.
\end{eq}
Since $\frac{a-n+1}{n-1}\in(-1\comma 0]$ and $\mu_2>-\mu_1>0$, there holds
\begin{eq} \label{eq: concavity inequality, -mu_1 is large, sum_{j=1}^{n-1}frac{frac{1}{n-1}}{mu_n+varepsilon+frac{a-n+1}{n-1}mu_j}} \begin{aligned}
\sum_{j=1}^{n-1}\frac{\frac{1}{n-1}}{\mu_n+\varepsilon+\frac{a-n+1}{n-1}\mu_j}&\geqslant\frac{\frac{1}{n-1}}{\mu_n+\varepsilon+\frac{a-n+1}{n-1}\mu_1}+\frac{\frac{n-2}{n-1}}{\mu_n+\varepsilon-\frac{a-n+1}{n-1}\mu_1} \\
&=\frac{\mu_n+\varepsilon+\frac{(n-3)(a-n+1)}{(n-1)^2}\mu_1}{(\mu_n+\varepsilon)^2-\left(\frac{a-n+1}{n-1}\right)^2\mu_1^2}\geqslant\frac{1}{\mu_n+\varepsilon}.
\end{aligned} \end{eq}
Combining \eqref{eq: concavity inequality, -mu_1 is large, -sigma_n(mu)>} and \eqref{eq: concavity inequality, -mu_1 is large, sum_{j=1}^{n-1}frac{frac{1}{n-1}}{mu_n+varepsilon+frac{a-n+1}{n-1}mu_j}}, we obtain \eqref{eq: concavity inequality, -mu_1 is large, aim, new} and therefore $T_2\geqslant 0$. Recalling \eqref{eq: concavity inequality, -mu_1 is large, 1} and \eqref{eq: concavity inequality, -mu_1 is large, T_1 geqslant 0}, we have justified \eqref{eq: concavity inequality, -mu_1 is large, conclusion}.
\end{proof}

The following proposition is essentially the concavity inequality in \cite[p.~4]{Chen2025}, whose proof is similar to that in \cite[pp.~10--20]{Zhang2025}. See also \cite[p.~10]{Tu2024} and \cite[p.~6]{Lu2023}.

\begin{prop} \label{prop: concavity inequality, -mu_1 is little}
Assume that $n\geqslant 3$, $p\in\{1\comma2\comma\cdots\comma n\}$, $\tau\in\left(0\comma\frac12\right]$, $\varepsilon\in(0\comma{+}\infty)$, $C\in[0\comma{+}\infty)$ and $\bm\mu\in\Gamma_p\colon-C\leqslant\mu_1\leqslant\mu_2\leqslant\cdots\leqslant\mu_n$. Then there exists such a positive constant $M$ depending only on $n$, $p$, $\tau$, $\varepsilon$, $C$ and $\sigma_p(\bm\mu)$ that if $\mu_n\geqslant M$, there holds
\begin{eq} \label{eq: concavity inequality, -mu_1 is little} \begin{aligned}
&\mathrel{\hphantom{=}}-\sum_{\substack{1\leqslant j\comma k\leqslant n \\ j\ne k}}  \sigma_{p-2}(\bm\mu|jk)w_j\overline{w}_k-\frac{1-\tau}{\mu_n}\sigma_{p-1}(\bm\mu|n)\left|w_n\right|^2 \\
&\geqslant -\frac{(p+1)^2}{\sigma_p(\bm\mu)}\left|\sum_{j=1}^n \sigma_{p-1}(\bm\mu|j)w_j\right|^2-(1-\tau)\sum_{j=1}^{n-1}\frac{\sigma_{p-1}(\bm\mu|j)\left|w_j\right|^2}{\mu_n+\varepsilon-\mu_j}
\end{aligned} \end{eq}
for any $\bm w=(w_1\comma w_2\comma\cdots\comma w_n)^{\mathrm T}\in\mathbb{C}^n$.
\end{prop}

\begin{proof}
Since $\tau\in\left(0\comma\frac12\right]$, there hold
\begin{eq}
1-\tau<\frac{2\left(1-\frac{7\tau}{4(1+\tau)}\right)}{2-\frac{7\tau}{4(1+\tau)}}
\end{eq}
and
\begin{eq}
\frac{1-\tau}{\mu_n+\varepsilon-\mu_j}\geqslant\frac{1-\tau}{\mu_n+\varepsilon+C}>\left(1-\frac{7\tau}{4(1+\tau)}\right)\frac{1}{\mu_n}
\end{eq}
if $M$ is large enough. Then refer to \cite[pp.~7--9]{Chen2025}, and note that the term ``$(2-\varepsilon_0)$'' in the inequality~(3.5) of \cite[p.~9]{Chen2025} can be replaced with any positive number less than 2 as long as the ``$M$'' there is large enough.
\end{proof}

Combining \propref{prop: concavity inequality, -mu_1 is large} and \propref{prop: concavity inequality, -mu_1 is little}, we obtain the following concavity inequality related to complex $(n-1)$-Hessian equation. A real version of this concavity inequality was given in \cite[p.~4]{Tu2024}, which in turn followed from the ideas in \cite{Zhang2025,Lu2024}.

\begin{thm} \label{thm: concavity inequality}
Assume that $n\geqslant 3$, $\tau\in\left(0\comma\frac12\right]$, $\varepsilon\in(0\comma{+}\infty)$ and $\bm\mu\in\Gamma_{n-1}\colon\mu_1\leqslant\mu_2\leqslant\cdots\leqslant\mu_n$. Then there exists such a positive constant $M$ depending only on $n$, $\tau$, $\varepsilon$ and $\sigma_{n-1}(\bm\mu)$ that if
$\mu_n\geqslant M$, there holds
\begin{eq} \label{eq: concavity inequality} \begin{aligned}
&\mathrel{\hphantom{=}}-\sum_{\substack{1\leqslant j\comma k\leqslant n \\ j\ne k}}  \sigma_{n-3}(\bm\mu|jk)w_j\overline{w}_k-\frac{1-\tau}{\mu_n}\sigma_{n-2}(\bm\mu|n)\left|w_n\right|^2 \\
&\geqslant -\frac{n^2}{\sigma_{n-1}(\bm\mu)}\left|\sum_{j=1}^n \sigma_{n-2}(\bm\mu|j)w_j\right|^2-(1-\tau)\sum_{j=1}^{n-1}\frac{\sigma_{n-2}(\bm\mu|j)\left|w_j\right|^2}{\mu_n+\varepsilon-\mu_j}
\end{aligned} \end{eq}
for any $\bm w=(w_1\comma w_2\comma\cdots\comma w_n)^{\mathrm T}\in\mathbb{C}^n$.
\end{thm}


\section{Second-order estimates for general $p$-Hessian equations} \label{sec: Second-order estimates for general p-Hessian equations}

In this section, we prove \thmref{thm: general p-Hessian, second-order estimates}. Let $(\mathcal{M}\comma\bm g)$ be a closed connected Riemannian manifold of dimension $n$. Throughout this section, we assume that $n\geqslant 3$, $p\in\{2\comma3\comma\cdots\comma n\}$ and $u\in\mathrm{C}^4(\mathcal M)$ satisfies the general $p$-Hessian equation \eqref{eq: general p-Hessian} on $\mathcal{M}$. Recall \eqref{eq: A(alpha, t)}, \eqref{eq: du and nabla^2u}, \eqref{eq: lambda(g^{-1}(A(du, u)+nabla^2u))}, \defnref{defn: lambda_q}, \defnref{defn: lambda}, \eqref{eq: sigma_p}, \defnref{defn: Garding cone}, \eqref{eq: moduli of du and nabla^2u} and that $\varphi(\bm{\alpha}\comma t)$ is a positive smooth function on $\mathrm{T}^*\mathcal{M}\times\mathbb{R}$. We assume additionally that $\underline{u}\in\mathrm{C}^2(\mathcal{M})$ satisfies pseudo-subsolution condition (cf. \defnref{defn: pseudo-solution}), and $\bm{A}(\bm{\alpha}\comma t)$ satisfies ``weak MTW condition'' \eqref{eq: A(alpha, t), weak MTW condition}. We need to prove for $p\in\{2\comma n-1\comma n\}$ that
\begin{eq} \label{eq: general p-Hessian, second-order estimates, conclusion}
\max\limits_{\mathcal M} \left|\nabla^2u\right|_{\bm g}\leqslant M\comma
\end{eq}
where $M$ is some positive constant depending only on $\mathcal{M}$, $\bm{g}$, $n$, $p$, $\bm{A}(\bm{\alpha}\comma t)$, $\varphi(\bm{\alpha}\comma t)$, $\underline{u}$, $\max\limits_{\mathcal M} \left|u\right|$ and $\max\limits_{\mathcal M} \left|\dif u\right|_{\bm g}$. The results of this section include \thmref{thm: p=2 or A(alpha, t) satisfies MTW condition or varphi(alpha, t) satisfies the convexity condition, second-order estimates}, \thmref{thm: p=n-1 or n, second-order estimates} and \thmref{thm: semi-convex, second-order estimates}.

The case $p=2$ is easier to some extent, so we first give the proof of this case which generalizes those in \cite{Guan2015a} (the Dirichlet problem with $\bm{A}(\bm\alpha\comma t)\equiv 0$ and the prescribed scalar curvature equation) and is somewhat similar to that in \cite{Spruck2017} (the prescribed scalar curvature equation), see also \cite{Chu2024a} (the almost Hermitian case with special $\bm{A}(\bm\alpha\comma t)$). Our proof also applies to the case of general $p$ with $\bm{A}(\bm\alpha\comma t)$ additionally satisfying ``MTW condition'' \eqref{eq: A(alpha, t), MTW condition} or $\varphi(\bm\alpha\comma t)$ additionally satisfying the convexity condition \eqref{eq: varphi(alpha, t) is convex with respect to alpha}, one can refer to \cite{Jiang2018,Jiang2020,Guan2015}. Like e.g. \cite{Szekelyhidi2018,Chu2021a}, we use the logarithmic maximum eigenvalue function as dominant term of the auxiliary function \eqref{eq: second-order estimates, Phi^{(0)}}, and our trick of perturbation (cf. e.g. \eqref{eq: second-order estimates, B_{jk}^{(varepsilon)}}, \eqref{eq: second-order estimates, Phi^{(varepsilon)}}, \eqref{eq: second-order estimates, max Phi^{(varepsilon)}=Phi^{(varepsilon)}(x_0)}) seems to have its own merits. Moreover, although we calculate mainly by ordinary derivatives like e.g. \cite{Li1990}, a trivial term will not be dropped unless it is the component of some tensor (cf. e.g. \eqref{eq: second-order estimates, tensor-1}--\eqref{eq: second-order estimates, tensor-3}).

\begin{thm} \label{thm: p=2 or A(alpha, t) satisfies MTW condition or varphi(alpha, t) satisfies the convexity condition, second-order estimates}
When $p=2$ or $\bm{A}(\bm\alpha\comma t)$ additionally satisfies ``MTW condition'' \eqref{eq: A(alpha, t), MTW condition} or $\varphi(\bm\alpha\comma t)$ additionally satisfies the convexity condition \eqref{eq: varphi(alpha, t) is convex with respect to alpha}, \eqref{eq: general p-Hessian, second-order estimates, conclusion} holds with $M$ independent of
\begin{eq} \label{eq: second-order estimates, min varphi(alpha, t)}
\min\limits_{\substack{(\bm\alpha\comma t)\in\mathrm{T}^*\mathcal{M}\times\mathbb{R} \\ \left|\bm\alpha\right|_{\bm g}\leqslant\max\limits_{\mathcal{M}}\left|\dif u\right|_{\bm g}\comma|t|\leqslant\max\limits_{\mathcal{M}}\left|u\right|}}\varphi(\bm{\alpha}\comma t).
\end{eq}
\end{thm}

\begin{proof}
Let $\bm{B}$ denote the symmetric $(0\comma2)$-tensor field $\bm{A}(\dif u\comma u)+\nabla^2u$ on $\mathcal{M}$. Recall the last paragraph before \subsecref{subsec: Hessian equations on compact Riemannian manifolds}. Then \eqref{eq: general p-Hessian} reads
\begin{eq} \label{eq: second-order estimates, equation}
\left\{ \begin{gathered}
\sigma_p^{\frac{1}{p}}\Bigl(\bigl(g^{jk}(\bm{x})\bigr)\bigl(B_{jk}(\bm{x})\bigr)\Bigr)=\varphi\bigl(\dif u(\bm{x})\comma u(\bm{x})\bigr)\comma\forall\bm{x}\in\mathcal{M} \\
\bm\lambda\Bigl(\bigl(g^{jk}(\bm{x})\bigr)\bigl(B_{jk}(\bm{x})\bigr)\Bigr)\in\Gamma_p\comma\forall\bm{x}\in\mathcal{M}
\end{gathered} \right. .
\end{eq}
Define the auxiliary function
\begin{eq} \label{eq: second-order estimates, Phi^{(0)}}
\Phi^{(0)}(\bm{x})\triangleq\log\biggl(1+\lambda_n\Bigl(\bigl(g^{jk}(\bm{x})\bigr)\bigl(B_{jk}(\bm{x})\bigr)\Bigr)\biggr)+\eta\bigl(\left|\dif u(\bm{x})\right|_{\bm g}^2\bigr)+\zeta\bigl((\underline{u}-u)(\bm{x})\bigr)\comma\bm{x}\in\mathcal{M}\comma
\end{eq}
where $\eta$ and $\zeta$ are both smooth functions on $\mathbb{R}$ to be determined later so that
\begin{al}
\eta'(t)>0\comma\eta''(t)\geqslant 0\comma&\forall t\in\left[0\comma\max_{\mathcal{M}}\left|\dif u\right|_{\bm g}^2\right]; \label{eq: second-order estimates, eta} \\
\zeta'(t)>0\comma\zeta''(t)\geqslant 0\comma&\forall t\in\left[\min_{\mathcal{M}}(\underline{u}-u)\comma\max_{\mathcal{M}}(\underline{u}-u)\right]. \label{eq: second-order estimates, zeta}
\end{al}
Assume that $\bm{x}_0\in\mathcal{M}$ satisfies
\begin{eq} \label{eq: second-order estimates, max_M Phi^{(0)}=Phi^{(0)}(x_0)}
\max_{\mathcal{M}} \Phi^{(0)}=\Phi^{(0)}(\bm{x}_0).
\end{eq}
Recall \eqref{eq: connection}, \eqref{eq: properties of connection} and choose a normal coordinate system $(\mathcal{U}_0\comma\bm{\psi}_{\mathcal{U}_0}\mbox{; }x^j)$ containing $\bm{x}_0$ so that
\begin{ga}
g_{jk}(\bm{x}_0)=\updelta_{jk}\comma\quad \Gamma_{jk}^l(\bm{x}_0)=0\comma\quad \mathrm{D}_lg_{jk}(\bm{x}_0)=0\comma \label{eq: second-order estimates, g_{jk}(x_0), Gamma_{jk}^l(x_0), D_lg_{jk}(x_0)} \\
\bigl(B_{jk}(\bm{x}_0)\bigr)=\Bigl(A_{jk}\bigl(\dif u(\bm{x}_0)\comma u(\bm{x}_0)\bigr)+\nabla_{jk}u(\bm{x}_0)\Bigr)=\mathbf{D}\triangleq\diag\{\mu_1\comma\mu_2\comma\cdots\comma\mu_n\}\comma \label{eq: second-order estimates, B_{jk}(x_0)} \\
\mu_1\leqslant\mu_2\leqslant\cdots\leqslant\mu_n \label{eq: second-order estimates, mu_j increases}.
\end{ga}
By \eqref{eq: g^{jk}, first-order} and \eqref{eq: g^{jk}, second-order}, we also have
\begin{eq} \label{eq: second-order estimates, g^{jk}(x_0), D_lg^{jk}(x_0), D_{lm}g^{jk}(x_0)}
g^{jk}(\bm{x}_0)=\updelta_{jk}\comma\quad \mathrm{D}_lg^{jk}(\bm{x}_0)=0\comma\quad \mathrm{D}_{lm}g^{jk}(\bm{x}_0)=-\mathrm{D}_{lm}g_{kj}(\bm{x}_0).
\end{eq}
Thus, there hold
\begin{ga}
\bm\mu\triangleq(\mu_1\comma\mu_2\comma\cdots\comma\mu_n)^{\mathrm T}\in\Gamma_p\comma \\
\sigma_p^{\frac1p}(\bm\mu)=\varphi\bigl(\dif u(\bm{x}_0)\comma u(\bm{x}_0)\bigr)\comma \label{eq: second-order estimates, sigma_p^{frac1p}(mu)} \\
\Phi^{(0)}(\bm{x}_0)=\log(1+\mu_n)+\eta\bigl(\left|\dif u(\bm{x}_0)\right|_{\bm g}^2\bigr)+\zeta\bigl((\underline{u}-u)(\bm{x}_0)\bigr)\comma \label{eq: second-order estimates, Phi^{(0)}(x_0)}
\end{ga}
and
\begin{al}
\mathrm{D}_{lm}g_{jk}(\bm{x}_0)&=\mathrm{D}_m\Gamma_{jl}^k(\bm{x}_0)+\mathrm{D}_m\Gamma_{kl}^j(\bm{x}_0)\comma \label{eq: second-order estimates, D_{lm}g_{jk}(x_0)} \\
R_{jklm}(\bm{x}_0)&=\mathrm{D}_l\Gamma_{jm}^k(\bm{x}_0)-\mathrm{D}_m\Gamma_{jl}^k(\bm{x}_0)\comma \label{eq: second-order estimates, R_{jklm}(x_0)} \\
\mathrm{D}_rR_{jklm}(\bm{x}_0)&=\mathrm{D}_{lr}\Gamma_{jm}^k(\bm{x}_0)-\mathrm{D}_{mr}\Gamma_{jl}^k(\bm{x}_0) \label{eq: second-order estimates, D_rR_{jklm}(x_0)}
\end{al}
by \eqref{eq: properties of connection} and \eqref{eq: R_{jklm}}. For any $\varepsilon\in(0\comma{+}\infty)$, define
\begin{ga}
B_{jk}^{(\varepsilon)}(\bm{x})\triangleq B_{jk}(\bm{x})+\varepsilon\frac{\updelta_{jn}\updelta_{kn}}{g^{nn}(\bm{x})}\comma\bm{x}\in\mathcal{M}\comma \label{eq: second-order estimates, B_{jk}^{(varepsilon)}} \\
\mu_j^{(\varepsilon)}\triangleq\mu_j+\varepsilon\updelta_{jn}\comma\quad \mathbf{D}^{(\varepsilon)}\triangleq\diag\{\mu_1^{(\varepsilon)}\comma\mu_2^{(\varepsilon)}\comma\cdots\comma\mu_n^{(\varepsilon)}\}=\bigl(B_{jk}^{(\varepsilon)}(\bm{x}_0)\bigr). \label{eq: second-order estimates, mu_j^{(varepsilon)}}
\end{ga}
Let $\mathcal{U}_\varepsilon$ be such a connected open neighbourhood of $\bm{x}_0$ that $\overline{\mathcal{U}}_\varepsilon$ is compact, $\overline{\mathcal{U}}_\varepsilon\subset\mathcal{U}_0$ and
\begin{eq}
\lambda_n\Bigl(\bigl(g^{jk}(\bm{x})\bigr)\bigl(B_{jk}^{(\varepsilon)}(\bm{x})\bigr)\Bigr)>\varepsilon\comma\forall\bm{x}\in\overline{\mathcal{U}}_\varepsilon.
\end{eq}
Then we can define the auxiliary function with perturbation:
\begin{eq} \label{eq: second-order estimates, Phi^{(varepsilon)}}
\Phi^{(\varepsilon)}(\bm{x})\triangleq\log\biggl(1+\lambda_n\Bigl(\bigl(g^{jk}(\bm{x})\bigr)\bigl(B_{jk}^{(\varepsilon)}(\bm{x})\bigr)\Bigr)-\varepsilon\biggr)+\eta\bigl(\left|\dif u(\bm{x})\right|_{\bm g}^2\bigr)+\zeta\bigl((\underline{u}-u)(\bm{x})\bigr)\comma\bm{x}\in\overline{\mathcal{U}}_\varepsilon.
\end{eq}
By \lemref{lem: Weyl}, we find that
\begin{eq} \begin{aligned}
\lambda_n\Bigl(\bigl(g^{jk}(\bm{x})\bigr)\bigl(B_{jk}^{(\varepsilon)}(\bm{x})\bigr)\Bigr)&\leqslant\lambda_n\Bigl(\bigl(g^{jk}(\bm{x})\bigr)\bigl(B_{jk}(\bm{x})\bigr)\Bigr)+\lambda_n\Biggl(\bigl(g^{jk}(\bm{x})\bigr)\left(\varepsilon\frac{\updelta_{jn}\updelta_{kn}}{g^{nn}(\bm{x})}\right)\Biggr) \\
&=\lambda_n\Bigl(\bigl(g^{jk}(\bm{x})\bigr)\bigl(B_{jk}(\bm{x})\bigr)\Bigr)+\varepsilon\comma\forall\bm{x}\in\overline{\mathcal{U}}_\varepsilon\comma
\end{aligned} \end{eq}
and therefore
\begin{eq}
\Phi^{(\varepsilon)}(\bm{x})\leqslant\Phi^{(0)}(\bm{x})\leqslant\Phi^{(0)}(\bm{x}_0)=\Phi^{(\varepsilon)}(\bm{x}_0)\comma\forall\bm{x}\in\overline{\mathcal{U}}_\varepsilon\comma
\end{eq}
namely
\begin{eq} \label{eq: second-order estimates, max Phi^{(varepsilon)}=Phi^{(varepsilon)}(x_0)}
\max_{\overline{\mathcal{U}}_\varepsilon} \Phi^{(\varepsilon)}=\Phi^{(\varepsilon)}(\bm{x}_0).
\end{eq}
Note that \eqref{eq: second-order estimates, max Phi^{(varepsilon)}=Phi^{(varepsilon)}(x_0)} holds for any $\varepsilon\in(0\comma{+}\infty)$, and $\varepsilon$ will be determined later. For any $j\in\{1\comma2\comma\cdots\comma n\}$, by \eqref{eq: second-order estimates, max Phi^{(varepsilon)}=Phi^{(varepsilon)}(x_0)} we have
\begin{eq} \label{eq: second-order estimates, D_j Phi^{(varepsilon)}(x_0), D_{jj} Phi^{(varepsilon)}(x_0)}
\mathrm{D}_j\Phi^{(\varepsilon)}(\bm{x}_0)=0\comma\quad\mathrm{D}_{jj}\Phi^{(\varepsilon)}(\bm{x}_0)\leqslant 0.
\end{eq}
Now recall \eqref{eq: F_{u, x}^{jk}} and define the following notations:
\begin{al}
G_{u\comma\bm{x}}^{jk}&\triangleq\left.\frac{\partial}{\partial a_{jk}}\sigma_p^{\frac 1p}\biggl(\bm\lambda\Bigl(\mathbf{A}\comma\bigl(B_{jk}(\bm{x})\bigr)\Bigr)\biggr)\right|_{\mathbf{A}=\left(g^{jk}(\bm{x})\right)}\comma \\
F_{u\comma\bm{x}}^{jk\comma lm}&\triangleq\left.\frac{\partial^2}{\partial b_{lm}\partial b_{jk}}\sigma_p^{\frac 1p}\biggl(\bm\lambda\Bigl(\bigl(g^{jk}(\bm{x})\bigr)\comma\mathbf{B}\Bigr)\biggr)\right|_{\mathbf{B}=\left(B_{jk}(\bm{x})\right)}.
\end{al}
Obviously, there hold
\begin{al}
G^{jk}\triangleq G_{u\comma\bm{x}_0}^{jk}&=\left.\frac{\partial}{\partial a_{jk}}\sigma_p^{\frac 1p}\bigl(\bm\lambda(\mathbf{A}\comma\mathbf{D})\bigr)\right|_{\mathbf{A}=\mathbf{I}_n}\comma \label{eq: second-order estimates, G^{jk}} \\
F^{jk}\triangleq F_{u\comma\bm{x}_0}^{jk}&=\left.\frac{\partial}{\partial b_{jk}}\sigma_p^{\frac 1p}\bigl(\bm\lambda(\mathbf{I}_n\comma\mathbf{B})\bigr)\right|_{\mathbf{B}=\mathbf{D}}\comma \label{eq: second-order estimates, F^{jk}} \\
F^{jk\comma lm}\triangleq F_{u\comma\bm{x}_0}^{jk\comma lm}&=\left.\frac{\partial^2}{\partial b_{lm}\partial b_{jk}}\sigma_p^{\frac 1p}\bigl(\bm\lambda(\mathbf{I}_n\comma\mathbf{B})\bigr)\right|_{\mathbf{B}=\mathbf{D}}. \label{eq: second-order estimates, F^{jk, lm}}
\end{al}

For the rest of this proof, we always calculate at $\bm{x}_0$. By \eqref{eq: du and nabla^2u}, \eqref{eq: second-order estimates, g_{jk}(x_0), Gamma_{jk}^l(x_0), D_lg_{jk}(x_0)} and \eqref{eq: second-order estimates, B_{jk}(x_0)}, we have
\begin{ga}
\mathrm{D}_{jk}u=\nabla_{jk}u=\mu_j\updelta_{jk}-A_{jk}(\dif u\comma u)\comma \label{eq: second-order estimates, D_{jk}u} \\
\mathrm{D}_l\nabla_{jk}u=\mathrm{D}_{jkl}u-\mathrm{D}_l\Gamma_{jk}^m\mathrm{D}_mu\comma \label{eq: second-order estimates, D_l nabla_{jk}u} \\
\mathrm{D}_{jj}\nabla_{kk}u=\mathrm{D}_{kkjj}u-2\mathrm{D}_j\Gamma_{kk}^l\mathrm{D}_{lj}u-\mathrm{D}_{jj}\Gamma_{kk}^l\mathrm{D}_lu. \label{eq: second-order estimates, D_{jj} nabla_{kk}u}
\end{ga}
Let $(\tilde{x}^j\comma\tilde\alpha_j)$ be the canonical coordinate on $\mathrm{T}^*\mathcal{U}_0$ with respect to $(\mathcal{U}_0\comma\bm{\psi}_{\mathcal{U}_0}\mbox{; }x^j)$ (cf. \eqref{eq: tilde x^j, tilde alpha_j}). Recalling \eqref{eq: tilde{nabla}_{tilde{x}^j}varphi}--\eqref{eq: tilde{nabla}_{tilde{x}^j tilde{x}^k}A_{x^rx^s}} and \eqref{eq: second-order estimates, R_{jklm}(x_0)}, each of the following is the component of some tensor (at $\bm{x}_0$):
\begin{ga}
\left.\frac{\partial\varphi(\cdot\comma u)}{\partial\tilde{x}^j}\right|_{\dif u}\comma\quad \left.\frac{\partial A_{rs}(\cdot\comma u)}{\partial\tilde{x}^j}\right|_{\dif u}\comma\quad \left.\frac{\partial\varphi(\cdot\comma u)}{\partial\tilde{\alpha}_j}\right|_{\dif u}\comma\quad \left.\frac{\partial^2\varphi(\cdot\comma u)}{\partial\tilde{\alpha}_k\partial\tilde{\alpha}_j}\right|_{\dif u}\comma\quad \left.\frac{\partial A_{rs}(\cdot\comma u)}{\partial\tilde{\alpha}_j}\right|_{\dif u}\comma \label{eq: second-order estimates, tensor-1} \\%
\left.\frac{\partial^2A_{rs}(\cdot\comma u)}{\partial\tilde{\alpha}_k\partial\tilde{\alpha}_j}\right|_{\dif u}\comma\quad \left.\frac{\partial^2\varphi(\cdot\comma u)}{\partial\tilde{\alpha}_k\partial\tilde{x}^j}\right|_{\dif u}\comma\quad \left.\frac{\partial^2A_{rs}(\cdot\comma u)}{\partial\tilde{\alpha}_k\partial\tilde{x}^j}\right|_{\dif u}\comma\quad \left.\frac{\partial^2\varphi(\cdot\comma u)}{\partial\tilde{x}^k\partial\tilde{x}^j}\right|_{\dif u}+\left.\frac{\partial\varphi(\cdot\comma u)}{\partial\tilde{\alpha}_m}\right|_{\dif u}\mathrm{D}_m\Gamma_{jk}^l\mathrm{D}_lu\comma \label{eq: second-order estimates, tensor-2} \\%
\left.\frac{\partial^2A_{rs}(\cdot\comma u)}{\partial\tilde{x}^k\partial\tilde{x}^j}\right|_{\dif u}+\left.\frac{\partial A_{rs}(\cdot\comma u)}{\partial\tilde{\alpha}_m}\right|_{\dif u}\mathrm{D}_m\Gamma_{jk}^l\mathrm{D}_lu-A_{qs}(\dif u\comma u)\mathrm{D}_r\Gamma_{jk}^q-A_{rq}(\dif u\comma u)\mathrm{D}_s\Gamma_{jk}^q. \label{eq: second-order estimates, tensor-3}
\end{ga}
For any $j\in\{1\comma2\comma\cdots\comma n\}$, differentiating both sides of the equality in \eqref{eq: second-order estimates, equation}, by \eqref{eq: second-order estimates, g^{jk}(x_0), D_lg^{jk}(x_0), D_{lm}g^{jk}(x_0)} we obtain
\begin{ga}
F^{lm}\mathrm{D}_jB_{lm}=\mathrm{D}_j\bigl(\varphi(\dif u\comma u)\bigr)\comma \label{eq: second-order estimates, equation, first-order} \\
-G^{lm}\mathrm{D}_{jj}g_{ml}+F^{lm\comma rs}\mathrm{D}_jB_{lm}\mathrm{D}_jB_{rs}+F^{lm}\mathrm{D}_{jj}B_{lm}=\mathrm{D}_{jj}\bigl(\varphi(\dif u\comma u)\bigr). \label{eq: second-order estimates, equation, second-order}
\end{ga}
Straightforward calculations show that
\begin{al}
\mathrm{D}_j\bigl(\varphi(\dif u\comma u)\bigr)&=\left.\frac{\partial\varphi(\cdot\comma u)}{\partial\tilde{x}^j}\right|_{\dif u}+\left.\frac{\partial\varphi(\cdot\comma u)}{\partial\tilde{\alpha}_l}\right|_{\dif u}\mathrm{D}_{lj}u+\left.\frac{\partial\varphi(\dif u\comma\cdot)}{\partial t}\right|_u\mathrm{D}_ju\comma \label{eq: second-order estimates, D_j varphi} \\%
\mathrm{D}_{jj}\bigl(\varphi(\dif u\comma u)\bigr)&=\left.\frac{\partial\varphi(\cdot\comma u)}{\partial\tilde{\alpha}_l}\right|_{\dif u}\mathrm{D}_{ljj}u-\left.\frac{\partial\varphi(\cdot\comma u)}{\partial\tilde{\alpha}_m}\right|_{\dif u}\mathrm{D}_m\Gamma_{jj}^l\mathrm{D}_lu+\kappa_j^{(1)}\comma \label{eq: second-order estimates, D_{jj} varphi}
\end{al}
where ($C_0$ is a non-negative constant)
\begin{eq} \label{eq: second-order estimates, kappa_j^{(1)}}
\kappa_j^{(1)}\geqslant-C_0\mu_j^2-C_1(1+\left|\mu_j\right|)
\end{eq}
in view of \eqref{eq: second-order estimates, D_{jk}u}. By \eqref{eq: second-order estimates, G^{jk}}, \eqref{eq: second-order estimates, F^{jk}}, \propref{prop: sigma_p circ lambda, derivatives}, \propref{prop: increasing with respect to each variable}, \eqref{eq: sum_{k=1}^n sigma_p(mu|k)} and \eqref{eq: product of sigma_{p-1}(mu|j), lower bound}, we have
\begin{al}
G^{jk}&=\frac1p\sigma_p^{\frac1p-1}(\bm\mu)\mu_j\sigma_{p-1}(\bm\mu|j)\updelta_{jk}\comma \label{eq: second-order estimates, G^{jk}, new} \\
F^{jk}&=\frac1p\sigma_p^{\frac1p-1}(\bm\mu)\sigma_{p-1}(\bm\mu|j)\updelta_{jk}\comma\quad F^{jj}>0\comma \label{eq: second-order estimates, F^{jk}, new} \\
\mathscr{F}\triangleq\sum_{j=1}^n F^{jj}&=\frac{n-p+1}{p}\sigma_p^{\frac1p-1}(\bm\mu)\sigma_{p-1}(\bm\mu)\geqslant(\mathrm{C}_n^p)^{\frac1p}. \label{eq: second-order estimates, F}
\end{al}
It follows from \eqref{eq: second-order estimates, equation, first-order}, \eqref{eq: second-order estimates, D_j varphi} and \eqref{eq: second-order estimates, F^{jk}, new} that
\begin{eq} \label{eq: second-order estimates, equation, first-order, new}
\left.\frac{\partial\varphi(\cdot\comma u)}{\partial\tilde{x}^j}\right|_{\dif u}+\left.\frac{\partial\varphi(\cdot\comma u)}{\partial\tilde{\alpha}_l}\right|_{\dif u}\mathrm{D}_{lj}u+\left.\frac{\partial\varphi(\dif u\comma\cdot)}{\partial t}\right|_u\mathrm{D}_ju=F^{kk}\mathrm{D}_jB_{kk}.
\end{eq}
Recalling \eqref{eq: second-order estimates, D_l nabla_{jk}u} and \eqref{eq: second-order estimates, D_{jj} nabla_{kk}u}, straightforward calculations also show that
\begin{ga}
\begin{aligned}
\mathrm{D}_jB_{kk}&=\mathrm{D}_j\bigl(A_{kk}(\dif u\comma u)\bigr)+\mathrm{D}_j\nabla_{kk}u \\
&=\left.\frac{\partial A_{kk}(\cdot\comma u)}{\partial\tilde{x}^j}\right|_{\dif u}+\left.\frac{\partial A_{kk}(\cdot\comma u)}{\partial\tilde{\alpha}_l}\right|_{\dif u}\mathrm{D}_{lj}u+\left.\frac{\partial A_{kk}(\dif u\comma\cdot)}{\partial t}\right|_u\mathrm{D}_ju+\mathrm{D}_{kkj}u-\mathrm{D}_j\Gamma_{kk}^l\mathrm{D}_lu\comma
\end{aligned} \label{eq: second-order estimates, D_jB_{kk}} \\%
\begin{aligned}
\mathrm{D}_{jj}B_{kk}&=\mathrm{D}_{jj}\bigl(A_{kk}(\dif u\comma u)\bigr)+\mathrm{D}_{jj}\nabla_{kk}u \\
&=\left.\frac{\partial A_{kk}(\cdot\comma u)}{\partial\tilde{\alpha}_l}\right|_{\dif u}\mathrm{D}_{ljj}u+\left.\frac{\partial^2A_{kk}(\cdot\comma u)}{\partial\tilde{\alpha}_m\partial\tilde{\alpha}_l}\right|_{\dif u}\mathrm{D}_{lj}u\mathrm{D}_{mj}u-\left.\frac{\partial A_{kk}(\cdot\comma u)}{\partial\tilde{\alpha}_m}\right|_{\dif u}\mathrm{D}_m\Gamma_{jj}^l\mathrm{D}_lu \\
&\quad+2A_{kq}(\dif u\comma u)\mathrm{D}_k\Gamma_{jj}^q+\kappa_{jk}^{(2)}+\mathrm{D}_{kkjj}u-2\mathrm{D}_j\Gamma_{kk}^l\mathrm{D}_{lj}u-\mathrm{D}_{jj}\Gamma_{kk}^l\mathrm{D}_lu\comma
\end{aligned} \label{eq: second-order estimates, D_{jj}B_{kk}}
\end{ga}
where
\begin{eq} \label{eq: second-order estimates, kappa_{jk}^{(2)}}
|\kappa_{jk}^{(2)}|\leqslant C_2(1+\left|\mu_j\right|)
\end{eq}
in view of \eqref{eq: second-order estimates, D_{jk}u}. By \eqref{eq: second-order estimates, equation, first-order, new}, \eqref{eq: second-order estimates, D_jB_{kk}} and \eqref{eq: second-order estimates, F}, we have
\begin{eq} \label{eq: second-order estimates, F^{kk}(D_{kkj}u-D_j Gamma_{kk}^lD_lu)}
F^{kk}(\mathrm{D}_{kkj}u-\mathrm{D}_j\Gamma_{kk}^l\mathrm{D}_lu)=\left(-F^{kk}\left.\frac{\partial A_{kk}(\cdot\comma u)}{\partial\tilde{\alpha}_l}\right|_{\dif u}+\left.\frac{\partial\varphi(\cdot\comma u)}{\partial\tilde{\alpha}_l}\right|_{\dif u}\right)\mathrm{D}_{lj}u+\kappa_j^{(3)}\comma
\end{eq}
where
\begin{eq} \label{eq: second-order estimates, kappa_j^{(3)}}
|\kappa_j^{(3)}|\leqslant C_3\mathscr{F}.
\end{eq}
It follows from \eqref{eq: second-order estimates, equation, second-order}, \eqref{eq: second-order estimates, D_{jj} varphi}, \eqref{eq: second-order estimates, G^{jk}, new} and \eqref{eq: second-order estimates, F^{jk}, new} that
\begin{eq} \label{eq: second-order estimates, equation, second-order, new}
\left.\frac{\partial\varphi(\cdot\comma u)}{\partial\tilde{\alpha}_l}\right|_{\dif u}(\mathrm{D}_{jjl}u-\mathrm{D}_l\Gamma_{jj}^m\mathrm{D}_mu)+\kappa_j^{(1)}=-F^{kk}\mu_k\mathrm{D}_{jj}g_{kk}+F^{lm\comma rs}\mathrm{D}_jB_{lm}\mathrm{D}_jB_{rs}+F^{kk}\mathrm{D}_{jj}B_{kk}.
\end{eq}
By \eqref{eq: second-order estimates, D_rR_{jklm}(x_0)} we have
\begin{eq} \label{eq: second-order estimates, D_{kk} Gamma_{jj}^l-D_{jj} Gamma_{kk}^l}
\mathrm{D}_{kk}\Gamma_{jj}^l-\mathrm{D}_{jj}\Gamma_{kk}^l=(\mathrm{D}_{kk}\Gamma_{jj}^l-\mathrm{D}_{jk}\Gamma_{jk}^l)-(\mathrm{D}_{jj}\Gamma_{kk}^l-\mathrm{D}_{kj}\Gamma_{kj}^l)=\mathrm{D}_kR_{jlkj}-\mathrm{D}_jR_{kljk}\comma
\end{eq}
where $\mathrm{D}_kR_{jlkj}$ and $\mathrm{D}_jR_{kljk}$ are both components of the first-order covariant differentiation (at $\bm{x}_0$) of the curvature tensor field $\bm{R}$ (cf. \eqref{eq: curvature tensor field}). Note that
\begin{eq} \label{eq: second-order estimates, A_{jq}(du, u)D_k Gamma_{jk}^q+D_j Gamma_{kk}^lD_{lj}u}
A_{jq}(\dif u\comma u)\mathrm{D}_j\Gamma_{kk}^q+\mathrm{D}_j\Gamma_{kk}^l\mathrm{D}_{lj}u=\mu_j\updelta_{lj}\mathrm{D}_j\Gamma_{kk}^l
\end{eq}
in view of \eqref{eq: second-order estimates, D_{jk}u}. Combining \eqref{eq: second-order estimates, D_{jj}B_{kk}}, \eqref{eq: second-order estimates, D_{kk} Gamma_{jj}^l-D_{jj} Gamma_{kk}^l}, \eqref{eq: second-order estimates, A_{jq}(du, u)D_k Gamma_{jk}^q+D_j Gamma_{kk}^lD_{lj}u} and \eqref{eq: second-order estimates, kappa_{jk}^{(2)}}, we obtain
\begin{eq} \label{eq: second-order estimates, D_{kk}B_{jj}-D_{jj}B_{kk}} \begin{aligned}
\mathrm{D}_{kk}B_{jj}-\mathrm{D}_{jj}B_{kk}&=\left.\frac{\partial A_{jj}(\cdot\comma u)}{\partial\tilde{\alpha}_l}\right|_{\dif u}(\mathrm{D}_{kkl}u-\mathrm{D}_l\Gamma_{kk}^m\mathrm{D}_mu)+\left.\frac{\partial^2A_{jj}(\cdot\comma u)}{\partial\tilde{\alpha}_m\partial\tilde{\alpha}_l}\right|_{\dif u}\mathrm{D}_{lk}u\mathrm{D}_{mk}u \\
&\quad-\left.\frac{\partial A_{kk}(\cdot\comma u)}{\partial\tilde{\alpha}_l}\right|_{\dif u}(\mathrm{D}_{jjl}u-\mathrm{D}_l\Gamma_{jj}^m\mathrm{D}_mu)-\left.\frac{\partial^2A_{kk}(\cdot\comma u)}{\partial\tilde{\alpha}_m\partial\tilde{\alpha}_l}\right|_{\dif u}\mathrm{D}_{lj}u\mathrm{D}_{mj}u \\
&\quad+2\mu_j\updelta_{lj}\mathrm{D}_j\Gamma_{kk}^l-2\mu_k\updelta_{lk}\mathrm{D}_k\Gamma_{jj}^l+\kappa_{jk}^{(4)}\comma
\end{aligned} \end{eq}
where
\begin{eq} \label{eq: second-order estimates, kappa_{jk}^{(4)}}
|\kappa_{jk}^{(4)}|\leqslant C_4(1+\left|\mu_j\right|+\left|\mu_k\right|).
\end{eq}
Recalling \eqref{eq: second-order Frechet}, ``weak MTW condition'' \eqref{eq: A(alpha, t), weak MTW condition} implies
\begin{eq} \label{eq: second-order estimates, A_{kk}, weak MTW condition}
\left.\frac{\partial^2A_{kk}(\cdot\comma u)}{\partial\tilde{\alpha}_j^2}\right|_{\dif u}\leqslant-c_1\leqslant0\comma\forall j\comma k\in\{1\comma2\comma\cdots\comma n\}\colon j\ne k
\end{eq}
and therefore (cf. \eqref{eq: second-order estimates, F^{jk}, new})
\begin{eq} \label{eq: second-order estimates, weak MTW condition, new}
F^{kk}\left(\left.\frac{\partial^2A_{jj}(\cdot\comma u)}{\partial\tilde{\alpha}_k^2}\right|_{\dif u}\mu_k^2-\left.\frac{\partial^2A_{kk}(\cdot\comma u)}{\partial\tilde{\alpha}_j^2}\right|_{\dif u}\mu_j^2\right)\geqslant\sum_{\substack{1\leqslant k\leqslant n \\ k\ne j}} F^{kk}\left(\left.\frac{\partial^2A_{jj}(\cdot\comma u)}{\partial\tilde{\alpha}_k^2}\right|_{\dif u}\mu_k^2+c_1\mu_j^2\right).
\end{eq}
Combining \eqref{eq: second-order estimates, D_{kk}B_{jj}-D_{jj}B_{kk}}, \eqref{eq: second-order estimates, D_{jk}u}, \eqref{eq: second-order estimates, weak MTW condition, new}, \eqref{eq: second-order estimates, kappa_{jk}^{(4)}}, \eqref{eq: second-order estimates, F^{kk}(D_{kkj}u-D_j Gamma_{kk}^lD_lu)} and \eqref{eq: second-order estimates, kappa_j^{(3)}}, for any $j\in\{1\comma2\comma\cdots\comma n\}$ we have
\begin{eq} \label{eq: second-order estimates, F^{kk}(D_{kk}B_{jj}-D_{jj}B_{kk})} \begin{aligned}
&\mathrel{\hphantom{=}}F^{kk}\left(\mathrm{D}_{kk}B_{jj}-\mathrm{D}_{jj}B_{kk}\right) \\
&\geqslant\left.\frac{\partial A_{jj}(\cdot\comma u)}{\partial\tilde{\alpha}_l}\right|_{\dif u}F^{kk}(\mathrm{D}_{kkl}u-\mathrm{D}_l\Gamma_{kk}^m\mathrm{D}_mu)-F^{kk}\left.\frac{\partial A_{kk}(\cdot\comma u)}{\partial\tilde{\alpha}_l}\right|_{\dif u}(\mathrm{D}_{jjl}u-\mathrm{D}_l\Gamma_{jj}^m\mathrm{D}_mu) \\
&\quad+2F^{kk}(\mu_j\updelta_{lj}\mathrm{D}_j\Gamma_{kk}^l-\mu_k\mathrm{D}_k\Gamma_{jj}^k)-C_5\sum_{\substack{1\leqslant k\leqslant n \\ k\ne j}} F^{kk}(1+\left|\mu_j\right|+\left|\mu_k\right|+\mu_k^2)+c_1\mu_j^2\sum_{\substack{1\leqslant k\leqslant n \\ k\ne j}} F^{kk} \\
&\geqslant\left.\frac{\partial A_{jj}(\cdot\comma u)}{\partial\tilde{\alpha}_l}\right|_{\dif u}\left(-F^{kk}\left.\frac{\partial A_{kk}(\cdot\comma u)}{\partial\tilde{\alpha}_m}\right|_{\dif u}+\left.\frac{\partial\varphi(\cdot\comma u)}{\partial\tilde{\alpha}_m}\right|_{\dif u}\right)\mathrm{D}_{ml}u-C_6\bigl((1+\left|\mu_j\right|)\mathscr{F}+F^{kk}\mu_k^2\bigr) \\
&\quad+2F^{kk}(\mu_j\updelta_{lj}\mathrm{D}_j\Gamma_{kk}^l-\mu_k\mathrm{D}_k\Gamma_{jj}^k)-F^{kk}\left.\frac{\partial A_{kk}(\cdot\comma u)}{\partial\tilde{\alpha}_l}\right|_{\dif u}(\mathrm{D}_{jjl}u-\mathrm{D}_l\Gamma_{jj}^m\mathrm{D}_mu)+c_1\mu_j^2(\mathscr{F}-F^{jj})\comma
\end{aligned} \end{eq}
recall \eqref{eq: second-order estimates, F}. Combining \eqref{eq: second-order estimates, equation, second-order, new}, \eqref{eq: second-order estimates, F^{kk}(D_{kk}B_{jj}-D_{jj}B_{kk})}, \eqref{eq: second-order estimates, D_jB_{kk}} and \eqref{eq: second-order estimates, kappa_j^{(1)}}, we obtain
\begin{eq} \label{eq: second-order estimates, F^{kk}D_{kk}B_{jj}} \begin{aligned}
F^{kk}\mathrm{D}_{kk}B_{jj}&=F^{kk}\left(\mathrm{D}_{kk}B_{jj}-\mathrm{D}_{jj}B_{kk}\right)+\left.\frac{\partial\varphi(\cdot\comma u)}{\partial\tilde{\alpha}_l}\right|_{\dif u}(\mathrm{D}_{jjl}u-\mathrm{D}_l\Gamma_{jj}^m\mathrm{D}_mu) \\
&\quad+\kappa_j^{(1)}+F^{kk}\mu_k\mathrm{D}_{jj}g_{kk}-F^{lm\comma rs}\mathrm{D}_jB_{lm}\mathrm{D}_jB_{rs} \\
&\geqslant\left(-F^{kk}\left.\frac{\partial A_{kk}(\cdot\comma u)}{\partial\tilde{\alpha}_l}\right|_{\dif u}+\left.\frac{\partial\varphi(\cdot\comma u)}{\partial\tilde{\alpha}_l}\right|_{\dif u}\right)\left(\mathrm{D}_{jjl}u-\mathrm{D}_l\Gamma_{jj}^m\mathrm{D}_mu+\left.\frac{\partial A_{jj}(\cdot\comma u)}{\partial\tilde{\alpha}_m}\right|_{\dif u}\mathrm{D}_{ml}u\right) \\
&\quad-C_6\bigl((1+\left|\mu_j\right|)\mathscr{F}+F^{kk}\mu_k^2\bigr)+2F^{kk}(\mu_j\updelta_{lj}\mathrm{D}_j\Gamma_{kk}^l-\mu_k\mathrm{D}_k\Gamma_{jj}^k) \\
&\quad+\kappa_j^{(1)}+F^{kk}\mu_k\mathrm{D}_{jj}g_{kk}-F^{lm\comma rs}\mathrm{D}_jB_{lm}\mathrm{D}_jB_{rs}+c_1\mu_j^2(\mathscr{F}-F^{jj}) \\
&\geqslant\left(-F^{kk}\left.\frac{\partial A_{kk}(\cdot\comma u)}{\partial\tilde{\alpha}_l}\right|_{\dif u}+\left.\frac{\partial\varphi(\cdot\comma u)}{\partial\tilde{\alpha}_l}\right|_{\dif u}\right)\mathrm{D}_lB_{jj}-C_0\mu_j^2-C_7(1+\left|\mu_j\right|)\mathscr{F}-C_6F^{kk}\mu_k^2 \\
&\quad+2F^{kk}(\mu_j\updelta_{lj}\mathrm{D}_j\Gamma_{kk}^l-\mu_k\mathrm{D}_k\Gamma_{jj}^k)+F^{kk}\mu_k\mathrm{D}_{jj}g_{kk}-F^{lm\comma rs}\mathrm{D}_jB_{lm}\mathrm{D}_jB_{rs}+c_1\mu_j^2(\mathscr{F}-F^{jj}).
\end{aligned} \end{eq}

Now recall \eqref{eq: second-order estimates, Phi^{(varepsilon)}}, \eqref{eq: second-order estimates, D_j Phi^{(varepsilon)}(x_0), D_{jj} Phi^{(varepsilon)}(x_0)}, \eqref{eq: second-order estimates, B_{jk}^{(varepsilon)}}, \eqref{eq: second-order estimates, mu_j^{(varepsilon)}} and that
\begin{eq}
\mu_1^{(\varepsilon)}\leqslant\cdots\leqslant\mu_{n-1}^{(\varepsilon)}<\mu_n^{(\varepsilon)}.
\end{eq}
By \eqref{eq: second-order estimates, g^{jk}(x_0), D_lg^{jk}(x_0), D_{lm}g^{jk}(x_0)} and \lemref{lem: lambda_q, derivatives}, for any $r\in\{1\comma2\comma\cdots\comma n\}$ we have
\begin{eq}
\mathrm{D}_r\Bigl(\lambda_n\bigl((g^{jk})(B_{jk}^{(\varepsilon)})\bigr)\Bigr)=\left.\frac{\partial\lambda_n(\mathbf{I}_n\comma\mathbf{B})}{\partial b_{jk}}\right|_{\mathbf{B}=\mathbf{D}^{(\varepsilon)}}\mathrm{D}_rB_{jk}^{(\varepsilon)}=\mathrm{D}_rB_{nn}^{(\varepsilon)}=\mathrm{D}_rB_{nn}\comma
\end{eq}
and
\begin{eq} \begin{aligned}
&\mathrel{\hphantom{=}}\mathrm{D}_{rr}\Bigl(\lambda_n\bigl((g^{jk})(B_{jk}^{(\varepsilon)})\bigr)\Bigr) \\
&=-\left.\frac{\partial\lambda_n(\mathbf{A}\comma\mathbf{D}^{(\varepsilon)})}{\partial a_{jk}}\right|_{\mathbf{A}=\mathbf{I}_n}\mathrm{D}_{rr}g_{kj}+\left.\frac{\partial\lambda_n(\mathbf{I}_n\comma\mathbf{B})}{\partial b_{jk}}\right|_{\mathbf{B}=\mathbf{D}^{(\varepsilon)}}\mathrm{D}_{rr}B_{jk}^{(\varepsilon)}+\left.\frac{\partial^2\lambda_n(\mathbf{I}_n\comma\mathbf{B})}{\partial b_{lm}\partial b_{jk}}\right|_{\mathbf{B}=\mathbf{D}^{(\varepsilon)}}\mathrm{D}_rB_{jk}^{(\varepsilon)}\mathrm{D}_rB_{lm}^{(\varepsilon)} \\
&=\mathrm{D}_{rr}B_{nn}^{(\varepsilon)}-\mu_n^{(\varepsilon)}\mathrm{D}_{rr}g_{nn}+\sum_{k=1}^{n-1} \frac{\mathrm{D}_rB_{nk}^{(\varepsilon)}\mathrm{D}_rB_{nk}^{(\varepsilon)}+\mathrm{D}_rB_{nk}^{(\varepsilon)}\mathrm{D}_rB_{kn}^{(\varepsilon)}}{2(\mu_n^{(\varepsilon)}-\mu_k^{(\varepsilon)})}+\sum_{j=1}^{n-1} \frac{\mathrm{D}_rB_{jn}^{(\varepsilon)}\mathrm{D}_rB_{jn}^{(\varepsilon)}+\mathrm{D}_rB_{jn}^{(\varepsilon)}\mathrm{D}_rB_{nj}^{(\varepsilon)}}{2(\mu_n^{(\varepsilon)}-\mu_j^{(\varepsilon)})} \\
&=\mathrm{D}_{rr}B_{nn}-\mu_n\mathrm{D}_{rr}g_{nn}+\sum_{j=1}^{n-1} \frac{2\left|\mathrm{D}_rB_{jn}\right|^2}{\mu_n+\varepsilon-\mu_j}.
\end{aligned} \end{eq}
Thus, we have
\begin{eq} \label{eq: second-order estimates, 0=D_r Phi^{(varepsilon)}} \begin{aligned}
0=\mathrm{D}_r\Phi^{(\varepsilon)}&=\frac{\mathrm{D}_r\Bigl(\lambda_n\bigl((g^{jk})(B_{jk}^{(\varepsilon)})\bigr)\Bigr)}{1+\lambda_n\bigl((g^{jk})(B_{jk}^{(\varepsilon)})\bigr)-\varepsilon}+\eta'(\left|\dif u\right|_{\bm g}^2)\mathrm{D}_r\left|\dif u\right|_{\bm g}^2+\zeta'(\underline{u}-u)\mathrm{D}_r(\underline{u}-u) \\
&=\frac{\mathrm{D}_rB_{nn}}{1+\mu_n}+\eta'(\left|\dif u\right|_{\bm g}^2)\mathrm{D}_r\left|\dif u\right|_{\bm g}^2+\zeta'(\underline{u}-u)\mathrm{D}_r(\underline{u}-u)\comma
\end{aligned} \end{eq}
and
\begin{eq} \label{eq: second-order estimates, 0 geqslant F^{rr}D_{rr} Phi^{(varepsilon)}} \begin{aligned}
0\geqslant F^{rr}\mathrm{D}_{rr}\Phi^{(\varepsilon)}&=\frac{1}{1+\mu_n}F^{rr}\left(\mathrm{D}_{rr}B_{nn}-\mu_n\mathrm{D}_{rr}g_{nn}+\sum_{j=1}^{n-1} \frac{2\left|\mathrm{D}_rB_{jn}\right|^2}{\mu_n+\varepsilon-\mu_j}\right) \\
&\quad-\frac{F^{rr}\left|\mathrm{D}_rB_{nn}\right|^2}{(1+\mu_n)^2}+\eta'(\left|\dif u\right|_{\bm g}^2)F^{rr}\mathrm{D}_{rr}\left|\dif u\right|_{\bm g}^2+\zeta'(\underline{u}-u)F^{rr}\mathrm{D}_{rr}(\underline{u}-u) \\
&\quad+F^{rr}\bigl(\eta''(\left|\dif u\right|_{\bm g}^2)\bigl|\mathrm{D}_r\left|\dif u\right|_{\bm g}^2\bigr|^2+\zeta''(\underline{u}-u)\left|\mathrm{D}_r(\underline{u}-u)\right|^2\bigr).
\end{aligned} \end{eq}
In view of \eqref{eq: moduli of du and nabla^2u} and \eqref{eq: second-order estimates, g^{jk}(x_0), D_lg^{jk}(x_0), D_{lm}g^{jk}(x_0)}, there hold
\begin{ga}
\mathrm{D}_r\left|\dif u\right|_{\bm g}^2=\mathrm{D}_rg^{jk}\mathrm{D}_ju\mathrm{D}_ku+g^{jk}(\mathrm{D}_{jr}u\mathrm{D}_ku+\mathrm{D}_ju\mathrm{D}_{kr}u)=2\sum_{j=1}^n \mathrm{D}_{jr}u\mathrm{D}_ju\comma \label{eq: second-order estimates, D_r|du|_g^2} \\
\begin{aligned}
\mathrm{D}_{rr}\left|\dif u\right|_{\bm g}^2&=\mathrm{D}_{rr}g^{jk}\mathrm{D}_ju\mathrm{D}_ku+g^{jk}(\mathrm{D}_{jrr}u\mathrm{D}_ku+2\mathrm{D}_{jr}u\mathrm{D}_{kr}u+\mathrm{D}_ju\mathrm{D}_{krr}u) \\
&=-\sum_{1\leqslant j\comma k\leqslant n} \mathrm{D}_{rr}g_{kj}\mathrm{D}_ju\mathrm{D}_ku+2\sum_{j=1}^n (\mathrm{D}_{jrr}u\mathrm{D}_ju+\left|\mathrm{D}_{jr}u\right|^2).
\end{aligned} \label{eq: second-order estimates, D_{rr}|du|_g^2}
\end{ga}
By \eqref{eq: second-order estimates, F^{kk}(D_{kkj}u-D_j Gamma_{kk}^lD_lu)}, \eqref{eq: second-order estimates, kappa_j^{(3)}} and \eqref{eq: second-order estimates, D_r|du|_g^2} we have
\begin{eq} \label{eq: second-order estimates, 2F^{rr} sum_{j=1}^n (D_{jrr}u-D_j Gamma_{rr}^lD_lu)D_ju}
2F^{rr}\sum_{j=1}^n (\mathrm{D}_{jrr}u-\mathrm{D}_j\Gamma_{rr}^l\mathrm{D}_lu)\mathrm{D}_ju\geqslant\left(-F^{kk}\left.\frac{\partial A_{kk}(\cdot\comma u)}{\partial\tilde{\alpha}_l}\right|_{\dif u}+\left.\frac{\partial\varphi(\cdot\comma u)}{\partial\tilde{\alpha}_l}\right|_{\dif u}\right)\mathrm{D}_l\left|\dif u\right|_{\bm g}^2-C_8\mathscr{F}.
\end{eq}
Recall \eqref{eq: second-order estimates, D_{lm}g_{jk}(x_0)}, \eqref{eq: second-order estimates, R_{jklm}(x_0)} and note that
\begin{eq} \label{eq: second-order estimates, -F^{rr}sum_{1 leqslant j, k leqslant n} D_{rr}g_{kj}D_juD_ku+2F^{rr}sum_{j=1}^n D_j Gamma_{rr}^lD_luD_ju} \begin{aligned}
&\mathrel{\hphantom{=}}-F^{rr}\sum_{1\leqslant j\comma k\leqslant n} \mathrm{D}_{rr}g_{kj}\mathrm{D}_ju\mathrm{D}_ku+2F^{rr}\sum_{j=1}^n \mathrm{D}_j\Gamma_{rr}^l\mathrm{D}_lu\mathrm{D}_ju \\
&=-F^{rr}\sum_{1\leqslant j\comma k\leqslant n} (\mathrm{D}_r\Gamma_{kr}^j+\mathrm{D}_r\Gamma_{jr}^k)\mathrm{D}_ju\mathrm{D}_ku+F^{rr}\sum_{1\leqslant j\comma k\leqslant n} (\mathrm{D}_j\Gamma_{rr}^k+\mathrm{D}_k\Gamma_{rr}^j)\mathrm{D}_ku\mathrm{D}_ju \\
&=F^{rr}\sum_{1\leqslant j\comma k\leqslant n} (R_{rkjr}+R_{rjkr})\mathrm{D}_ju\mathrm{D}_ku.
\end{aligned} \end{eq}
It follows from \eqref{eq: second-order estimates, D_{rr}|du|_g^2}, \eqref{eq: second-order estimates, 2F^{rr} sum_{j=1}^n (D_{jrr}u-D_j Gamma_{rr}^lD_lu)D_ju} and \eqref{eq: second-order estimates, -F^{rr}sum_{1 leqslant j, k leqslant n} D_{rr}g_{kj}D_juD_ku+2F^{rr}sum_{j=1}^n D_j Gamma_{rr}^lD_luD_ju} that
\begin{eq} \label{eq: second-order estimates, F^{rr}D_{rr}|du|_g^2}
F^{rr}\mathrm{D}_{rr}\left|\dif u\right|_{\bm g}^2\geqslant\left(-F^{kk}\left.\frac{\partial A_{kk}(\cdot\comma u)}{\partial\tilde{\alpha}_l}\right|_{\dif u}+\left.\frac{\partial\varphi(\cdot\comma u)}{\partial\tilde{\alpha}_l}\right|_{\dif u}\right)\mathrm{D}_l\left|\dif u\right|_{\bm g}^2+2F^{rr}\sum_{j=1}^n \left|\mathrm{D}_{jr}u\right|^2-C_9\mathscr{F}.
\end{eq}
Since $\eta'>0$, combining \eqref{eq: second-order estimates, F^{kk}D_{kk}B_{jj}}, \eqref{eq: second-order estimates, F^{rr}D_{rr}|du|_g^2} and \eqref{eq: second-order estimates, 0=D_r Phi^{(varepsilon)}}, we obtain
\begin{eq} \label{eq: second-order estimates, frac{1}{1+mu_n}F^{rr}D_{rr}B_{nn}+eta'(|du|_g^2)F^{rr}D_{rr}(|du|_g^2)} \begin{aligned}
&\mathrel{\hphantom{=}}\frac{1}{1+\mu_n}F^{rr}\mathrm{D}_{rr}B_{nn}+\eta'(\left|\dif u\right|_{\bm g}^2)F^{rr}\mathrm{D}_{rr}\left|\dif u\right|_{\bm g}^2 \\
&\geqslant\left(-F^{kk}\left.\frac{\partial A_{kk}(\cdot\comma u)}{\partial\tilde{\alpha}_l}\right|_{\dif u}+\left.\frac{\partial\varphi(\cdot\comma u)}{\partial\tilde{\alpha}_l}\right|_{\dif u}\right)\bigl(-\zeta'(\underline{u}-u)\mathrm{D}_l(\underline{u}-u)\bigr)-C_0\mu_n-C_7\mathscr{F} \\
&\quad+\frac{1}{1+\mu_n}\bigl(2F^{kk}(\mu_n\mathrm{D}_n\Gamma_{kk}^n-\mu_k\mathrm{D}_k\Gamma_{nn}^k)+F^{kk}\mu_k\mathrm{D}_{nn}g_{kk}-C_6F^{kk}\mu_k^2+c_1\mu_n^2(\mathscr{F}-F^{nn})\bigr) \\
&\quad-\frac{1}{1+\mu_n}F^{lm\comma rs}\mathrm{D}_nB_{lm}\mathrm{D}_nB_{rs}+2\eta'(\left|\dif u\right|_{\bm g}^2)F^{rr}\sum_{j=1}^n \left|\mathrm{D}_{jr}u\right|^2-C_9\eta'(\left|\dif u\right|_{\bm g}^2)\mathscr{F}.
\end{aligned} \end{eq}
Recalling \eqref{eq: second-order estimates, F^{jk}, new} and \eqref{eq: sum_{k=1}^n mu_k sigma_p(mu|k)}, note that
\begin{eq} \label{eq: second-order estimates, F^{kk}mu_k}
F^{kk}\mu_k=\sigma_p^{\frac1p}(\bm\mu)>0.
\end{eq}
By \eqref{eq: second-order estimates, D_{lm}g_{jk}(x_0)}, \eqref{eq: second-order estimates, R_{jklm}(x_0)}, \eqref{eq: R(Z, W, X, Y)=R(X, Y, Z, W)} and \eqref{eq: second-order estimates, F^{kk}mu_k}, we have
\begin{eq} \label{eq: second-order estimates, F^{kk}(2mu_nD_n Gamma_{kk}^n-2mu_kD_k Gamma_{nn}^k+mu_kD_{nn}g_{kk}-mu_nD_{kk}g_{nn})} \begin{aligned}
&\mathrel{\hphantom{=}}F^{kk}(2\mu_n\mathrm{D}_n\Gamma_{kk}^n-2\mu_k\mathrm{D}_k\Gamma_{nn}^k+\mu_k\mathrm{D}_{nn}g_{kk}-\mu_n\mathrm{D}_{kk}g_{nn}) \\
&=F^{kk}(2\mu_n\mathrm{D}_n\Gamma_{kk}^n-2\mu_n\mathrm{D}_k\Gamma_{nk}^n-2\mu_k\mathrm{D}_k\Gamma_{nn}^k+2\mu_k\mathrm{D}_n\Gamma_{kn}^k) \\
&=2F^{kk}R_{knnk}(\mu_n-\mu_k)\geqslant-2\max_{1\leqslant k\leqslant n-1} \left|R_{knnk}\right|\mu_n\mathscr{F}\comma
\end{aligned} \end{eq}
where in the ``$\geqslant$'' we have used $\mu_n\geqslant\mu_k$. Combining \eqref{eq: second-order estimates, 0 geqslant F^{rr}D_{rr} Phi^{(varepsilon)}}, \eqref{eq: second-order estimates, frac{1}{1+mu_n}F^{rr}D_{rr}B_{nn}+eta'(|du|_g^2)F^{rr}D_{rr}(|du|_g^2)} and \eqref{eq: second-order estimates, F^{kk}(2mu_nD_n Gamma_{kk}^n-2mu_kD_k Gamma_{nn}^k+mu_kD_{nn}g_{kk}-mu_nD_{kk}g_{nn})}, we obtain
\begin{eq} \label{eq: second-order estimates, 0 geqslant F^{rr}D_{rr} Phi^{(varepsilon)}, new} \begin{aligned}
0&\geqslant\frac{1}{1+\mu_n}\left(\sum_{j=1}^{n-1} \frac{2F^{kk}\left|\mathrm{D}_kB_{jn}\right|^2}{\mu_n+\varepsilon-\mu_j}-\frac{F^{kk}\left|\mathrm{D}_kB_{nn}\right|^2}{1+\mu_n}-F^{lm\comma rs}\mathrm{D}_nB_{lm}\mathrm{D}_nB_{rs}\right) \\
&\quad+F^{kk}\bigl(\eta''(\left|\dif u\right|_{\bm g}^2)\bigl|\mathrm{D}_k\left|\dif u\right|_{\bm g}^2\bigr|^2+\zeta''(\underline{u}-u)\left|\mathrm{D}_k(\underline{u}-u)\right|^2\bigr)-C_0\mu_n+c_1\mu_n(\mathscr{F}-F^{nn}) \\
&\quad+\zeta'(\underline{u}-u)\left(F^{kk}\mathrm{D}_{kk}(\underline{u}-u)+F^{kk}\left.\frac{\partial A_{kk}(\cdot\comma u)}{\partial\tilde{\alpha}_l}\right|_{\dif u}\mathrm{D}_l(\underline{u}-u)-C_{10}\right)-C_{11}\mathscr{F} \\
&\quad+2\eta'(\left|\dif u\right|_{\bm g}^2)F^{kk}\sum_{j=1}^n \left|\mathrm{D}_{jk}u\right|^2-\frac{C_6}{1+\mu_n}F^{kk}\mu_k^2-C_9\eta'(\left|\dif u\right|_{\bm g}^2)\mathscr{F}\comma
\end{aligned} \end{eq}
recall $\zeta'>0$. In view of \eqref{eq: sigma_{p-1}(mu|j) decreases} and \eqref{eq: second-order estimates, F^{jk}, new}, there holds
\begin{eq} \label{eq: second-order estimates, F^{jj} decreases}
F^{11}\geqslant F^{22}\geqslant\cdots\geqslant F^{nn}>0.
\end{eq}
Since $\underline{u}$ satisfies pseudo-subsolution condition, recalling \eqref{eq: general p-Hessian, pseudo-subsolution condition}, we have
\begin{eq} \label{eq: second-order estimates, pseudo-subsolution condition}
F^{kk}\mathrm{D}_{kk}(\underline{u}-u)+F^{kk}\left.\frac{\partial A_{kk}(\cdot\comma u)}{\partial\tilde{\alpha}_l}\right|_{\dif u}\mathrm{D}_l(\underline{u}-u)\geqslant\delta_1\mathscr{F}-M_1F^{nn}-M_1\comma
\end{eq}
where $\delta_1>0$ and $M_1\geqslant 0$ are both constants depending only on some trivial quantities. It follows from \eqref{eq: second-order estimates, F^{jk, lm}} and \propref{prop: sigma_p circ lambda, derivatives} that
\begin{eq}
F^{lm\comma rs}=\frac1p\sigma_p^{\frac1p-1}(\bm\mu)\left.\frac{\partial^2\sigma_p\bigl(\bm\lambda(\mathbf{I}_n\comma\mathbf{B})\bigr)}{\partial b_{rs}\partial b_{lm}}\right|_{\mathbf{B}=\mathbf{D}}-\frac{p-1}{p^2}\sigma_p^{\frac1p-2}(\bm\mu)\sigma_{p-1}(\bm\mu|l)\sigma_{p-1}(\bm\mu|r)\updelta_{lm}\updelta_{rs}
\end{eq}
and therefore
\begin{eq} \label{eq: second-order estimates, F^{lm, rs}D_nB_{lm}D_nB_{rs}} \begin{aligned}
&\mathrel{\hphantom{=}}F^{lm\comma rs}\mathrm{D}_nB_{lm}\mathrm{D}_nB_{rs} \\
&=\frac1p\sigma_p^{\frac1p-1}(\bm\mu)\sum_{\substack{1\leqslant l\comma r\leqslant n \\ l\ne r}} \sigma_{p-2}(\bm\mu|lr)\mathrm{D}_nB_{ll}\mathrm{D}_nB_{rr}-\frac1p\sigma_p^{\frac1p-1}(\bm\mu)\sum_{\substack{1\leqslant l\comma m\leqslant n \\ l\ne m}} \sigma_{p-2}(\bm\mu|lm)\left|\mathrm{D}_nB_{lm}\right|^2 \\
&\quad-\frac{p-1}{p^2}\sigma_p^{\frac1p-2}(\bm\mu)\sum_{1\leqslant l\comma r\leqslant n} \sigma_{p-1}(\bm\mu|l)\sigma_{p-1}(\bm\mu|r)\mathrm{D}_nB_{ll}\mathrm{D}_nB_{rr} \\
&\leqslant\left.\frac{\partial^2\sigma_p^{\frac1p}}{\partial\mu_r\partial\mu_l}\right|_{\bm\mu}\mathrm{D}_nB_{ll}\mathrm{D}_nB_{rr}-\frac2p\sigma_p^{\frac1p-1}(\bm\mu)\sum_{j=1}^{n-1} \sigma_{p-2}(\bm\mu|jn)\left|\mathrm{D}_nB_{jn}\right|^2\comma
\end{aligned} \end{eq}
where the ``$\leqslant$'' uses \corref{cor: corollary of increasing with respect to each variable}. Recalling $\eta'>0$, we may as well assume that
\begin{eq} \label{eq: second-order estimates, mu_n geqslant frac{2C_6}{min_M eta'(|du|_g^2)}-1}
\mu_n\geqslant\frac{2C_6}{\min\limits_{\mathcal{M}}\eta'(\left|\dif u\right|_{\bm g}^2)}-1\comma
\end{eq}
which implies
\begin{eq} \label{eq: second-order estimates, 2eta'(|du|_g^2)F^{kk}sum_{j=1}^n |D_{jk}u|^2-frac{C_6}{1+mu_n}F^{kk}mu_k^2-C_9eta'(|du|_g^2)F}
2\eta'(\left|\dif u\right|_{\bm g}^2)F^{kk}\sum_{j=1}^n \left|\mathrm{D}_{jk}u\right|^2-\frac{C_6}{1+\mu_n}F^{kk}\mu_k^2\geqslant \eta'(\left|\dif u\right|_{\bm g}^2)F^{kk}\mu_k^2-C_{12}\eta'(\left|\dif u\right|_{\bm g}^2)\mathscr{F}
\end{eq}
by \eqref{eq: second-order estimates, D_{jk}u} and the elementary inequality $\left|a-b\right|^2\geqslant\frac34\left|a\right|^2-3\left|b\right|^2$. Combining \eqref{eq: second-order estimates, 0 geqslant F^{rr}D_{rr} Phi^{(varepsilon)}, new}, \eqref{eq: second-order estimates, pseudo-subsolution condition}, \eqref{eq: second-order estimates, F^{lm, rs}D_nB_{lm}D_nB_{rs}} and \eqref{eq: second-order estimates, 2eta'(|du|_g^2)F^{kk}sum_{j=1}^n |D_{jk}u|^2-frac{C_6}{1+mu_n}F^{kk}mu_k^2-C_9eta'(|du|_g^2)F}, we obtain the following key estimate (recall that $c_1$ is only non-negative):
\begin{eq} \label{eq: second-order estimates, key estimate} \begin{aligned}
0&\geqslant\frac{1}{1+\mu_n}\left(\sum_{j=1}^{n-1} \frac{2F^{jj}\left|\mathrm{D}_jB_{jn}\right|^2}{\mu_n+\varepsilon-\mu_j}-\left.\frac{\partial^2\sigma_p^{\frac1p}}{\partial\mu_k\partial\mu_j}\right|_{\bm\mu}\mathrm{D}_nB_{jj}\mathrm{D}_nB_{kk}-\frac{F^{nn}\left|\mathrm{D}_nB_{nn}\right|^2}{1+\mu_n}\right) \\
&\quad+\frac{1}{1+\mu_n}\sum_{j=1}^{n-1}\left(\frac{2F^{nn}\left|\mathrm{D}_nB_{jn}\right|^2}{\mu_n+\varepsilon-\mu_j}+\frac2p\sigma_p^{\frac1p-1}(\bm\mu)\sigma_{p-2}(\bm\mu|jn)\left|\mathrm{D}_nB_{jn}\right|^2-\frac{F^{jj}\left|\mathrm{D}_jB_{nn}\right|^2}{1+\mu_n}\right) \\
&\quad+F^{jj}\bigl(\eta''(\left|\dif u\right|_{\bm g}^2)\bigl|\mathrm{D}_j\left|\dif u\right|_{\bm g}^2\bigr|^2+\zeta''(\underline{u}-u)\left|\mathrm{D}_j(\underline{u}-u)\right|^2\bigr)-\zeta'(\underline{u}-u)(M_1F^{nn}+C_{13}) \\
&\quad+\eta'(\left|\dif u\right|_{\bm g}^2)F^{jj}\mu_j^2-C_0\mu_n+c_1\mu_n(\mathscr{F}-F^{nn})+\bigl(\delta_1\zeta'(\underline{u}-u)-C_{14}\eta'(\left|\dif u\right|_{\bm g}^2)-C_{11}\bigr)\mathscr{F}.
\end{aligned} \end{eq}
Next, we fix a constant $\tau_1\in\left(0\comma\frac12\right]$ to be determined later and apply the idea of classification in \cite{Hou2010}, see also \cite{Chou2001,Szekelyhidi2018}.

Case 1: $\mu_1\leqslant-\tau_1\mu_n$. In this case, since $F^{11}\geqslant\frac1n\mathscr{F}$ (cf. \eqref{eq: second-order estimates, F^{jj} decreases}), there holds
\begin{eq} \label{eq: second-order estimates, p=2 or A(alpha, t) satisfies MTW condition or varphi(alpha, t) satisfies the convexity condition, F^{jj}mu_j^2}
F^{jj}\mu_j^2\geqslant F^{11}\mu_1^2+F^{nn}\mu_n^2\geqslant\frac{\tau_1^2}{n}\mathscr{F}\mu_n^2+F^{nn}\mu_n^2.
\end{eq}
By \eqref{eq: second-order estimates, key estimate} and \eqref{eq: second-order estimates, 0=D_r Phi^{(varepsilon)}}, for any $j\in\{1\comma2\comma\cdots\comma n\}$ we have
\begin{eq} \label{eq: second-order estimates, p=2 or A(alpha, t) satisfies MTW condition or varphi(alpha, t) satisfies the convexity condition, -frac{|D_jB_{nn}|^2}{(1+mu_n)^2}+eta''(|du|_g^2)|D_j|du|_g^2|^2} \begin{aligned}
&\mathrel{\hphantom{=}}-\frac{\left|\mathrm{D}_jB_{nn}\right|^2}{(1+\mu_n)^2}+\eta''(\left|\dif u\right|_{\bm g}^2)\bigl|\mathrm{D}_j\left|\dif u\right|_{\bm g}^2\bigr|^2 \\
&=-\bigl|\eta'(\left|\dif u\right|_{\bm g}^2)\mathrm{D}_j\left|\dif u\right|_{\bm g}^2+\zeta'(\underline{u}-u)\mathrm{D}_j(\underline{u}-u)\bigr|^2+\eta''(\left|\dif u\right|_{\bm g}^2)\bigl|\mathrm{D}_j\left|\dif u\right|_{\bm g}^2\bigr|^2 \\
&\geqslant\Bigl(\eta''(\left|\dif u\right|_{\bm g}^2)-2\bigl(\eta'(\left|\dif u\right|_{\bm g}^2)\bigr)^2\Bigr)\bigl|\mathrm{D}_j\left|\dif u\right|_{\bm g}^2\bigr|^2-2\bigl(\zeta'(\underline{u}-u)\bigr)^2\left|\mathrm{D}_j(\underline{u}-u)\right|^2.
\end{aligned} \end{eq}
Now we require that
\begin{eq} \label{eq: second-order estimates, p=2 or A(alpha, t) satisfies MTW condition or varphi(alpha, t) satisfies the convexity condition, eta''(t)-2(eta'(t))^2}
\eta''(t)-2\bigl(\eta'(t)\bigr)^2\geqslant 0\comma\forall t\in\left[0\comma\max_{\mathcal{M}}\left|\dif u\right|_{\bm g}^2\right].
\end{eq}
Recalling \eqref{eq: second-order estimates, eta} and \eqref{eq: second-order estimates, zeta}, it follows from \eqref{eq: second-order estimates, key estimate}, \eqref{eq: second-order estimates, p=2 or A(alpha, t) satisfies MTW condition or varphi(alpha, t) satisfies the convexity condition, F^{jj}mu_j^2}, \eqref{eq: second-order estimates, p=2 or A(alpha, t) satisfies MTW condition or varphi(alpha, t) satisfies the convexity condition, -frac{|D_jB_{nn}|^2}{(1+mu_n)^2}+eta''(|du|_g^2)|D_j|du|_g^2|^2}, \eqref{eq: second-order estimates, p=2 or A(alpha, t) satisfies MTW condition or varphi(alpha, t) satisfies the convexity condition, eta''(t)-2(eta'(t))^2} and \corref{cor: concavity-1} that
\begin{eq} \begin{aligned}
0&\geqslant\eta'(\left|\dif u\right|_{\bm g}^2)\left(\frac{\tau_1^2}{n}\mathscr{F}\mu_n^2+F^{nn}\mu_n^2\right)-\Bigl(C_{15}\bigl(\zeta'(\underline{u}-u)\bigr)^2+C_{14}\eta'(\left|\dif u\right|_{\bm g}^2)+C_{11}\Bigr)\mathscr{F} \\
&\quad-\zeta'(\underline{u}-u)(M_1F^{nn}+C_{13})-C_0\mu_n.
\end{aligned} \end{eq}
We may as well assume that
\begin{eq} \label{eq: second-order estimates, p=2 or A(alpha, t) satisfies MTW condition or varphi(alpha, t) satisfies the convexity condition, case 1, mu_n^2 geqslant}
\mu_n^2\geqslant\max\left\{\frac{M_1\max\limits_{\mathcal{M}}\zeta'(\underline{u}-u)}{\min\limits_{\mathcal{M}}\eta'(\left|\dif u\right|_{\bm g}^2)}\comma \frac{2n\Bigl(C_{15}\max\limits_{\mathcal{M}}\bigl(\zeta'(\underline{u}-u)\bigr)^2+C_{14}\max\limits_{\mathcal{M}}\eta'(\left|\dif u\right|_{\bm g}^2)+C_{11}\Bigr)}{\tau_1^2\min\limits_{\mathcal{M}}\eta'(\left|\dif u\right|_{\bm g}^2)}\right\}\comma
\end{eq}
which implies
\begin{eq}
0\geqslant\frac{\tau_1^2}{2n}\eta'(\left|\dif u\right|_{\bm g}^2)\mathscr{F}\mu_n^2-C_0\mu_n-C_{13}\zeta'(\underline{u}-u).
\end{eq}
In view of $\mathscr{F}\geqslant1$ (cf. \eqref{eq: second-order estimates, F}) and the elementary inequality $a\mu_n\leqslant\frac12\mu_n^2+\frac12a^2$, there holds
\begin{eq} \label{eq: second-order estimates, p=2 or A(alpha, t) satisfies MTW condition or varphi(alpha, t) satisfies the convexity condition, case 1, conclusion} 
\mu_n^2\leqslant\left(\frac{2nC_0}{\tau_1^2\min\limits_{\mathcal{M}}\eta'(\left|\dif u\right|_{\bm g}^2)}\right)^2+\frac{4nC_{13}\max\limits_{\mathcal{M}}\zeta'(\underline{u}-u)}{\tau_1^2\min\limits_{\mathcal{M}}\eta'(\left|\dif u\right|_{\bm g}^2)}.
\end{eq}

Case 2: $\mu_1>-\tau_1\mu_n$. In this case, for any $j_0\in\{1\comma2\comma\cdots\comma n-1\}$ we have
\begin{eq} \begin{aligned}
\sigma_{p-2}(\bm\mu|j_0n)+\frac{\sigma_{p-1}(\bm\mu|n)}{\mu_n+\varepsilon-\mu_{j_0}}\geqslant\frac{\sigma_{p-1}(\bm\mu|j_0)}{\mu_n+\varepsilon-\mu_{j_0}}\geqslant\frac{\sigma_{p-1}(\bm\mu|j_0)}{(1+\tau_1)\mu_n+\varepsilon}\comma
\end{aligned} \end{eq}
where the first ``$\geqslant$'' follows from \eqref{eq: sigma_{p-1}(mu|j_1)-sigma_{p-1}(mu|j_2)} and \corref{cor: corollary of increasing with respect to each variable}. Thus, recalling \eqref{eq: second-order estimates, F^{jk}, new}, there holds
\begin{eq} \label{eq: second-order estimates, p=2 or A(alpha, t) satisfies MTW condition or varphi(alpha, t) satisfies the convexity condition, case 2, 1} \begin{aligned}
&\mathrel{\hphantom{=}}\frac{2F^{nn}\left|\mathrm{D}_nB_{j_0n}\right|^2}{\mu_n+\varepsilon-\mu_{j_0}}+\frac2p\sigma_p^{\frac1p-1}(\bm\mu)\sigma_{p-2}(\bm\mu|j_0n)\left|\mathrm{D}_nB_{j_0n}\right|^2-\frac{2}{1+\tau_1}\frac{F^{j_0j_0}\left|\mathrm{D}_nB_{j_0n}\right|^2}{1+\mu_n} \\
&=\frac1p\sigma_p^{\frac1p-1}(\bm\mu)\left(\frac{2\sigma_{p-1}(\bm\mu|n)}{\mu_n+\varepsilon-\mu_{j_0}}+2\sigma_{p-2}(\bm\mu|j_0n)-\frac{2\sigma_{p-1}(\bm\mu|j_0)}{(1+\tau_1)(1+\mu_n)}\right)\left|\mathrm{D}_nB_{j_0n}\right|^2 \\
&\geqslant\frac1p\sigma_p^{\frac1p-1}(\bm\mu)\left(\frac{2\sigma_{p-1}(\bm\mu|j_0)}{(1+\tau_1)\mu_n+\varepsilon}-\frac{2\sigma_{p-1}(\bm\mu|j_0)}{(1+\tau_1)(1+\mu_n)}\right)\left|\mathrm{D}_nB_{j_0n}\right|^2\geqslant0\comma
\end{aligned} \end{eq}
where in the last ``$\geqslant$ we have required that
\begin{eq} \label{eq: second-order estimates, p=2 or A(alpha, t) satisfies MTW condition or varphi(alpha, t) satisfies the convexity condition, case 2, varepsilon}
\varepsilon\in(0\comma1+\tau_1).
\end{eq}
Recalling \eqref{eq: second-order estimates, D_l nabla_{jk}u}, note that
\begin{eq} \begin{aligned}
&\mathrel{\hphantom{=}}\mathrm{D}_lB_{jk}-\mathrm{D}_kB_{jl} \\
&=\mathrm{D}_l\bigl(A_{jk}(\dif u\comma u)\bigr)+\mathrm{D}_l\nabla_{jk}u-\mathrm{D}_k\bigl(A_{jl}(\dif u\comma u)\bigr)-\mathrm{D}_k\nabla_{jl}u \\
&=\left.\frac{\partial A_{jk}(\cdot\comma u)}{\partial\tilde{x}^l}\right|_{\dif u}+\left.\frac{\partial A_{jk}(\cdot\comma u)}{\partial\tilde{\alpha}_m}\right|_{\dif u}\mathrm{D}_{ml}u+\left.\frac{\partial A_{jk}(\dif u\comma\cdot)}{\partial t}\right|_u\mathrm{D}_lu+\mathrm{D}_{jkl}u-\mathrm{D}_l\Gamma_{jk}^m\mathrm{D}_mu \\
&\quad-\left.\frac{\partial A_{jl}(\cdot\comma u)}{\partial\tilde{x}^k}\right|_{\dif u}-\left.\frac{\partial A_{jl}(\cdot\comma u)}{\partial\tilde{\alpha}_m}\right|_{\dif u}\mathrm{D}_{mk}u-\left.\frac{\partial A_{jl}(\dif u\comma\cdot)}{\partial t}\right|_u\mathrm{D}_ku-\mathrm{D}_{jlk}u+\mathrm{D}_k\Gamma_{jl}^m\mathrm{D}_mu
\end{aligned} \end{eq}
and therefore
\begin{eq} \label{eq: second-order estimates, D_jB_{nn}-D_nB_{jn}, D_jB_{jn}-D_nB_{jj}}
\max\left\{|\mathrm{D}_jB_{nn}-\mathrm{D}_nB_{jn}|\comma|\mathrm{D}_jB_{jn}-\mathrm{D}_nB_{jj}|\right\}\leqslant C_{16}(1+\left|\mu_j\right|+\mu_n)\leqslant C_{16}(1+n\mu_n)
\end{eq}
in view of \eqref{eq: second-order estimates, D_{jk}u}, \eqref{eq: second-order estimates, R_{jklm}(x_0)} and \eqref{eq: mu_1>-frac{n-p}{p(n-1)} sum_{j=2}^n mu_j}. Combining \eqref{eq: second-order estimates, p=2 or A(alpha, t) satisfies MTW condition or varphi(alpha, t) satisfies the convexity condition, case 2, 1} and \eqref{eq: second-order estimates, D_jB_{nn}-D_nB_{jn}, D_jB_{jn}-D_nB_{jj}}, we find that
\begin{eq} \label{eq: second-order estimates, p=2 or A(alpha, t) satisfies MTW condition or varphi(alpha, t) satisfies the convexity condition, case 2, 2} \begin{aligned}
&\mathrel{\hphantom{=}}\frac{1}{1+\mu_n}\sum_{j=1}^{n-1}\left(\frac{2F^{nn}\left|\mathrm{D}_nB_{jn}\right|^2}{\mu_n+\varepsilon-\mu_j}+\frac2p\sigma_p^{\frac1p-1}(\bm\mu)\sigma_{p-2}(\bm\mu|jn)\left|\mathrm{D}_nB_{jn}\right|^2-\frac{F^{jj}\left|\mathrm{D}_jB_{nn}\right|^2}{1+\mu_n}\right) \\
&\geqslant\frac{1}{1+\mu_n}\sum_{j=1}^{n-1} \left(\frac{2}{1+\tau_1}\frac{F^{jj}\left|\mathrm{D}_nB_{jn}\right|^2}{1+\mu_n}-\frac{F^{jj}\left|\mathrm{D}_jB_{nn}\right|^2}{1+\mu_n}\right)\geqslant-\frac{C_{17}}{1-\tau_1}\mathscr{F}\comma
\end{aligned} \end{eq}
where in the last ``$\geqslant$'' we have used the elementary inequality $\left|a+b\right|^2\leqslant\frac{2}{1+\tau_1}\left|a\right|^2+\frac{2}{1-\tau_1}\left|b\right|^2$ since $\tau_1\in\left(0\comma\frac12\right]$. By \eqref{eq: second-order estimates, p=2 or A(alpha, t) satisfies MTW condition or varphi(alpha, t) satisfies the convexity condition, -frac{|D_jB_{nn}|^2}{(1+mu_n)^2}+eta''(|du|_g^2)|D_j|du|_g^2|^2} and \eqref{eq: second-order estimates, p=2 or A(alpha, t) satisfies MTW condition or varphi(alpha, t) satisfies the convexity condition, eta''(t)-2(eta'(t))^2} we have
\begin{eq} \label{eq: second-order estimates, p=2 or A(alpha, t) satisfies MTW condition or varphi(alpha, t) satisfies the convexity condition, -frac{F^{nn}|D_nB_{nn}|^2}{(1+mu_n)^2}+F^{nn}eta''(|du|_g^2)|D_n|du|_g^2|^2}
-\frac{F^{nn}\left|\mathrm{D}_nB_{nn}\right|^2}{(1+\mu_n)^2}+F^{nn}\eta''(\left|\dif u\right|_{\bm g}^2)\bigl|\mathrm{D}_n\left|\dif u\right|_{\bm g}^2\bigr|^2\geqslant-C_{18}\bigl(\zeta'(\underline{u}-u)\bigr)^2F^{nn}.
\end{eq}
Recalling \eqref{eq: second-order estimates, eta} and \eqref{eq: second-order estimates, zeta}, it follows from \eqref{eq: second-order estimates, key estimate}, \eqref{eq: second-order estimates, p=2 or A(alpha, t) satisfies MTW condition or varphi(alpha, t) satisfies the convexity condition, case 2, 2}, \eqref{eq: second-order estimates, p=2 or A(alpha, t) satisfies MTW condition or varphi(alpha, t) satisfies the convexity condition, -frac{F^{nn}|D_nB_{nn}|^2}{(1+mu_n)^2}+F^{nn}eta''(|du|_g^2)|D_n|du|_g^2|^2} and \corref{cor: concavity-1} that
\begin{eq} \begin{aligned}
0&\geqslant-\frac{C_{17}}{1-\tau_1}\mathscr{F}-C_{18}\bigl(\zeta'(\underline{u}-u)\bigr)^2F^{nn}-\zeta'(\underline{u}-u)(M_1F^{nn}+C_{13})+\eta'(\left|\dif u\right|_{\bm g}^2)F^{nn}\mu_n^2 \\
&\quad-C_0\mu_n+c_1\mu_n(\mathscr{F}-F^{nn})+\bigl(\delta_1\zeta'(\underline{u}-u)-C_{14}\eta'(\left|\dif u\right|_{\bm g}^2)-C_{11}\bigr)\mathscr{F}.
\end{aligned} \end{eq}
We may as well assume that
\begin{ga}
\mu_n^2\geqslant\frac{2M_1\max\limits_{\mathcal{M}}\zeta'(\underline{u}-u)+2C_{18}\max\limits_{\mathcal{M}}\bigl(\zeta'(\underline{u}-u)\bigr)^2}{\min\limits_{\mathcal{M}}\eta'(\left|\dif u\right|_{\bm g}^2)}\comma \label{eq: second-order estimates, p=2 or A(alpha, t) satisfies MTW condition or varphi(alpha, t) satisfies the convexity condition, case 2, mu_n^2 geqslant} \\
\min\limits_{\mathcal{M}}\zeta'(\underline{u}-u)\geqslant\frac{2}{\delta_1}\left(C_{14}\max\limits_{\mathcal{M}}\eta'(\left|\dif u\right|_{\bm g}^2)+C_{11}+\frac{C_{17}}{1-\tau_1}\right)\comma \label{eq: second-order estimates, p=2 or A(alpha, t) satisfies MTW condition or varphi(alpha, t) satisfies the convexity condition, case 2, min_M zeta'(underline{u}-u) geqslant}
\end{ga}
which imply
\begin{eq} \label{eq: second-order estimates, p=2 or A(alpha, t) satisfies MTW condition or varphi(alpha, t) satisfies the convexity condition, case 2, conclusion}
0\geqslant\frac12\eta'(\left|\dif u\right|_{\bm g}^2)F^{nn}\mu_n^2+\frac{\delta_1}{2}\zeta'(\underline{u}-u)\mathscr{F}-C_0\mu_n+c_1\mu_n(\mathscr{F}-F^{nn})-C_{13}\zeta'(\underline{u}-u).
\end{eq}
In view of \eqref{eq: sigma_p^{1/p-1}(mu)sigma_{p-1}(mu|n)}, \eqref{eq: second-order estimates, F^{jk}, new} and \eqref{eq: second-order estimates, F}, there holds
\begin{eq} \label{eq: second-order estimates, p=2 or A(alpha, t) satisfies MTW condition or varphi(alpha, t) satisfies the convexity condition, case 2, F^{nn}geqslant}
F^{nn}\geqslant\frac{1}{C_{19}\mathscr{F}^{p-1}}.
\end{eq}
So far, we have not used the assumption ``$p=2$ or $\bm{A}(\bm\alpha\comma t)$ additionally satisfies `MTW condition' \eqref{eq: A(alpha, t), MTW condition} or $\varphi(\bm\alpha\comma t)$ additionally satisfies the convexity condition \eqref{eq: varphi(alpha, t) is convex with respect to alpha}''. If $p=2$, then we have
\begin{eq}
0\geqslant\frac{1}{C_{20}}\min\limits_{\mathcal{M}}\bigl(\eta'(\left|\dif u\right|_{\bm g}^2)\bigr)^{\frac12}\min\limits_{\mathcal{M}}\bigl(\zeta'(\underline{u}-u)\bigr)^{\frac12}\mu_n-C_0\mu_n-C_{13}\max\limits_{\mathcal{M}}\zeta'(\underline{u}-u)\comma
\end{eq}
and therefore
\begin{eq} \label{eq: second-order estimates, p=2, case 2, conclusion}
\mu_n\leqslant\frac{2C_{20}C_{13}\max\limits_{\mathcal{M}}\zeta'(\underline{u}-u)}{\min\limits_{\mathcal{M}}\bigl(\eta'(\left|\dif u\right|_{\bm g}^2)\bigr)^{\frac12}\min\limits_{\mathcal{M}}\bigl(\zeta'(\underline{u}-u)\bigr)^{\frac12}}
\end{eq}
as long as
\begin{eq} \label{eq: second-order estimates, p=2, case 2, requirement}
\min\limits_{\mathcal{M}}\bigl(\eta'(\left|\dif u\right|_{\bm g}^2)\bigr)^{\frac12}\min\limits_{\mathcal{M}}\bigl(\zeta'(\underline{u}-u)\bigr)^{\frac12}\geqslant 2C_0C_{20}.
\end{eq}
If $\bm{A}(\bm\alpha\comma t)$ satisfies \eqref{eq: A(alpha, t), MTW condition}, then $c_1$ can be $c_0>0$ (cf. \eqref{eq: second-order estimates, A_{kk}, weak MTW condition}) and \eqref{eq: second-order estimates, p=2 or A(alpha, t) satisfies MTW condition or varphi(alpha, t) satisfies the convexity condition, case 2, conclusion} implies
\begin{eq}
0\geqslant\frac{1}{2C_{19}}\eta'(\left|\dif u\right|_{\bm g}^2)\frac{\mu_n^2}{\mathscr{F}^{p-1}}+\frac{(n-1)c_0}{n}\mu_n\mathscr{F}-C_0\mu_n-C_{13}\max\limits_{\mathcal{M}}\zeta'(\underline{u}-u)
\end{eq}
since $F^{nn}\leqslant\frac1n\mathscr{F}$ (cf. \eqref{eq: second-order estimates, F^{jj} decreases}); recalling $p\in\{2\comma3\comma\cdots\comma n\}$, by the elementary inequalities
\begin{ga}
a+b\geqslant\frac{p}{(p-1)^{\frac{p-1}{p}}}a^{\frac1p}b^{\frac{p-1}{p}}\comma\forall a\comma b\in(0\comma{+}\infty)\comma \label{eq: second-order estimates, A(alpha, t) satisfies MTW condition, case 2, a+b geqslant} \\
ab\leqslant ca^{\frac{p+1}{p}}+\frac{p^p}{(p+1)^{p+1}c^p}b^{p+1}\comma\forall a\comma b\comma c\in(0\comma{+}\infty)\comma
\end{ga}
we obtain
\begin{eq} \begin{aligned}
0&\geqslant\frac{1}{C_{21}}\min\limits_{\mathcal{M}}\bigl(\eta'(\left|\dif u\right|_{\bm g}^2)\bigr)^{\frac1p}c_0^{\frac{p-1}{p}}\mu_n^{\frac{p+1}{p}}-C_0\mu_n-C_{13}\max\limits_{\mathcal{M}}\zeta'(\underline{u}-u) \\
&\geqslant\frac{1}{2C_{21}}\min\limits_{\mathcal{M}}\bigl(\eta'(\left|\dif u\right|_{\bm g}^2)\bigr)^{\frac1p}c_0^{\frac{p-1}{p}}\mu_n^{\frac{p+1}{p}}-\frac{C_{22}}{c_0^{p-1}\min\limits_{\mathcal{M}}\eta'(\left|\dif u\right|_{\bm g}^2)}-C_{13}\max\limits_{\mathcal{M}}\zeta'(\underline{u}-u)
\end{aligned} \end{eq}
and therefore
\begin{eq} \label{eq: second-order estimates, A(alpha, t) satisfies MTW condition, case 2, conclusion}
\mu_n^{\frac{p+1}{p}}\leqslant\frac{\frac{2C_{21}C_{22}}{c_0^{p-1}\min\limits_{\mathcal{M}}\eta'(\left|\dif u\right|_{\bm g}^2)}+2C_{21}C_{13}\max\limits_{\mathcal{M}}\zeta'(\underline{u}-u)}{\min\limits_{\mathcal{M}}\bigl(\eta'(\left|\dif u\right|_{\bm g}^2)\bigr)^{\frac1p}c_0^{\frac{p-1}{p}}}.
\end{eq}
If $\varphi(\bm\alpha\comma t)$ satisfies \eqref{eq: varphi(alpha, t) is convex with respect to alpha}, then $C_0$ can be 0 (cf. \eqref{eq: second-order estimates, kappa_j^{(1)}}) and \eqref{eq: second-order estimates, p=2 or A(alpha, t) satisfies MTW condition or varphi(alpha, t) satisfies the convexity condition, case 2, conclusion} implies
\begin{eq}
0\geqslant\frac{1}{C_{23}}\min\limits_{\mathcal{M}}\bigl(\eta'(\left|\dif u\right|_{\bm g}^2)\bigr)^{\frac1p}\min\limits_{\mathcal{M}}\bigl(\zeta'(\underline{u}-u)\bigr)^{\frac{p-1}{p}}\mu_n^{\frac2p}-C_{13}\max\limits_{\mathcal{M}}\zeta'(\underline{u}-u)
\end{eq}
in view of \eqref{eq: second-order estimates, p=2 or A(alpha, t) satisfies MTW condition or varphi(alpha, t) satisfies the convexity condition, case 2, F^{nn}geqslant} and \eqref{eq: second-order estimates, A(alpha, t) satisfies MTW condition, case 2, a+b geqslant}; thus, we have
\begin{eq} \label{eq: second-order estimates, varphi(alpha, t) satisfies the convexity condition, case 2, conclusion}
\mu_n^{\frac 2p}\leqslant\frac{C_{23}C_{13}\max\limits_{\mathcal{M}}\zeta'(\underline{u}-u)}{\min\limits_{\mathcal{M}}\bigl(\eta'(\left|\dif u\right|_{\bm g}^2)\bigr)^{\frac1p}\min\limits_{\mathcal{M}}\bigl(\zeta'(\underline{u}-u)\bigr)^{\frac{p-1}{p}}}.
\end{eq}

At last, we let $\tau_1\triangleq\frac12$, $\varepsilon\triangleq1$, and
\begin{ga}
\eta(t)\triangleq-\frac12\log\left(1+\max\limits_{\mathcal{M}}\left|\dif u\right|_{\bm g}^2-t\right)\comma t\in\left[0\comma\max_{\mathcal{M}}\left|\dif u\right|_{\bm g}^2\right]; \\
\zeta(t)\triangleq C_{21}t\comma t\in\left[\min_{\mathcal{M}}(\underline{u}-u)\comma\max_{\mathcal{M}}(\underline{u}-u)\right].
\end{ga}
Here $C_{21}$ is a positive constant large enough. Then the requirements \eqref{eq: second-order estimates, eta}, \eqref{eq: second-order estimates, zeta}, \eqref{eq: second-order estimates, p=2 or A(alpha, t) satisfies MTW condition or varphi(alpha, t) satisfies the convexity condition, eta''(t)-2(eta'(t))^2}, \eqref{eq: second-order estimates, p=2 or A(alpha, t) satisfies MTW condition or varphi(alpha, t) satisfies the convexity condition, case 2, varepsilon}, \eqref{eq: second-order estimates, p=2 or A(alpha, t) satisfies MTW condition or varphi(alpha, t) satisfies the convexity condition, case 2, min_M zeta'(underline{u}-u) geqslant} and \eqref{eq: second-order estimates, p=2, case 2, requirement} are all satisfied. Recall \eqref{eq: second-order estimates, mu_n geqslant frac{2C_6}{min_M eta'(|du|_g^2)}-1}, \eqref{eq: second-order estimates, p=2 or A(alpha, t) satisfies MTW condition or varphi(alpha, t) satisfies the convexity condition, case 1, mu_n^2 geqslant}, \eqref{eq: second-order estimates, p=2 or A(alpha, t) satisfies MTW condition or varphi(alpha, t) satisfies the convexity condition, case 1, conclusion}, \eqref{eq: second-order estimates, p=2 or A(alpha, t) satisfies MTW condition or varphi(alpha, t) satisfies the convexity condition, case 2, mu_n^2 geqslant}, \eqref{eq: second-order estimates, p=2, case 2, conclusion}, \eqref{eq: second-order estimates, A(alpha, t) satisfies MTW condition, case 2, conclusion} and \eqref{eq: second-order estimates, varphi(alpha, t) satisfies the convexity condition, case 2, conclusion}. By \eqref{eq: second-order estimates, Phi^{(0)}}, \eqref{eq: second-order estimates, max_M Phi^{(0)}=Phi^{(0)}(x_0)} and \eqref{eq: second-order estimates, Phi^{(0)}(x_0)} we obtain
\begin{eq}
\max_{\mathcal{M}}\lambda_n\Bigl((g^{jk})\bigl(A_{jk}(\dif u\comma u)+\nabla_{jk}u\bigr)\Bigr)\leqslant C_{24}
\end{eq}
with $C_{24}$ independent of \eqref{eq: second-order estimates, min varphi(alpha, t)}. Since $\bm\lambda\Bigl((g^{jk})\bigl(A_{jk}(\dif u\comma u)+\nabla_{jk}u\bigr)\Bigr)\in\Gamma_p\subset\Gamma_1$, there also holds
\begin{eq}
\min_{\mathcal{M}}\lambda_1\Bigl((g^{jk})\bigl(A_{jk}(\dif u\comma u)+\nabla_{jk}u\bigr)\Bigr)\geqslant-(n-1)C_{24}.
\end{eq}
Thus, by \lemref{lem: Weyl} we find that
\begin{eq}
\max_{\mathcal{M}}\left|\bm\lambda\bigl((g^{jk})(\nabla_{jk}u)\bigr)\right|\leqslant C_{25}
\end{eq}
and therefore \eqref{eq: general p-Hessian, second-order estimates, conclusion} holds with $M$ independent of \eqref{eq: second-order estimates, min varphi(alpha, t)}. 
\end{proof}

\begin{rmk} \label{rmk: second-order estimates, p=2 or A(alpha, t) satisfies MTW condition}
From the above proof, it's easy to find that if $\bm{A}(\bm\alpha\comma t)$ satisfies ``MTW condition'' \eqref{eq: A(alpha, t), MTW condition}, then the requirement \eqref{eq: second-order estimates, p=2 or A(alpha, t) satisfies MTW condition or varphi(alpha, t) satisfies the convexity condition, case 2, min_M zeta'(underline{u}-u) geqslant} is unnecessary and we can let $\zeta\equiv 0$ --- thus, in this case, the assumption ``$\underline{u}\in\mathrm{C}^2(\mathcal{M})$ satisfies pseudo-subsolution condition'' can be omitted.
\end{rmk}

Next, we give the proof of the case $p=n-1$ which generalizes those in \cite{Ren2019} (the Dirichlet problem with $\bm{A}(\bm\alpha\comma t)\equiv 0$ and the prescribed $(n-1)$-th Weingarten curvature equation) and is somewhat similar to that in \cite{Tu2024} (the estimate of Pogorelov type for the Dirichlet problem with $\bm{A}(\bm\alpha\comma t)\equiv 0$), see also \cite{Lu2024} (the prescribed $(n-1)$-th Weingarten curvature equation). The first half of the proof is the same as that of \thmref{thm: p=2 or A(alpha, t) satisfies MTW condition or varphi(alpha, t) satisfies the convexity condition, second-order estimates} and therefore omitted. Our proof also applies to the case $p=n$, one can refer to \cite{Jiang2014,Guan2015a}.

\begin{thm} \label{thm: p=n-1 or n, second-order estimates}
When $p\in\{n-1\comma n\}$, \eqref{eq: general p-Hessian, second-order estimates, conclusion} holds.
\end{thm}

\begin{proof}
We start from the key estimate \eqref{eq: second-order estimates, key estimate} which holds for any $p\in\{2\comma3\comma\cdots\comma n\}$, and always calculate at $\bm{x}_0$. Since $\mu_1>-\frac{n-p}{p}\mu_n$ (cf. \eqref{eq: mu_1>-frac{n-p}{p(n-1)} sum_{j=2}^n mu_j}), for any $j_0\in\{1\comma2\comma\cdots\comma n-1\}$ we have
\begin{eq} \begin{aligned}
\sigma_{p-2}(\bm\mu|j_0n)+\frac{\sigma_{p-1}(\bm\mu|n)}{\mu_n+\varepsilon-\mu_{j_0}}\geqslant\frac{\sigma_{p-1}(\bm\mu|j_0)}{\mu_n+\varepsilon-\mu_{j_0}}\geqslant\frac{\sigma_{p-1}(\bm\mu|j_0)}{\frac np\mu_n+\varepsilon}\comma
\end{aligned} \end{eq}
where the first ``$\geqslant$'' follows from \eqref{eq: sigma_{p-1}(mu|j_1)-sigma_{p-1}(mu|j_2)} and \corref{cor: corollary of increasing with respect to each variable}. Recalling \eqref{eq: second-order estimates, F^{jk}, new}, there holds
\begin{eq} \label{eq: second-order estimates, p=n-1 or n, 1} \begin{aligned}
&\mathrel{\hphantom{=}}\frac{2F^{nn}\left|\mathrm{D}_nB_{j_0n}\right|^2}{\mu_n+\varepsilon-\mu_{j_0}}+\frac2p\sigma_p^{\frac1p-1}(\bm\mu)\sigma_{p-2}(\bm\mu|j_0n)\left|\mathrm{D}_nB_{j_0n}\right|^2-\frac{2p}{n}\frac{F^{j_0j_0}\left|\mathrm{D}_nB_{j_0n}\right|^2}{1+\mu_n} \\
&=\frac1p\sigma_p^{\frac1p-1}(\bm\mu)\left(\frac{2\sigma_{p-1}(\bm\mu|n)}{\mu_n+\varepsilon-\mu_{j_0}}+2\sigma_{p-2}(\bm\mu|j_0n)-\frac{2p\sigma_{p-1}(\bm\mu|j_0)}{n(1+\mu_n)}\right)\left|\mathrm{D}_nB_{j_0n}\right|^2 \\
&\geqslant\frac1p\sigma_p^{\frac1p-1}(\bm\mu)\left(\frac{2p\sigma_{p-1}(\bm\mu|j_0)}{n\mu_n+p\varepsilon}-\frac{2p\sigma_{p-1}(\bm\mu|j_0)}{n(1+\mu_n)}\right)\left|\mathrm{D}_nB_{j_0n}\right|^2\geqslant0\comma
\end{aligned} \end{eq}
where in the last ``$\geqslant$ we have required that
\begin{eq} \label{eq: second-order estimates, p=n-1 or n, varepsilon}
\varepsilon\in\left(0\comma\frac np\right).
\end{eq}
Combining \eqref{eq: second-order estimates, p=n-1 or n, 1} and \eqref{eq: second-order estimates, D_jB_{nn}-D_nB_{jn}, D_jB_{jn}-D_nB_{jj}}, we find that
\begin{eq} \label{eq: second-order estimates, p=n-1 or n, 2} \begin{aligned}
&\mathrel{\hphantom{=}}\frac{1}{1+\mu_n}\sum_{j=1}^{n-1}\left(\frac{2F^{nn}\left|\mathrm{D}_nB_{jn}\right|^2}{\mu_n+\varepsilon-\mu_j}+\frac2p\sigma_p^{\frac1p-1}(\bm\mu)\sigma_{p-2}(\bm\mu|jn)\left|\mathrm{D}_nB_{jn}\right|^2-\frac{F^{jj}\left|\mathrm{D}_jB_{nn}\right|^2}{1+\mu_n}\right) \\
&\geqslant\frac{1}{1+\mu_n}\sum_{j=1}^{n-1} \left(\frac{2p}{n}\frac{F^{jj}\left|\mathrm{D}_nB_{jn}\right|^2}{1+\mu_n}-\frac{F^{jj}\left|\mathrm{D}_jB_{nn}\right|^2}{1+\mu_n}\right)\geqslant-C_{26}\mathscr{F}\comma
\end{aligned} \end{eq}
where in the last ``$\geqslant$'' we have used $p>\frac n2$ and the elementary inequality
\begin{eq}
\left|a+b\right|^2\leqslant\frac{2p}{n}\left|a\right|^2+\frac{2p}{2p-n}\left|b\right|^2.
\end{eq}
By \eqref{eq: mu_n sigma_{p-1}(mu|n)} and \eqref{eq: second-order estimates, sigma_p^{frac1p}(mu)} we have
\begin{eq} \label{eq: second-order estimates, p=n-1 or n, F^{nn}mu_n}
F^{nn}\mu_n\geqslant\frac 1n\sigma_p^{\frac1p}(\bm\mu)=\frac 1n\varphi(\dif u\comma u)\geqslant\frac{1}{C_{27}}.
\end{eq}
Recalling $\eta'>0$, we may as well assume that
\begin{ga}
\mu_n^2\geqslant\frac{3M_1\max\limits_{\mathcal{M}}\zeta'(\underline{u}-u)}{\min\limits_{\mathcal{M}}\eta'(\left|\dif u\right|_{\bm g}^2)}\comma \label{eq: second-order estimates, p=n-1 or n, mu_n^2 geqslant} \\
\min\limits_{\mathcal{M}}\eta'(\left|\dif u\right|_{\bm g}^2)\geqslant 3C_{27}C_0. \label{eq: second-order estimates, p=n-1 or n, requirement}
\end{ga}
It follows from \eqref{eq: second-order estimates, p=n-1 or n, mu_n^2 geqslant}, \eqref{eq: second-order estimates, p=n-1 or n, F^{nn}mu_n} and \eqref{eq: second-order estimates, p=n-1 or n, requirement} that
\begin{eq} \label{eq: second-order estimates, p=n-1 or n, zeta'(underline{u}-u)M_1F^{nn}+eta'(|du|_g^2)F^{jj}mu_j^2-C_0mu_n}
-\zeta'(\underline{u}-u)M_1F^{nn}+\eta'(\left|\dif u\right|_{\bm g}^2)F^{jj}\mu_j^2-C_0\mu_n\geqslant\eta'(\left|\dif u\right|_{\bm g}^2)\left(\sum_{j=1}^{n-1} F^{jj}\mu_j^2+\frac{\mu_n}{3C_{27}}\right).
\end{eq}
Now we fix a constant $\tau_2\in\left(0\comma\frac12\right]$ to be determined later, and require that
\begin{al}
\eta''(t)\geqslant 2\tau_2\bigl(\eta'(t)\bigr)^2\comma &\forall t\in\left[0\comma\max_{\mathcal{M}}\left|\dif u\right|_{\bm g}^2\right]; \label{eq: second-order estimates, p=n-1 or n, eta''(t) geqslant 2tau_2(eta'(t))^2} \\
\zeta''(t)\geqslant 2\tau_2\bigl(\zeta'(t)\bigr)^2\comma &\forall t\in\left[\min_{\mathcal{M}}(\underline{u}-u)\comma\max_{\mathcal{M}}(\underline{u}-u)\right]. \label{eq: second-order estimates, p=n-1 or n, zeta''(t) geqslant 2tau_2(zeta'(t))^2}
\end{al}
Then \eqref{eq: second-order estimates, 0=D_r Phi^{(varepsilon)}} implies
\begin{eq} \label{eq: second-order estimates, p=n-1 or n, frac{tau_2F^{nn}|D_nB_{nn}|^2}{(1+mu_n)^2}} \begin{aligned}
\frac{\tau_2F^{nn}\left|\mathrm{D}_nB_{nn}\right|^2}{(1+\mu_n)^2}&=F^{nn}\tau_2\bigl|\eta'(\left|\dif u\right|_{\bm g}^2)\mathrm{D}_n\left|\dif u\right|_{\bm g}^2+\zeta'(\underline{u}-u)\mathrm{D}_n(\underline{u}-u)\bigr|^2 \\
&\leqslant F^{nn}\bigl(\eta''(\left|\dif u\right|_{\bm g}^2)\bigl|\mathrm{D}_n\left|\dif u\right|_{\bm g}^2\bigr|^2+\zeta''(\underline{u}-u)\left|\mathrm{D}_n(\underline{u}-u)\right|^2\bigr).
\end{aligned} \end{eq}
Combining \eqref{eq: second-order estimates, key estimate}, \eqref{eq: second-order estimates, p=n-1 or n, 2}, \eqref{eq: second-order estimates, p=n-1 or n, frac{tau_2F^{nn}|D_nB_{nn}|^2}{(1+mu_n)^2}} and \eqref{eq: second-order estimates, p=n-1 or n, zeta'(underline{u}-u)M_1F^{nn}+eta'(|du|_g^2)F^{jj}mu_j^2-C_0mu_n}, we obtain
\begin{eq} \label{eq: second-order estimates, p=n-1 or n, key estimate} \begin{aligned}
C_{13}\zeta'(\underline{u}-u)&\geqslant\frac{1}{1+\mu_n}\left(\sum_{j=1}^{n-1} \frac{2F^{jj}\left|\mathrm{D}_jB_{jn}\right|^2}{\mu_n+\varepsilon-\mu_j}-\left.\frac{\partial^2\sigma_p^{\frac1p}}{\partial\mu_k\partial\mu_j}\right|_{\bm\mu}\mathrm{D}_nB_{jj}\mathrm{D}_nB_{kk}-\frac{(1-\tau_2)F^{nn}\left|\mathrm{D}_nB_{nn}\right|^2}{1+\mu_n}\right) \\
&\quad+\eta'(\left|\dif u\right|_{\bm g}^2)\left(\sum_{j=1}^{n-1} F^{jj}\mu_j^2+\frac{\mu_n}{3C_{27}}\right)+\bigl(\delta_1\zeta'(\underline{u}-u)-C_{14}\eta'(\left|\dif u\right|_{\bm g}^2)-C_{11}-C_{26}\bigr)\mathscr{F}
\end{aligned} \end{eq}
for any $p\in\left(\frac n2\comma n\right]\cap\mathbb{N}$. Recalling \eqref{eq: second-order estimates, F^{jk}, new}, note that
\begin{eq} \label{eq: second-order estimates, p=n-1 or n, -frac{partial^2sigma_p^{frac1p}}{partial mu_k partial mu_j}|_{mu}D_nB_{jj}D_nB_{kk}} \begin{aligned}
-\left.\frac{\partial^2\sigma_p^{\frac1p}}{\partial\mu_k\partial\mu_j}\right|_{\bm\mu}\mathrm{D}_nB_{jj}\mathrm{D}_nB_{kk}&=-\frac1p\sigma_p^{\frac1p-1}(\bm\mu)\sum_{\substack{1\leqslant j\comma k\leqslant n \\ j\ne k}} \sigma_{p-2}(\bm\mu|jk)\mathrm{D}_nB_{jj}\mathrm{D}_nB_{kk} \\
&\quad+\frac{p-1}{p^2}\sigma_p^{\frac1p-2}(\bm\mu)\left|\sum_{j=1}^n \sigma_{p-1}(\bm\mu|j)\mathrm{D}_nB_{jj}\right|^2\comma
\end{aligned} \end{eq}
and
\begin{eq} \label{eq: second-order estimates, p=n-1 or n, frac{1}{p^2}sigma_p^{frac1p-2}(mu)|sum_{j=1}^n sigma_{p-1}(mu|j)D_nB_{jj}|^2}
\frac{1}{p^2}\sigma_p^{\frac1p-2}(\bm\mu)\left|\sum_{j=1}^n \sigma_{p-1}(\bm\mu|j)\mathrm{D}_nB_{jj}\right|^2=\sigma_p^{-\frac1p}(\bm\mu)\left|\sum_{j=1}^n F^{jj}\mathrm{D}_nB_{jj}\right|^2=\frac{1}{\varphi(\dif u\comma u)}\bigl|\mathrm{D}_n\bigl(\varphi(\dif u\comma u)\bigr)\bigr|^2
\end{eq}
in view of \eqref{eq: second-order estimates, sigma_p^{frac1p}(mu)} and \eqref{eq: second-order estimates, equation, first-order}.

If $p=n$, then there holds
\begin{eq} \label{eq: second-order estimates, p=n, 1} \begin{aligned}
&\mathrel{\hphantom{=}}-\sigma_p^{-1}(\bm\mu)\sum_{\substack{1\leqslant j\comma k\leqslant n \\ j\ne k}} \sigma_{p-2}(\bm\mu|jk)\mathrm{D}_nB_{jj}\mathrm{D}_nB_{kk}+\sigma_p^{-2}(\bm\mu)\left|\sum_{j=1}^n \sigma_{p-1}(\bm\mu|j)\mathrm{D}_nB_{jj}\right|^2 \\
&=-\sum_{\substack{1\leqslant j\comma k\leqslant n \\ j\ne k}} \frac{\mathrm{D}_nB_{jj}\mathrm{D}_nB_{kk}}{\mu_j\mu_k}+\left|\sum_{j=1}^n \frac{\mathrm{D}_nB_{jj}}{\mu_j}\right|^2=\sum_{j=1}^n \frac{\left|\mathrm{D}_nB_{jj}\right|^2}{\mu_j^2}\geqslant\frac{\sigma_{p-1}(\bm\mu|n)\left|\mathrm{D}_nB_{nn}\right|^2}{\mu_n\sigma_p(\bm\mu)}.
\end{aligned} \end{eq}
Combining \eqref{eq: second-order estimates, p=n-1 or n, key estimate}, \eqref{eq: second-order estimates, p=n-1 or n, -frac{partial^2sigma_p^{frac1p}}{partial mu_k partial mu_j}|_{mu}D_nB_{jj}D_nB_{kk}}, \eqref{eq: second-order estimates, p=n-1 or n, frac{1}{p^2}sigma_p^{frac1p-2}(mu)|sum_{j=1}^n sigma_{p-1}(mu|j)D_nB_{jj}|^2} and \eqref{eq: second-order estimates, p=n, 1}, we obtain
\begin{eq} \label{eq: second-order estimates, p=n, key estimate}
C_{13}\zeta'(\underline{u}-u)\geqslant-\frac{C_{28}}{1+\mu_n}+\frac{\eta'(\left|\dif u\right|_{\bm g}^2)}{3C_{27}}\mu_n+\bigl(\delta_1\zeta'(\underline{u}-u)-C_{14}\eta'(\left|\dif u\right|_{\bm g}^2)-C_{11}-C_{26}\bigr)\mathscr{F}\comma
\end{eq}
and therefore
\begin{eq} \label{eq: second-order estimates, p=n, conclusion}
\mu_n\leqslant 3C_{27}\frac{C_{28}+C_{13}\max\limits_{\mathcal{M}}\zeta'(\underline{u}-u)}{\min\limits_{\mathcal{M}}\eta'(\left|\dif u\right|_{\bm g}^2)}
\end{eq}
as long as
\begin{eq} \label{eq: second-order estimates, p=n, requirement}
\min\limits_{\mathcal{M}}\zeta'(\underline{u}-u)\geqslant\frac{1}{\delta_1}\left(C_{14}\max\limits_{\mathcal{M}}\eta'(\left|\dif u\right|_{\bm g}^2)+C_{11}+C_{26}\right).
\end{eq}

If $p=n-1$, then \thmref{thm: concavity inequality} implies that
\begin{eq} \label{eq: second-order estimates, p=n-1, concavity inequality} \begin{aligned}
&\mathrel{\hphantom{=}}-\sum_{\substack{1\leqslant j\comma k\leqslant n \\ j\ne k}} \sigma_{p-2}(\bm\mu|jk)\mathrm{D}_nB_{jj}\mathrm{D}_nB_{kk}-\frac{1-\tau_2}{\mu_n}\sigma_{p-1}(\bm\mu|n)\left|\mathrm{D}_nB_{nn}\right|^2 \\
&\geqslant-\frac{n^2}{\sigma_p(\bm\mu)}\left|\sum_{j=1}^n \sigma_{p-1}(\bm\mu|j)\mathrm{D}_nB_{jj}\right|^2-(1-\tau_2)\sum_{j=1}^{n-1}\frac{\sigma_{p-1}(\bm\mu|j)\left|\mathrm{D}_nB_{jj}\right|^2}{\mu_n+\varepsilon-\mu_j}
\end{aligned} \end{eq}
as long as $\mu_n\geqslant M$. Here $M$ is a positive constant depending only on $n$, $\tau_2$, $\varepsilon$ and $\varphi(\dif u\comma u)$. Now we may as well assume that
\begin{eq} \label{eq: second-order estimates, p=n-1, mu_n geqslant max{M, 1}}
\mu_n\geqslant\max\{M\comma 1\}.
\end{eq}
By \eqref{eq: second-order estimates, D_jB_{nn}-D_nB_{jn}, D_jB_{jn}-D_nB_{jj}} we have
\begin{eq} \label{eq: second-order estimates, p=n-1, 1}
\frac{1}{1+\mu_n}\left(\sum_{j=1}^{n-1} \frac{2F^{jj}\left|\mathrm{D}_jB_{jn}\right|^2}{\mu_n+\varepsilon-\mu_j}-\sum_{j=1}^{n-1} \frac{(1-\tau_2)F^{jj}\left|\mathrm{D}_nB_{jj}\right|^2}{\mu_n+\varepsilon-\mu_j}\right)\geqslant-\sum_{j=1}^{n-1} \frac{C_{29}\mu_nF^{jj}}{\mu_n+\varepsilon-\mu_j}.
\end{eq}
Combining \eqref{eq: second-order estimates, p=n-1 or n, key estimate}, \eqref{eq: second-order estimates, p=n-1 or n, -frac{partial^2sigma_p^{frac1p}}{partial mu_k partial mu_j}|_{mu}D_nB_{jj}D_nB_{kk}}, \eqref{eq: second-order estimates, p=n-1, concavity inequality}, \eqref{eq: second-order estimates, p=n-1, 1} and \eqref{eq: second-order estimates, p=n-1 or n, frac{1}{p^2}sigma_p^{frac1p-2}(mu)|sum_{j=1}^n sigma_{p-1}(mu|j)D_nB_{jj}|^2}, we obtain
\begin{eq} \label{eq: second-order estimates, p=n-1, key estimate} \begin{aligned}
C_{13}\zeta'(\underline{u}-u)&\geqslant-\sum_{j=1}^{n-1} \frac{C_{29}\mu_nF^{jj}}{\mu_n+\varepsilon-\mu_j}-\frac{C_{30}}{1+\mu_n}+\eta'(\left|\dif u\right|_{\bm g}^2)\sum_{j=1}^{n-1}F^{jj}\mu_j^2+\frac{\eta'(\left|\dif u\right|_{\bm g}^2)}{3C_{27}}\mu_n \\
&\quad+\bigl(\delta_1\zeta'(\underline{u}-u)-C_{14}\eta'(\left|\dif u\right|_{\bm g}^2)-C_{11}-C_{26}\bigr)\mathscr{F}.
\end{aligned} \end{eq}
Note that
\begin{eq} \label{eq: second-order estimates, p=n-1, sum_{j=1}^{n-1} frac{F^{jj}}{mu_n+varepsilon-mu_j}}
\sum_{j=1}^{n-1} \frac{\mu_nF^{jj}}{\mu_n+\varepsilon-\mu_j}\leqslant\sum_{\substack{1\leqslant j\leqslant n-1 \\ \mu_j\geqslant\frac{\mu_n}{2}}} \frac{4F^{jj}\mu_j^2}{\varepsilon\mu_n}+\sum_{\substack{1\leqslant j\leqslant n-1 \\ \mu_j<\frac{\mu_n}{2}}} \frac{\mu_nF^{jj}}{\frac{\mu_n}{2}+\varepsilon}\leqslant\frac{4}{\varepsilon\mu_n}\sum_{j=1}^{n-1}F^{jj}\mu_j^2+2\mathscr{F}.
\end{eq}
It follows from \eqref{eq: second-order estimates, p=n-1, key estimate} and \eqref{eq: second-order estimates, p=n-1, sum_{j=1}^{n-1} frac{F^{jj}}{mu_n+varepsilon-mu_j}} that
\begin{eq} \label{eq: second-order estimates, p=n-1, 2} \begin{aligned}
C_{13}\zeta'(\underline{u}-u)&\geqslant\left(\eta'(\left|\dif u\right|_{\bm g}^2)-\frac{4C_{29}}{\varepsilon\mu_n}\right)\sum_{j=1}^{n-1}F^{jj}\mu_j^2-\frac{C_{30}}{1+\mu_n}+\frac{\eta'(\left|\dif u\right|_{\bm g}^2)}{3C_{27}}\mu_n \\
&\quad+\bigl(\delta_1\zeta'(\underline{u}-u)-C_{14}\eta'(\left|\dif u\right|_{\bm g}^2)-C_{11}-C_{26}-2C_{29}\bigr)\mathscr{F}\comma
\end{aligned} \end{eq}
and therefore
\begin{eq} \label{eq: second-order estimates, p=n-1, conclusion}
\mu_n\leqslant\max\left\{\frac{4C_{29}}{\varepsilon\min\limits_{\mathcal{M}}\eta'(\left|\dif u\right|_{\bm g}^2)}\comma 3C_{27}\frac{C_{30}+C_{13}\max\limits_{\mathcal{M}}\zeta'(\underline{u}-u)}{\min\limits_{\mathcal{M}}\eta'(\left|\dif u\right|_{\bm g}^2)}\right\}
\end{eq}
as long as
\begin{eq} \label{eq: second-order estimates, p=n-1, requirement}
\min\limits_{\mathcal{M}}\zeta'(\underline{u}-u)\geqslant\frac{1}{\delta_1}\left(C_{14}\max\limits_{\mathcal{M}}\eta'(\left|\dif u\right|_{\bm g}^2)+C_{11}+C_{26}+2C_{29}\right).
\end{eq}

At last, we let $\varepsilon\triangleq\frac12$,
\begin{eq}
\eta(t)\triangleq\mathrm{e}^{C_{31}t}\comma\quad \zeta(t)\triangleq\mathrm{e}^{C_{32}\left(t-\min\limits_{\mathcal{M}}(\underline{u}-u)\right)}\comma\quad \tau_2\triangleq\frac12\min\left\{\mathrm{e}^{-C_{31}\max\limits_{\mathcal{M}}\left|\dif u\right|_{\bm g}^2}\comma\mathrm{e}^{-C_{32}\osc\limits_{\mathcal{M}}(\underline{u}-u)}\right\}.
\end{eq}
Here $C_{31}$ is a positive constant large enough and $C_{32}$ is a positive constant sufficiently larger than $C_{31}$. Then the requirements \eqref{eq: second-order estimates, eta}, \eqref{eq: second-order estimates, zeta}, \eqref{eq: second-order estimates, p=n-1 or n, varepsilon}, \eqref{eq: second-order estimates, p=n-1 or n, requirement}, \eqref{eq: second-order estimates, p=n-1 or n, eta''(t) geqslant 2tau_2(eta'(t))^2}, \eqref{eq: second-order estimates, p=n-1 or n, zeta''(t) geqslant 2tau_2(zeta'(t))^2}, \eqref{eq: second-order estimates, p=n, requirement} and \eqref{eq: second-order estimates, p=n-1, requirement} are all satisfied. Recall \eqref{eq: second-order estimates, mu_n geqslant frac{2C_6}{min_M eta'(|du|_g^2)}-1}, \eqref{eq: second-order estimates, p=n-1 or n, mu_n^2 geqslant}, \eqref{eq: second-order estimates, p=n, conclusion}, \eqref{eq: second-order estimates, p=n-1, mu_n geqslant max{M, 1}} and \eqref{eq: second-order estimates, p=n-1, conclusion}. By \eqref{eq: second-order estimates, Phi^{(0)}}, \eqref{eq: second-order estimates, max_M Phi^{(0)}=Phi^{(0)}(x_0)} and \eqref{eq: second-order estimates, Phi^{(0)}(x_0)} we obtain
\begin{eq} \label{eq: second-order estimates, p=n-1 or n, conclusion}
\max_{\mathcal{M}}\lambda_n\Bigl((g^{jk})\bigl(A_{jk}(\dif u\comma u)+\nabla_{jk}u\bigr)\Bigr)\leqslant C_{33}
\end{eq}
and therefore \eqref{eq: general p-Hessian, second-order estimates, conclusion} holds.
\end{proof}

Combining \thmref{thm: p=2 or A(alpha, t) satisfies MTW condition or varphi(alpha, t) satisfies the convexity condition, second-order estimates} and \thmref{thm: p=n-1 or n, second-order estimates}, we have proved \thmref{thm: general p-Hessian, second-order estimates} in particular. At last, we give the second-order estimates for ``semi-convex solutions'' to the general $p$-Hessian equation \eqref{eq: general p-Hessian}. Here, by a ``semi-convex solution'' $u$ to \eqref{eq: general p-Hessian} we mean $u\in\mathrm{C}^2(\mathcal{M})$ satisfies \eqref{eq: general p-Hessian} and \eqref{eq: second-order estimates, semi-convex} below. One can refer to \cite{Lu2023,Zhang2025,Chen2025} for similar results. Note that every general $(p+1)$-admissible solution to \eqref{eq: general p-Hessian} is a ``semi-convex solution'', cf. \eqref{eq: mu_1, lower bound}.

\begin{thm} \label{thm: semi-convex, second-order estimates}
For general $p$, assume additionally that there exists a non-negative constant $C$, depending only on $\mathcal{M}$, $\bm{g}$, $n$, $p$, $\bm{A}(\bm{\alpha}\comma t)$, $\varphi(\bm{\alpha}\comma t)$, $\max\limits_{\mathcal M} \left|u\right|$ and $\max\limits_{\mathcal M} \left|\dif u\right|_{\bm g}$, so that
\begin{eq} \label{eq: second-order estimates, semi-convex}
\bm\lambda\Bigl({\bm g}^{-1}\bigl(\bm{A}(\dif u\comma u)+\nabla^2u\bigr)\Bigr)(\bm{x})+C\bm{1}_n\in\overline{\Gamma}_n\comma\forall\bm{x}\in\mathcal{M}.
\end{eq}
Then \eqref{eq: general p-Hessian, second-order estimates, conclusion} also holds.
\end{thm}

\begin{proof}
The proof is almost the same as that of the case $p=n-1$ in \thmref{thm: p=n-1 or n, second-order estimates}, so the details are omitted. Recall the key estimate \eqref{eq: second-order estimates, key estimate} which holds for any $p\in\{2\comma3\comma\cdots\comma n\}$. We always calculate at $\bm{x}_0$, and may as well assume that $\mu_n>(n-1)C$. Then \eqref{eq: second-order estimates, semi-convex} implies
\begin{eq}
\mu_1\geqslant-C>-\frac{1}{n-1}\mu_n\comma
\end{eq}
and therefore
\begin{eq} \begin{aligned}
&\mathrel{\hphantom{=}}\frac{1}{1+\mu_n}\sum_{j=1}^{n-1}\left(\frac{2F^{nn}\left|\mathrm{D}_nB_{jn}\right|^2}{\mu_n+\varepsilon-\mu_j}+\frac2p\sigma_p^{\frac1p-1}(\bm\mu)\sigma_{p-2}(\bm\mu|jn)\left|\mathrm{D}_nB_{jn}\right|^2-\frac{F^{jj}\left|\mathrm{D}_jB_{nn}\right|^2}{1+\mu_n}\right) \\
&\geqslant\frac{1}{1+\mu_n}\sum_{j=1}^{n-1} \left(\frac{2(n-1)}{n}\frac{F^{jj}\left|\mathrm{D}_nB_{jn}\right|^2}{1+\mu_n}-\frac{F^{jj}\left|\mathrm{D}_jB_{nn}\right|^2}{1+\mu_n}\right)\geqslant-C_{26}\mathscr{F}
\end{aligned} \end{eq}
analogous to \eqref{eq: second-order estimates, p=n-1 or n, 2}. Thus, in this case, \eqref{eq: second-order estimates, p=n-1 or n, key estimate} holds for any $p\in\{2\comma3\comma\cdots\comma n\}$. By \propref{prop: concavity inequality, -mu_1 is little} we have
\begin{eq} \label{eq: second-order estimates, semi-convex, concavity inequality} \begin{aligned}
&\mathrel{\hphantom{=}}-\sum_{\substack{1\leqslant j\comma k\leqslant n \\ j\ne k}} \sigma_{p-2}(\bm\mu|jk)\mathrm{D}_nB_{jj}\mathrm{D}_nB_{kk}-\frac{1-\tau_2}{\mu_n}\sigma_{p-1}(\bm\mu|n)\left|\mathrm{D}_nB_{nn}\right|^2 \\
&\geqslant-\frac{(p+1)^2}{\sigma_p(\bm\mu)}\left|\sum_{j=1}^n \sigma_{p-1}(\bm\mu|j)\mathrm{D}_nB_{jj}\right|^2-(1-\tau_2)\sum_{j=1}^{n-1}\frac{\sigma_{p-1}(\bm\mu|j)\left|\mathrm{D}_nB_{jj}\right|^2}{\mu_n+\varepsilon-\mu_j}
\end{aligned} \end{eq}
analogous to \eqref{eq: second-order estimates, p=n-1, concavity inequality}, as long as $\mu_n\geqslant M$. Here $M$ is a positive constant depending only on $n$, $p$, $\tau_2$, $\varepsilon$, $C$ and $\varphi(\dif u\comma u)$. \eqref{eq: general p-Hessian, second-order estimates, conclusion} now follows from \eqref{eq: second-order estimates, p=n-1, mu_n geqslant max{M, 1}}--\eqref{eq: second-order estimates, p=n-1 or n, conclusion}.
\end{proof}


\section{$\mathrm{C}^0$ estimates for general $p$-Hessian equations} \label{sec: C^0 estimates for general p-Hessian equations}

In this section, we prove \thmref{thm: general p-Hessian, C^0 estimates}. Let $(\mathcal{M}\comma\bm g)$ be a closed connected Riemannian manifold of dimension $n$. Throughout this section, we assume that $n\geqslant 3$, $p\in\{2\comma3\comma\cdots\comma n\}$ and $u\in\mathrm{C}^2(\mathcal M)$ satisfies the general $p$-Hessian equation \eqref{eq: general p-Hessian} on $\mathcal{M}$. Recall \eqref{eq: A(alpha, t)}, \eqref{eq: du and nabla^2u}, \eqref{eq: lambda(g^{-1}(A(du, u)+nabla^2u))}, \defnref{defn: lambda_q}, \defnref{defn: lambda}, \eqref{eq: sigma_p}, \defnref{defn: Garding cone}, \eqref{eq: moduli of du and nabla^2u} and that $\varphi(\bm{\alpha}\comma t)$ is a positive smooth function on $\mathrm{T}^*\mathcal{M}\times\mathbb{R}$. We assume additionally that $\underline{u}\in\mathrm{C}^2(\mathcal{M})$ satisfies pseudo-supersolution condition (cf. \defnref{defn: pseudo-solution}), and there exist non-negative continuous functions $\phi_1$, $\phi_2$ on $\mathcal{M}$ so that \eqref{eq: C^0 estimates, structural condition} holds for any $(\bm{x}\comma t)\in\mathcal{M}\times\mathbb{R}$. We need to prove that
\begin{eq} \label{eq: general p-Hessian, C^0 estimates, conclusion}
\osc\limits_{\mathcal M} u\leqslant M\comma
\end{eq}
where $M$ is some positive constant depending only on $\mathcal{M}$, $\bm{g}$, $n$, $p$, $\bm{A}(\bm{\alpha}\comma t)$, $\varphi(\bm{\alpha}\comma t)$ and $\underline{u}$; if $\max\limits_{\mathcal{M}}u=0$ or $\max\limits_{\mathcal{M}}(u-\underline{u})=0$, the assumption \eqref{eq: C^0 estimates, structural condition} can be replaced with \eqref{eq: C^0 estimates, structural condition, weaker}.

We mainly follow the strategy of \cite[pp.~346--348]{Szekelyhidi2018} and \cite[pp.~5--7]{Sui2025}. First, we recall the classical weak Harnack inequality for $\mathrm{W}^{1\comma 2}$ weak supersolutions to quasi-linear elliptic equations of divergence form (cf. e.g. \cite{Gilbarg2001}).

\begin{lem} \label{lem: classical weak Harnack inequality}
Let $\Omega$ be a domain in $\mathbb{R}^n$ ($n\geqslant3$), $\bm{x}_0\in\Omega$, $d\in(0\comma{+}\infty)\colon\mathrm{B}_{4d}(\bm{x}_0)\subset\Omega$, $q\in(n\comma{+}\infty)$, $r\in\left[1\comma\frac{n}{n-2}\right)$, and $w\in\mathrm{W}^{1\comma2}(\Omega)$ is a $\mathrm{W}^{1\comma 2}$ weak supersolution, a.e. non-negative in $\mathrm{B}_{4d}(\bm{x}_0)$, to the equation
\begin{eq}
-\mathrm{D}_j\bigl(a^j(\cdot\comma\mathrm{D}w\comma w)\bigr)+b(\cdot\comma\mathrm{D}w\comma w)=0 \mbox{ in } \Omega\comma
\end{eq}
where $a^1(\bm{x}\comma\bm{v}\comma t)$, $a^2(\bm{x}\comma\bm{v}\comma t)$, $\cdots$, $a^n(\bm{x}\comma\bm{v}\comma t)$, $b(\bm{x}\comma\bm{v}\comma t)\in\mathrm{L}_{\mathrm{loc}}^1(\Omega\times\mathbb{R}^n\times\mathbb{R})$. Assume that there exist a positive constant $\varepsilon$ and a.e. non-negative functions $\xi\in\mathrm{L}_{\mathrm{loc}}^{\frac q2}(\Omega)$, $\xi_1$, $\xi_2\in\mathrm{L}_{\mathrm{loc}}^{\infty}(\Omega)$ so that for a.e. $(\bm{x}\comma\bm{v}\comma t)\in\mathrm{B}_{4d}(\bm{x}_0)\times\mathbb{R}^n\times\mathbb{R}$, there hold
\begin{al}
a^j(\bm{x}\comma\bm{v}\comma t)v_j&\geqslant\varepsilon\left|\bm{v}\right|^2-\xi(\bm{x})\bigl(\left|t\right|+\xi_1(\bm{x})\bigr)^2\comma \label{eq: classical weak Harnack inequality, a^j(x, v, t)v_j geqslant} \\
\sum_{j=1}^n\left|a^j(\bm{x}\comma\bm{v}\comma t)\right|^2&\leqslant\xi_2(\bm{x})\left|\bm{v}\right|^2+\xi(\bm{x})\bigl(\left|t\right|+\xi_1(\bm{x})\bigr)^2\comma \\
b(\bm{x}\comma\bm{v}\comma t)&\leqslant \xi^{\frac12}(\bm{x})\left|\bm{v}\right|+\xi(\bm{x})\bigl(\left|t\right|+\xi_1(\bm{x})\bigr). \label{eq: classical weak Harnack inequality, b(x, v, t) leqslant}
\end{al}
Then there exists a positive constant $M$ depending only on $d^{2-\frac{2n}{q}}\left\|\xi\right\|_{\mathrm{L}^{\frac q2}\left(\vphantom{B_1^1}\mathrm{B}_{3d}(\bm{x}_0)\right)}$, $\left\|\xi_2\right\|_{\mathrm{L}^{\infty}\left(\vphantom{B_1^1}\mathrm{B}_{3d}(\bm{x}_0)\right)}$, $n$, $q$, $r$ and $\frac{1}{\varepsilon}$ so that
\begin{eq}
d^{-\frac nr}\left\|w\right\|_{\mathrm{L}^r\left(\vphantom{B_1^1}\mathrm{B}_{2d}(\bm{x}_0)\right)}\leqslant M\left(\essinf_{\mathrm{B}_d(\bm{x}_0)}w+\left\|\xi_1\right\|_{\mathrm{L}^{\infty}\left(\vphantom{B_1^1}\mathrm{B}_d(\bm{x}_0)\right)}\right).
\end{eq}
\end{lem}

\begin{proof}
Refer to e.g. \cite[pp.~194--198]{Gilbarg2001}.
\end{proof}

Note that \lemref{lem: classical weak Harnack inequality} also holds for $n=2$ as long as $p\in[1\comma q)$. Now we use the idea in \cite[p.~347]{Szekelyhidi2018} to prove the $\mathrm{L}^1$ estimates for \eqref{eq: general p-Hessian}, recall \eqref{eq: dV_g} and \eqref{eq: connection}.

\begin{prop} \label{prop: L^1 estimates}
Assuming \eqref{eq: C^0 estimates, structural condition}, there exists a positive constant $M$ depending only on $\mathcal{M}$, $\bm g$, $n$, $\max\limits_{\mathcal{M}}\phi_1$ and $\max\limits_{\mathcal{M}}\phi_2$ so that
\begin{eq} \label{eq: L^1 estimates}
\int_{\mathcal M} \left(\max_{\mathcal M}u-u\right)\dif V_{\bm g}\leqslant M.
\end{eq}
\end{prop}

\begin{proof}
Since $(\mathcal{M}\comma\bm{g})$ is a closed connected Riemannian manifold, it's well-known that there exist a finite number of points $\bm{x}_1$, $\bm{x}_2$, $\cdots$, $\bm{x}_N\in\mathcal{M}$ ($N\in\mathbb{N}$), local coordinate systems
\begin{eq}
(\mathcal{U}_1\comma\bm{\psi}_{\mathcal{U}_1}\mbox{; }x_1^j)\comma(\mathcal{U}_2\comma\bm{\psi}_{\mathcal{U}_2}\mbox{; }x_2^j)\comma\cdots\comma(\mathcal{U}_N\comma\bm{\psi}_{\mathcal{U}_N}\mbox{; }x_N^j)\comma
\end{eq}
and a positive constant $d_0$ depending only on $\mathcal{M}$ and $\bm{g}$, so that
\begin{ga}
\bm{x}_\alpha\in\mathcal{U}_\alpha\comma\mathrm{B}_{4d_0}\bigl(\bm{\psi}_{\mathcal{U}_\alpha}(\bm{x}_\alpha)\bigr)\subset\bm{\psi}_{\mathcal{U}_\alpha}(\mathcal{U}_\alpha)\comma\forall\alpha\in\{1\comma2\comma\cdots\comma N\}\comma \label{eq: L^1 estimates, x_alpha in U_alpha, B_{4d_0}(psi_{U_alpha}(x_alpha))} \\
\bigcup_{\alpha=1}^N \mathcal{B}_{d_0}(\bm{x}_\alpha)=\mathcal{M}\ \biggl(\mathcal{B}_{d_0}(\bm{x}_\alpha)\triangleq\bm{\psi}_{\mathcal{U}_\alpha}^{-1}\Bigl(\mathrm{B}_{d_0}\bigl(\bm{\psi}_{\mathcal{U}_\alpha}(\bm{x}_\alpha)\bigr)\Bigr)\biggr)\comma \label{eq: L^1 estimates, bigcup_{alpha=1}^N B_{d_0}(x_alpha)=M}
\end{ga}
and
\begin{eq} \label{eq: L^1 estimates, (g_{x_alpha^jx_alpha^k}(y)), |Gamma_{x_alpha^jx_alpha^k}^{x_alpha^l}(y)|}
\frac12\mathbf{I}_n\leqslant\bigl(g_{x_\alpha^jx_\alpha^k}(\bm{y})\bigr)\leqslant2\mathbf{I}_n\comma\max_{1\leqslant j\comma k\comma l\leqslant n}\left|\Gamma_{x_\alpha^jx_\alpha^k}^{x_\alpha^l}(\bm{y})\right|\leqslant 1\comma\forall\bm{y}\in\mathcal{B}_{4d_0}(\bm{x}_\alpha).
\end{eq}
For any $\alpha\in\{1\comma2\comma\cdots\comma N\}$ and $\bm{y}\in\mathcal{B}_{4d_0}(\bm{x}_\alpha)$, since $\bm\lambda\Bigl({\bm g}^{-1}\bigl(\bm{A}(\dif u\comma u)+\nabla^2u\bigr)\Bigr)(\bm{y})\in\Gamma_p\subset\Gamma_1$, there holds
\begin{eq} \label{eq: L^1 estimates, 0<g^{x_alpha^jx_alpha^k}(y)A_{x_alpha^jx_alpha^k}(du(y), u(y))+g^{x_alpha^jx_alpha^k}(y)nabla_{x_alpha^jx_alpha^k}u(y)} \begin{aligned}
0&<g^{x_\alpha^jx_\alpha^k}(\bm{y})A_{x_\alpha^jx_\alpha^k}\bigl(\dif u(\bm{y})\comma u(\bm{y})\bigr)+g^{x_\alpha^jx_\alpha^k}(\bm{y})\nabla_{x_\alpha^jx_\alpha^k}u(\bm{y}) \\
&\leqslant\max_{\mathcal{M}}\phi_1\cdot\left|\dif u(\bm{y})\right|_{\bm g}+\max_{\mathcal{M}}\phi_2+\mathrm{D}_{x_\alpha^j}(g^{x_\alpha^jx_\alpha^k}\mathrm{D}_{x_\alpha^k}u)(\bm{y})+\Gamma_{x_\alpha^lx_\alpha^j}^{x_\alpha^j}(\bm{y})g^{x_\alpha^lx_\alpha^k}(\bm{y})\mathrm{D}_{x_\alpha^k}u(\bm{y})
\end{aligned} \end{eq}
in view of \eqref{eq: C^0 estimates, structural condition}, \eqref{eq: du and nabla^2u} and \eqref{eq: properties of connection}. Let $w$ denote the $\mathrm{C}^2$ function
\begin{eq} \label{eq: L^1 estimates, w}
\left(\max\limits_{\mathcal{M}}u-u\right)\circ\bm{\psi}_{\mathcal{U}_\alpha}^{-1}
\end{eq}
in $\mathrm{B}_{4d_0}\bigl(\bm{\psi}_{\mathcal{U}_\alpha}(\bm{x}_\alpha)\bigr)$. Then we have $w\geqslant 0$ and
\begin{eq}
-\mathrm{D}_j\bigl(a^j(\cdot\comma\mathrm{D}w\comma w)\bigr)\bigl(\bm{\psi}_{\mathcal{U}_\alpha}(\bm{y})\bigr)+b\Bigl(\bm{\psi}_{\mathcal{U}_\alpha}(\bm{y})\comma\mathrm{D}w\bigl(\bm{\psi}_{\mathcal{U}_\alpha}(\bm{y})\bigr)\comma w\bigl(\bm{\psi}_{\mathcal{U}_\alpha}(\bm{y})\bigr)\Bigr)>0\comma
\end{eq}
where
\begin{ga}
a^j(\bm{x}\comma\bm{v}\comma t)\triangleq g^{x_\alpha^jx_\alpha^k}\bigl(\bm{\psi}_{\mathcal{U}_\alpha}^{-1}(\bm{x})\bigr)v_k\in\mathrm{C}^\infty\Bigl(\mathrm{B}_{4d_0}\bigl(\bm{\psi}_{\mathcal{U}_\alpha}(\bm{x}_\alpha)\bigr)\times\mathbb{R}^n\times\mathbb{R}\Bigr)\comma \label{eq: L^1 estimates, a^j(x, v, t)} \\
\begin{aligned}
b(\bm{x}\comma\bm{v}\comma t)&\triangleq-\Gamma_{x_\alpha^lx_\alpha^j}^{x_\alpha^j}\bigl(\bm{\psi}_{\mathcal{U}_\alpha}^{-1}(\bm{x})\bigr)g^{x_\alpha^lx_\alpha^k}\bigl(\bm{\psi}_{\mathcal{U}_\alpha}^{-1}(\bm{x})\bigr)v_k+\max_{\mathcal{M}}\phi_1\cdot\left(g^{x_\alpha^jx_\alpha^k}\bigl(\bm{\psi}_{\mathcal{U}_\alpha}^{-1}(\bm{x})\bigr)v_jv_k\right)^{\frac12} \\
&\quad+\max_{\mathcal{M}}\phi_2\in\mathrm{C}\Bigl(\mathrm{B}_{4d_0}\bigl(\bm{\psi}_{\mathcal{U}_\alpha}(\bm{x}_\alpha)\bigr)\times\mathbb{R}^n\times\mathbb{R}\Bigr).
\end{aligned}
\end{ga}
Thus, recalling \eqref{eq: L^1 estimates, (g_{x_alpha^jx_alpha^k}(y)), |Gamma_{x_alpha^jx_alpha^k}^{x_alpha^l}(y)|}, $w$ is a non-negative (classical) supersolution to the equation
\begin{eq}
-\mathrm{D}_j\bigl(a^j(\cdot\comma\mathrm{D}w\comma w)\bigr)+b(\cdot\comma\mathrm{D}w\comma w)=0 \mbox{ in } \mathrm{B}_{4d_0}\bigl(\bm{\psi}_{\mathcal{U}_\alpha}(\bm{x}_\alpha)\bigr)\comma
\end{eq}
where
\begin{ga}
a^j(\bm{x}\comma\bm{v}\comma t)v_j\geqslant\frac12\left|\bm{v}\right|^2\comma\quad \sum_{j=1}^n\left|a^j(\bm{x}\comma\bm{v}\comma t)\right|^2\leqslant4n\left|\bm{v}\right|^2\comma \label{eq: L^1 estimates, a^j(x, v, t), properties} \\
b(\bm{x}\comma\bm{v}\comma t)\leqslant\left(2n^2+\sqrt{2}\max_{\mathcal{M}}\phi_1\right)\left|\bm{v}\right|+\max_{\mathcal{M}}\phi_2. 
\end{ga}
It follows from \lemref{lem: classical weak Harnack inequality} that
\begin{eq} \label{eq: L^1 estimates, classical weak Harnack inequality}
\int_{\mathrm{B}_{2d_0}\left(\vphantom{B_1^1}\bm{\psi}_{\mathcal{U}_\alpha}(\bm{x}_\alpha)\right)} w(\bm{x})\dif\bm{x}\leqslant C_1\left(\inf_{\mathrm{B}_{d_0}\left(\vphantom{B_1^1}\bm{\psi}_{\mathcal{U}_\alpha}(\bm{x}_\alpha)\right)}w+\max_{\mathcal{M}}\phi_2\right).
\end{eq}
Here $C_1$ is a positive constant depending only on $\max\limits_{\mathcal{M}}\phi_1$, $n$ and $d_0$.

\eqref{eq: L^1 estimates, bigcup_{alpha=1}^N B_{d_0}(x_alpha)=M}, \eqref{eq: L^1 estimates, w} and \eqref{eq: L^1 estimates, classical weak Harnack inequality} imply for any $\alpha\in\{1\comma2\comma\cdots\comma N\}$ that
\begin{eq} \label{eq: L^1 estimates, key estimate} \begin{aligned}
\int_{\mathcal{B}_{2d_0}(\bm{x}_\alpha)} \left(\max_{\mathcal M}u-u\right)\dif V_{\bm g}&=\int_{\mathrm{B}_{2d_0}\left(\vphantom{B_1^1}\bm{\psi}_{\mathcal{U}_\alpha}(\bm{x}_\alpha)\right)} w(\bm{x})\biggl(\det\Bigl(g_{x_\alpha^jx_\alpha^k}\bigl(\bm{\psi}_{\mathcal{U}_\alpha}^{-1}(\bm{x})\bigr)\Bigr)\biggr)^{\frac12}\dif\bm{x} \\
&\leqslant 2^{\frac n2}C_1\left(\max_{\mathcal M}u-\sup_{\mathcal{B}_{d_0}(\bm{x}_\alpha)}u+\max_{\mathcal{M}}\phi_2\right).
\end{aligned} \end{eq}
Recall \eqref{eq: L^1 estimates, bigcup_{alpha=1}^N B_{d_0}(x_alpha)=M}. We may as well assume that $\max\limits_{\mathcal{M}}u=u(\bm{x}_1)$, and therefore
\begin{eq}
\int_{\mathcal{B}_{2d_0}(\bm{x}_1)} \left(\max_{\mathcal M}u-u\right)\dif V_{\bm g}\leqslant 2^{\frac n2}C_1\max_{\mathcal{M}}\phi_2\triangleq C_2.
\end{eq}
Since $\mathcal{M}$ is connected, there exists such $\alpha_0\in\{2\comma3\comma\cdots\comma N\}$ that $\mathcal{B}_{d_0}(\bm{x}_{\alpha_0})\cap\mathcal{B}_{d_0}(\bm{x}_1)\ne\varnothing$. Note that
\begin{eq} \label{eq: L^1 estimates, 1}
\int_{\mathcal{B}_{d_0}(\bm{x}_{\alpha_0})\cap\mathcal{B}_{2d_0}(\bm{x}_1)} \left(\max_{\mathcal M}u-u\right)\dif V_{\bm g}\leqslant C_2\comma\quad \int_{\mathcal{B}_{d_0}(\bm{x}_{\alpha_0})\cap\mathcal{B}_{2d_0}(\bm{x}_1)} \dif V_{\bm g}>0.
\end{eq}
Combining \eqref{eq: L^1 estimates, key estimate} and \eqref{eq: L^1 estimates, 1}, we obtain
\begin{eq} \begin{aligned}
\int_{\mathcal{B}_{2d_0}(\bm{x}_{\alpha_0})} \left(\max_{\mathcal M}u-u\right)\dif V_{\bm g}&\leqslant 2^{\frac n2}C_1\left(\max_{\mathcal M}u-\sup_{\mathcal{B}_{d_0}(\bm{x}_{\alpha_0})}u+\max_{\mathcal{M}}\phi_2\right) \\
&\leqslant 2^{\frac n2}C_1\left(\frac{C_2}{\int_{\mathcal{B}_{d_0}(\bm{x}_{\alpha_0})\cap\mathcal{B}_{2d_0}(\bm{x}_1)} \dif V_{\bm g}}+\max_{\mathcal{M}}\phi_2\right).
\end{aligned} \end{eq}
By induction, it's easy to prove \eqref{eq: L^1 estimates}.
\end{proof}

\begin{rmk} \label{rmk: L^1 estimates}
In \propref{prop: L^1 estimates}, if $\max\limits_{\mathcal{M}}u=0$, then the assumption \eqref{eq: C^0 estimates, structural condition} can be replaced with \eqref{eq: C^0 estimates, structural condition, weaker}, refer to the proof of \propref{prop: L^1 estimates, max_M(u-underline{u})=0} below.
\end{rmk}

\begin{prop} \label{prop: L^1 estimates, max_M(u-underline{u})=0}
Assuming \eqref{eq: C^0 estimates, structural condition, weaker} instead of \eqref{eq: C^0 estimates, structural condition}, if $\max\limits_{\mathcal{M}}(u-\underline{u})=0$, then there exists a positive constant $M$ depending only on $\mathcal{M}$, $\bm g$, $n$, $\max\limits_{\mathcal{M}}\phi_1$, $\max\limits_{\mathcal{M}}\phi_2$ and $\max\limits_{\mathcal{M}}\left|\underline{u}\right|$ so that
\begin{eq} \label{eq: L^1 estimates, max_M(u-underline{u})=0}
\int_{\mathcal M} \left(\max_{\mathcal M}u-u\right)\dif V_{\bm g}\leqslant M.
\end{eq}
\end{prop}

\begin{proof}
The proof is almost the same as that of \propref{prop: L^1 estimates}, so the details are omitted. Recall the settings \eqref{eq: L^1 estimates, x_alpha in U_alpha, B_{4d_0}(psi_{U_alpha}(x_alpha))}--\eqref{eq: L^1 estimates, (g_{x_alpha^jx_alpha^k}(y)), |Gamma_{x_alpha^jx_alpha^k}^{x_alpha^l}(y)|}. For any $\alpha\in\{1\comma2\comma\cdots\comma N\}$ and $\bm{y}\in\mathcal{B}_{4d_0}(\bm{x}_\alpha)$, by \eqref{eq: C^0 estimates, structural condition, weaker} there holds
\begin{eq}
0<\max_{\mathcal{M}}\phi_1\cdot\left|\dif u(\bm{y})\right|_{\bm g}+\max_{\mathcal{M}}\phi_2\cdot\bigl(\left|u(\bm{y})\right|+1\bigr)+\mathrm{D}_{x_\alpha^j}(g^{x_\alpha^jx_\alpha^k}\mathrm{D}_{x_\alpha^k}u)(\bm{y})+\Gamma_{x_\alpha^lx_\alpha^j}^{x_\alpha^j}(\bm{y})g^{x_\alpha^lx_\alpha^k}(\bm{y})\mathrm{D}_{x_\alpha^k}u(\bm{y})
\end{eq}
analogous to \eqref{eq: L^1 estimates, 0<g^{x_alpha^jx_alpha^k}(y)A_{x_alpha^jx_alpha^k}(du(y), u(y))+g^{x_alpha^jx_alpha^k}(y)nabla_{x_alpha^jx_alpha^k}u(y)}. Note that $\max\limits_{\mathcal{M}}(u-\underline{u})=0$ implies
\begin{eq}
\left|\max_{\mathcal{M}}u\right|\leqslant\max_{\mathcal{M}}\left|\underline{u}\right|.
\end{eq}
Let $w$ denote the non-negative $\mathrm{C}^2$ function \eqref{eq: L^1 estimates, w}. Then we have
\begin{eq}
-\mathrm{D}_j\bigl(a^j(\cdot\comma\mathrm{D}w\comma w)\bigr)\bigl(\bm{\psi}_{\mathcal{U}_\alpha}(\bm{y})\bigr)+\tilde{b}\Bigl(\bm{\psi}_{\mathcal{U}_\alpha}(\bm{y})\comma\mathrm{D}w\bigl(\bm{\psi}_{\mathcal{U}_\alpha}(\bm{y})\bigr)\comma w\bigl(\bm{\psi}_{\mathcal{U}_\alpha}(\bm{y})\bigr)\Bigr)>0\comma
\end{eq}
where $a^j(\bm{x}\comma\bm{v}\comma t)$ is defined as \eqref{eq: L^1 estimates, a^j(x, v, t)}, satisfying \eqref{eq: L^1 estimates, a^j(x, v, t), properties}, and
\begin{eq} \begin{aligned}
\tilde{b}(\bm{x}\comma\bm{v}\comma t)&\triangleq-\Gamma_{x_\alpha^lx_\alpha^j}^{x_\alpha^j}\bigl(\bm{\psi}_{\mathcal{U}_\alpha}^{-1}(\bm{x})\bigr)g^{x_\alpha^lx_\alpha^k}\bigl(\bm{\psi}_{\mathcal{U}_\alpha}^{-1}(\bm{x})\bigr)v_k+\max_{\mathcal{M}}\phi_1\cdot\left(g^{x_\alpha^jx_\alpha^k}\bigl(\bm{\psi}_{\mathcal{U}_\alpha}^{-1}(\bm{x})\bigr)v_jv_k\right)^{\frac12} \\
&\quad+\max_{\mathcal{M}}\phi_2\cdot\left(\left|t\right|+\max_{\mathcal{M}}\left|\underline{u}\right|+1\right) \\
&\leqslant\left(2n^2+\sqrt{2}\max_{\mathcal{M}}\phi_1\right)\left|\bm{v}\right|+\max_{\mathcal{M}}\phi_2\cdot\left(\left|t\right|+\max_{\mathcal{M}}\left|\underline{u}\right|+1\right).
\end{aligned} \end{eq}
By \lemref{lem: classical weak Harnack inequality} we obtain
\begin{eq}
\int_{\mathrm{B}_{2d_0}\left(\vphantom{B_1^1}\bm{\psi}_{\mathcal{U}_\alpha}(\bm{x}_\alpha)\right)} w(\bm{x})\dif\bm{x}\leqslant C_3\left(\inf_{\mathrm{B}_{d_0}\left(\vphantom{B_1^1}\bm{\psi}_{\mathcal{U}_\alpha}(\bm{x}_\alpha)\right)}w+\max_{\mathcal{M}}\left|\underline{u}\right|+1\right)\comma
\end{eq}
analogous to \eqref{eq: L^1 estimates, classical weak Harnack inequality}. Here $C_3$ is a positive constant depending only on $\max\limits_{\mathcal{M}}\phi_1$, $\max\limits_{\mathcal{M}}\phi_2$, $n$ and $d_0$. Now it's easy to prove \eqref{eq: L^1 estimates, max_M(u-underline{u})=0} following the latter half of the proof of \propref{prop: L^1 estimates}.
\end{proof}

It's well-known that the Alexandrov maximum principle (cf. e.g. \cite[p.~220]{Gilbarg2001}) is useful for proving the $\mathrm{C}^0$ estimates for fully non-linear elliptic equations, refer to e.g. \cite{Guo2023,Chen2021a} (see also e.g. \cite{Guo2022,Guo2023b,Guedj2025,Qiao2025} for $\mathrm{C}^0$ estimates bypassing the Alexandrov maximum principle). The following proposition can be viewed as a generalized Alexandrov lemma (cf. e.g. \cite[p.~221]{Gilbarg2001}), which is similar to that in \cite[p.~346]{Szekelyhidi2018}.

\begin{prop} \label{prop: generalized Alexandrov}
Assume that $\Omega$ is a bounded domain in $\mathbb{R}^n$, $w\in\mathrm{C}^2(\Omega)\cap\mathrm{C}(\overline{\Omega})$, $\bm{x}_0\in\Omega$ and $\varepsilon\in\left(0\comma\min\limits_{\partial\Omega}w-w(\bm{x}_0)\right]$. Let $d\triangleq\max\limits_{\bm{x}\in\partial\Omega}\left|\bm{x}-\bm{x}_0\right|$, $\upomega_n$ denote the volume of $\mathrm{B}_1(\bm{0})$ and $\mathcal{P}$ denote the set
\begin{eq}
\left\{\bm{x}\in\Omega\middle|\left|\mathrm{D}w(\bm{x})\right|<\frac{\varepsilon}{d}\comma \mbox{and } w(\bm{y})\geqslant w(\bm{x})+\sum_{j=1}^n\mathrm{D}_jw(\bm{x})(y_j-x_j) \mbox{ for any } \bm{y}\in\Omega\right\}.
\end{eq}
Then there holds
\begin{eq} \label{eq: generalized Alexandrov, conclusion}
\frac{\upomega_n}{d^n}\varepsilon^n\leqslant\int_{\mathcal{P}} \det\bigl(\mathrm{D}^2w(\bm{x})\bigr)\dif\bm{x}.
\end{eq}
\end{prop}

\begin{proof}
Fix $\bm{v}\in\mathbb{R}^n\colon\left|\bm{v}\right|<\frac{\varepsilon}{d}$ and define a linear function
\begin{eq}
\ell(\bm{x})\triangleq w(\bm{x}_0)+\sum_{j=1}^n v_j(x_j-x_{0\comma j})\comma\bm{x}\in\mathbb{R}^n.
\end{eq}
Obviously, $\ell(\bm{x}_0)=w(\bm{x}_0)$ and $\max\limits_{\partial\Omega}\ell<w(\bm{x}_0)+\varepsilon\leqslant\min\limits_{\partial\Omega}w$. Thus, there exists such $\bm{x}_1\in\Omega$ that
\begin{eq}
\min_{\overline{\Omega}}(w-\ell)=w(\bm{x}_1)-\ell(\bm{x}_1)\leqslant 0.
\end{eq}
This implies $\mathrm{D}w(\bm{x}_1)=\mathrm{D}\ell(\bm{x}_1)=\bm{v}$ and
\begin{eq}
w(\bm{y})-\ell(\bm{y})\geqslant w(\bm{x}_1)-\ell(\bm{x}_1)\comma \mbox{i.e. } w(\bm{y})\geqslant w(\bm{x}_1)+\sum_{j=1}^n v_j(y_j-x_{1\comma j})
\end{eq}
for any $\bm{y}\in\Omega$. It follows that $\bm{x}_1\in\mathcal{P}$ and $\bm{v}\in\mathrm{D}w(\mathcal{P})$, therefore $\mathrm{B}_{\frac{\varepsilon}{d}}(\bm{0})\subset\mathrm{D}w(\mathcal{P})$. Now since $\mathrm{D}^2w(\bm{x})$ is positively semi-definite for any $\bm{x}\in\mathcal{P}$, it's easy to prove \eqref{eq: generalized Alexandrov, conclusion} (cf. \cite[p.~221]{Gilbarg2001}).
\end{proof}

So far, we have been prepared for the proof of \thmref{thm: general p-Hessian, C^0 estimates}. Now we start to prove it, applying the ideas in \cite{Szekelyhidi2018,Sui2025}.

\begin{proof}
Let $v\triangleq u-\underline{u}$ and assume that $\min\limits_{\mathcal{M}}v=v(\bm{x}_0)$ ($\bm{x}_0\in\mathcal{M}$). Similar to the setting of the proof of \propref{prop: L^1 estimates}, there exist a local coordinate system $(\mathcal{U}\comma\bm{\psi}_{\mathcal{U}}\mbox{; }x^j)$ containing $\bm{x}_0$ and a positive constant $d_0$ depending only on $\mathcal{M}$ and $\bm{g}$, so that
\begin{ga}
\mathrm{B}_{2d_0}\bigl(\bm{\psi}_{\mathcal{U}}(\bm{x}_0)\bigr)\subset\bm{\psi}_{\mathcal{U}}(\mathcal{U})\comma\quad \mathcal{B}_{2d_0}(\bm{x}_0)\triangleq\bm{\psi}_{\mathcal{U}}^{-1}\Bigl(\mathrm{B}_{2d_0}\bigl(\bm{\psi}_{\mathcal{U}}(\bm{x}_0)\bigr)\Bigr)\subset\mathcal{U}\comma \label{eq: C^0 estimates, B_{2d_0}(x_0)} \\
\frac12\mathbf{I}_n\leqslant\bigl(g_{jk}(\bm{y})\bigr)\leqslant2\mathbf{I}_n\comma\sum_{1\leqslant j\comma k\comma l\leqslant n} |\Gamma_{jk}^l(\bm{y})|^2\leqslant 1\comma\forall\bm{y}\in\mathcal{B}_{2d_0}(\bm{x}_0). \label{eq: C^0 estimates, (g_{jk}(y)), sum_{1 leqslant j, k, l leqslant n} |Gamma_{jk}^l(y)|^2}
\end{ga}
Fix $\varepsilon\in(0\comma{+}\infty)$ to be determined later, and define the comparison function
\begin{eq} \label{eq: C^0 estimates, comparison function}
w(\bm{x})\triangleq v\bigl(\bm{\psi}_{\mathcal{U}}^{-1}(\bm{x})\bigr)+\varepsilon\left|\bm{x}-\bm{\psi}_{\mathcal{U}}(\bm{x}_0)\right|^2\comma\bm{x}\in\mathrm{B}_{2d_0}\bigl(\bm{\psi}_{\mathcal{U}}(\bm{x}_0)\bigr).
\end{eq}
Then there holds
\begin{eq}
\min_{\partial\mathrm{B}_{d_0}\left(\vphantom{B_1^2}\bm{\psi}_{\mathcal{U}}(\bm{x}_0)\right)}w-w\bigl(\bm{\psi}_{\mathcal{U}}(\bm{x}_0)\bigr)\geqslant\varepsilon d_0^2.
\end{eq}
Let $\Omega\triangleq\mathrm{B}_{d_0}\bigl(\bm{\psi}_{\mathcal{U}}(\bm{x}_0)\bigr)$, $\upomega_n$ denote the volume of $\mathrm{B}_1(\bm{0})$ and $\mathcal{P}$ denote the set
\begin{eq} \label{eq: C^0 estimates, P}
\left\{\bm{x}\in\Omega\middle|\left|\mathrm{D}w(\bm{x})\right|<d_0\varepsilon\comma \mbox{and } w(\bm{y})\geqslant w(\bm{x})+\sum_{j=1}^n\mathrm{D}_jw(\bm{x})(y_j-x_j) \mbox{ for any } \bm{y}\in\Omega\right\}.
\end{eq}
\propref{prop: generalized Alexandrov} implies
\begin{eq} \label{eq: C^0 estimates, generalized Alexandrov}
\upomega_nd_0^n\varepsilon^n\leqslant\int_{\mathcal{P}} \det\bigl(\mathrm{D}^2w(\bm{x})\bigr)\dif\bm{x}.
\end{eq}
By \eqref{eq: C^0 estimates, comparison function} and \eqref{eq: du and nabla^2u} we have
\begin{ga}
\mathrm{D}_jw(\bm{x})=\mathrm{D}_jv\bigl(\bm{\psi}_{\mathcal{U}}^{-1}(\bm{x})\bigr)+2\varepsilon\bigl(x_j-x^j(\bm{x}_0)\bigr)\comma \\
\mathrm{D}_{jk}w(\bm{x})=\nabla_{jk}v\bigl(\bm{\psi}_{\mathcal{U}}^{-1}(\bm{x})\bigr)+\Gamma_{jk}^l\bigl(\bm{\psi}_{\mathcal{U}}^{-1}(\bm{x})\bigr)\mathrm{D}_lv\bigl(\bm{\psi}_{\mathcal{U}}^{-1}(\bm{x})\bigr)+2\varepsilon\updelta_{jk}.
\end{ga}
For any $\bm{x}\in\mathcal{P}$, we find that
\begin{ga}
v\bigl(\bm{\psi}_{\mathcal{U}}^{-1}(\bm{x})\bigr)\leqslant w(\bm{x})\leqslant w\bigl(\bm{\psi}_{\mathcal{U}}(\bm{x}_0)\bigr)+\left|\mathrm{D}w(\bm{x})\right|\left|\bm{x}-\bm{\psi}_{\mathcal{U}}(\bm{x}_0)\right|<\min_{\mathcal{M}}v+d_0^2\varepsilon\comma \label{eq: C^0 estimates, v(psi_U^{-1}(x)) leqslant} \\
\left(\sum_{j=1}^n\bigl|\mathrm{D}_jv\bigl(\bm{\psi}_{\mathcal{U}}^{-1}(\bm{x})\bigr)\bigr|^2\right)^{\frac12}\leqslant\left|\mathrm{D}w(\bm{x})\right|+2\varepsilon\left|\bm{x}-\bm{\psi}_{\mathcal{U}}(\bm{x}_0)\right|<3d_0\varepsilon\comma
\end{ga}
and therefore
\begin{eq} \label{eq: C^0 estimates, D^2w(x) leqslant}
\mathbf{O}\leqslant\mathrm{D}^2w(\bm{x})\leqslant\Bigl(\nabla_{jk}v\bigl(\bm{\psi}_{\mathcal{U}}^{-1}(\bm{x})\bigr)\Bigr)+(3d_0+2)\varepsilon\mathbf{I}_n\leqslant\Bigl(\nabla_{jk}v\bigl(\bm{\psi}_{\mathcal{U}}^{-1}(\bm{x})\bigr)\Bigr)+2(3d_0+2)\varepsilon\Bigl(g_{jk}\bigl(\bm{\psi}_{\mathcal{U}}^{-1}(\bm{x})\bigr)\Bigr)
\end{eq}
in view of \eqref{eq: C^0 estimates, (g_{jk}(y)), sum_{1 leqslant j, k, l leqslant n} |Gamma_{jk}^l(y)|^2}. Recall \eqref{eq: moduli of du and nabla^2u}, \eqref{eq: lambda(g^{-1}(A(du, u)+nabla^2u))} and $v=u-\underline{u}$. It follows that
\begin{ga}
\bigl|\dif\left(u-\underline{u}\right)\bigl(\bm{\psi}_{\mathcal{U}}^{-1}(\bm{x})\bigr)\bigr|_{\bm g}\leqslant3\sqrt{2}d_0\varepsilon\comma \\
\bm\lambda\Bigl({\bm g}^{-1}\bigl(\nabla^2(u-\underline{u})\bigr)\Bigr)\bigl(\bm{\psi}_{\mathcal{U}}^{-1}(\bm{x})\bigr)+2(3d_0+2)\varepsilon\bm{1}_n\in\overline{\Gamma}_n.
\end{ga}
Now let $\varepsilon\triangleq\frac{\delta_2}{2(3d_0+2)}$. Since $\underline{u}$ satisfies pseudo-supersolution condition, recalling \eqref{eq: general p-Hessian, pseudo-supersolution condition}, there holds
\begin{eq} \label{eq: C^0 estimates, pseudo-supersolution condition}
\sigma_n\biggl(\bm\lambda\Bigl({\bm g}^{-1}\bigl(\delta_2\bm{g}+\nabla^2(u-\underline{u})\bigr)\Bigr)\bigl(\bm{\psi}_{\mathcal{U}}^{-1}(\bm{x})\bigr)\biggr)\leqslant M_2\comma
\end{eq}
where $\delta_2>0$ and $M_2\geqslant 0$ are both constants depending only on some trivial quantities. Combining \eqref{eq: C^0 estimates, generalized Alexandrov}, \eqref{eq: C^0 estimates, D^2w(x) leqslant}, \eqref{eq: C^0 estimates, (g_{jk}(y)), sum_{1 leqslant j, k, l leqslant n} |Gamma_{jk}^l(y)|^2} and \eqref{eq: C^0 estimates, pseudo-supersolution condition}, we obtain
\begin{eq} \label{eq: C^0 estimates, key estimate}
\frac{\upomega_nd_0^n\delta_2^n}{2^n(3d_0+2)^n}\leqslant\int_{\mathcal{P}} 2^n\det\biggl(\Bigl(g^{jk}\bigl(\bm{\psi}_{\mathcal{U}}^{-1}(\bm{x})\bigr)\Bigr)\Bigl(\nabla_{jk}v\bigl(\bm{\psi}_{\mathcal{U}}^{-1}(\bm{x})\bigr)+\delta_2g_{jk}\bigl(\bm{\psi}_{\mathcal{U}}^{-1}(\bm{x})\bigr)\Bigr)\biggr)\dif\bm{x}\leqslant 2^nM_2\int_{\mathcal{P}} \dif\bm{x}.
\end{eq}
By \eqref{eq: C^0 estimates, v(psi_U^{-1}(x)) leqslant} and \eqref{eq: C^0 estimates, P} we have
\begin{eq} \label{eq: C^0 estimates, osc_Mv int_P dx leqslant} \begin{aligned}
\osc_{\mathcal{M}}v\int_{\mathcal{P}} \dif\bm{x}&\leqslant\int_{\mathcal{P}}\left(\max_{\mathcal{M}}v-v\bigl(\bm{\psi}_{\mathcal{U}}^{-1}(\bm{x})\bigr)+\frac{d_0^2\delta_2}{2(3d_0+2)}\right)\dif\bm{x} \\
&\leqslant\int_{\mathrm{B}_{d_0}\left(\vphantom{B_1^2}\bm{\psi}_{\mathcal{U}}(\bm{x}_0)\right)} \left(\max_{\mathcal{M}}v-v\bigl(\bm{\psi}_{\mathcal{U}}^{-1}(\bm{x})\bigr)\right)\dif\bm{x}+\frac{\upomega_nd_0^{n+2}\delta_2}{2(3d_0+2)}.
\end{aligned} \end{eq}
On the other hand, by \propref{prop: L^1 estimates} (see also \rmkref{rmk: L^1 estimates}, \propref{prop: L^1 estimates, max_M(u-underline{u})=0}), \eqref{eq: C^0 estimates, B_{2d_0}(x_0)}, \eqref{eq: dV_g} and \eqref{eq: C^0 estimates, (g_{jk}(y)), sum_{1 leqslant j, k, l leqslant n} |Gamma_{jk}^l(y)|^2} we have
\begin{eq} \label{eq: C^0 estimates, L^1 estimates}
C_4\geqslant\int_{\mathcal{B}_{d_0}(\bm{x}_0)} \left(\max_{\mathcal M}v-v\right)\dif V_{\bm g}\geqslant2^{-\frac n2}\int_{\mathrm{B}_{d_0}\left(\vphantom{B_1^2}\bm{\psi}_{\mathcal{U}}(\bm{x}_0)\right)} \left(\max_{\mathcal{M}}v-v\bigl(\bm{\psi}_{\mathcal{U}}^{-1}(\bm{x})\bigr)\right)\dif\bm{x}.
\end{eq}
Combining \eqref{eq: C^0 estimates, key estimate}, \eqref{eq: C^0 estimates, osc_Mv int_P dx leqslant} and \eqref{eq: C^0 estimates, L^1 estimates}, we obtain $\osc\limits_{\mathcal{M}}v\leqslant C_5$ and therefore \eqref{eq: general p-Hessian, C^0 estimates, conclusion}.
\end{proof}


\section{First-order estimates for general $p$-Hessian equations} \label{sec: First-order estimates for general p-Hessian equations}

In this section, we prove \thmref{thm: general p-Hessian, first-order estimates}. Let $(\mathcal{M}\comma\bm g)$ be a closed connected Riemannian manifold of dimension $n$. Throughout this section, we assume that $n\geqslant 3$, $p\in\{2\comma3\comma\cdots\comma n\}$ and $u\in\mathrm{C}^3(\mathcal M)$ satisfies the general $p$-Hessian equation \eqref{eq: general p-Hessian} on $\mathcal{M}$. Recall \eqref{eq: A(alpha, t)}, \eqref{eq: du and nabla^2u}, \eqref{eq: lambda(g^{-1}(A(du, u)+nabla^2u))}, \defnref{defn: lambda_q}, \defnref{defn: lambda}, \eqref{eq: sigma_p}, \defnref{defn: Garding cone}, \eqref{eq: moduli of du and nabla^2u} and that $\varphi(\bm{\alpha}\comma t)$ is a positive smooth function on $\mathrm{T}^*\mathcal{M}\times\mathbb{R}$. We assume additionally that there exist non-negative continuous functions $\phi_3$, $\phi_4$, $\cdots$, $\phi_8$ on $\mathcal{M}\times\mathbb{R}$ and
\begin{eq} \label{eq: first-order estimates, omega_j}
\omega_1\comma\omega_2\comma\cdots\comma\omega_5\in\left\{\omega\in\mathrm{C}\bigl([0\comma{+}\infty)\comma[0\comma{+}\infty)\bigr)\middle|\lim_{t\to{+}\infty}\frac{\omega(t)}{t}=0\comma\sup_{t\in[0\comma{+}\infty)}\bigl(\omega(t)-t\bigr)\leqslant1\right\}
\end{eq}
so that \eqref{eq: first-order estimates, structural conditions, A_{vv}(alpha, t) leqslant phi_3(x, t)g(v, v)omega_1(|alpha|_g^2)}--\eqref{eq: first-order estimates, structural conditions, D^2 varphi(alpha, t) leqslant phi_8(x, t)(|alpha|_g^2+1)max{1, frac{omega_5(|alpha|_g^{2(p-1)})}{varphi^{p-1}(alpha, t)}}} hold for any $(\bm{x}\comma t)\in\mathcal{M}\times\mathbb{R}$ and $(\bm\alpha\comma\bm{v})\in\mathrm{T}_{\bm x}^*\mathcal{M}\times\mathrm{T}_{\bm x}\mathcal{M}$. We need to prove that
\begin{eq} \label{eq: general p-Hessian, first-order estimates, conclusion}
\max\limits_{\mathcal M} \left|\dif u\right|_{\bm g}\leqslant M\comma
\end{eq}
where $M$ is some positive constant depending only on $\mathcal{M}$, $\bm{g}$, $n$, $p$, $\bm{A}(\bm{\alpha}\comma t)$, $\varphi(\bm{\alpha}\comma t)$ and $\max\limits_{\mathcal{M}}\left|u\right|$; if $p=n$, the assumptions \eqref{eq: first-order estimates, structural conditions, A_{vw}(alpha, t) leqslant phi_4(x, t)(g(v, v)g(w, w))^{frac 12}omega_2(|alpha|_g^2)}--\eqref{eq: first-order estimates, structural conditions, D^2 varphi(alpha, t) leqslant phi_8(x, t)(|alpha|_g^2+1)max{1, frac{omega_5(|alpha|_g^{2(p-1)})}{varphi^{p-1}(alpha, t)}}} can be omitted and \eqref{eq: first-order estimates, structural conditions, A_{vv}(alpha, t) leqslant phi_3(x, t)g(v, v)omega_1(|alpha|_g^2)} can be replaced with \eqref{eq: first-order estimates, structural conditions, A_{vv}(alpha, t) leqslant phi_3(x, t)g(v, v)(|alpha|_g^2+1)}. The results of this section include \thmref{thm: general p-admissible solutions, first-order estimates} and \thmref{thm: semi-convex, first-order estimates}.

First, we prove the case of general $p$. One can refer to \cite{Jiang2018,Jiang2020,Guan2016,Sui2025} for similar results. Note that our assumptions \eqref{eq: first-order estimates, structural conditions, D^2A_{vv}(alpha, t) geqslant -phi_6(x, t)g(v, v)(|alpha|_g^2+1)}, \eqref{eq: first-order estimates, structural conditions, D^2 varphi(alpha, t) leqslant phi_8(x, t)(|alpha|_g^2+1)max{1, frac{omega_5(|alpha|_g^{2(p-1)})}{varphi^{p-1}(alpha, t)}}} are weaker than (1.16) or (1.14) in \cite[p.~5207]{Jiang2020}.

\begin{thm} \label{thm: general p-admissible solutions, first-order estimates}
Assuming \eqref{eq: first-order estimates, structural conditions, A_{vv}(alpha, t) leqslant phi_3(x, t)g(v, v)omega_1(|alpha|_g^2)}--\eqref{eq: first-order estimates, structural conditions, D^2 varphi(alpha, t) leqslant phi_8(x, t)(|alpha|_g^2+1)max{1, frac{omega_5(|alpha|_g^{2(p-1)})}{varphi^{p-1}(alpha, t)}}}, \eqref{eq: general p-Hessian, first-order estimates, conclusion} holds.
\end{thm}

\begin{proof}
Let $\bm{B}$ denote the symmetric $(0\comma2)$-tensor field $\bm{A}(\dif u\comma u)+\nabla^2u$ on $\mathcal{M}$, and recall \eqref{eq: second-order estimates, equation}. Define the auxiliary function
\begin{eq} \label{eq: first-order estimates, Phi}
\Phi(\bm{x})\triangleq\log\bigl(1+\left|\dif u(\bm{x})\right|_{\bm g}^2\bigr)+\zeta\bigl(u(\bm{x})\bigr)\comma\bm{x}\in\mathcal{M}\comma
\end{eq}
where $\zeta$ is a smooth function on $\mathbb{R}$ to be determined later so that
\begin{eq}
\zeta'(t)>0\comma\zeta''(t)\geqslant 0\comma\forall t\in\left[\min_{\mathcal{M}}u\comma\max_{\mathcal{M}}u\right]. \label{eq: first-order estimates, zeta}
\end{eq}
Assume that $\bm{x}_0\in\mathcal{M}$ satisfies
\begin{eq} \label{eq: first-order estimates, max_M Phi=Phi(x_0)}
\max_{\mathcal{M}} \Phi=\Phi(\bm{x}_0).
\end{eq}
Recall \eqref{eq: connection}, \eqref{eq: properties of connection} and choose a normal coordinate system $(\mathcal{U}\comma\bm{\psi}_{\mathcal{U}}\mbox{; }x^j)$ containing $\bm{x}_0$ so that
\begin{ga}
g_{jk}(\bm{x}_0)=\updelta_{jk}\comma\quad \Gamma_{jk}^l(\bm{x}_0)=0\comma\quad \mathrm{D}_lg_{jk}(\bm{x}_0)=0\comma \label{eq: first-order estimates, g_{jk}(x_0), Gamma_{jk}^l(x_0), D_lg_{jk}(x_0)} \\
\bigl(B_{jk}(\bm{x}_0)\bigr)=\Bigl(A_{jk}\bigl(\dif u(\bm{x}_0)\comma u(\bm{x}_0)\bigr)+\nabla_{jk}u(\bm{x}_0)\Bigr)=\mathbf{D}\triangleq\diag\{\mu_1\comma\mu_2\comma\cdots\comma\mu_n\}\comma \label{eq: first-order estimates, B_{jk}(x_0)}.
\end{ga}
By \eqref{eq: g^{jk}, first-order} and \eqref{eq: g^{jk}, second-order}, we also have
\begin{eq} \label{eq: first-order estimates, g^{jk}(x_0), D_lg^{jk}(x_0), D_{lm}g^{jk}(x_0)}
g^{jk}(\bm{x}_0)=\updelta_{jk}\comma\quad \mathrm{D}_lg^{jk}(\bm{x}_0)=0\comma\quad \mathrm{D}_{lm}g^{jk}(\bm{x}_0)=-\mathrm{D}_{lm}g_{kj}(\bm{x}_0).
\end{eq}
Thus, there hold
\begin{ga}
\bm\mu\triangleq(\mu_1\comma\mu_2\comma\cdots\comma\mu_n)^{\mathrm T}\in\Gamma_p\comma \\
\sigma_p^{\frac1p}(\bm\mu)=\varphi\bigl(\dif u(\bm{x}_0)\comma u(\bm{x}_0)\bigr)\comma \label{eq: first-order estimates, sigma_p^{frac1p}(mu)} \\
\left|\dif u(\bm{x}_0)\right|_{\bm g}^2=\sum_{j=1}^n\left|\mathrm{D}_ju(\bm{x}_0)\right|^2\comma \label{eq: first-order estimates, |du(x_0)|^2}
\end{ga}
and \eqref{eq: second-order estimates, D_{lm}g_{jk}(x_0)}, \eqref{eq: second-order estimates, R_{jklm}(x_0)}. For any $j\in\{1\comma2\comma\cdots\comma n\}$, by \eqref{eq: first-order estimates, max_M Phi=Phi(x_0)} we have
\begin{eq} \label{eq: first-order estimates, D_j Phi(x_0), D_{jj} Phi(x_0)}
\mathrm{D}_j\Phi(\bm{x}_0)=0\comma\quad\mathrm{D}_{jj}\Phi(\bm{x}_0)\leqslant 0.
\end{eq}
Recalling \eqref{eq: F_{u, x}^{jk}}, there hold
\begin{eq} \label{eq: first-order estimates, F^{jk}}
F^{jk}\triangleq F_{u\comma\bm{x}_0}^{jk}=\left.\frac{\partial}{\partial b_{jk}}\sigma_p^{\frac 1p}\bigl(\bm\lambda(\mathbf{I}_n\comma\mathbf{B})\bigr)\right|_{\mathbf{B}=\mathbf{D}}\comma
\end{eq}
and \eqref{eq: second-order estimates, F^{jk}, new}, \eqref{eq: second-order estimates, F}.

For the rest of this proof, we always calculate at $\bm{x}_0$. By \eqref{eq: du and nabla^2u}, \eqref{eq: first-order estimates, g_{jk}(x_0), Gamma_{jk}^l(x_0), D_lg_{jk}(x_0)} and \eqref{eq: first-order estimates, B_{jk}(x_0)}, we have
\begin{eq} \label{eq: first-order estimates, D_{jk}u}
\mathrm{D}_{jk}u=\nabla_{jk}u=\mu_j\updelta_{jk}-A_{jk}(\dif u\comma u).
\end{eq}
Let $(\tilde{x}^j\comma\tilde\alpha_j)$ be the canonical coordinate on $\mathrm{T}^*\mathcal{U}$ with respect to $(\mathcal{U}\comma\bm{\psi}_{\mathcal{U}}\mbox{; }x^j)$ (cf. \eqref{eq: tilde x^j, tilde alpha_j}). Recalling \eqref{eq: tilde{nabla}_{tilde{x}^j}varphi} and \eqref{eq: tilde{nabla}_{tilde{x}^j}A_{x^rx^s}}, each of the following is the component of some tensor (at $\bm{x}_0$):
\begin{eq} \label{eq: first-order estimates, tensor}
\left.\frac{\partial\varphi(\cdot\comma u)}{\partial\tilde{x}^j}\right|_{\dif u}=\tilde{\nabla}_{\tilde{x}^j}\varphi(\bm\alpha\comma t)\comma\quad \left.\frac{\partial A_{rs}(\cdot\comma u)}{\partial\tilde{x}^j}\right|_{\dif u}=\tilde{\nabla}_{\tilde{x}^j}A_{rs}(\bm\alpha\comma t)\comma\quad \left.\frac{\partial\varphi(\cdot\comma u)}{\partial\tilde{\alpha}_j}\right|_{\dif u}\comma\quad \left.\frac{\partial A_{rs}(\cdot\comma u)}{\partial\tilde{\alpha}_j}\right|_{\dif u}.
\end{eq}
It follows from \eqref{eq: second-order estimates, equation, first-order, new} and \eqref{eq: second-order estimates, D_jB_{kk}} that
\begin{eq} \label{eq: first-order estimates, F^{kk}D_{kkj}u} \begin{aligned}
F^{kk}\mathrm{D}_{kkj}u&=F^{kk}\mathrm{D}_j\Gamma_{kk}^l\mathrm{D}_lu-\left(F^{kk}\left.\frac{\partial A_{kk}(\cdot\comma u)}{\partial\tilde{\alpha}_l}\right|_{\dif u}-\left.\frac{\partial\varphi(\cdot\comma u)}{\partial\tilde{\alpha}_l}\right|_{\dif u}\right)\mathrm{D}_{lj}u \\
&\quad-F^{kk}\left(\left.\frac{\partial A_{kk}(\cdot\comma u)}{\partial\tilde{x}^j}\right|_{\dif u}+\left.\frac{\partial A_{kk}(\dif u\comma\cdot)}{\partial t}\right|_u\mathrm{D}_ju\right)+\left(\left.\frac{\partial\varphi(\cdot\comma u)}{\partial\tilde{x}^j}\right|_{\dif u}+\left.\frac{\partial\varphi(\dif u\comma\cdot)}{\partial t}\right|_u\mathrm{D}_ju\right).
\end{aligned} \end{eq}
\eqref{eq: first-order estimates, Phi} and \eqref{eq: first-order estimates, D_j Phi(x_0), D_{jj} Phi(x_0)} imply
\begin{ga}
0=\mathrm{D}_r\Phi=\frac{\mathrm{D}_r\left|\dif u\right|_{\bm g}^2}{1+\left|\dif u\right|_{\bm g}^2}+\zeta'(u)\mathrm{D}_ru\comma \label{eq: first-order estimates, 0=D_r Phi} \\
0\geqslant F^{rr}\mathrm{D}_{rr}\Phi=\frac{F^{rr}\mathrm{D}_{rr}\left|\dif u\right|_{\bm g}^2}{1+\left|\dif u\right|_{\bm g}^2}-\frac{F^{rr}\bigl|\mathrm{D}_r\left|\dif u\right|_{\bm g}^2\bigr|^2}{(1+\left|\dif u\right|_{\bm g}^2)^2}+\zeta'(u)F^{rr}\mathrm{D}_{rr}u+\zeta''(u)F^{rr}\left|\mathrm{D}_ru\right|^2. \label{eq: first-order estimates, 0 geqslant F^{rr}D_{rr} Phi}
\end{ga}
Recalling \eqref{eq: first-order estimates, omega_j}, for some positive constant $C$ large enough, there exists a positive constant $C_1\in[1\comma{+}\infty)$ (depending on $C$, but independent of $t$) so that if $t\geqslant C_1$, then
\begin{eq} \label{eq: first-order estimates, omega_j(t) leqslant}
\max\left\{\omega_1(t)\comma\omega_2(t)\comma\cdots\comma\omega_5(t)\right\}\leqslant\frac{t}{C}.
\end{eq}
We may as well assume $\left|\dif u\right|_{\bm g}>0$. In view of \eqref{eq: first-order estimates, |du(x_0)|^2}, there exists such $n_0\in\{1\comma2\comma\cdots\comma n\}$ that
\begin{eq} \label{eq: first-order estimates, |D_{n_0}u|}
\left|\mathrm{D}_{n_0}u\right|=\max_{1\leqslant j\leqslant n} \left|\mathrm{D}_ju\right|\geqslant\frac{\left|\dif u\right|_{\bm g}}{\sqrt{n}}>0.
\end{eq}
By \eqref{eq: first-order estimates, 0=D_r Phi}, \eqref{eq: second-order estimates, D_r|du|_g^2}, \eqref{eq: first-order estimates, |D_{n_0}u|} and \eqref{eq: first-order estimates, D_{jk}u} we have
\begin{eq} \label{eq: first-order estimates, 0=D_{n_0} Phi} \begin{aligned}
\frac12\zeta'(u)(1+\left|\dif u\right|_{\bm g}^2)&=-\frac{1}{\mathrm{D}_{n_0}u}\sum_{j=1}^n \mathrm{D}_{jn_0}u\mathrm{D}_ju=-\mu_{n_0}+A_{n_0n_0}(\dif u\comma u)+\frac{1}{\mathrm{D}_{n_0}u}\sum_{\substack{1\leqslant j\leqslant n \\ j\ne n_0}} A_{jn_0}(\dif u\comma u)\mathrm{D}_ju \\
&\leqslant-\mu_{n_0}+\phi_3(\bm{x}_0\comma u)\omega_1(\left|\dif u\right|_{\bm g}^2)+\frac{1}{\left|\mathrm{D}_{n_0}u\right|}\phi_4(\bm{x}_0\comma u)\left(\sum_{\substack{1\leqslant j\leqslant n \\ j\ne n_0}} \left|\mathrm{D}_ju\right|^2\right)^{\frac12}\omega_2(\left|\dif u\right|_{\bm g}^2) \\
&\leqslant-\mu_{n_0}+\phi_3(\bm{x}_0\comma u)\omega_1(\left|\dif u\right|_{\bm g}^2)+\sqrt{n-1}\phi_4(\bm{x}_0\comma u)\omega_2(\left|\dif u\right|_{\bm g}^2)\comma
\end{aligned} \end{eq}
where the first ``$\leqslant$'' is due to \eqref{eq: first-order estimates, structural conditions, A_{vv}(alpha, t) leqslant phi_3(x, t)g(v, v)omega_1(|alpha|_g^2)}, \eqref{eq: first-order estimates, structural conditions, A_{vw}(alpha, t) leqslant phi_4(x, t)(g(v, v)g(w, w))^{frac 12}omega_2(|alpha|_g^2)}, \eqref{eq: first-order estimates, g_{jk}(x_0), Gamma_{jk}^l(x_0), D_lg_{jk}(x_0)} and the fact that
\begin{eq}
\sum_{\substack{1\leqslant j\leqslant n \\ j\ne n_0}} A_{jn_0}(\dif u\comma u)\mathrm{D}_ju=\bigl(\bm{A}(\dif u\comma u)\bigr)\left(\sum_{\substack{1\leqslant j\leqslant n \\ j\ne n_0}} \mathrm{D}_ju\frac{\partial}{\partial x^j}\comma\frac{\partial}{\partial x^{n_0}}\right).
\end{eq}
Recalling $\zeta'>0$, \eqref{eq: first-order estimates, 0=D_{n_0} Phi} and \eqref{eq: first-order estimates, omega_j(t) leqslant} imply
\begin{eq} \label{eq: first-order estimates, mu_{n_0} leqslant}
\mu_{n_0}\leqslant-\frac12\zeta'(u)(1+\left|\dif u\right|_{\bm g}^2)+\frac{C_2}{C}\left|\dif u\right|_{\bm g}^2\leqslant-\frac14\zeta'(u)\left|\dif u\right|_{\bm g}^2<0
\end{eq}
as long as
\begin{ga}
\left|\dif u\right|_{\bm g}^2\geqslant C_1\comma \label{eq: first-order estimates, |du|_g^2 geqslant C_1} \\
C\geqslant\frac{4C_2}{\min\limits_{\mathcal{M}}\zeta'(u)}. \label{eq: first-order estimates, C geqslant frac{4C_2}{min_M zeta'(u)}}
\end{ga}
It follows from \eqref{eq: second-order estimates, F^{jk}, new}, \eqref{eq: sigma_p(mu)=mu_j sigma_{p-1}(mu|j)+sigma_p(mu|j)}, \propref{prop: increasing with respect to each variable}, \eqref{eq: first-order estimates, mu_{n_0} leqslant} and \eqref{eq: second-order estimates, F} that
\begin{eq} \label{eq: first-order estimates, F^{n_0n_0}>frac{1}{n-p+1}F}
F^{n_0n_0}=\frac1p\sigma_p^{\frac1p-1}(\bm\mu)\bigl(\sigma_{p-1}(\bm\mu)-\mu_{n_0}\sigma_{p-2}(\bm\mu|n_0)\bigr)>\frac{1}{n-p+1}\mathscr{F}.
\end{eq}
On the other hand, by \eqref{eq: sigma_p(mu)=mu_j sigma_{p-1}(mu|j)+sigma_p(mu|j)}, \propref{prop: increasing with respect to each variable}, \eqref{eq: sigma_p(mu|j_1j_2...j_m)=sigma_p(mu_{k_1},mu_{k_2},...,mu_{k_{n-m}})} and \propref{prop: Maclaurin} we have
\begin{eq}
0<\sigma_p(\bm\mu)=\mu_{n_0}\sigma_{p-1}(\bm\mu|n_0)+\sigma_p(\bm\mu|n_0)\leqslant\mu_{n_0}\sigma_{p-1}(\bm\mu|n_0)+\mathrm{C}_{n-1}^{p}\left(\frac{\sigma_{p-1}(\bm\mu|n_0)}{\mathrm{C}_{n-1}^{p-1}}\right)^{\frac{p}{p-1}}\comma
\end{eq}
and therefore
\begin{eq} \label{eq: first-order estimates, F^{n_0n_0}>frac{(zeta'(u))^{p-1}|du|_g^{2(p-1)}}{C_3varphi^{p-1}(du, u)}}
F^{n_0n_0}>\frac1p\sigma_p^{\frac1p-1}(\bm\mu)(\mathrm{C}_{n-1}^{p-1})^p\left(\frac{-\mu_{n_0}}{\mathrm{C}_{n-1}^{p}}\right)^{p-1}\geqslant\frac{\bigl(\zeta'(u)\bigr)^{p-1}\left|\dif u\right|_{\bm g}^{2(p-1)}}{C_3\varphi^{p-1}(\dif u\comma u)}
\end{eq}
in view of \eqref{eq: first-order estimates, mu_{n_0} leqslant} and \eqref{eq: first-order estimates, sigma_p^{frac1p}(mu)}. Combining \eqref{eq: first-order estimates, |D_{n_0}u|}, \eqref{eq: first-order estimates, F^{n_0n_0}>frac{1}{n-p+1}F} and \eqref{eq: first-order estimates, F^{n_0n_0}>frac{(zeta'(u))^{p-1}|du|_g^{2(p-1)}}{C_3varphi^{p-1}(du, u)}}, we obtain
\begin{eq} \label{eq: first-order estimates, F^{rr}|D_ru|^2 geqslant}
F^{rr}\left|\mathrm{D}_ru\right|^2\geqslant F^{n_0n_0}\left|\mathrm{D}_{n_0}u\right|^2\geqslant\frac{1}{C_4}\max\left\{\mathscr{F}\comma\frac{\bigl(\zeta'(u)\bigr)^{p-1}\left|\dif u\right|_{\bm g}^{2(p-1)}}{\varphi^{p-1}(\dif u\comma u)}\right\}\left|\dif u\right|_{\bm g}^2.
\end{eq}

By \eqref{eq: first-order estimates, F^{kk}D_{kkj}u}, \eqref{eq: second-order estimates, D_r|du|_g^2}, \eqref{eq: first-order estimates, structural conditions, g^{jk}(x)tilde{nabla}_{tilde{x}^j}A_{rs}(alpha, t)v^rv^s alpha_k+D_tA_{vv}(alpha, t)|alpha|_g^2 leqslant phi_5(x, t)g(v, v)omega_3(|alpha|_g^4)}, \eqref{eq: first-order estimates, structural conditions, g^{jk}(x)tilde{nabla}_{tilde{x}^j}varphi(alpha, t)alpha_k+D_t varphi(alpha, t)|alpha|_g^2 geqslant -phi_7(x, t)omega_4(|alpha|_g^4)max{1, frac{|alpha|_g^{2(p-1)}}{varphi^{p-1}(alpha, t)}}} and \eqref{eq: first-order estimates, g^{jk}(x_0), D_lg^{jk}(x_0), D_{lm}g^{jk}(x_0)} we have
\begin{eq} \label{eq: first-order estimates, 2F^{rr} sum_{j=1}^n (D_{jrr}u-D_j Gamma_{rr}^lD_lu)D_ju} \begin{aligned}
&\mathrel{\hphantom{=}}2F^{rr}\sum_{j=1}^n (\mathrm{D}_{jrr}u-\mathrm{D}_j\Gamma_{rr}^l\mathrm{D}_lu)\mathrm{D}_ju+\left(F^{kk}\left.\frac{\partial A_{kk}(\cdot\comma u)}{\partial\tilde{\alpha}_l}\right|_{\dif u}-\left.\frac{\partial\varphi(\cdot\comma u)}{\partial\tilde{\alpha}_l}\right|_{\dif u}\right)\mathrm{D}_l\left|\dif u\right|_{\bm g}^2 \\
&\geqslant-2\phi_5(\bm{x}_0\comma u)\omega_3(\left|\dif u\right|_{\bm g}^4)\mathscr{F}-2\phi_7(\bm{x}_0\comma u)\omega_4(\left|\dif u\right|_{\bm g}^4)\max\left\{1\comma\frac{\left|\dif u\right|_{\bm g}^{2(p-1)}}{\varphi^{p-1}(\dif u\comma u)}\right\} \\
&\geqslant-\frac{C_5}{C}(\mathscr{F}+1)\left|\dif u\right|_{\bm g}^4-\frac{C_5}{C\bigl(\zeta'(u)\bigr)^{p-1}}\frac{\bigl(\zeta'(u)\bigr)^{p-1}\left|\dif u\right|_{\bm g}^{2(p-1)}}{\varphi^{p-1}(\dif u\comma u)}\left|\dif u\right|_{\bm g}^4\comma
\end{aligned} \end{eq}
where the last ``$\geqslant$'' is due to \eqref{eq: first-order estimates, omega_j(t) leqslant} and \eqref{eq: first-order estimates, |du|_g^2 geqslant C_1}. It follows from \eqref{eq: second-order estimates, D_{rr}|du|_g^2}, \eqref{eq: first-order estimates, 2F^{rr} sum_{j=1}^n (D_{jrr}u-D_j Gamma_{rr}^lD_lu)D_ju}, \eqref{eq: first-order estimates, F^{rr}|D_ru|^2 geqslant} and \eqref{eq: second-order estimates, -F^{rr}sum_{1 leqslant j, k leqslant n} D_{rr}g_{kj}D_juD_ku+2F^{rr}sum_{j=1}^n D_j Gamma_{rr}^lD_luD_ju} that
\begin{eq} \label{eq: first-order estimates, F^{rr}D_{rr}|du|_g^2} \begin{aligned}
F^{rr}\mathrm{D}_{rr}\left|\dif u\right|_{\bm g}^2&\geqslant 2F^{rr}\sum_{j=1}^n \left|\mathrm{D}_{jr}u\right|^2-\left(F^{kk}\left.\frac{\partial A_{kk}(\cdot\comma u)}{\partial\tilde{\alpha}_l}\right|_{\dif u}-\left.\frac{\partial\varphi(\cdot\comma u)}{\partial\tilde{\alpha}_l}\right|_{\dif u}\right)\mathrm{D}_l\left|\dif u\right|_{\bm g}^2 \\
&\quad-\frac{C_5}{C}(\mathscr{F}+1)\left|\dif u\right|_{\bm g}^4-\frac{C_5C_4}{C\bigl(\zeta'(u)\bigr)^{p-1}}F^{n_0n_0}\left|\mathrm{D}_{n_0}u\right|^2\left|\dif u\right|_{\bm g}^2-C_6\mathscr{F}\left|\dif u\right|_{\bm g}^2.
\end{aligned} \end{eq}
\eqref{eq: first-order estimates, 0=D_r Phi} and \eqref{eq: first-order estimates, D_{jk}u} imply
\begin{eq} \begin{aligned}
&\mathrel{\hphantom{=}}-\left(F^{kk}\left.\frac{\partial A_{kk}(\cdot\comma u)}{\partial\tilde{\alpha}_l}\right|_{\dif u}-\left.\frac{\partial\varphi(\cdot\comma u)}{\partial\tilde{\alpha}_l}\right|_{\dif u}\right)\frac{\mathrm{D}_l\left|\dif u\right|_{\bm g}^2}{1+\left|\dif u\right|_{\bm g}^2}+\zeta'(u)F^{rr}\mathrm{D}_{rr}u \\
&=\left(F^{kk}\left.\frac{\partial A_{kk}(\cdot\comma u)}{\partial\tilde{\alpha}_l}\right|_{\dif u}-\left.\frac{\partial\varphi(\cdot\comma u)}{\partial\tilde{\alpha}_l}\right|_{\dif u}\right)\zeta'(u)\mathrm{D}_lu+\zeta'(u)F^{kk}\bigl(\mu_k-A_{kk}(\dif u\comma u)\bigr) \\
&=\zeta'(u)F^{kk}\Bigl(\bigl(\mathrm{D}_{\dif u}[A_{kk}(\cdot\comma u)]\bigr)(\dif u)-A_{kk}(\dif u\comma u)\Bigr)+\zeta'(u)\Bigl(\varphi(\dif u\comma u)-\bigl(\mathrm{D}_{\dif u}[\varphi(\cdot\comma u)]\bigr)(\dif u)\Bigr)\comma
\end{aligned} \end{eq}
where in the last ``$=$'' we have used \eqref{eq: Frechet}, \eqref{eq: second-order estimates, F^{kk}mu_k} and \eqref{eq: first-order estimates, sigma_p^{frac1p}(mu)}. Note that for any $\psi\in\mathrm{C}^2(\mathrm{T}_{\bm{x}_0}^*\mathcal{M})$ and $\bm\alpha$, $\bm\beta\in\mathrm{T}_{\bm{x}_0}^*\mathcal{M}$, there holds
\begin{eq}
\psi(\bm\alpha+\bm\beta)=\psi(\bm\alpha)+\bigl(\mathrm{D}_{\bm\alpha}[\psi]\bigr)(\bm\beta)+\int_0^1 (1-\theta)\bigl(\mathrm{D}_{\bm\alpha+\theta\bm\beta}^2[\psi]\bigr)(\bm\beta\comma\bm\beta)\dif\theta\comma
\end{eq}
recall \eqref{eq: second-order Frechet}. Thus, we have
\begin{eq} \label{eq: first-order estimates, -(F^{kk}frac{partial A_{kk}(cdot, u)}{partial tilde{alpha}_l}|_{du}-frac{partial varphi(cdot, u)}{partial tilde{alpha}_l}|_{du})frac{D_l|du|_g^2}{1+|du|_g^2}+zeta'(u)F^{rr}D_{rr}u} \begin{aligned}
&\mathrel{\hphantom{=}}-\left(F^{kk}\left.\frac{\partial A_{kk}(\cdot\comma u)}{\partial\tilde{\alpha}_l}\right|_{\dif u}-\left.\frac{\partial\varphi(\cdot\comma u)}{\partial\tilde{\alpha}_l}\right|_{\dif u}\right)\frac{\mathrm{D}_l\left|\dif u\right|_{\bm g}^2}{1+\left|\dif u\right|_{\bm g}^2}+\zeta'(u)F^{rr}\mathrm{D}_{rr}u \\
&=\zeta'(u)F^{kk}\left(-A_{kk}(\bm{0}\comma u)+\int_0^1 \theta\bigl(\mathrm{D}_{\theta\!\dif u}^2[A_{kk}(\cdot\comma u)]\bigr)(\dif u\comma\dif u)\dif\theta\right) \\
&\quad+\zeta'(u)\left(\varphi(\bm{0}\comma u)-\int_0^1 \theta\bigl(\mathrm{D}_{\theta\!\dif u}^2[\varphi(\cdot\comma u)]\bigr)(\dif u\comma\dif u)\dif\theta\right) \\
&\geqslant-\zeta'(u)\left(C_7\mathscr{F}+\phi_6(\bm{x}_0\comma u)(\left|\dif u\right|_{\bm g}^2+1)\mathscr{F}+\phi_8(\bm{x}_0\comma u)(\left|\dif u\right|_{\bm g}^2+1)\max\left\{1\comma\frac{\omega_5(\left|\dif u\right|_{\bm g}^{2(p-1)})}{\varphi^{p-1}(\dif u\comma u)}\right\}\right) \\
&\geqslant-\zeta'(u)\left(C_7\mathscr{F}+C_8(\mathscr{F}+1)\left|\dif u\right|_{\bm g}^2+\frac{C_8C_4}{C\bigl(\zeta'(u)\bigr)^{p-1}}F^{n_0n_0}\left|\mathrm{D}_{n_0}u\right|^2\right)
\end{aligned} \end{eq}
in view of \eqref{eq: first-order estimates, structural conditions, D^2A_{vv}(alpha, t) geqslant -phi_6(x, t)g(v, v)(|alpha|_g^2+1)}, \eqref{eq: first-order estimates, structural conditions, D^2 varphi(alpha, t) leqslant phi_8(x, t)(|alpha|_g^2+1)max{1, frac{omega_5(|alpha|_g^{2(p-1)})}{varphi^{p-1}(alpha, t)}}}, \eqref{eq: first-order estimates, omega_j(t) leqslant}, \eqref{eq: first-order estimates, |du|_g^2 geqslant C_1} and \eqref{eq: first-order estimates, F^{rr}|D_ru|^2 geqslant}. Combining \eqref{eq: first-order estimates, F^{rr}D_{rr}|du|_g^2} and \eqref{eq: first-order estimates, -(F^{kk}frac{partial A_{kk}(cdot, u)}{partial tilde{alpha}_l}|_{du}-frac{partial varphi(cdot, u)}{partial tilde{alpha}_l}|_{du})frac{D_l|du|_g^2}{1+|du|_g^2}+zeta'(u)F^{rr}D_{rr}u}, we find that
\begin{eq} \label{eq: first-order estimates, frac{F^{rr}D_{rr}|du|_g^2}{1+|du|_g^2}+zeta'(u)F^{rr}D_{rr}u} \begin{aligned}
\frac{F^{rr}\mathrm{D}_{rr}\left|\dif u\right|_{\bm g}^2}{1+\left|\dif u\right|_{\bm g}^2}+\zeta'(u)F^{rr}\mathrm{D}_{rr}u&\geqslant\frac{2}{1+\left|\dif u\right|_{\bm g}^2}F^{rr}\sum_{j=1}^n \left|\mathrm{D}_{jr}u\right|^2-\frac12\zeta''(u)F^{n_0n_0}\left|\mathrm{D}_{n_0}u\right|^2-\frac{2C_5}{C}\mathscr{F}\left|\dif u\right|_{\bm g}^2 \\
&\quad-\zeta'(u)C_9\mathscr{F}\left|\dif u\right|_{\bm g}^2-C_6\mathscr{F}
\end{aligned} \end{eq}
as long as
\begin{ga}
\zeta''(t)>0\comma\forall t\in\left[\min_{\mathcal{M}}u\comma\max_{\mathcal{M}}u\right]\comma \label{eq: first-order estimates, zeta''(t)>0} \\
C\geqslant\max\left\{\frac{4C_5C_4}{\min\limits_{\mathcal{M}}\bigl(\zeta'(u)\bigr)^{p-1}\min\limits_{\mathcal{M}}\zeta''(u)}\comma\frac{4C_8C_4}{\min\limits_{\mathcal{M}}\bigl(\zeta'(u)\bigr)^{p-2}\min\limits_{\mathcal{M}}\zeta''(u)}\right\}. \label{eq: first-order estimates, C geqslant frac{4C_5C_4}{min_M(zeta'(u))^{p-1}min_M zeta''(u)}, frac{4C_8C_4}{min_M(zeta'(u))^{p-2}min_M zeta''(u)}}
\end{ga}
By \eqref{eq: second-order estimates, D_r|du|_g^2} and \eqref{eq: first-order estimates, |du(x_0)|^2} we have
\begin{eq} \label{eq: first-order estimates, frac12F^{rr}|D_r|du|_g^2|^2}
\frac12F^{rr}\bigl|\mathrm{D}_r\left|\dif u\right|_{\bm g}^2\bigr|^2\leqslant 2F^{rr}\sum_{j=1}^n \left|\mathrm{D}_{jr}u\right|^2\left|\dif u\right|_{\bm g}^2.
\end{eq}
It follows from \eqref{eq: first-order estimates, 0 geqslant F^{rr}D_{rr} Phi}, \eqref{eq: first-order estimates, frac{F^{rr}D_{rr}|du|_g^2}{1+|du|_g^2}+zeta'(u)F^{rr}D_{rr}u}, \eqref{eq: first-order estimates, frac12F^{rr}|D_r|du|_g^2|^2}, \eqref{eq: first-order estimates, 0=D_r Phi} and \eqref{eq: first-order estimates, F^{rr}|D_ru|^2 geqslant} that
\begin{eq} \label{eq: first-order estimates, key estimate} \begin{aligned}
0&\geqslant\frac12\Bigl(\zeta''(u)-\bigl(\zeta'(u)\bigr)^2\Bigr)F^{rr}\left|\mathrm{D}_ru\right|^2-\zeta'(u)C_9\mathscr{F}\left|\dif u\right|_{\bm g}^2-\frac{2C_5}{C}\mathscr{F}\left|\dif u\right|_{\bm g}^2-C_6\mathscr{F} \\
&\geqslant\frac{1}{2C_4}\Bigl(\zeta''(u)-\bigl(\zeta'(u)\bigr)^2-2C_4C_9\zeta'(u)\Bigr)\mathscr{F}\left|\dif u\right|_{\bm g}^2-\frac{2C_5}{C}\mathscr{F}\left|\dif u\right|_{\bm g}^2-C_6\mathscr{F} \\
&\geqslant\frac{1}{4C_4}\zeta''(u)\mathscr{F}\left|\dif u\right|_{\bm g}^2-C_6\mathscr{F}
\end{aligned} \end{eq}
as long as
\begin{ga}
\frac13\zeta''(t)-\bigl(\zeta'(t)\bigr)^2-2C_4C_9\zeta'(t)\geqslant 0\comma\forall t\in\left[\min_{\mathcal{M}}u\comma\max_{\mathcal{M}}u\right]\comma \label{eq: first-order estimates, frac13zeta''(t)-(zeta'(t))^2-2C_4C_9zeta'(t) geqslant 0} \\
C\geqslant\frac{24C_4C_5}{\min\limits_{\mathcal{M}}\zeta''(u)}. \label{eq: first-order estimates, C geqslant frac{24C_4C_5}{min_M zeta''(u)}}
\end{ga}
\eqref{eq: first-order estimates, key estimate} implies
\begin{eq} \label{eq: first-order estimates, conclusion}
\left|\dif u\right|_{\bm g}^2\leqslant\frac{4C_4C_6}{\min\limits_{\mathcal{M}}\zeta''(u)}.
\end{eq}

At last, we let
\begin{eq}
\zeta(t)\triangleq-\frac13\log\left(1-\mathrm{e}^{-(6C_4C_9+1)\left(1+\max\limits_{\mathcal{M}}u-t\right)}\right).
\end{eq}
Then the requirements \eqref{eq: first-order estimates, zeta}, \eqref{eq: first-order estimates, zeta''(t)>0} and \eqref{eq: first-order estimates, frac13zeta''(t)-(zeta'(t))^2-2C_4C_9zeta'(t) geqslant 0} are both satisfied. Recall \eqref{eq: first-order estimates, omega_j(t) leqslant}, \eqref{eq: first-order estimates, |du|_g^2 geqslant C_1}, \eqref{eq: first-order estimates, C geqslant frac{4C_2}{min_M zeta'(u)}}, \eqref{eq: first-order estimates, C geqslant frac{4C_5C_4}{min_M(zeta'(u))^{p-1}min_M zeta''(u)}, frac{4C_8C_4}{min_M(zeta'(u))^{p-2}min_M zeta''(u)}}, \eqref{eq: first-order estimates, C geqslant frac{24C_4C_5}{min_M zeta''(u)}} and \eqref{eq: first-order estimates, conclusion}. By \eqref{eq: first-order estimates, Phi} and \eqref{eq: first-order estimates, max_M Phi=Phi(x_0)} we obtain \eqref{eq: general p-Hessian, first-order estimates, conclusion}.
\end{proof}

Next, we give the sharp first-order estimates for ``semi-convex solutions'' to \eqref{eq: general p-Hessian}. Here, by a ``semi-convex solution'' $u$ to \eqref{eq: general p-Hessian} we mean $u\in\mathrm{C}^2(\mathcal{M})$ satisfies \eqref{eq: general p-Hessian} and \eqref{eq: first-order estimates, semi-convex} below. One can compare \eqref{eq: first-order estimates, semi-convex} with \eqref{eq: second-order estimates, semi-convex}. Note that, if there exists such a non-negative continuous function $\phi_9$ on $\mathcal{M}\times\mathbb{R}$ that
\begin{eq}
\varphi(\bm\alpha\comma t)\leqslant\phi_9(\bm{x}\comma t)(\left|\bm\alpha\right|_{\bm g}^2+1)\comma\forall\bm\alpha\in\mathrm{T}_{\bm x}^*\mathcal{M}
\end{eq}
for any $(\bm{x}\comma t)\in\mathcal{M}\times\mathbb{R}$, then every general $(p+1)$-admissible solution to \eqref{eq: general p-Hessian} is a ``semi-convex solution'', cf. \eqref{eq: mu_1, lower bound}.

\begin{thm} \label{thm: semi-convex, first-order estimates}
Assuming \eqref{eq: first-order estimates, structural conditions, A_{vv}(alpha, t) leqslant phi_3(x, t)g(v, v)(|alpha|_g^2+1)} instead of \eqref{eq: first-order estimates, structural conditions, A_{vv}(alpha, t) leqslant phi_3(x, t)g(v, v)omega_1(|alpha|_g^2)}--\eqref{eq: first-order estimates, structural conditions, D^2 varphi(alpha, t) leqslant phi_8(x, t)(|alpha|_g^2+1)max{1, frac{omega_5(|alpha|_g^{2(p-1)})}{varphi^{p-1}(alpha, t)}}}, if there exists a non-negative constant $C$, depending only on $\mathcal{M}$, $\bm{g}$, $n$, $p$, $\bm{A}(\bm{\alpha}\comma t)$, $\varphi(\bm{\alpha}\comma t)$ and $\max\limits_{\mathcal M} \left|u\right|$, so that
\begin{eq} \label{eq: first-order estimates, semi-convex}
\bm\lambda\Bigl({\bm g}^{-1}\bigl(\bm{A}(\dif u\comma u)+\nabla^2u\bigr)\Bigr)(\bm{x})+C\bigl(\left|\dif u(\bm{x})\right|_{\bm g}^2+1\bigr)\bm{1}_n\in\overline{\Gamma}_n\comma\forall\bm{x}\in\mathcal{M}\comma
\end{eq}
then \eqref{eq: general p-Hessian, first-order estimates, conclusion} also holds.
\end{thm}

\begin{proof}
The proof is almost the same as the first half of that of \thmref{thm: general p-admissible solutions, first-order estimates}, so the details are omitted. We always calculate at $\bm{x}_0$. Recalling \eqref{eq: first-order estimates, 0=D_r Phi}, by \eqref{eq: second-order estimates, D_r|du|_g^2} and \eqref{eq: first-order estimates, D_{jk}u} we have
\begin{eq}
\frac12\zeta'(u)(1+\left|\dif u\right|_{\bm g}^2)\mathrm{D}_ru=-\sum_{j=1}^n \mathrm{D}_{jr}u\mathrm{D}_ju=\sum_{j=1}^n A_{jr}(\dif u\comma u)\mathrm{D}_ju-\mu_r\mathrm{D}_ru\comma
\end{eq}
and therefore
\begin{eq}
\sum_{r=1}^n \left(\frac12\zeta'(u)(1+\left|\dif u\right|_{\bm g}^2)+\mu_r\right)\left|\mathrm{D}_ru\right|^2=\sum_{1\leqslant j\comma r\leqslant n} A_{jr}(\dif u\comma u)\mathrm{D}_ju\mathrm{D}_ru\leqslant C_{10}\left|\dif u\right|_{\bm g}^2(1+\left|\dif u\right|_{\bm g}^2)
\end{eq}
in view of \eqref{eq: first-order estimates, structural conditions, A_{vv}(alpha, t) leqslant phi_3(x, t)g(v, v)(|alpha|_g^2+1)}, \eqref{eq: first-order estimates, g_{jk}(x_0), Gamma_{jk}^l(x_0), D_lg_{jk}(x_0)} and \eqref{eq: first-order estimates, |du(x_0)|^2}. Note that \eqref{eq: first-order estimates, semi-convex}, \eqref{eq: first-order estimates, g^{jk}(x_0), D_lg^{jk}(x_0), D_{lm}g^{jk}(x_0)} and \eqref{eq: first-order estimates, B_{jk}(x_0)} imply
\begin{eq}
\mu_r\geqslant-C(\left|\dif u\right|_{\bm g}^2+1)\comma\forall r\in\{1\comma2\comma\cdots\comma n\}.
\end{eq}
Now we let
\begin{eq}
\zeta(t)\triangleq2(C+C_{10}+1)t.
\end{eq}
It follows that $\left|\dif u\right|_{\bm g}=0$ (at $\bm{x}_0$). By \eqref{eq: first-order estimates, Phi} and \eqref{eq: first-order estimates, max_M Phi=Phi(x_0)} we obtain \eqref{eq: general p-Hessian, first-order estimates, conclusion}.
\end{proof}


\phantomsection
\addcontentsline{toc}{section}{Acknowledgements}
\section*{Acknowledgements}

The author would like to thank his advisor Professor Gang Tian for his helpful guidance.

\phantomsection
\addcontentsline{toc}{section}{Declarations}
\section*{Declarations}

\noindent{\bfseries Conflict of interest} The author has no relevant financial or non-financial interests to disclose.

\noindent{\bfseries Data availability} No datasets were generated or analysed during the current study.

\phantomsection
\addcontentsline{toc}{section}{References}
\setstretch{1}

\end{document}